\documentclass[12pt]{article}
 \usepackage{graphicx}
 \usepackage{latexsym}
 \usepackage{amsmath}
 \usepackage{verbatim}
 \usepackage{amssymb}
 \usepackage{mathrsfs}
 \usepackage{amsmath,amssymb}
 \DeclareMathAlphabet      {\mathbf}{OT1}{cmr}{bx}{n}
 \DeclareFontFamily{OT1}{pzc}{}
 \DeclareFontShape{OT1}{pzc}{m}{it}%
 {<-> s * [1.15] pzcmi7t}{}
 \DeclareMathAlphabet{\mathpzc}{OT1}{pzc}{m}{it}
 \usepackage{enumitem}
 \usepackage{ dsfont }
 \usepackage{ mathrsfs }
 \usepackage[utf8]{inputenc}
 \usepackage{amsbsy }
 \usepackage{graphics}

 \usepackage{amsthm}
 \newtheorem{thm}{Theorem}
 \newtheorem{theorem}{Theorem}[section]

 \newtheorem{lemma}[theorem]{Lemma}

 \newtheorem{definition}[theorem]{Definition}

 \newtheorem{prop}[theorem]{Proposition}

 \newtheorem{remark}{Remark}[section]

 \usepackage{romannum}
 \usepackage[toc,page,title,titletoc,header]{appendix}
 \usepackage{hyperref}
 \usepackage{amscd}
 \usepackage[all,cmtip]{xy}

 \def\ack{\section*{Acknowledgements}%
 	\addtocontents{toc}{\protect\vspace{6pt}}%
 	\addcontentsline{toc}{section}{Acknowledgements}%
 }

 \usepackage{abstract}

\begin{document}   \small
 	\pagenumbering{arabic}

	\pagenumbering{arabic}
	\title{On PFH and HF spectral invariants}
	\author{{ Guanheng Chen}}
	\date{}
	\maketitle
	\thispagestyle{empty}
	\begin{abstract}
For a closed symplectic surface, there are two types of spectral invariants: one defined by periodic Floer homology (PFH) and another by quantitative Heegaard Floer homology (QHF). The theme  of this paper is to investigate the relationship between these two invariants. We begin by defining intermediate invariants using the cylindrical formulation of QHF, which we call HF spectral invariants. These invariants are shown to be equivalent to the link spectral invariants in the author's  previous work. In the case of the sphere, we prove that the homogenized HF spectral invariants at the unit are equal to the homogenized PFH spectral invariants. This result is derived by constructing homomorphisms from quantitative Heegaard Floer homology to periodic Floer homology, which we refer to as open-closed morphisms.  In addition,  we show that  the homogenized PFH spectral invariants are quasimorphisms.	
	\end{abstract}
	\tableofcontents

\section{Introduction and Main results}
Let $\Sigma$ be a closed surface with genus $g$ and $\omega$ a volume form  of volume 1.  Given a \textbf{Hamiltonian function} $H: S^1_t \times \Sigma \to \mathbb{R}$, there is a unique vector field $X_{H_t}$, called  the  \textbf{Hamiltonian vector field},    satisfying  the relation $\omega(X_{H_t}, \cdot) = d_{\Sigma} H_t $.  Let $\varphi_H^t$ be the flow generated by $X_{H_t}$, i.e., $\partial_t \varphi_H^t =X_{H_t} \circ \varphi_H^t$ and $\varphi_H^0 =\operatorname{Id}$.  The time-1 flow is denoted by $\varphi_H: =\varphi_H^1$. A symplectomorphism $\varphi$ is called a \textbf{Hamiltonian symplectomorphism} if  $\varphi =\varphi_H$ for some $H$.  The collection of all Hamiltonian symplectomorphisms forms a group $Ham(\Sigma, \omega). $

Given  $\varphi \in \text{Ham}(\Sigma, \omega)$ and an integer $d$, M. Hutchings defines a Floer-type invariant $\widetilde{PFH}(\Sigma, \varphi, d)$ for $\varphi$, referred to as \textbf{periodic Floer homology} \cite{H1, H3}, abbreviated as PFH.  Roughly speaking,  PFH is the homology of a chain complex generated by certain sets of periodic orbits  with total degree $d$, and the differential is defined by counting holomorphic curves in $\mathbb{R} \times S^1 \times \Sigma$.
Fix a suitable \textbf{link} $\underline{L}$ (a disjoint union of simple closed curves, Definition \ref{definition6})   on $\Sigma$;   there exists a ``relative version of PFH" associated with $(\underline{L}, \varphi)$, called \textbf{quantitative Heegaard Floer homology}, abbreviated as QHF,   introduced by D. Cristofaro-Gardiner, V. Humili$\grave{e}$re, C. Mak, S. Seyfaddini, and I. Smith \cite{CHMSS}. QHF is defined as the Lagrangian Floer homology of  the  Lagrangian  pair $(\operatorname{Sym}^d \varphi(\underline{L}), \operatorname{Sym}^d \underline{L}) \subset \operatorname{Sym}^d \Sigma$, denoted by $HF(\operatorname{Sym}^d \varphi(\underline{L}), \operatorname{Sym}^d \underline{L})$.  Moreover,  these two kinds of Floer homologies are non-vanishing \cite{EH,CPPZ, CHMSS} and independent of the choice of the Hamiltonian symplectomorphism. Therefore, we have two abstract homologies $ \widetilde{PFH}(\Sigma,  d )$ and $HF(\operatorname{Sym}^d\underline{L})$ with   canonical  isomorphisms 
\begin{equation}  \label{eq62}
\begin{split}
&\mathfrak{j}_H: \widetilde{PFH}(\Sigma, \varphi_H, d) \to  \widetilde{PFH}(\Sigma,   d )\\
&\mathbf{j}_H: HF(\operatorname{Sym}^d \varphi_H(\underline{L}), \operatorname{Sym}^d \underline{L}) \to HF(\operatorname{Sym}^d\underline{L})
\end{split}
\end{equation}
  for  every $H \in C^{\infty}(S^1 \times \Sigma)$. PFH and QHF are respectively  $\Lambda$-modules and $R$-modules, where $\Lambda$ (\ref{eq85}) and $R$  (\ref{eq76}) are certain Novikov rings. 
Finally, each theory carries a degree-$(-2)$ map,   called the \textbf{$U$ map} (Section 2.5 of \cite{HT3} and Section \ref{U-map}).

There are two families  of numerical invariants defined from these two Floer theories
		\begin{equation*}
\begin{split}
&		c_{d}^{pfh} : C^{\infty}( S^1 \times \Sigma )  \times \widetilde{PFH}(\Sigma, d ) \to \{-\infty\}\cup \mathbb{R},\\
&c_{\underline{L}}^{link} : C^{\infty}( [0,1]\times \Sigma )  \times  HF(\operatorname{Sym}^d\underline{L}) \to \{-\infty\}\cup \mathbb{R}.
\end{split}
	\end{equation*}
The former  are  called \textbf{PFH spectral invariants}  \cite{CHS1, EH, CPZ}, and the latter  are  called \textbf{link spectral invariants}\cite{CHMSS}.

  Although these two types of spectral invariants stem from different Floer theories, they satisfy many parallel properties, such as Hofer-Lipschitz continuity, spectrality and the Weyl  law \cite{CHMSS, EH, CPZ}. So  it is natural to study whether they have any relation. The aim of this paper is to  try to answer   this question at least in some special cases.
To this end, our strategy is to construct morphisms between these two Floer homologies.
Because these two Floer theories are defined by counting holomorphic curves in manifolds of different dimensions, it is hard to define the morphisms directly.   To overcome this issue,   the author follows  R. Lipshitz's approach \cite{RL1}  to define an   intermediate Floer homology  by counting holomorphic curves in the  4-manifold $\mathbb{R} \times [0,1]\times \Sigma$, denoted by $HF(\Sigma,  \varphi_H(\underline{L}),  \underline{L})$ \cite{GHC2}.  Similarly,  we have an abstract homology  $HF(\Sigma,    \underline{L})$  and a  canonical isomorphism  
\begin{equation} \label{eq22}
\begin{split}
j_H: HF(\Sigma,  \varphi_H(\underline{L}),  \underline{L}) \to HF(\Sigma,    \underline{L})
\end{split}
\end{equation}
for every $H \in C^{\infty}([0, 1] \times \Sigma)$. 
Furthermore,   the author proves that there is a canonical  isomorphism (Theorem 1 of \cite{GHC2})
\begin{equation} \label{eq30}
\begin{split}
\Phi_H: HF(\Sigma,  \varphi_H(\underline{L}), \underline{L}) \to HF( \operatorname{Sym}^d \varphi_H(\underline{L}), \operatorname{Sym}^d  \underline{L}) 
\end{split}
\end{equation} 
  preserving  the action filtration. The isomorphism  $\Phi_H$ descends  to an isomorphism  $$\Phi: HF(\Sigma,   \underline{L}) \to HF(  \operatorname{Sym}^d  \underline{L}). $$ Therefore,  $HF(\Sigma,  \varphi_H(\underline{L}), \underline{L}) $ can be viewed as an alternative formulation of the quantitative Heegaard Floer homology. When the context is clear,  we also call it QHF.  It serves as a bridge between the QHF and PFH. 
  
  Using  $HF(\Sigma, \underline{L})$, we define a family of spectral invariants $c_{\underline{L}}$ as in \cite{CHMSS, LZ, FOOO}.   To distinguish them from  link spectral invariants  $c^{link}_{\underline{L}}$, we call $c_{\underline{L}}$ the \textbf{HF spectral invariants} instead.  The properties of HF spectral invariants are summarized in Theorem \ref{thm4}.  Via the filtered  isomorphism (\ref{eq30}), we know that  $c_{\underline{L} }$ is equivalent to  $c^{link}_{\underline{L}}$     (Corollary 1.9 of \cite{GHC2}) in the sense that 
   \begin{equation} \label{eq31}
 \frac{1}{d}c_{\underline{L}}(H,  a) =c^{link}_{\underline{L}}(H,   \Phi(a))
\end{equation}
for every $H\in C^{\infty}([0,1] \times \Sigma)$ and $a \in    HF(\Sigma,  \underline{L})$.

We now return to the morphisms  between PFH and QHF. In \cite{GHC2},   the author  establishes  a non-vanishing  $\Lambda$-module homomorphism from PFH to QHF
\begin{equation} \label{eq19}
 \mathcal{CO}(\underline{L}, H):   \widetilde{PFH}(\Sigma, \varphi_H, d)   \to HF(\Sigma,  \varphi_H(\underline{L}), \underline{L})
\end{equation}
 which is called the \textbf{closed-open morphism}.  The above homomorphism is invariant under changes in  the Hamiltonian function. In other words, given another Hamiltonian function $G$, we have $ j_H \circ \mathcal{CO}(\underline{L}, H) \circ \mathfrak{j}_H^{-1}= j_G \circ \mathcal{CO}(\underline{L}, G) \circ \mathfrak{j}_G^{-1}$,  and $\mathcal{CO}(\underline{L}, H) $ descends to a homomorphism  
\begin{equation*} 
\mathcal{CO}(\underline{L}):   \widetilde{PFH}(\Sigma, d)   \to HF(\Sigma, \underline{L}).
\end{equation*}
There are two special classes in PFH and  QHF called the \textbf{units}, denoted by $\mathfrak{e}  \in  \widetilde{PFH}(\Sigma, d)$ (Section 6 of \cite{GHC2}) and $e_{\underline{L}} \in HF(\Sigma,  \underline{L})$ (Definition \ref{definition4}) respectively.  The closed-open morphism preserves the units.  Moreover, it decreases the spectral invariants, i.e., 
\begin{equation} \label{eq21}
c_{\underline{L}}(H,    \mathcal{CO}(\underline{L})(\mathfrak{v}))  \le c^{pfh}_d(H, \mathfrak{v}). 
 \end{equation} 
 In particular, the above inequality is true at the units.

  Applying  the same methods in \cite{GHC2},    we  construct   in Theorem \ref{thm2} a  reverse $\Lambda$-module  homomorphism
 \begin{equation}  \label{eq75}
 \mathcal{OC}(\underline{L}, H):   HF(\Sigma,  \varphi_H(\underline{L}), \underline{L}) \to  \widetilde{PFH}(\Sigma, \varphi_H, d)
\end{equation}
  from QHF to PFH called the   \textbf{open-closed morphism}. The open-closed morphisms have similar properties:  invariance, commutativity with $U$ maps,  and monotonicity of the spectral invariants. Furthermore, we will show that they map the ``point class''  $\sigma_{\underline{L}} \in HF(\Sigma, \underline{L})$ to a nonzero class $\mathfrak{d}$ in PFH.  Thus, we have 
  \begin{equation}  \label{eq80}
c^{pfh}_d(H, \mathfrak{d}) \le c_{\underline{L}}(H, \sigma_{\underline{L}}  ). 
 \end{equation} 
There is a natural way to regard  $R$ as a field extension of  $\Lambda$. Then the open-closed morphisms can be ungraded to be $R$-modules homomorphisms: 
 \begin{equation}   
 \mathcal{OC}(\underline{L}, H; R):   HF(\Sigma,  \varphi_H(\underline{L}), \underline{L}) \to  \widetilde{PFH}(\Sigma, \varphi_H, d)\otimes_{\Lambda} R.
\end{equation}
 The homomorphism  $\mathcal{OC}(\underline{L}, H; R)$  has properties parallel to those of (\ref{eq75}). 
  
The closed-open/open-closed morphisms (\ref{eq19}) (\ref{eq75})      are analogues of the usual  closed-open/open-closed morphisms between  symplectic Floer homology  and  Lagrangian Floer homology   defined by P. Albers \cite{PA}.  See also the    construction by P. Biran and  O. Cornea \cite{BC}.  These maps  have also been constructed  by V. Colin, P. Ghiggini, and K. Honda \cite{VPK} for a  similar setting to ours.  We refer the reader to Pages 4--5 of  \cite{GHC2} to  for a discussion of the differences between our construction and the one in \cite{VPK}.

 In the case of the sphere,  the class $\mathfrak{d}$  and the unit  of PFH are related by the $U$ map.   
Combining (\ref{eq21}) and (\ref{eq80}),  we prove the equivalence between PFH spectral invariants and HF spectral invariants at the units.
\begin{thm}  \label{thm0}
 Suppose that $\underline{L}$ is a  $0$-admissible link   on  $\mathbb{S}^2$ with $d$ components. Then for any Hamiltonian function $H \in C^{\infty}(S^1 \times \mathbb{S}^2)$, we have
$$c^{pfh}_{d}(H,  \mathfrak{e})  -1   \le   c_{\underline{L}}(H, e_{\underline{L}}) \le c^{pfh}_{d}(H,   \mathfrak{e}). $$
Moreover,  for any $\varphi \in Ham(\mathbb{S}^2, \omega)$, we have
\begin{equation} \label{eq32}
 \mu_{\underline{L}}(\varphi, e_{\underline{L}}) = d \mu^{link}_{\underline{L}}(\varphi, \mathbf{1}_{\underline{L}})=  \mu_{d}^{pfh}(\varphi, \mathfrak{e}),
\end{equation}
where  $\mathbf{1}_{\underline{L}} \in HF(\operatorname{Sym}^d \underline{L})$ is the unit of QHF, $\mu_{\underline{L}}, \mu^{link}_{\underline{L}}, \mu_{d}^{pfh}$ are the homogenizations  of $c_{\underline{L}}$, $c^{link}_{\underline{L}}$, and $c^{pfh}_{d}$ respectively (see (\ref{eq60}) and (\ref{eq61})).   In particular, for any two $0$-admissible links $\underline{L}, \underline{L}'$ with the same number of components, then we have  $ \mu_{\underline{L}}(\varphi, e_{\underline{L}}) = \mu_{\underline{L}'} (\varphi, e_{\underline{L}'}) $.
\end{thm}

 \paragraph{Coefficients}
All the Novikov rings are  defined with coefficients in   $\mathbb{F} = \mathbb{Z}/2\mathbb{Z}$; unless otherwise stated.

\paragraph{Structure of the paper.}
Besides Theorem \ref{thm0}, our main results are stated in Section \ref{section2}. Theorem \ref{thm0} is a consequence of the results in  Section \ref{section2}.  The necessary preliminaries on PFH and QHF are reviewed in Section \ref{section0}.  In Section \ref{section1}, we  study the quantum product  and $U$ maps on $HF(\Sigma, \underline{L})$. Section \ref{section4} is devoted to studying the properties of the HF spectral invariants. After that,  the  open-closed  morphisms are constructed and computed respectively  in Section \ref{open-closed} and Section \ref{computing-open-closed}. Building on these preparations, in Section \ref{section6}, we  prove the main results mentioned in Section \ref{section2}. Finally,   Section \ref{section7} presents a heuristic discussion of the correspondence between PFH and
quantum homology, as well as the correspondence between the open-closed morphisms
and the quantum inclusion.

\ack
The author is grateful to the anonymous referees for their helpful suggestions, which greatly improved the exposition of this paper. In particular, he deeply appreciates one of the referees for suggesting the use of continuation morphisms (\ref{eq71}) to construct the open-closed/closed-open morphisms. This approach significantly simplifies the original formulation (Theorem \ref{thm6} and Theorem 3 of \cite{GHC2}) and enhances the readability of the results.

The initial version of the manuscript was just over 50 pages, and it has now expanded to over 100 pages. The author is deeply grateful to the referee for their constructive comments, which prompted a deeper reflection and the inclusion of further details in the proofs. Furthermore, a substantial amount of new content has been incorporated, with Theorem \ref{thm7}, Theorem \ref{thm8}, and Section \ref{section7} directly arising from addressing these comments. These additions have substantially enhanced the completeness of the paper.
  
 The author is supported by the National Natural Science Foundation of China Youth Fund Project, Grant number: 12201106.

\subsection{Preliminaries} \label{section0} 

\subsubsection{Twisted periodic Floer homology}
In this section, we review the definition of twisted periodic Floer homology and PFH spectral invariants.  For more details, please refer to \cite{H3, H4, CHS1, EH}.

Fix a  Hamiltonian symplectomorphism $\varphi \in Ham(\Sigma, \omega)$.
Define the \textbf{mapping torus} by
\begin{equation*}
Y_{\varphi} : = [0,1]_t \times \Sigma / (0, \varphi(x)) \sim (1, x).
\end{equation*}
There is a natural vector field $R : =\partial_t$ and  a closed 2-form $\omega_{\varphi}$ on $Y_{\varphi}$ induced from the above quotient.  The pair $(dt, \omega_{\varphi})$ forms a {stable Hamiltonian structure} and $R$ is the Reeb vector field.  Let $\xi : = \ker \pi_*$ denote the vertical bundle of $\pi: Y_{\varphi}  \to S^1. $ Suppose that $\varphi$ is generated by $H$. Then  we  have the following global trivialization of $Y_{\varphi}$:
\begin{equation}  \label{eq34}
\begin{split}
\Psi_H: & S^1_t \times \Sigma  \to   Y_{\varphi_H} \\
&(t,x) \to (t, (\varphi_H^t)^{-1}(x)).
\end{split}
\end{equation}
It is easy to show  that $\Psi_H^*(\omega_{\varphi}) = \omega + d(H_t dt)$ and $(\Psi_H)_*(\partial_t +X_H) =R$.

\paragraph{Periodic orbits.}
 A \textbf{periodic orbit} is a map $\gamma : \mathbb{R} /q\mathbb{Z} \to Y_{\varphi}$ satisfying the ODE $\partial_t \gamma(t) =R \circ \gamma(t)$.   Here $\gamma$ could be multiply covered. The number $q>0$ is called the \textbf{period} or \textbf{degree} of $\gamma$. Note that $q$ is equal to the intersection number $[\gamma] \cdot [\Sigma]$.

Let $\gamma$ be a periodic orbit with degree $q$.   $\gamma$ is called \textbf{nondegenerate} if the    linearized return map $d\varphi^q: T_{\gamma(0)} \Sigma \to T_{\gamma(q)} $  does not have 1 as an eigenvalue. A nondegenerate periodic orbit $\gamma$  is called  \textbf{hyperbolic}  if   $d\varphi^q \vert_{\gamma(0)}$ has real eigenvalues,  and  \textbf{elliptic} otherwise.   The symplectomorphism $\varphi$ is called \textbf{$d$-nondegenerate}  if  every closed orbit with degree at most $d$ is nondegenerate.

Let $\gamma$ be an elliptic periodic orbit with period $q$.  We can find a trivialization of $\xi$ such that the linearized flow  is a rotation $e^{i2\pi\theta_t}$, where $\{\theta_t\}_{t \in[0,q]}$ is a continuous function with $\theta_0=0$.  The number $\theta = \theta_t \vert_{t=q} \in \mathbb{R}/\mathbb{Z}$ is called the \textbf{rotation number} of $\gamma$ (see Page 37 of \cite{H4} for details).   The following definition will be used later to  define the PFH cobordism maps by holomorphic curves.
	\begin{definition} (see \cite{H5} Definition 4.1) \label{definition2}
		Fix  $d>0$.  Let  $\gamma$ be an embedded elliptic orbit with degree $0<q \le d$.
		\begin{itemize}
			\item
			$\gamma$ is called $d$-positive elliptic if the rotation number $\theta  $ is in $ (0, \frac{q}{d}) \mod 1$.
			\item
			$\gamma$ is called $d$-negative elliptic if the rotation number $\theta $ is in $ ( -\frac{q}{d},0) \mod 1$.
		\end{itemize}
	\end{definition}

An \textbf{orbit set} is a finite set of pairs $\alpha = \{(\alpha_i, m_i)\}$, where $\{\alpha_i\}$  are distinct embedded periodic orbits and $\{m_i\}$ are positive integers.   An orbit set is called \textbf{a PFH generator} if it satisfies a further condition: If $\alpha_i$ is  hyperbolic, then $m_i=1$.

\paragraph{ECH index and $J_0$ index.} \
Given orbit sets $\alpha_{\pm}=\{(\alpha_{\pm, i})\}$, let  $H_2(Y_{\varphi}, \alpha_+, \alpha_-)$ denote the set   of 2-chains $Z$ in $Y_{\varphi}$ with $\partial Z = \alpha_+ - \alpha_-$, modulo the boundary of 3-chains. We call the  element  $Z \in H_2(Y_{\varphi}, \alpha_+, \alpha_-)$  a \textbf{relative homology class}. This is an affine space of $H_2(Y_{\varphi},  \mathbb{Z}) \cong \mathbb{Z}[\Sigma] \oplus  (H_1(S^1) \otimes H_1(\Sigma))$.    

For a  relative homology class $Z \in H_2(Y_{\varphi}, \alpha_+, \alpha_-)$, Hutchings defines  a topological  index  called \textbf{ECH index}.  It is defined as follows. Fix a trivialization $\tau$ of $\xi \vert_{\alpha_{\pm}}$ along the orbits.  The ECH index  is  defined by
\begin{equation*}
I(Z) : =c_{\tau}(\xi \vert_{Z}) + Q_{\tau}(Z) + \sum_i \sum_{p=1}^{m_i} CZ_{\tau}(\alpha_{+, i}^{p}) -\sum_j \sum_{q=1}^{n_j} CZ_{\tau}(\alpha_{-, j}^{q}),
\end{equation*}
where $\alpha_{\pm, i}^{p}$ denote the $p$-fold covers of $\alpha_{\pm, i}$, $c_{\tau}(\xi \vert_{Z}) $ is the relative Chern number,  $Q_{\tau}(Z)$ is the relative self-intersection number and $CZ_{\tau}$ is the Conley-Zehnder index (see Section 2.2, 2.3, and 2.5 of \cite{H1} for details).

There is another topological index called  \textbf{$J_0$ index} \cite{H2} that measures  the topological  complexity of the curves. The $J_0$ index is  given  by the following formula:
\begin{equation*}
J_0(Z) : =-c_{\tau}(\xi \vert_{Z}) + Q_{\tau}(Z) + \sum_i \sum_{p=1}^{m_i-1} CZ_{\tau}(\alpha_{+, i}^{p}) -\sum_j \sum_{q=1}^{n_j-1} CZ_{\tau}(\alpha_{-, j}^{q}).
\end{equation*}

The $J_0$ index will be used when we define the open-closed morphisms. The role of this index actually comes from the definition of the link $\underline{L}$.  We will explain this point in Remark \ref{remark3} later.

\paragraph{Holomorphic curves and holomorphic currents.}
To  define the PFH differential, consider the symplectization  $$X:=\mathbb{R}_s \times Y_{\varphi},  \  \ \Omega:=\omega_{\varphi} + ds \wedge dt. $$
An almost complex structure on $X$ is called \textbf{admissible} if it preserves $\xi$, is $\mathbb{R}$-invariant, sends $\partial_s$ to $R$, and its restriction to $\xi$ is compatible with  $\omega_{\varphi}$.  The set of admissible almost complex structures is denoted by $\mathcal{J}(Y_{\varphi}, \omega_{\varphi})$.

Given $J \in \mathcal{J}(Y_{\varphi}, \omega_{\varphi})$ and orbit sets $\alpha_+=\{ (\alpha_{+, i}, m_i) \}$,  $\alpha_-=\{(\alpha_{-, j}, n_j)\}$,   let $\mathcal{M}^J(\alpha_+, \alpha_-, Z)$ be the set of equivalence classes   of punctured holomorphic curves  $u : \dot{F} \to  X$ with the following properties: $u$ has positive ends at covers of $\alpha_{+, i}$ with total multiplicity $m_i$, negative ends at covers of $\alpha_{-, j}$ with total multiplicity $n_j$, and no other ends. Moreover, the relative homology class of $u$ is $Z$. Two holomorphic curves $u_i: \dot{F}_i \to X$ are   equivalent if there exists a biholomorphic $\phi: F_1\to F_2$ preserving the (ordered) punctures  such that $u_1=u_2 \circ \phi$.  To distinguish these from the  HF curves or HF-PFH curves  defined later, sometimes we also call  an element of  $\mathcal{M}^J(\alpha_+, \alpha_-, Z)$  a  \textbf{PFH curve}. 
A holomorphic curve $u$ is called \textbf{simple} if it does not  factor as
 \begin{equation}\label{eq46}
 \dot{F} \xrightarrow{\phi} \dot{F}'  \xrightarrow{v}  X,
 \end{equation}
where $\phi$ is a branched cover of degree 2 or more, and $v$ is a $J$-holomorphic curve.

In  the ECH/PFH setting, we often consider a weaker concept called \textbf{holomorphic currents}.  A $J$-holomorphic current from $\alpha_+$  to $\alpha_-$ is a formal sum $\mathcal{C} =\sum_a d_a C_a$ such that $\mathcal{C}$ is asymptotic to $\alpha_{\pm}$ as $s\to \pm\infty $  in the sense of currents,  where $\{C_a\}$ are distinct simple holomorphic curves with finite energy  and  $\{d_a\}$ are positive integers. When $d_a=1$ for all $a$, then the  holomorphic current is just the same as the concept of simple holomorphic curves.  Conversely, we can obtain a holomorphic current from a holomorphic curve  as follows:   Let  $u=\bigcup_a u_a$ be a holomorphic curve, where  $u_a$ is
irreducible.  We factorize  $u_a=v_a \circ \phi_{a}$ as in (\ref{eq46}) such that $v_a$ is simple. Then the underlying holomorphic current of $u$ is $\mathcal{C}=\sum_a \operatorname{deg} (\phi_a) v_a(\dot{F}_a')$.

A fact that will be used later is that  the $J_0$ index is nonnegative for the holomorphic currents  in the symplectization of $(Y_{\varphi}, \omega_{\varphi})$.
\begin{lemma}[Lemma 2.4 of \cite{GHC2}] \label{lem13}
Let $J \in \mathcal{J}(Y_{\varphi}, \omega_{\varphi})$  be an admissible almost complex structure in  the symplectization of $\mathbb{R} \times Y_{\varphi}$. Let $\mathcal{C} $ be a holomorphic current from $\alpha_+ $ to $\alpha_-$ in $\mathbb{R} \times Y_{\varphi}$  with no closed components. Then $J_0(\mathcal{C}) \ge 0$.
\end{lemma}

\paragraph{Definition of periodic Floer homology.}
Assume that $d >g(\Sigma)$ throughout.   Fix a tuple of  $d$ points $\mathbf{x}=(x_1, ...,x_d)$  on $\Sigma$ (not necessarily to be distinct). Define a reference 1-cycle   $\gamma_H^{\mathbf{x}}:=\Psi_H(S^1 \times  \textbf{x})$, where $\Psi_H$ is the trivialization (\ref{eq34}).    An \textbf{anchored orbit set} is a pair $(\alpha, [Z])$, where $\alpha$ is an orbit set   and $[Z] \in  H_2(Y_{\varphi}, \alpha,\gamma_H^{\mathbf{x}})/ \ker \omega_{\varphi} $. We call it an   \textbf{anchored PFH generator} if $\alpha$ is a PFH generator. Note that $H_2(Y_{\varphi}, \alpha, \gamma_H^{\mathbf{x}})/ \ker \omega_{\varphi}$ is an affine space  of $\mathbb{Z} [\Sigma]$.

The chain complex $\widetilde{PFC}(\Sigma, \varphi_H, d)$ is  the set of the possibly infinite formal sums 
\begin{equation} \label{eq26}
\sum a_{(\alpha, [Z])}(\alpha, [Z]),
\end{equation}
where $a_{(\alpha, [Z])} \in \mathbb{F}$ and each $(\alpha, [Z])$ is an anchored PFH generator. Also, for any $C\in \mathbb{R}$, we require that there  are  only finitely many  $(\alpha, [Z])$ such that $\int_{Z} \omega_{\varphi_H} >C$ and $ a_{(\alpha, [Z])} \ne 0$.

Fix  $J \in \mathcal{J}(Y_{\varphi}, \omega_{\varphi})$. The differential $\partial_J $ on   $\widetilde{PFC}(\Sigma, \varphi_H, d)$  is defined by
\begin{equation*}
\partial_J(\alpha_+, [Z_+]): = \sum_{\alpha_-} \sum_{Z,  I(Z)=1}  \# \left(\mathcal{M}^J(\alpha_+, \alpha_-,  Z) / \mathbb{R} \right) (\alpha_-,  [Z_+ \#(- Z)]).
\end{equation*}
The homology of $(\widetilde{PFC}(\Sigma, \varphi_H,d), \partial_J)$ is called the \textbf{twisted periodic Floer homology}, denoted by $\widetilde{PFH}(\Sigma, \varphi_H, d)$.   $\widetilde{PFH}(\Sigma, \varphi_H, d)$   is independent of the choices of almost complex structure and base point, and invariant under Hamiltonian isotopy of $\varphi_H$ (see 
Corollary 1.1 of \cite{LT} and Page 12 of \cite{GHC2}).  Then we obtain the abstract homology $\widetilde{PFH}(\Sigma,  d)$ and the canonical isomorphism (\ref{eq62}). 


\paragraph{Notation.} Given relative homology classes $Z_1 \in  H_2(X, \alpha_+, \alpha_0)$ and $Z_2 \in  H_2(X, \alpha_0, \alpha_-)$, $Z_1\#Z_2$ denotes  the  relative homology class in   $H_2(X, \alpha_+, \alpha_-)$ by gluing along their common boundary $\alpha_0$. For $Z\in  H_2(X, \alpha_+, \alpha_-)$, $-Z \in H_2(X, \alpha_-, \alpha_+)$ denotes the orientation reversal of  $Z$.  Later, in the HF setting or the open-closed setting, we use ``$\#$" to denote these operations (gluing along common boundary and orientation reversing).

\begin{remark}
We need   $d >g(\Sigma)$  for the following reasons: If $d\le g(\Sigma)$, then PFH is still well defined but using a larger class of almost complex structures (see (1.6) of \cite{HT}).  These  kinds of almost complex structures  are $\Omega$-tame. This causes an issue in defining the PFH cobordism maps via Seiberg-Witten equations.  When we define the PFH cobordism maps, we need to perturb the Seiberg-Witten equations by the symplectic form $r\Omega_X$. However,    $r\Omega_X$ is not self-dual with respect to  the natural metric 
$$g_J(\cdot, \cdot):=\frac{1}{2}\left(\Omega_X(\cdot, J\cdot)- \Omega_X( J\cdot, \cdot)\right).$$
Additional work would be required  to modify the construction in \cite{GHC}.

Another reason is that $d$ is chosen to be the number of components of an admissible link $\underline{L}$ (Definition \ref{definition6}) for our purposes. Such a class of links has  $(g(\Sigma)+k)$-components, where $k \ge 1$. Thus,  the condition $d>g(\Sigma)$ imposes no additional restriction  on the comparison of the spectral invariants.  
\end{remark}

\paragraph{Grading.}
The twisted PFH  admits a $\mathbb Z$-grading ((16) of \cite{CHS1}). It  is defined as follows. Fix a constant trivialization $\tau_{\mathbf{x}}$ of $T_{\mathbf{x}}\Sigma$. Pushing forward this   trivialization  via $\Psi_H$ (\ref{eq34}) induces a trivialization of  $\xi\vert_{_{\gamma_H^{\mathbf{x}}}}, $ still denoted by  $\tau_{\mathbf{x}}$ . Then the grading of a PFH generator $(\alpha, [Z])$ is
\begin{equation} \label{eq69}
 \operatorname{gr}\left((\alpha, [Z]) \right):= c_{\tau, \tau_{\mathbf{x}} } (Z) + Q_{\tau, \tau_{\mathbf{x}} } (Z) +CZ^{ech}_{\tau}(\alpha).
\end{equation}

\paragraph{The $U$ map.} There is a well-defined map
\begin{equation*}
U:  \widetilde{PFH}_*(\Sigma, \varphi_H, d)  \to \widetilde{PFH}_{*-2}(\Sigma, \varphi_H, d)
\end{equation*}
called  the \textbf{$U$ map.} Fix $z \in Y_{\varphi_H}$.  The definition of the $U$  map is similar to the differential. Instead of counting $I=1$ holomorphic curves  modulo $\mathbb{R}$ translation, the $U $ map is defined by counting $I=2$ holomorphic curves  that   pass through the fixed point $(0, z) \in X$.   A homotopy argument shows  that the  $U$ map is independent of the choice of $z$.  For more details, please see Section 2.5 of \cite{HT3}.

\paragraph{PFH unit.} In Section 6 of \cite{GHC2}, the author defines  a nonzero class $\mathfrak{e} \in \widetilde{PFH}(\Sigma,d )$.  It is  an analogue of the  HF unit defined in  Definition \ref{definition4} later.  If we take $H$ to be a small Morse function on $\Sigma$, then $\mathfrak{j}_H^{-1}(\mathfrak{e})$ is represented by anchored PFH generators  consisting of the constant orbits at the local maximum of $H$ (see Lemma 5.2 of \cite{GHC2}).

\paragraph{Cobordism maps on PFH.} Let $(X, \Omega_X)$ be a symplectic 4-manifold. Suppose that there exists a compact subset $K$ such that
\begin{equation} \label{eq18}
 (X \setminus K, \Omega_X) \cong \left([0, \infty) \times Y_{\varphi_+}, \omega_{\varphi_+} +  ds \wedge dt \right) \cup \left((-\infty, 0] \times Y_{\varphi_-}, \omega_{\varphi_-} +  ds \wedge dt \right)
\end{equation}
We allow $Y_{\varphi_+} =\emptyset$ or  $Y_{\varphi_-} =\emptyset$.  We call $(X, \Omega_X )$ a \textbf{symplectic cobordism} from $(Y_{\varphi_+}, \omega_{\varphi_+})$ to $(Y_{\varphi_-}, \omega_{\varphi_-})$.  Let $\gamma_{\pm}^{ref}$ be reference 1-cycles on $Y_{\varphi_{\pm}}$.  Fix a reference homology class $Z_{ref} \in H_2(X, \gamma_+^{ref}, \gamma_-^{ref})$.  The symplectic manifold   $(X, \Omega_X)$ induces a  homomorphism
\begin{equation*}
PFH_{Z_{ref}}^{sw}(X, \Omega_X) : \widetilde{PFH}(\Sigma, \varphi_+, \gamma^{ref}_+) \to  \widetilde{PFH}(\Sigma, \varphi_-, \gamma_-^{ref}).
\end{equation*}
This homomorphism is called a \textbf{PFH cobordism map.}

Building on Hutchings-Taubes's  idea \cite{HT}, the cobordism map $ PFH_{Z_{ref}}^{sw}(X, \Omega_X)$ is  defined by using   Seiberg-Witten theory \cite{KM} and Lee-Taubes's isomorphism \cite{LT}.   Even though  the cobordism maps are defined by Seiberg-Witten theory, they satisfy   some nice properties called \textbf{``holomorphic curves axioms"}.  This  means that the PFH cobordism maps count holomorphic curves in  a certain sense. For the precise statement, we refer readers to Theorem 1 of   \cite{GHC} and   Appendix  B of \cite{GHC2}.

In this paper, we  focus on the following two special cases of $(X, \Omega_X)$.

\begin{enumerate}
\item

Given two Hamiltonian functions $H_+, H_-$,  define a homotopy $H_s :=\chi(s)H_+ + (1- \chi(s)) H_-$, where $\chi: \mathbb{R} \to \mathbb{R}$ is a  nondecreasing cutoff function such that $\chi(s)=1$ for $s \ge R_0>0$ and $\chi(s)=0 $ for $s \le 0$.   Define
\begin{equation} \label{eq23}
\begin{split}
&X:= \mathbb{R}_s \times S^1_t \times \Sigma,\\
 &\omega_X: =\omega +dH_s \wedge dt,  \  \  \Omega_X :=\omega_X +ds \wedge dt.
\end{split}
\end{equation}
This is a symplectic cobordism if $R_0$ is sufficiently large. Note that we have  identified  $Y_{\varphi_{H_{\pm}}}$ with $S^1\times \Sigma$ implicitly by using (\ref{eq34}).  Fix a reference relative homology class  $Z_{ref}=[\mathbb{R} \times S^1 \times \mathbf{x}] \in H_2(X, \gamma_{H_+}^{\mathbf{x}},\gamma_{H_-}^{\mathbf{x}})$.
Then we have a cobordism map
\begin{equation*}
PFH_{Z_{ref}}^{sw}(X, \Omega_{X}) :  \widetilde{PFH}(\Sigma, \varphi_{H_+}, d) \to \widetilde{PFH}(\Sigma, \varphi_{H_-}, d).
\end{equation*}
This map only depends on  $H_+$, $H_-$ and the relative homology class $Z_{ref}$. For simplicity, we write $\mathfrak{I}_{H_+, H_-}=PFH_{Z_{ref}}^{sw}(X, \Omega_{X})$ instead.  By the composition rule and holomorphic curve axioms, we have
\begin{equation}\label{eq63}
\mathfrak{I}_{H, H}= \operatorname{Id} \mbox{ and } \mathfrak{I}_{H_2, H_3} \circ \mathfrak{I}_{H_1, H_2}= \mathfrak{I}_{H_1, H_3},
\end{equation}
for any Hamiltonian functions $H_1, H_2, H_3.$    The abstract homology  
$\widetilde{PFH}(\Sigma,  d)$ is the direct limit of the direct system $\{(\widetilde{PFH}(\Sigma,  \varphi_{H_i}, d), \mathfrak{I}_{H_i, H_j})\}$.  Hence,  we have $\mathfrak{I}_{H, G} =\mathfrak{j}_G^{-1} \circ  \mathfrak{j}_H$. 

\item
 Let $(B_-, \omega_{B_-}, j_{B_-})$ be a sphere with a puncture $p$, where $(\omega_{B_-}, j_{B_-})$  is a  K\"ahler structure.   Suppose that we have a neighbourhood $U$ of $p$ so that  we have the following identification
\begin{equation*}
(B_-, \omega_{B_-}, j_{B_-}) \vert_U \cong ( [0, \infty)_s \times S_t^1, ds \wedge dt, j),
\end{equation*}
where $j$ is a complex structure that maps $\partial_s$ to $\partial_t$.   Let $\chi: \mathbb{R} \to \mathbb{R}  $   be a cutoff function such that $\chi(s)=1$ when $s  \ge R_0$ and $\chi(s) =0$ when $s \le  R_0/10$. Take
\begin{equation}  \label{eq6}
\begin{split}
&X_- =:B_- \times \Sigma, \\
 &\omega_{X_-}: =\omega + d(\chi(s) H dt),  \  \  \Omega_{X_-} :=\omega_{X_-} +  \omega_{B_-}.
\end{split}
\end{equation}
For sufficiently large $R_0>0$, $(X_-, \Omega_{X_-})$ is a symplectic manifold satisfying (\ref{eq18}).

\item
 Let $(B_+, \omega_{B_+}, j_{B_+})$ be a sphere with a puncture $p$.  Suppose that we have neighbourhood $U$ of $p$ so that  we have identification
\begin{equation*}
(B_+, \omega_{B_+}, j_{B_+}) \vert_U \cong ( (-\infty, 0]_s \times S_t^1, ds \wedge dt, j),
\end{equation*}
where $j$ is a complex structure that maps $\partial_s$ to $\partial_t$.   Let $\chi: \mathbb{R} \to \mathbb{R}  $   be cut off function such that $\chi=1$ when $s \le -R_0$ and $\chi(s) =0$ when $s \ge 0$. Take
\begin{equation} \label{eq98}
\begin{split}
&X_+ =:B_+ \times \Sigma \\
 &\omega_{X_+}: =\omega + d(\chi(s) H dt)  \  \  \Omega_{X_+} :=\omega_{X_+} +  \omega_{B_+}.
\end{split}
\end{equation}
For sufficiently large $R_0>0$, $(X_+, \Omega_{X_+})$ is a symplectic manifold satisfying (\ref{eq18}). 

\end{enumerate}

\paragraph{Filtered PFH and PFH spectral invariants.} We define a functional $\mathbb{A}_H$ on the  anchored orbit sets by:
\begin{equation} \label{eq58}
\mathbb{A}_H(\alpha, [Z]) := \int_{Z} \omega_{\varphi_H} + \int_0^1 H_t(\mathbf{x})dt,
\end{equation}
where  $ \int_0^1 H_t(\mathbf{x})dt  $ is short for  $\sum_{i=1}^d  \int_0^1 H_t(x_i) dt. $

Let $\widetilde{PFC}^{L}(\Sigma, \varphi_H, d)$ be the set of formal sums  (\ref{eq26})  satisfying $\mathbb{A}_H(\alpha, [Z])<L$. 
 It is easy to check that it is  a subcomplex of  $(\widetilde{PFC}(\Sigma, \varphi_H, d), \partial_J)$. The homology is denoted by  $\widetilde{PFH}^{L}(\Sigma, \varphi_H, d)$.   Let
$\mathfrak{i}_L: \widetilde{PFH}^{L}(\Sigma, \varphi_H,  d) \to \widetilde{PFH}(\Sigma, \varphi_H,  d)$ be the map induced by the  inclusion.

 Fix $\mathfrak{v}\in  \widetilde{PFH}(\Sigma,     d) $. The \textbf{PFH spectral invariant} at $\mathfrak{v}$ is defined by
\begin{equation*}
c_{d}^{pfh}(H, \mathfrak{v}) :=\inf \{ L \in \mathbb{R} \vert  \mathfrak{j}_H^{-1}(\mathfrak{v})  \mbox{ belongs to the image of } \mathfrak{i}_L \}.
\end{equation*}
If   $\varphi_H$ is degenerate, take  a sequence  $\{{H_n}\}_{n=1}^{\infty}$ such that $\varphi_{H_n}$ are  nondegenerate and   $\{{H_n}\}_{n=1}^{\infty}$ converges to $H$ in the  $C^{\infty}$ topology. Then,  define
\begin{equation*}
c_{d}^{pfh}(H, \mathfrak{v})=\lim_{n \to \infty } c_{d}^{pfh}(H_n, \mathfrak{v}).
\end{equation*}

\begin{remark}
 Unlike the action functional defined in \cite{EH, CPZ}, our definition includes an additional term  $\int_0^1 H_t(\mathbf{x})dt$. This adjustment ensures that the format of the PFH action functional aligns with that in the HF setting (\ref{eq57}). Another advantage of this definition is that the spectral invariant  defined by (\ref{eq58})  is independent of the choice of the base point $\mathbf{x}$ (see (1.7) of \cite{GHC2}).
\end{remark}

Let $\widetilde{Ham}(\Sigma, \omega)$ be the universal  cover of   ${Ham}(\Sigma, \omega)$.  An element  of  $\widetilde{Ham}(\Sigma, \omega)$  is a homotopy class of paths $\{\varphi_t\}_{t\in [0, 1]} \subset Ham(\Sigma, \omega)$ with fixed endpoints  $\varphi_0=\operatorname{Id}$ and $\varphi_1 =\varphi$.
Let  $\tilde{\varphi} \in \widetilde{Ham}(\Sigma, \omega)$  be a class represented by a path generated by a  mean-normalized Hamiltonian $H$.   Define
\begin{equation*}
c_{d}^{pfh}(\tilde{\varphi}, \mathfrak{v}): =c_{d}^{pfh}(H,  \mathfrak{v}).
\end{equation*}
By Proposition 3.2 of \cite{CHS2},   $c_{d}^{pfh}(\tilde{\varphi},  \mathfrak{v})$ is well defined because it is independent of the choice of $H$.

 To obtain numerical invariants for the elements  of  $Ham(\Sigma, \omega)$ rather than its universal cover, we need the homogenized  spectral invariants.
Fix $\varphi \in Ham(\Sigma, \omega)$ and $\mathfrak{v}\in \widetilde{PFH}(\Sigma, d)$. We  define the  \textbf{homogenized PFH spectral invariant}  by
 \begin{equation}\label{eq60}
\mu^{pfh}_d(\varphi,  \mathfrak{v}) := \limsup_{n \to \infty} \frac{c_{d}^{pfh}(\tilde{\varphi}^n, \mathfrak{v})}{n}.
\end{equation}
By Proposition 3.5 of \cite{CHS2} and $\widetilde{Ham}(\Sigma, \omega) =Ham(\Sigma, \omega)$ when $g(\Sigma) \ge 1$,  $\mu^{pfh}_d$  descends to  $ Ham(\Sigma, \omega) $.

\paragraph{Novikov ring module}
 Let $\Lambda$   be the ring of formal Laurent series:
\begin{equation} \label{eq85}
\Lambda :=\mathbb{F}[[S^{-1}, S] =\left\{\sum_{i=0}^{\infty} a_i S^{b_i} : a_i \in \mathbb{F}, b_i \in \mathbb{Z}, b_0 >b_1>....\right\},
\end{equation}
  where $S$ is a formal variable.  Then the chain complex $\widetilde{PFC}(\Sigma, \varphi_H, d)$ is a $\Lambda$-module by defining the following  action:
\begin{equation} \label{eq24}
\left(\sum_i a_i S^{b_i}\right)\cdot (\alpha, [Z]): = \sum_i a_i(\alpha, [Z+ b_i\Sigma]).
\end{equation}
Obviously, the action descends to the homology, and hence, $\widetilde{PFH}(\Sigma, \varphi_H, d)$ is also a $\Lambda$-module. It is straightforward to verify that the $U$ map and the cobordism maps $\mathfrak{I}_{H, G}$   are all $\Lambda$-module homomorphisms. 

By definition, we also have the following properties for the action of $S^b$:
\begin{equation*}
\mathbb{A}_H(S^{b}\cdot(\alpha, [Z])) = \mathbb{A}_H(\alpha, [Z]) + b \mbox{ and } \operatorname{gr}(S^{b}\cdot(\alpha, [Z])) =\operatorname{gr}((\alpha, [Z])) +2b(d-g+1).
\end{equation*}
Given $\hat{r}=\sum_{i}a_i S^{b_i} \in \Lambda$,  let  $|\hat{r}| := \max_i \{b_i \vert a_i \ne 0\}$ denote the maximal exponent of $S$ with a non-zero coefficient. By Proposition 4.2 of \cite{EH},   we have:
\begin{equation*}
 c_{d}^{pfh}(H, \hat{r} \cdot \mathfrak{v})=c_{d}^{pfh}(H, \mathfrak{v}) + |\hat{r}|.  
\end{equation*}

\paragraph{Periodic Floer homology of $\mathbb{S}^2$.}
In the simplest case, where   $\Sigma=\mathbb{S}^2$,  the $U$-module structure of  $\widetilde{PFH}_k(\mathbb{S}^2, d) $ is known. 

\begin{prop}[Example 2.19 of  \cite{EH}] \label{lem25}For $d \ge 1$, we have 
$$
\widetilde{PFH}_k(\mathbb{S}^2, d) = \begin{cases} 
\mathbb{F}, & \text{if } {k = \operatorname{gr}(\mathfrak{e}) \!\! \mod 2  } \\ 
0, &  \text{otherwise}.
\end{cases}
$$
For each   $k$ with $ k=\operatorname{gr}(\mathfrak{e}) \mod 2$,   $U: \widetilde{PFH}_k(\mathbb{S}^2, d)  \to \widetilde{PFH}_{k-2}(\mathbb{S}^2, d) $  is an isomorphism. Furthermore,    $U^{d+1} \mathfrak{v}   = S^{-1}  \mathfrak{v}$ for all $\mathfrak{v}$. 
\end{prop}

\subsubsection{Quantitative Heegaard Floer homology}
 In this section, we review the cylindrical formulation of  QHF defined in \cite{GHC2}. One will find that most of the definitions of QHF are parallel to those of  PFH.   Therefore, QHF could be regarded as a relative version of PFH.

\paragraph{Admissible Lagrangian links.}To begin with,  let us recall  a class of  links considered in \cite{GHC2}.
\begin{definition}\label{definition6}
 Let $\underline{L}=\bigsqcup_{i=1}^d L_i$ be a  disjoint union of simple closed curves on $\Sigma$.  We call $\underline{L}$ a \textbf{link} on $\Sigma$.     For  a nonnegative constant $\eta$,  we say  a link $\underline{L}$ is $\eta$-\textbf{admissible} if it satisfies the following properties:
	\begin{enumerate} [label=\textbf{A.\arabic*}]
		\item \label{assumption3}
		here is an integer $k\ge1$ such that $d=k+g$, where $g$ is the genus of $\Sigma$. $\bigsqcup_{i=1}^k L_i$  is a disjoint union of contractible simple closed curves.  For $k+1\le i\le d$, $L_i$   is the cocore of the 1-handle. For each 1-handle, we have exactly one corresponding $L_i$.
		\item  \label{assumption4}
		We require  that $\Sigma \setminus \underline{L} =\bigcup_{i=1}^{k+1} \mathring{B}_i$.    Then  $\mathring{B}_i$ is a disk for $1 \le i\le k$ and $\mathring{B}_{k+1} $ is a planar domain.  Let $B_i$ be the closure of   $\mathring{B}_i$.  For $1\le i\le k $,	the  circle $L_i$ is the boundary of   $B_i$.
		\item \label{assumption5}
		$\mathring{B}_i \cap \mathring{B}_j =\emptyset$. 
		\item \label{assumption6}
			For $ 1\le i<j \le k$, we have  $\int_{B_i} \omega =\int_{B_j} \omega=\lambda$. Also, $\lambda=2\eta(2g+k-1)+ \int_{B_{k+1}} \omega$.
			If $g=0$ and $k=1$, we set $\eta=0.$
	\end{enumerate}
\end{definition}
A picture of an admissible link is shown in Figure \ref{figure1}.
	\begin{figure}[h]
		\begin{center}
			\includegraphics[width=10cm, height=5cm]{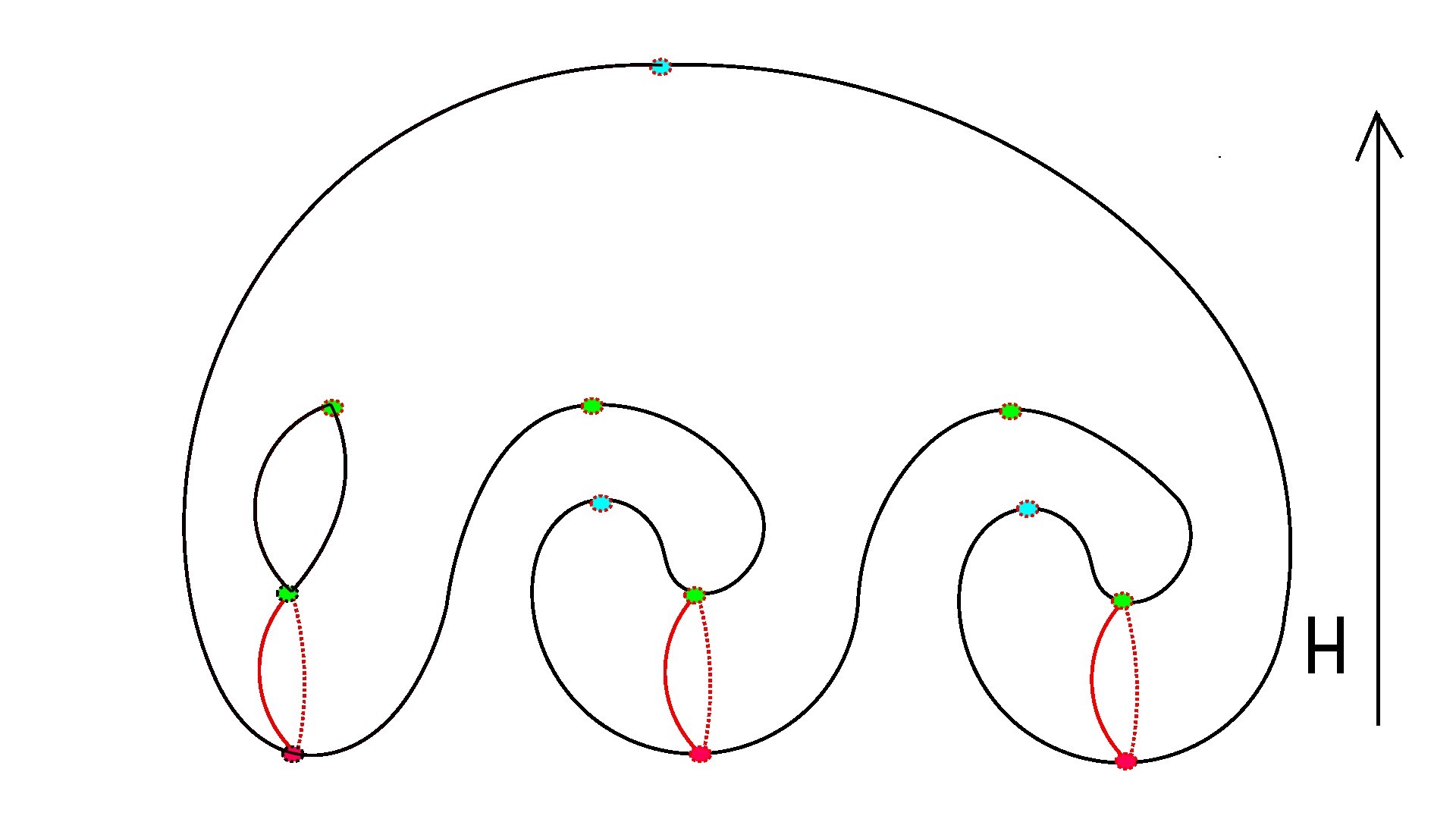}
		\end{center}
		\caption{The red circles are the admissible link. }
		\label{figure1}
	\end{figure}
Note that  if $\underline{L}$ is admissible,  so is $\varphi(\underline{L})$, where  $\varphi$	 is  any Hamiltonian symplectomorphism.  \textbf{We assume that the link is $\eta$-admissible throughout.}
	
 \paragraph{Cylindrical formulation of QHF.}
Fix an admissible link $\underline{L} =\bigcup_{i=1}^d L_i$ and   $ \varphi_H \in Ham(\Sigma, \omega)$. Throughout we  assume that $\varphi_H$ is \textbf{nondegenerate} in the sense that  $\varphi_H(\underline{L})$ intersects $\underline{L}$ transversely.

 \begin{definition} A \textbf{Reeb chord}  of $\varphi_H$ is a  union of paths $$\mathbf{y}=[0,1]\times (y_1, ...,y_d) \subset [0,1] \times \Sigma, $$ where $y_i \in L_i \cap \varphi_H(L_{\sigma(i)})$ and $\sigma: \{1,..., d\} \to \{1, ..., d\}$ is a permutation.
\end{definition}
	Fix a base point $\textbf{x}=(x_1, ..., x_d)$, where $x_i \in L_i$.  Define    a reference chord  from $\{0\}\times \varphi_H(\underline{L})$ to $\{1\} \times \underline{L}$ by
\begin{equation*}
\textbf{x}_H (t):=\varphi_H \circ (\varphi_H^{t})^{-1} (\textbf{x}) \subset [0,1]_t \times \Sigma.
\end{equation*}

Let $(E:=\mathbb{R}_s \times [0,1]_t \times \Sigma, \Omega: =\omega + ds \wedge dt)$ be a symplectic manifold.  Let $\mathcal{L} :=\mathbb{R} \times (\{0\} \times \varphi_H(\underline{L}) \cup \{1\} \times \underline{L})$ be a disjoint union of Lagrangian submanifolds in $(E, \Omega)$. Let $\textbf{y}_{\pm}$ be two Reeb chords. Then we have a concept called $d$-\textbf{multisection} in $E$.  Roughly speaking, this is  a map $u: \dot{F} \to E$ which is asymptotic to $\mathbf{y}_{\pm}$ as $s \to \pm\infty$ and satisfies  the  Lagrangian boundary conditions $u(\partial \dot{F}) \subset \mathcal{L} $, where $\dot{F}$ is a Riemann surface with boundary punctures. If a $d$-multisection is holomorphic, we call it an \textbf{HF curve}.   The set of equivalence classes  of the $d$-multisections  is denoted by $H_2(E, \textbf{y}_+, \textbf{y}_-)$.
An element of  $H_2(E, \textbf{y}_+, \textbf{y}_-)$ is also called a \textbf{relative homology class} because it is the counterpart of a relative homology class in the PFH setting.   Here   two  $d$-multisections represent the same relative homology class if they  are equivalent in $H_2(E, \mathcal{L}\cup \{\infty\} \times \textbf{y}_+ \cup \{-\infty\} \times \textbf{y}_-; \mathbb{Z})$.

 Fix $A \in H_2(E,  \textbf{y}_+, \textbf{y}_-)$. The ECH index and $J_0$ index  can also be generalized  to the current setting, denoted by $I(A)$ and $J_0(A)$ respectively.  The definition of  relative homology class, HF curves, ECH index and $J_0$ index will be postponed to  Section \ref{section1}.  We will define   these concepts  for a   slightly more general setting.



Fix an $\eta$-admissible link $\underline{L}$. Given     a Reeb chord $\textbf{y}$, 	 a  \textbf{capping} of $\textbf{y}$    is an  equivalence class $[A]$ in $H_2(E,  \textbf{y}, \textbf{x}_H)/ \ker (\omega + \eta J_0)$.  Define  a  complex ${CF}(\Sigma, \varphi_H(\underline{L}), \underline{L}) $  to be  the set of formal sums
\begin{equation} \label{eq27}
 \sum_{(\mathbf{y}, [A])}  a_{(\mathbf{y},[A])}( \mathbf{y}, [A])
 \end{equation}  
 such that   $a_{(\mathbf{y},[A])} \in \mathbb{F}  $ and for any $C\in \mathbb{R}$, there are only finitely many   $(\mathbf{y}, [A])$ such that  $\int_{A} \omega +\eta J_0(A) >C$ and  $a_{(\mathbf{y},[A])}  \ne 0$.

 \begin{remark}
 To keep the format consistent with that of the PFH setting, our convention for cappings here is opposite to  that  in \cite{CHMSS}. Specifically, the negative end of a capping here is asymptotic to the reference chords, whereas in \cite{CHMSS}, it is the positive end that is asymptotic to the reference chords.
 \end{remark}

Let $\mathcal{J}_E$ denote the set of $\Omega$-compatible almost complex structures such that $J$ is $\mathbb{R}_s$-invariant, $J(\partial_s) = \partial_t$,  $J$ sends $T\Sigma$ to itself and $J \vert_{T\Sigma}$ is $\omega$-compatible.	Fix $J \in \mathcal{J}_E$.  Let $\mathcal{M}^J(\textbf{y}_+, \textbf{y}_-, A)$ denote the moduli space of HF curves that are asymptotic to $\textbf{y}_{\pm}$ as $s \to \pm \infty$ and have  relative homology class $A$.   Because $J$ is $\mathbb{R}_s$-invariant, this induces a  natural $\mathbb{R}$-action on $\mathcal{M}^J(\textbf{y}_+, \textbf{y}_-, A)$.

Fix a generic    $J \in \mathcal{J}_E$. The differential is defined by
		\begin{equation*}
		d_J (\textbf{y}_+, [A_+]):=\sum_{A \in H_2(E, \textbf{y}_+, \textbf{y}_-), I(A)=1} \#\left(\mathcal{M}^J( \textbf{y}_+, \textbf{y}_-, A)/\mathbb{R} \right) ( \textbf{y}_-, [  (-A) \#A_+]).
	\end{equation*}
The homology of $(CF_*(\Sigma, \varphi_H(\underline{L}), \underline{L}), d_J)$ is  well defined  and is denoted by $HF_*(\Sigma, \varphi_H(\underline{L}), \underline{L})_J$  \cite{GHC2}.    

By Proposition 3.9 of \cite{GHC2}, the homology is independent of the choices of $J$  and $H$.  More precisely, for two pairs $(H, J_H)$ and $(G, J_G)$, there is a  canonical   isomorphism
  \begin{equation*}  
  \mathcal{I}_{H, G}: HF(\Sigma, \varphi_H(\underline{L}), \underline{L})_{J_H} \to HF(\Sigma, \varphi_G(\underline{L}), \underline{L})_{J_G}
  \end{equation*}
  called  a \textbf{continuation morphism}.  More details about this point are  given in Section \ref{section1} later.    
Then $HF(\Sigma, \underline{L})$ is  the direct limit of the  direct system $\{(HF(\Sigma, \varphi_{H_i}(\underline{L}), \underline{L}) ,   \mathcal{I}_{H_i, H_j})\}$.


\begin{remark}
Even though we only define the QHF for a Hamiltonian symplectomorphism $\varphi_H$, the above construction  also works for a pair of   Hamiltonian symplectomorphisms $(\varphi_H, \varphi_K)$.  Because $\varphi_K(\underline{L})$ is also an admissible link,  we just need to replace  $\underline{L}$ by $\varphi_K(\underline{L})$.  The result is denoted by $HF(\Sigma, \varphi_H(\underline{L}), \varphi_K(\underline{L}))$. In this situation, the reference chord is chosen to be $\varphi_K(\mathbf{x}_{\bar{K}\#H}).$
\end{remark}

\begin{remark} \label{remark4}
The links under consideration are  slightly different  from those in \cite{CHMSS}.  The main reason is that the admissible links  make computations easier in the cylindrical setting (see  Remark 1.4 and Remark 2.2 of \cite{GHC2} for details). On the other hand, QHF $HF(\operatorname{Sym}^d \varphi_H(\underline{L}), \operatorname{Sym}^d \underline{L})$ in  \cite{CHMSS}  is  still well defined and we have the isomorphism  
$$HF(\operatorname{Sym}^d \varphi_H(\underline{L}), \operatorname{Sym}^d \underline{L}) \cong H^*(\mathbb{T}^d, R)$$ 
for admissible links (see the explanations in Remark 1.4 and
Remark 3.1  of \cite{GHC2}), where $\mathbb{T}^d$ is the $d$-torus and $R$ is  a Novikov ring (\ref{eq76}).
Combining this with the isomorphism (\ref{eq30}), we know that
 $HF(\Sigma,   \underline{L})$  is isomorphic to $H^*(\mathbb{T}^d, R) $
 as an $R$-module. 
\end{remark}

\paragraph{Filtered QHF and HF spectral invariants.}

Following   \cite{LZ, CHMSS, FOOO}, we  define an  \textbf{action functional} on 	the  generators  by
\begin{equation} \label{eq57}
	\mathcal{A}^{\eta}_{H}( \textbf{y}, [A]) : =\int_A \omega + \int_0^1 H_t(\textbf{x})dt+\eta J_0(A).
\end{equation}
When $\eta=0$, we simply write $\mathcal{A}^{\eta}_{H}$ as $\mathcal{A}_{H}$. 

\begin{remark} \label{remark3}
The  term $J_0(A)$ corresponds to $\Delta \cdot [\hat{\mathbf{y}}] $ in \cite{CHMSS}, where $\Delta$ is the diagonal of $\operatorname{Sym}^d\Sigma$ and $\hat{\mathbf{y}}$ is a capping  of a Reeb chord $\mathbf{y}$.  This viewpoint is proved in Proposition 3.2 of  \cite{GHC2}.

The reason why  the $J_0$ index  is  included  in the action functional  is that the Lagrangian torus $\operatorname{Sym}^d\underline{L}$  is  monotone with respect to $\operatorname{Sym}^d\omega + \eta PD(\Delta)$ rather than  $\operatorname{Sym}^d\omega$ (Lemma 4.21 of \cite{CHMSS}). If we want the isomorphism (\ref{eq30})  to preserve the action filtration, we have to add the term $\eta J _0$ to the action functional.

Moreover, in the computations  of  open-closed maps, we need the  nonnegativity of energy and monotonicity to rule out some holomorphic curves (Lemma \ref{lem6}).  Due to the form of (\ref{eq57}), the energy therein should be understood as $\int {\omega} + \eta J_0$ rather than just $\int \omega$. Therefore, we also need $J_0$ index  in the PFH and open-closed settings.
\end{remark}

Given $L \in \mathbb{R}$, let $CF^L(\Sigma, \varphi_H(\underline{L}), \underline{L})$ be the set of formal sums  (\ref{eq28}) satisfying $\mathcal{A}^{\eta}_H(\textbf{y}, [A]) < L$. It is easy to check that it is a subcomplex.   The \textbf{filtered} QHF, denoted by $HF^L(\Sigma, \varphi_H(\underline{L}), \underline{L})$, is the homology   of  $(CF^L(\Sigma, \varphi_H(\underline{L}), \underline{L}), d_J)$.  Let $$i_L: HF^L(\Sigma, \varphi_H(\underline{L}), \underline{L}) \to HF(\Sigma, \varphi_H(\underline{L}), \underline{L}) $$ be the homomorphism  induced by the    inclusion.

\begin{definition}
Fix  $a \in HF(\Sigma, \underline{L})$.  The \textbf{HF spectral invariant} at $a$  is
 $$ c_{\underline{L}}(H, a) : =\inf \{L \in \mathbb{R}  \vert  (j_H)^{-1}(a)  \mbox{ belongs to the image of } i_L\}. $$
\end{definition}

 Let $\textbf{c}=\sum a_{ ( \textbf{y}, [A])} ( \textbf{y}, [A])$ be a cycle in $CF(\Sigma, \varphi_H(\underline{L}), \underline{L}). $ The action of this cycle is  defined by
\begin{equation*}
\mathcal{A}^{\eta}_H(\textbf{c}) = \max\{\mathcal{A}^{\eta}_H( \textbf{y}, [A]) \vert a_{(\textbf{y}, [A])} \ne 0 \}.
\end{equation*}
Then the spectral invariant  can be expressed alternatively  as
\begin{equation} \label{eq20}
c_{\underline{L}}(H, a) =\inf \{\mathcal{A}^{\eta}_H( \textbf{c}) \vert [ \textbf{c}] = (j_H)^{-1} (a)\}.
\end{equation}
Fix $\varphi \in Ham(\Sigma, \omega)$ and $a \in HF(\Sigma, \underline{L})$. Define the  \textbf{homogenized  HF spectral invariant}  by
 \begin{equation} \label{eq61}
\mu_{\underline{L}}(\varphi, a) := \limsup_{n \to \infty} \frac{c_{\underline{L}}(\tilde{\varphi}^n, a)}{n},
\end{equation}
where $\tilde{\varphi} \in \widetilde{Ham}(\Sigma, \omega)$ is a lift of $\varphi$.

\paragraph{Novikov ring module.}

 Let 
 \begin{equation} \label{eq76}
 R:=\mathbb{F}[[T^{-1}, T] = \{ \sum_{i=0}^{\infty}a_iT^{b_i} \vert a_i \in \mathbb{F},  b_i \in \mathbb{Z}, b_0>b_1>....\}
 \end{equation} be the Novikov ring, where $T$ is a formal variable.   Similar to the PFH case,   $HF(\Sigma, \varphi_H(\underline{L}),  \underline{L}) $ is an $R$-module due to the following construction. 

For $1\le i \le k$, let $v_i: [0,1]_s \times [0,1]_t \to \Sigma$ be a map such that $v_i(0, t) = v_i(1, t)=v_i(s, 0) =x_i $ and $v_i(s,1) \in L_i$  and such that it represents   the class $[B_i] \in H_2(\Sigma, L_i, \mathbb{Z})$, where $B_i$ is the  closed disk in Definition  \ref{definition6}. Define
\begin{equation*}
\begin{split}
u_{x_i}: &[0,1]_s \times [0,1]_t \to [0, 1]_s \times [0,1]_t \times \Sigma\\
    &(s,t) \to (s, t, \varphi_H \circ (\varphi_H^t)^{-1}   \circ v_i(s,t)).
\end{split}
\end{equation*}
Together with the trivial strip at $x_j$ $(j\ne i)$, $u_{x_i}$ represents a class in $H_2(E, \textbf{x}_H, \textbf{x}_H)$, still denoted by $[B_i]$.  We also replace the map $v_i$   by  $v_i'$, where $v_i'$ satisfies   $v'_i(0, t) = v'_i(1, t)=v'_i(s, 1) =x_i $ and $v'_i(s,0) \in L_i$ and represents the class $[B_i] \in H_2(\Sigma, L_i,  \mathbb{Z})$.  Using  the same construction, we have  another  map $u_{x_i}'$. The  difference  between  $u_{x_i}$ and $u_{x_i}'$ is that  $u_{x_i} \vert_{t=1}$ wraps $\partial B_i$ once  while   $u'_{x_i} \vert_{t=0}$ wraps $\partial \varphi_H(B_i)$ one time.  So we denote the   equivalence class of  $u_{x_i}'$ in  $H_2(E, \textbf{x}_H, \textbf{x}_H)$ by $[\varphi_H(B_i)]$.
 By the monotonicity   (\ref{assumption6}) and Lemma 2.7 of \cite{GHC2},   all the classes $[B_i]$ and $[\varphi_H(B_i)]$ are equivalent in $H_2(E, \textbf{x}_H, \textbf{x}_H)/\ker (\omega + \eta J_0)$,  denote their common equivalence class by $\mathcal{B}$. 

Then  $HF(\Sigma, \varphi_H(\underline{L}), \underline{L})$ is an $R$-module because  we define  the following action
 \begin{equation}\label{eq28}
 \sum_i a_iT^{b_i} \cdot (\textbf{y}, [A]) :=\sum_i a_i( \mathbf{y}, [A]+ b_i\mathcal{B}).
 \end{equation}
 Note that $(\textbf{y}, [A] + m(k+1) \mathcal{B}) =(\textbf{y}, [A+ \sum_{i=1}^{k+1}mB_i]) =(\textbf{y}, [A] + m[\Sigma])$.  Therefore, we   can embed $\Lambda$ into $R$ by sending $S$ to $T^{k+1}$.  Consequently, $HF(\Sigma, \varphi_H(\underline{L}), \underline{L})$ is also  a   $\Lambda$-module.

Similar to the case of PFH, we have 
\begin{equation*}
\begin{split}
\mathcal{A}^{\eta}_H(T^b \cdot (\mathbf{y}, [A])) =\mathcal{A}^{\eta}_H( (\mathbf{y}, [A])) + b \lambda,
\end{split}
\end{equation*}
where $\lambda$ is the constant in Definition \ref{definition6}.  By the same argument in  Proposition 4.2 of \cite{EH}, for $r =\sum_{i=0}^{\infty} a_i T^{b_i}\in R$, we have  
\begin{equation*}
 c_{\underline{L}} (H, r\cdot \sigma) = c_{\underline{L}} (H, \sigma) + \lambda |r|, 
 \end{equation*}
 where $|r|:=\max_i\{ b_i  \vert a_i \ne 0\}$.

\subsubsection{Coefficients}
\paragraph{Comparison with previous complexes} In the previous work \cite{GHC2}, the PFH and QHF complexes are  defined   as finite formal sums of  (\ref{eq26}) and (\ref{eq27}) respectively, denoted by $(\widetilde{PFC}^{o}(\Sigma, \varphi_H,  d), \partial^o)$  and $(CF^o(\Sigma, \varphi_H(\underline{L}), \underline{L}), d^o)$ respectively, whereas the present complexes allow infinite sums.  The Floer homologies  are still well defined due to the monotonicity and they are canonically isomorphic to $\widetilde{PFH}^o(\Sigma,  d)$ and $HF^o(\Sigma,   \underline{L})$. 
Let  $c_d^{pfh}(H, \cdot)^{o}$ and $c_{\underline{L}}(H, \cdot)^{o}$ denote the spectral invariants defined  by  these  former complexes. In this subsection, we illustrate the relation between $c_d^{pfh}(H, \cdot)^{o},c_{\underline{L}}(H, \cdot)^{o}$ and $c_d^{pfh}(H, \cdot), c_{\underline{L}}(H, \cdot)$.

The complexes $(\widetilde{PFC}^o(\Sigma, \varphi_H,  d), \partial^{o}), (CF^o(\Sigma, \varphi_H(\underline{L}), \underline{L}), d^o)$ are   $\mathbb{F}[S^{-1}, S]$ and $\mathbb{F}[T^{-1}, T]$ modules respectively.  Let 
\begin{equation*}
\begin{split}
&\iota_H^{pfh}:  \widetilde{PFC}^o(\Sigma, \varphi_H,  d) \to    \widetilde{PFC}(\Sigma, \varphi_H,  d),\\
&\iota_H^{hf}: CF^o(\Sigma, \varphi_H(\underline{L}), \underline{L}), d^{o} \to CF(\Sigma, \varphi_H(\underline{L}), \underline{L})
\end{split}
\end{equation*}
denote the inclusions. They are compatible with the continuation  morphisms, and hence descend to homomorphisms at the homological  level:
 \begin{equation*}
\begin{split}
\iota_*^{pfh}:  \widetilde{PFH}^o(\Sigma,  d)  \to     \widetilde{PFH}(\Sigma,   d) 
\mbox{ and } \iota_*^{hf}: HF^o(\Sigma,   \underline{L}) \to HF(\Sigma,   \underline{L}). 
\end{split}
\end{equation*}
 Moreover, $\mathfrak{j}_H \circ (\iota_H^{pfh})_* = \iota_*^{pfh} \circ \mathfrak{j}_H^{o}$ and  $j_H \circ (\iota_H^{hf})_* = \iota_*^{hf} \circ j_H^{o}$. Then we have 
 \begin{equation}\label{eq83}
\begin{split}
c_d^{pfh}(H, \mathfrak{v})^{o} =c_d^{pfh}(H,  \iota_{*}^{pfh}(\mathfrak{v}) )  \mbox{ and } c_{\underline{L}}(H, \sigma )^{o} =c_{\underline{L}}(H,   \iota_{*}^{hf}(\sigma)),
\end{split}
\end{equation}
where $\mathfrak{v} \in \widetilde{PFH}^o(\Sigma, d) $ and $\sigma \in HF^o(\Sigma,  \underline{L}).$  Thus, the use of the larger coefficient rings  
$\Lambda, R$ does not introduce any discrepancies when comparing the spectral invariants.  Moreover,  (\ref{eq83})  implies that $ \iota_{*}^{pfh}$ and $ \iota_{*}^{hf}$ are injective, because $c_d^{pfh}(H, \mathfrak{v}) =-\infty, c_{\underline{L}}(H, \sigma ) =-\infty$ if and only if $\mathfrak{v}=0$ and $\sigma=0$. The proof of  (\ref{eq83})  is provided  in  the Appendix \ref{appendixA} (Lemma \ref{A1}).

  \paragraph{PFH with $R$ coefficients}
As mentioned in the last section,    we can regard $\Lambda$ as a subfield  of $R$ via $S=T^{d-g+1} $.  
Thus, we can   extend the coefficients of PFH to $R$ by taking the tensor product with $R$.   More specifically, define a chain complex $\widetilde{PFC}(\Sigma, \varphi_H, d) \otimes_{\Lambda} R$ with differential $ \partial \otimes  \operatorname{Id}$.  
 Then the homology of this complex  is   $\widetilde{PFH}(\Sigma, \varphi_H, d) \otimes_{\Lambda}  R$.     The  $U$ maps, cobordism maps and closed-open morphisms     extend naturally  to  $R$-module homomorphisms.

Viewing $R$ as a  $\Lambda$-vector space, $\operatorname{dim}_{\Lambda} R=d-g+1$.   Fix a  $\Lambda$-basis $\{1, T^{-1}, ...., T^{-(d-g)}\}$.
 Then  $r \in R$ can be uniquely  written as $r =\sum_{j=0}^{d-g} \hat{r}_j T^{-j}$, where  $\hat{r}_j \in \Lambda$. This   gives    the following   identifications 
 \begin{equation} \label{eq74}
 \begin{split}
& \widetilde{PFC}(\Sigma,  \varphi_H,  d) \otimes_{\Lambda}  R \cong \widetilde{PFC}(\Sigma, \varphi_H,  d) \otimes_{\mathbb{F}}\operatorname{Span}_{\mathbb{F}} \{1, T^{-1}, ..., T^{-(d-g)}\}, \\
&\widetilde{PFH}(\Sigma,  d) \otimes_{\Lambda}  R \cong \widetilde{PFH}(\Sigma,  d) \otimes_{\mathbb{F}} \operatorname{Span}_{\mathbb{F}}  \{1, T^{-1}, ..., T^{-(d-g)}\}  
\end{split}
\end{equation}
by sending $ \mathfrak{c}  \otimes r $ to $\sum_{j=0}^{d-g} (\hat{r}_{j} \mathfrak{c})  \otimes T^{-j}$ and  $   \mathfrak{v}  \otimes r $ to $\sum_{j=0}^{d-g} (\hat{r}_{j} \mathfrak{v})  \otimes T^{-j}$.  
  We extend the action functional $\mathbb{A}_H$ to $\widetilde{PFC}(\Sigma, \varphi_H, d) \otimes_{\Lambda}  R$   by 
  \begin{equation*}
  \mathbb{A}_H(\sum_{i=0}^{d-g}    \mathfrak{c}_i \otimes  T^{-i} ) : = \max_{ 0\le i \le d-g} \{ \mathbb{A}_H(\mathfrak{c}_i) -\frac {i}{d-g+1}  \}, 
  \end{equation*}
  where $\mathfrak{c}_i \in \widetilde{PFC}(\Sigma, \varphi_H, d) $.       It is crucial to note that  $\partial\otimes \operatorname{Id}$  still decreases the above action functional. Thus, we have spectral invariants $c^{pfh}_d(H, \mathfrak{v}; R)$ for $\mathfrak{v} \in \widetilde{PFH}(\Sigma,   d) \otimes_{\Lambda}  R$.

Let $\mathfrak{v}   \in \widetilde{PFH}(\Sigma, d) \otimes_{\Lambda}  R$. Under the identification (\ref{eq74}), write $\mathfrak{v} =\sum_{i=0}^{d-g} \mathfrak{v}_i T^{-i}$, where   $\mathfrak{v} _i \in \widetilde{PFH}(\Sigma,  d)$.  It is straightforward to verify that 
   \begin{equation*}
  c_{d}^{pfh}(H, \mathfrak{v} ; R) =\max_{0\le i \le d-g}\{c_{d}^{pfh}(H, \mathfrak{v} _i) -\frac{i}{d-g+1}\}. 
\end{equation*}
In particular, for $\mathfrak{v} \in \widetilde{PFH}(\Sigma,  d)$, we have  $  c_{d}^{pfh}(H, \mathfrak{v} ; R) = c_{d}^{pfh}(H, \mathfrak{v}). $

\subsection{Main results} \label{section2}
In this section, we give the precise statements about the results mentioned at the beginning of the paper. These include the properties of the HF spectral invariants, open-closed morphisms, and a general relation between HF spectral invariants and PFH spectral invariants.

In the  first part of this paper, we  study the properties of the spectral invariants  $c_{\underline{L}} $. The results are summarized in the following theorem. These properties are parallel to those in \cite{CHMSS}.
\begin{thm} \label{thm4}
Assume that  the link $\underline{L}$ is $\eta$-admissible. Then the spectral invariants
$$c_{\underline{L}} : C^{\infty}([0,1] \times \Sigma) \times HF(\Sigma, \underline{L}) \to \{-\infty\}\cup \mathbb{R}$$ satisfy  the following properties:
\begin{enumerate}
\item
(Spectrality) For any $H$ and $a  \ne 0 \in HF(\Sigma, \underline{L})$, we have
$c_{\underline{L}} (H, a ) \in \operatorname{Spec}(H: \underline{L})$, where  $ \operatorname{Spec}(H: \underline{L})$ is the action  spectrum  of $H$ defined in (\ref{eq59}).

\item (Hofer-Lipschitz)
For $a \ne 0 \in HF(\Sigma, \underline{L})$,  we have
\begin{equation*}
d\int_0^1 \min_{\Sigma} (H_t-K_t) dt  \le c_{\underline{L}}(H, a) - c_{\underline{L}}(K, a) \le d \int_0^1 \max_{\Sigma} (H_t-K_t) dt.
\end{equation*}

\item (Homotopy invariance)
Let $H$ and $ K$   be  two mean-normalized Hamiltonian functions. Suppose that they are homotopic in the sense of Definition \ref{definition3}.  Then
\begin{equation*}
c_{\underline{L}}(H, a) = c_{\underline{L}}(K, a).
\end{equation*}

\item
(Shift) Fix  $a \ne  0 \in HF(\Sigma, \underline{L})$. Let $c: [0,1]_ t \to \mathbb{R}$ be a smooth  function that only depends on $t$. Then
\begin{equation*}
c_{\underline{L}} (H + c, a) =  c_{\underline{L}} (H, a )  + d \int_0^1 c(t)dt.
\end{equation*}

\item (Lagrangian control)  If $H_t\vert_{L_i}=c_i(t)$ for $i=1,...,d$, then
\begin{equation*}
c_{\underline{L}}(H, a) = c_{\underline{L}}(0, a) +  \sum_{i=1}^d \int_0^1 c_i(t) dt.
\end{equation*}
Moreover,  for any  Hamiltonian function $H$, we have
\begin{equation*}
     \sum_{i=1}^d \int_0^1 \min_{L_i} H_t   dt +  c_{\underline{L}}(0, a) \le c_{\underline{L}}(H, a) \le  c_{\underline{L}}(0, a)   + \sum_{i=1}^d \int_0^1 \max_{L_i} H_t   dt.
\end{equation*}

\item
(Triangle inequality) For any Hamiltonian functions $H,K$ and $a, b \in HF(\Sigma, \underline{L})$, we have
$$ c_{\underline{L}}(H\#K, \mu_2(a \otimes b))  \le  c_{\underline{L}}(K, a) + c_{\underline{L}}(H, b), $$
 where $\mu_2$ is the quantum product defined in Section \ref{section1}, and     $$H\#K(t, x) : =H_t(x) + K_t((\varphi^t_H)^{-1}(x))$$ is the  \textbf{composition} of two Hamiltonian functions.  

\item
(Normalization) For the unit $e_{\underline{L}}$, we have $c_{\underline{L}}(0, e_{\underline{L}}) =0. $

\item
(Weyl law) Let $\{\underline{L}_m\}_{m=1}^{\infty}$ be a sequence of $\eta_m$-admissible links. Suppose that    $\{\underline{L}_m\}_{m=1}^{\infty}$ is  equidistributed in the sense that $\operatorname{diam} \underline{L}_m  \to 0$   (Section 3.1 of   \cite{CHMSS}).  Let $d_m$ denote the number of components of $\underline{L}_m$. Then,  we have
\begin{equation*}
\lim_{m\to \infty}\frac{1}{d_m} c_{\underline{L}_m}(H, e_{\underline{L}_m}) =\int_0^1\int_{\Sigma} H_t dt \wedge \omega.
\end{equation*}

\end{enumerate}

\end{thm}
The properties of $c_{\underline{L}}$ in the above theorem could be deduced from the equivalence relation (\ref{eq21}), possibly except the triangle inequality, because we  have not shown that the quantum product  $\mu_2$ here agrees with the one of $HF(\operatorname{Sym}^d \underline{L})$.
 For completeness, we prove these properties using HF curves in four-dimensional  setting instead. The methods are parallel to those in \cite{CHMSS,  FOOO, LZ}.

The next theorem is a summary of the properties of the open-closed morphisms.
\begin{thm} \label{thm2}
Let $\underline{L}$ be an $\eta$-admissible link and $\varphi_H$ a $d$-nondegenerate Hamiltonian symplectomorphism. 	Then,  	we have  a   $\Lambda$-module homomorphism
\begin{equation*}
\mathcal{OC}(\underline{L}, H) :   HF(\Sigma,  \varphi_H(\underline{L}), \underline{L}) \to \widetilde{PFH}(\Sigma, \varphi_H, d)
\end{equation*}
 satisfying  the following properties:
		\begin{itemize}
			\item (\textbf{Invariance})
			We have  the following commutative diagram:
\begin{equation}		\label{eq70}
\begin{CD}
				HF_{*}(\Sigma,  \varphi_H(\underline{L}),   \underline{L}) @>\mathcal{OC}(\underline{L}, H)>>   \widetilde{PFH}_*(\Sigma,  \varphi_H, d) \\
				@VV \mathcal{I}_{H, G}V @VV \mathfrak{I}_{H, G}V\\
				HF_{*}(\Sigma,  \varphi_G(\underline{L}),  \underline{L})  @> \mathcal{OC}(\underline{L}, G)>>    \widetilde{PFH}_*(\Sigma,  \varphi_G,d)
			\end{CD}
\end{equation}
Then $\mathcal{OC}(\underline{L}, H)$ descends to a $\Lambda$-module  homomorphism 
\begin{equation*}
\mathcal{OC}(\underline{L}) :   HF(\Sigma,\underline{L}) \to \widetilde{PFH}(\Sigma,  d).
\end{equation*}

\item (\textbf{Commutation with the $U$ maps})
$$U \circ  \mathcal{OC}(\underline{L}) =  \mathcal{OC}(\underline{L}) \circ U. $$

			\item
			(\textbf{Non-vanishing})	
			There  exist     nonzero classes    $\sigma_{\underline{L}}\in HF(\Sigma, \underline{L})$   and  $\mathfrak{d}  \in  \widetilde{PFH}(\Sigma,d)$ (see (\ref{eq84})  for the precise construction)   such that 
\begin{equation*}
\mathcal{OC} (\underline{L}) (\sigma_{\underline{L}})=\mathfrak{d}.
\end{equation*}
In particular, the open-closed morphism is non-vanishing. Furthermore, we have  
\begin{equation} \label{eq92}
\begin{split}
&U^{d-g+1} (\mathfrak{d}) =S^{-1}\mathfrak{d},\\
&c_{\underline{L}}(0, \sigma_{\underline{L}}) =0 \mbox{ and }c_d^{pfh}(0, \mathfrak{d})=0. 
\end{split}
\end{equation}
Here  $S^{-1}$ is the Novikov ring action (\ref{eq85}). 
\item
(\textbf{Decreasing spectral invariants}) Suppose that the link $\underline{L}$ is $0$-admissible. Assume that there are  nonzero classes   $\sigma  \in  HF(\Sigma, \underline{L})$ and $\mathfrak{v} \in  \widetilde{PFH}(\Sigma,d)$ such that $\mathcal{OC} (\underline{L} )(\sigma)= \mathfrak{v}.$ Then for any Hamiltonian function $H$, we have
\begin{equation} \label{eq65}
  c_{d}^{pfh}(H, \mathfrak{v}) \le c_{\underline{L}}(H, \sigma).
\end{equation}
		\end{itemize}
	\end{thm}

\begin{remark}
(\ref{eq92}) tells us that $ \mathfrak{d}$ is a  $U$-cyclic  element in the sense  of Definition 2.20  in  \cite{EH}.  Corollary 2 in \cite{CPPZ}  proves the existence of  $U$-cyclic  elements for PFH of  symplectomorphisms  in rational Hamiltonian isotopy classes. Our result here  gives an explicit  $U$-cyclic  element for the  Hamiltonian case. 
\end{remark}

 
 The  reason that Theorem \ref{thm2} gives a $\Lambda$-module homomorphism
rather than an $R$-module homomorphism is simply  that the PFH target is naturally only a
$\Lambda$-module.  The $T$-action corresponds to adding $[B_i]$ to the capping,  whereas this geometric interpretation is   meaningless   in the  PFH setting.  Thus, the  $T$-action is  not well defined  on $ \widetilde{PFH}(\Sigma,d)$, and  PFH has no natural $R$-module structure.   On the other hand, since $\Lambda$ is a subfield of $R$, we can extend the PFH target to be an  $R$-module by a purely algebraic construction (\ref{eq74}). 
After replacing  $\widetilde{PFH}(\Sigma,d)$ by  $\widetilde{PFH}(\Sigma,d)\otimes_{\Lambda} R$,   the open-closed morphisms can be upgraded  to $R$-module homomorphisms as follows:
Define  $\mathcal{OC}(\underline{L}, H; R) : HF(\Sigma,  \varphi_H(\underline{L}), \underline{L}) \to \widetilde{PFH}(\Sigma, \varphi_H, d)\otimes_{\Lambda} R$  by 
\begin{equation} \label{eq106}
\mathcal{OC}(\underline{L}, H; R) (\sigma): =\sum_{i=0}^{d-g} \mathcal{OC}(\underline{L}, H)(T^i \cdot \sigma)T^{-i}. 
\end{equation}
The above definition  is motivated by  Biran-Cornea's construction of quantum inclusion  \cite{BC}. Further discussion on this aspect is given in Appendix \ref{appendixB}.  The map (\ref{eq106}) yields the refinement in  the following theorem.

\begin{thm} \label{thm7}
Let $\underline{L}$ be an $\eta$-admissible link and $\varphi_H$ a $d$-nondegenerate Hamiltonian symplectomorphism. Then 
\begin{equation*}
\mathcal{OC}(\underline{L}, H; R):   HF(\Sigma,  \varphi_H(\underline{L}), \underline{L}) \to \widetilde{PFH}(\Sigma, \varphi_H, d)\otimes_{\Lambda} R
\end{equation*}
is an $R$-module homomorphism  satisfying  the following properties:
		\begin{itemize}
			\item (\textbf{Invariance})
			We have  the following commutative diagram:
\begin{equation}		 
\begin{CD}
				HF_{*}(\Sigma,  \varphi_H(\underline{L}),   \underline{L}) @>\mathcal{OC}(\underline{L}, H; R)>>   \widetilde{PFH}_*(\Sigma,  \varphi_H, d) \otimes_{\Lambda} R\\
				@VV \mathcal{I}_{H, G}V @VV \mathfrak{I}_{H, G}V\\
				HF_{*}(\Sigma,  \varphi_G(\underline{L}),  \underline{L})  @> \mathcal{OC}(\underline{L}, G; R)>>    \widetilde{PFH}_*(\Sigma,  \varphi_G, d)\otimes_{\Lambda} R
			\end{CD}
\end{equation}
Then $\mathcal{OC}(\underline{L}, H; R)$ descends to an  $R$-module  homomorphism 
\begin{equation*}
\mathcal{OC}(\underline{L}; R) :   HF(\Sigma,\underline{L}) \to \widetilde{PFH}(\Sigma,  d)\otimes_{\Lambda} R.
\end{equation*}

\item  (\textbf{Commutation with the $U$ maps})
$$U \circ \mathcal{OC}(\underline{L}; R) = \mathcal{OC}(\underline{L}; R)\circ U.$$
			\item
			(\textbf{Non-vanishing})	
			Let   $\sigma_{\underline{L}}\in HF(\Sigma, \underline{L})$  be the nonzero class in Theorem \ref{thm2}.  Then  each $T^{-i}$ coefficient $\mathcal{OC} (\underline{L}) (T^i\sigma_{\underline{L}})$ of 
\begin{equation*}
\mathfrak{d}_R:= \mathcal{OC} (\underline{L}; R ) (\sigma_{\underline{L}}) \in \widetilde{PFH}(\Sigma, d)\otimes_{\Lambda} R
\end{equation*}
 is nonzero.  Moreover, we have  
 $$U^{d-g+1} ( \mathfrak{d}_R) =S^{-1}\mathfrak{d}_R \mbox{ and } c_d^{pfh}(0, \mathfrak{d}_R; R)=0. $$ 

\item
(\textbf{Decreasing spectral invariants}) Suppose that the link $\underline{L}$ is $0$-admissible. Assume that there are  nonzero classes   $\sigma  \in  HF(\Sigma, \underline{L})$ and $\mathfrak{v} \in   \widetilde{PFH}(\Sigma, d)\otimes_{\Lambda} R$ such that $\mathcal{OC} (\underline{L}; R )(\sigma)= \mathfrak{v}.$  Let $\mathfrak{v}_i$ denote  the $T^{-i}$ coefficient  of $\mathfrak{v}$, where $0\le i\le d-g$.   Then for any Hamiltonian function $H$, we have
\begin{equation*}  
  c_{d}^{pfh}(H, \mathfrak{v}; R) =\max_{0\le  i \le d-g}\{c_{d}^{pfh}(H, \mathfrak{v}_i) -\frac{i}{d-g+1} \} \le c_{\underline{L}}(H, \sigma).
\end{equation*}
		\end{itemize}
	\end{thm}

\begin{remark}
The closed-open/open-closed morphisms are well defined for $\eta$-admissible links  while the third property in Theorems  \ref{thm2}, \ref{thm7} requires  that $\eta=0$.  This is because we do not have the $\eta \ne 0$-version of PFH spectral invariants.  One may try to define such a  version of PFH spectral invariants by adding an extra term  $\eta J_0$ to the action functional as in (\ref{eq57}). However, we cannot prove the Hofer-Lipschitz continuity of such spectral invariants by following the methods in \cite{LZ, CHMSS, FOOO}, because the $J_0$ index could be negative in cobordism (Section 5.5 of \cite{H4}). In particular,   the $\eta \ne 0$-version of PFH spectral invariants  a priori  could depend on the almost complex structures. Therefore, here we only focus  on the case $\eta=0$. 
\end{remark}

Combining Theorem \ref{thm2}, Theorem \ref{thm7}, and Theorem 2 of \cite{GHC2}, we deduce  the following relation between HF spectral invariants and PFH spectral invariants.
\begin{thm}   \label{thm5}
Suppose that the link $\underline{L}$ is $0$-admissible.    For any Hamiltonian function $H \in C^{\infty}(S^1 \times \Sigma)$, we have
\begin{equation*}
 c^{pfh}_{d}(H, \mathfrak{d} )  \le   c_{d}^{pfh}(H, \mathfrak{d}_R; R)  \le c_{\underline{L}}(H, \sigma_{\underline{L}}) \le c_{\underline{L}}(H, e_{\underline{L}}) \le c^{pfh}_{d}(H, \mathfrak{e}) .
\end{equation*}
\end{thm}

In the simplest case where   $\Sigma =\mathbb{S}^2$ is the sphere, we can determine the class $\mathfrak{d}_R$ by using the computations of the $U$ maps on PFH (Proposition \ref{lem25}) and QHF (Theorem \ref{thm9}).
\begin{thm} \label{thm8}
Suppose $\Sigma =\mathbb{S}^2$ is the sphere.  Let $\underline{L}$  be an $\eta$-admissible link. Then
\begin{equation*}
\mathfrak{d}_R= \mathcal{OC} (\underline{L}; R ) (\sigma_{\underline{L}}) = \sum_{i=0}^d \left( U^{d-i} \mathfrak{e}  \right) T^{-i}, 
\end{equation*}
where $\mathfrak{e}$ is the unit of PFH.  In particular, the classes   $\mathfrak{d}_R$ and $\mathfrak{d}$ are  independent of the choice of the admissible link  $\underline{L}$. 
\end{thm}
The computation of $\mathfrak{d}_R$ actually  has the same formal shape as     Biran-Cornea's computation in quantum homology (Theorem 2.3.1 of \cite{BC}).  We will further explain this point in Section \ref{section7}.

 \begin{remark}
The classes $\mathfrak{d} \in \widetilde{PFH}(\Sigma,d)$ and $\mathfrak{d}_R \in \widetilde{PFH}(\Sigma,d)\otimes_{\Lambda}R$ are obtained via open-closed morphisms and $\sigma_{\underline{L}}\in HF(\Sigma,\underline{L})$; thus, a priori, they may depend on $\underline{L}$. It is natural to ask precisely what these classes are and whether they are independent of the admissible link. The computation in Theorem \ref{thm8} answers these questions for the sphere case. For the torus case, we propose a method to compute $\mathfrak{d}_R$ (see the discussion in Section \ref{torus}). Motivated by the computations for the sphere and the torus, we conjecture that $\mathfrak{d}_R$ is independent of the admissible link for any genus.
\end{remark}

According to   Theorem 7.6 of \cite{CHMSS},  the  homogenized link spectral invariants   are  homogeneous quasimorphisms  in the case of the sphere. A byproduct of Theorem  \ref{thm0} is  that the homogenized PFH spectral invariants $\mu_d^{pfh}$ are also  homogeneous quasimorphisms.    Recall that a \textbf{homogeneous quasimorphism} on a group $G$ is a map $\mu: G \to \mathbb{R}$ such that
\begin{enumerate}
\item
$\mu(g^n) =n \mu(g)$;
\item
there exists a constant $D=D(\mu) \ge 0$,  called the \textbf{defect} of $\mu$, satisfying
$$|\mu(gh) - \mu(g) -\mu(h)| \le D. $$
\end{enumerate}

\begin{thm} \label{thm1}
The  homogenized spectral invariants   $\mu_{d}^{pfh}( \cdot ,\mathfrak{e}) : Ham(\mathbb{S}^2, \omega) \to \mathbb{R}$ are  homogeneous quasimorphisms with defect $2$.
\end{thm}

\paragraph{Related results.}
The open-closed morphisms here  is algebraically analogous to the  quantum inclusion introduced by  P. Biran and O. Cornea \cite{BC}.  We refer the discussion on this point to Section \ref{section7}. 

The ``Weyl law'' in Theorem \ref{thm4} in fact    is an analogue of the ``volume property''  for  embedded contact homology.  It was first discovered by D.  Cristofaro-Gardiner, M. Hutchings, and V. Ramos \cite{CHR}.  \textbf{Embedded contact homology} (abbreviated as  ``ECH'')  is a sister version of the periodic Floer homology.  The constructions  of ECH and PFH are the same. The only difference is that they are defined for different  stable Hamiltonian  structures.    If a result holds for  one of them, then one could expect that there should be a parallel result for  the other.  The Weyl law also holds for PFH.  This  is proved by  O. Edtmair and Hutchings \cite{EH}, also by   D. Cristofaro-Gardiner, R. Prasad and B. Zhang \cite{CPZ}  independently.  The ``$U$-cyclic sequences'' assumption in \cite{EH} is proved in \cite{CPPZ}. The   Weyl law   for QHF is discovered  by D. Cristofaro-Gardiner, V. Humili$\grave{e}$re, C. Mak, S. Seyfaddini
and I. Smith  \cite{CHMSS}.


Recently,   
the authors of \cite{CHMSS} show that the  homogenized link spectral invariants satisfy the ``two-term Weyl law'' for a class of  autonomous  Hamiltonian functions \cite{CHMSS1} on the sphere.  Such a law  is proposed by Hutchings  	for ECH capacities \cite{H6}. 
Theorem \ref{thm0} implies that homogenized PFH spectral invariants also satisfy the ``two-term Weyl law'' for the  class of autonomous  Hamiltonian functions in \cite{CHMSS1}.



\paragraph{Outline of the proof.}   To construct the open-closed morphisms, our strategy is  to define  an open-closed morphism $\widetilde{\mathcal{OC}}(\underline{L}, H)$ by counting holomorphic curves in an   ``open-closed''
symplectic manifold $W_H$ with Lagrangian boundary condition $\mathcal{L}_H \subset \partial W_H$. Topologically, $W_H=B\times \Sigma$ and $\mathcal{L}_H =\partial B \times \underline{L}$, where $B$ is a disk with one negative interior puncture and one positive boundary puncture. This can be done by using the ECH index defined in \cite{VPK} and  the  standard holomorphic curve  theory.  See Section \ref{open-closed-cobordism},  Section \ref{open-closed-ECH-index}, and Section \ref{construction-invariance}.  

 However, there is a major difficulty  in    proving  the  invariance  of the open-closed morphisms (Diagram (\ref{eq70})). To prove   the  invariance, one has  to  understand the interaction between the open-closed morphism  and the continuation morphisms, on both the QHF and PFH sides. The only known PFH cobordism maps are defined by Seiberg-Witten theory, whereas the open-closed morphisms  are defined by holomorphic curves and there is no known Seiberg-Witten theory on the  ``open-closed''
symplectic manifolds. Therefore, it is difficult to understand the  interaction between the open-closed morphism and the continuation morphisms.

A simple way to solve the above problem is to define the open-closed morphism alternatively  as 
\begin{equation}\label{eq71}
\mathcal{OC}(\underline{L}, H): =\mathfrak{I}_{H_0,  H} \circ  \widetilde{\mathcal{OC}}(\underline{L}, H_0)   \circ \mathcal{I}_{H, H_0},
\end{equation} 
where $H_0$ is a fixed Hamiltonian function.  By the property of the continuation morphisms, the $\mathcal{OC}$ morphisms  satisfy the diagram (\ref{eq70}). However, this definition creates  another problem  when we prove the final property in Theorem \ref{thm2}  concerning the spectral invariants. The  estimate (\ref{eq65}) comes from the  energy estimates of the holomorphic curves in $W_H$. However,    the definition (\ref{eq71})  \textbf{does not} establish the existence  of holomorphic curves in $W_H$,  while $\widetilde{\mathcal{OC}}(\underline{L}, H) \ne 0$ does  provide  holomorphic curves in $W_H$ and prove the final property of Theorem \ref{thm2}.

By the results in \cite{GHC},   we indeed know how to define the PFH continuation morphisms $\mathfrak{I}_{H, G}$ using holomorphic curves, but under certain technical assumptions  (\ref{assumption1}) and  (\ref{assumption2}) on $H$ and $G$ respectively.   An initial attempt is to prove the  diagram (\ref{eq70}) under these  technical  assumptions. This works and the strategy is just the  usual neck-stretching, homotopy  and gluing argument in Floer theory.  Roughly speaking, we  show that 
\begin{equation} \label{eq86}
\partial \mathcal{M} \cong \mathcal{M}_{PFH} \times \mathcal{M}^0_W \times \mathcal{M}_{QHF} \sqcup (-\mathcal{M}^1_W), 
\end{equation}
 where  $\mathcal{M} $  is a  moduli space of holomorphic  curves in $W=B \times \Sigma$ defined by a family of data, $\mathcal{M}_W^0$ and  $\mathcal{M}_W^1$   are   moduli spaces  of holomorphic  curves in $W$   used to define $\widetilde{\mathcal{OC}}(\underline{L}, H) $ and $\widetilde{\mathcal{OC}}(\underline{L}, G)$ respectively, and    $\mathcal{M}_{QHF}, \mathcal{M}_{PFH}$ are moduli spaces  of curves used to define the continuation morphisms on QHF and PFH.   Equation (\ref{eq86}) implies that  the diagram (\ref{eq70}) holds.   By Proposition 3.7 of \cite{GHC},  another useful fact is that   for a pair of Hamiltonian functions $H, G$,   we can  perform $C^1$ perturbations of  $H, G$ such that they satisfy the desired assumptions  (\ref{assumption1}) and  (\ref{assumption2}).  We prove the above property in Theorem \ref{thm6} and call it partial invariance.  
 
 Fix  $H_0$ satisfying  (\ref{assumption1}).   A nice  implication of the partial invariance  is   that two different definitions coincide, i.e., 
 $$ {\mathcal{OC}}(\underline{L}, H) = \widetilde{\mathcal{OC}}(\underline{L}, H), $$
provided  that   $H$ satisfies the  technical assumption (\ref{assumption2}).  If ${\mathcal{OC}}(\underline{L}, H) = \widetilde{\mathcal{OC}}(\underline{L}, H)\ne 0$, then  it produces  holomorphic curves in $W_H$. This can be used to prove 
 the estimate (\ref{eq65})  under (\ref{assumption2}).  The estimate  (\ref{eq65})  without  (\ref{assumption2})  follows from Proposition 3.7 of \cite{GHC} and Hofer-Lipschitz continuity of the spectral invariants. 
 
Our final mission is to prove the nonvanishing of  the open-closed morphisms.   The idea is  to  prove  (\ref{eq97}) for a special Hamiltonian function $H_0$ and extend to the general case by using the invariance property.  The partial invariance of $ \widetilde{\mathcal{OC}}$ is not sufficient for this task.  Thus, we pick (\ref{eq71}) as our final definition  and require  $H_0$   to satisfy   (\ref{assumption1}). 

The best  candidate for computing  $\widetilde{\mathcal{OC}}(\underline{L}, H_0)$ is taking $H_0$ to be  a small  Morse function $-\frac{1}{\kappa} f$ because  the dynamics of $\varphi_{- \frac{1}{\kappa} f}$  can be fully understood, where $\kappa \gg1$. We can use the index and energy argument to show that the ``constant holomorphic curves"  at  minimum  points are the   curves contributing to  $\widetilde{\mathcal{OC}}(\underline{L}, -\frac{1}{\kappa} f)$.  Specifically, we should  have 
 \begin{equation} \label{eq87}
 \widetilde{\mathcal{OC}}(\underline{L}, -\frac{1}{\kappa} f) ( \mathbf{c}_{\diamondsuit}) = \mathfrak{c}_{{\diamondsuit}} + \mbox{ ``higher order terms"}, 
 \end{equation}
where  $\mathbf{c}_{\diamondsuit} $ and $ \mathfrak{c}_{\diamondsuit}$ are certain  classes represented by minima of  $-f$.  The ``higher order terms"  arise from the holomorphic curves with nontrivial   degree   under the  projection  $W_{-\frac{1}{\kappa} f} \to \Sigma$.  We do not know   what their precise values are. On the other hand,    we can show that  these unknown terms are killed by the PFH cobordism maps induced by $(X_-, \Omega_{X_-})$, where   $(X_-, \Omega_{X_-})$ is the symplectic cobordism (\ref{eq6}). 
Furthermore,    $PFH(X_-, \Omega_{X_-}) (\mathfrak{c}_{\diamondsuit}) \ne 0 $ (cf.  Lemma \ref{lem11}). In sum,   this suggests 
\begin{equation} \label{eq97}
PFH(X_-, \Omega_{X_-}) \circ  \widetilde{\mathcal{OC}}(\underline{L}, -\frac{1}{\kappa} f) ( \mathbf{c}_{\diamondsuit}) \ne 0, 
\end{equation}
  and hence  the left hand side  of (\ref{eq87}) is  nonzero.  The strategy coincides with the characterization of quantum inclusion in Theorem A of \cite{BC}.  More explanation about this point is given in Section \ref{section7}.   We remark that we do not prove (\ref{eq87}) and (\ref{eq97}) in this paper. But  they are  the origin  of the proof of the nonvanishing of open-closed morphisms. 

 However,  an  issue  in the above argument is  that  a Morse function does not satisfy the requirement (\ref{assumption1}).  Hence, several adjustments are required. We apply  Proposition 3.7 of \cite{GHC} to  perturb $-\frac{1}{\kappa} f$  so that it satisfies  (\ref{assumption1}).   We pick $H_0$ to be such a perturbation.  It is worth  noting that the dynamics of  $H_0$  may  be quite different  from those  of  $-\frac{1}{\kappa} f$.  The good news is that such  a perturbation only affects a small part of $-\frac{1}{\kappa} f$ that  lies away from the  minima of  $-\frac{1}{\kappa} f$.  Thus, we can still  show that 
 $$\widetilde{\mathcal{OC}}(\underline{L}, H_0)  ( \mathbf{c}_{\diamondsuit}) = \mathfrak{c}_{\diamondsuit}  + \mbox{ ``unknown  terms''}. $$
Since we only know  very little about the PFH complex of  $H_0$,  we cannot determine whether  the right hand side of the above equation is nonvanishing.  To overcome this  problem,  we   apply the continuation morphism $\mathfrak{I}_{H_0, -\frac{1}{\kappa} f}$  to the above equation and transfer  back to the   familiar  Morse function case.   Eventually, we show that 
 \begin{equation} \label{eq88}
 \mathfrak{I}_{H_0, -\frac{1}{\kappa} f} \circ \widetilde{\mathcal{OC}}(\underline{L}, H_0)  ( \mathbf{c}_{\diamondsuit}) = \mathfrak{c}_{\diamondsuit} + \mbox{ ``higher order  terms''}. 
 \end{equation}
The same argument in the last paragraph shows that
 \begin{equation}  \label{eq104}
 PFH(X_-, \Omega_{X_-}) \circ \mathfrak{I}_{H_0, -\frac{1}{\kappa} f} \circ \widetilde{\mathcal{OC}}(\underline{L}, H_0)  ( \mathbf{c}_{\diamondsuit}) \ne 0, 
  \end{equation}
 and hence the right hand side of (\ref{eq88}) is nonzero.  Actually, (\ref{eq88}) and (\ref{eq104}) are  expected because  we should have the following composition law 
 $$\mathfrak{I}_{H_0, -\frac{1}{\kappa} f} \circ \widetilde{\mathcal{OC}}(\underline{L}, H_0)  =  \widetilde{\mathcal{OC}}(\underline{L}, -\frac{1}{\kappa} f). $$  
 However, again  because   $\mathfrak{I}_{H_0, -\frac{1}{\kappa} f} $  is defined by Seiberg-Witten theory, we cannot establish  this  composition law.   We prove (\ref{eq88}) by analysing  the holomorphic curves and using  the holomorphic curve  axioms of $\mathfrak{I}_{H_0, -\frac{1}{\kappa} f} $.  The above computations are carried out in Section \ref{computing-open-closed}.  Section \ref{proof of thm2} and   Section \ref{proof of thm7} then complete  the proofs  of Theorem \ref{thm2} and Theorem \ref{thm7} respectively.

Theorem \ref{thm5}    is a direct consequence of Theorem \ref{thm2},  Theorem \ref{thm7}, and Theorem 3 of \cite{GHC2}.  Theorem \ref{thm8} follows from  Theorem \ref{thm7},  the computations of $\widetilde{PFH}(\mathbb{S}^2, d)$ and the $U$ maps on QHF (Theorem \ref{thm9}). We give an alternative proof of  Theorem \ref{thm1} by using the computations of $\widetilde{PFH}(\mathbb{S}^2, d)$ and duality in Floer theory \cite{U}.

 \section{Morphisms on QHF}\label{section1}
 In this section, we  define the continuation morphisms,   quantum product and unit  on  $HF(\Sigma, \underline{L})$. The construction here is  similar to that in Section 4 of  \cite{CHT}.

 \subsection{Moduli space of HF curves}
To begin with,  we introduce the definitions of HF curves and relative homology classes. These definitions are mostly paraphrases of those in  Section 4 of \cite{VPK}.

 Let $\dot{D}_m$ be a disk  with boundary punctures $(p_0, p_1, ..., p_m)$.  The order of the punctures is  counter-clockwise.  See Figure \ref{figure2}.
Let $\partial_i \dot{D}_m$ denote  the boundary of $\dot{D}_m$ connecting $p_{i-1}$ and $p_i$ for $1\le i \le m$, and $\partial_{m+1} \dot{D}_m$  the boundary connecting $p_m$ and $p_0$.

Fix a complex structure $j_m$ and a K\"ahler form $\omega_{D_m}$ on $\dot{D}_m$  throughout.    We say that $\dot{D}_m$ is a \textbf{disk with strip-like ends} if  for each  $p_i$ we have a neighborhood $U_i$ of $p_i$ such that
\begin{equation}
( U_i,  \omega_{D_m}, j_m) \cong (  \mathbb{R}_{\epsilon_i} \times [0,1],  ds\wedge dt, j),
\end{equation}
where $j$ is the standard complex structure on $\mathbb{R} \times [0,1]$ that $j(\partial_s) =\partial_t$,  where $\epsilon_i = +$ for $1\le  i\le m$ and $\epsilon_0=-$.  Here $\mathbb{R}_+ =[0, \infty)$ and  $\mathbb{R}_- =(-\infty, 0]$.
	\begin{figure}[h]
		\begin{center}
			\includegraphics[width=10cm, height=6cm]{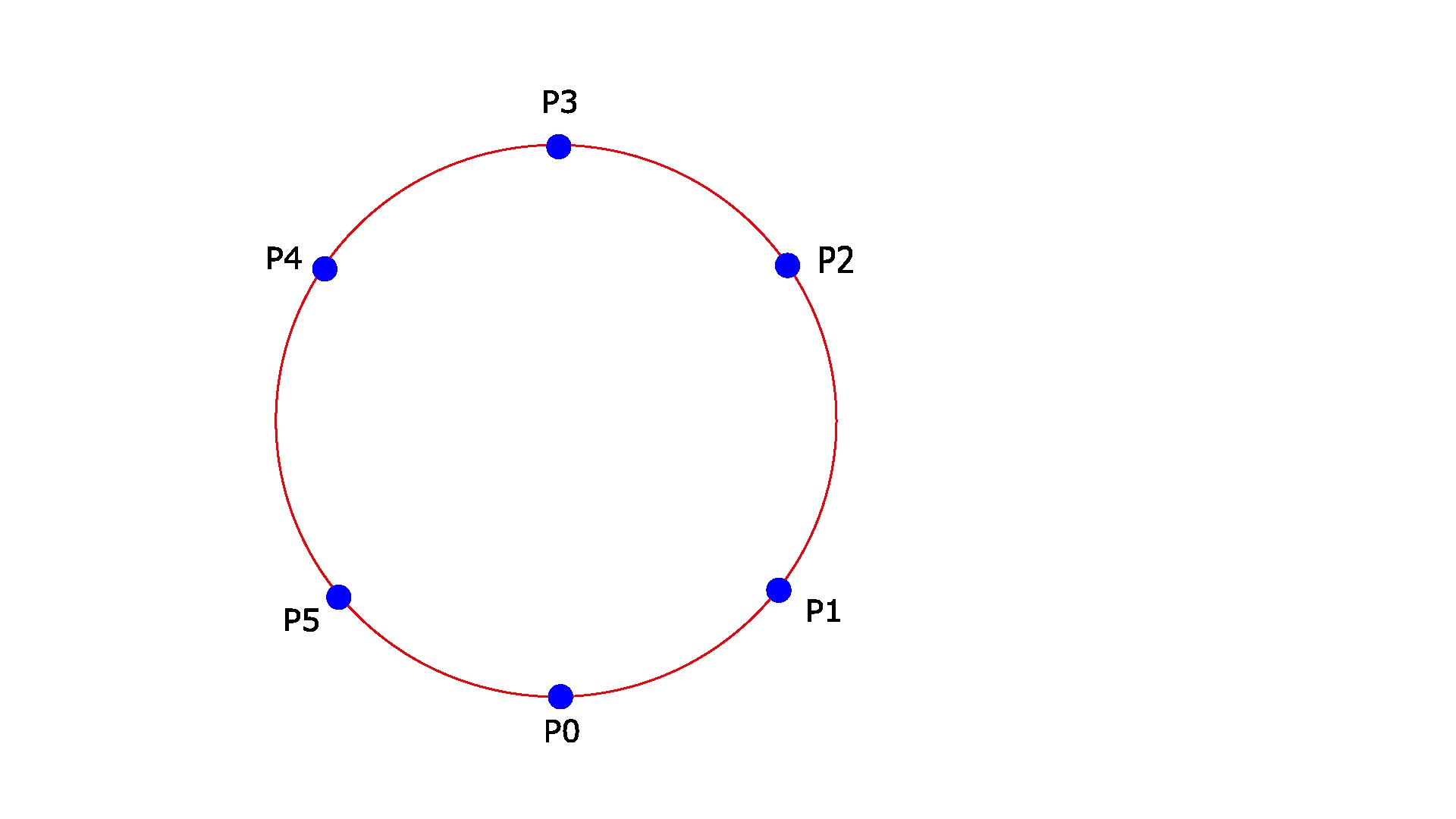}
		\end{center}
		\caption{A picture of the case $m=5$.  }
		\label{figure2}
	\end{figure}

Let $\pi_m: E_m =\dot{D}_m \times \Sigma \to \dot{D}_m$ be the trivial fibration.
 A closed 2-form $\omega_{E_m}$ is called \textbf{admissible} if $\omega_{E_m} \vert_{\Sigma} =\omega$ and  $\omega_{E_m}=\omega$ over the strip-like ends.  
 Note that $\Omega_{E_m} = \omega_{E_m} + \omega_{D_m}$ is a symplectic form on $E_m$ if $ \omega_{D_m}$ is large enough.  As a result, $(\pi: E_m \to \dot{D}_m, \Omega_m)$ over $U_i$ can be identified  with
 \begin{equation} \label{eq1}
( \pi_m: U_{i} \times \Sigma \to U_i,  \Omega_{D_m}) \cong (  \pi_{\mathbb{R} \times [0,1]}: \mathbb{R}_{\epsilon_i} \times [0,1] \times \Sigma \to \mathbb{R}_{\epsilon_i} \times [0,1],  \omega+  ds\wedge dt).
\end{equation}
 We call it  a  \textbf{(strip-like) end} of $(E_m, \Omega_{E_m})$ at $p_i$.

 Fix an $\eta$-admissible link $\underline{L}$. Let $\{\underline{L}_{p_i}\}_{i=0}^m$  be links such that $\underline{L}_{p_i}=\varphi_{H_i}(\underline{L})$ for some Hamiltonian functions $H_i$.    Let $\mathcal L\subset\partial E_m$ be a Lagrangian submanifold in $\partial E_m$ satisfying  the following conditions:

\begin{enumerate} [label=\textbf{C.\arabic*}]
\item \label{C0}
Let $\mathcal{L}_i= \mathcal{L} \vert_{\partial_i \dot{D}_m} \subset \pi^{-1}(\partial_i \dot{D}_m)$. $\mathcal{L}_i$ consists of  a disjoint union of $d$ Lagrangian submanifolds.
\item \label{C1}
For $1\le i\le m$, over the end at $p_i$ (under the identification \ref{eq1}), we have
  \begin{equation*}
 \mathcal{L} = (\mathbb{R}_{+} \times \{0\}  \times \underline{L}_{p_{i-1}} )\cup(\mathbb{R}_{+} \times \{1\} \times \underline{L}_{p_{i}} ).
\end{equation*}

\item \label{C2}
 Over the end at $p_0$ (under the identification \ref{eq1}), we have
  \begin{equation*}
 \mathcal{L} = (\mathbb{R}_{-} \times \{0\}  \times \underline{L}_{p_{0}} )\cup(\mathbb{R}_{-} \times \{1\} \times \underline{L}_{p_{m}} ).
\end{equation*}

 \item \label{C3}
 For any $z \in \partial \dot{D}_m$,    $\underline{L}_z=\mathcal{L} \cap \pi_m^{-1}(z)$  is an $\eta$-admissible link that is Hamiltonian isotopic to  the  fixed $\eta$-admissible  link $\underline{L}$.

 \end{enumerate}

Let $(E_m, \Omega_m, \mathcal{L}_m)$  and $(E_n, \Omega_n, \mathcal{L}_n)$  be two symplectic fibrations.  Suppose that the pair of links $(\underline{L}_{p_{i-1}},\underline{L}_{p_{i}} )$ at the $i$-th positive end of  $(E_n, \Omega_n, \mathcal{L}_n)$  coincides with  the links at the negative end of   $(E_m, \Omega_m, \mathcal{L}_m)$.  
Fix $R \ge 0$. Define the $R$-\textbf{stretched composition} $(E, \Omega, \mathcal{L}) :=  (E_n, \Omega_n, \mathcal{L}_n) \circ_R (E_m, \Omega_m, \mathcal{L}_m) $ by
\begin{equation}  \label{eq33}
(E, \Omega, \mathcal{L}) =   (E_n, \Omega_n, \mathcal{L}_n)\vert_{s_+ \le R}  \cup_{s_+-R =s_-+R} (E_m, \Omega_m, \mathcal{L}_m) \vert_{s_-\ge -R}.
\end{equation}  
Most of the time, the number $R$ is not important, so  we suppress  it from the notation.

\begin{definition}
An almost complex structure  is called  adapted to the fibration if
\begin{enumerate}
\item
$J$ is $\Omega_{E_m}$-tame.
\item
Over the strip-like ends, $J$ is $\mathbb{R}_s$-invariant,  $J(\partial_s) =\partial_t$, $J$ preserves $T\Sigma$ and $J \vert_{T\Sigma}$ is compatible with $\omega$.
\item
$\pi_m$ is complex linear with respect to $(J, j_m)$, i.e., $j_m \circ d\pi_m =d\pi_m \circ J.$
 \end{enumerate}
\end{definition}
Let $\mathcal{J}_{tame}(E_m)$ denote the set of almost complex structures adapted to the fibration. Using the admissible 2-form $\omega_{E_m}$, we have a  splitting $TE_m =TE_m^{hor} \oplus TE_m^{vert}$, where $TE_m^{vert}: =\ker d\pi_m$ and $TE_m^{hor}:=\{v \in TE_m \vert \omega_{E_m}(v ,w) =0, w \in T^vE_m\}$.  With respect to this splitting,  an  almost complex structure $J \in \mathcal{J}_{tame}(E_m)$ can be written as
$$J =\left(
 \begin{matrix}
   J^{hh} & 0 \\
   J^{hv} & J^{vv}
  \end{matrix}
  \right). $$ Therefore, $J$ is $\Omega_{E_m}$-compatible if and only if $J^{hv}=0$.

   Let $\mathcal{J}_{comp}(E_m) \subset\mathcal{J}_{tame}(E_m) $ denote the  set of almost complex structures  that are adapted to the fibration  and    $\Omega_{E_m}$-compatible.  Later,  we will use the almost complex structures in $\mathcal{J}_{comp}(E_m)$ for  computations.

We now define the concept of holomorphic curves in $(E_m, \Omega_{E_m}, \mathcal{L}_m)$.
\begin{definition} \label{definition1}
		Fix Reeb chords $\mathbf{y}_i \in  \underline{L}_{p_{i-1}} \cap  \underline{L}_{p_{i}}$ and $\mathbf{y}_0 \in  \underline{L}_{p_{0}} \cap  \underline{L}_{p_{m}}$. Let $(\dot{F},j)$ be a Riemann surface (possibly disconnected) with   boundary punctures. A \textbf{d-multisection} is a smooth map $u: (\dot{F}, \partial \dot{F}) \to E_m $ such that
		\begin{enumerate}
\item
 Each irreducible component of $\dot{F}$ has at least one puncture. 
			\item
			$u(\partial {\dot{F}}) \subset \mathcal{L}$.  Let $\{L_j^i\}_{i=1}^d$ be the connected components of $\mathcal{L} \vert_{\partial_j \dot{D}_m}$. For each $1 \le i\le d$,   $u^{-1}(L_j^i)$ consists of exactly one component of $\partial \dot{F}$.
			\item
			For $1\le i \le m$, $u$ is asymptotic to $\mathbf{y}_i$ as $s \to \infty $.
			\item
			$u$ is asymptotic to $\mathbf{y}_0$ as $s \to -\infty $.
			
		\end{enumerate}
The integer $d$ is called the \textbf{degree} of $u$. Fix an almost complex structure $J \in \mathcal{J}_{tame}(E_m)$.  If $u$ is a $J$-holomorphic $d$-multisection, then $u$ is called an \textbf{HF curve}.

\begin{remark} \label{remark5}
In our definition of HF curves, each irreducible component of $\dot{F}$ contains at least one puncture. This excludes the possibility that an HF curve contains an irreducible component entirely within a fiber or that the whole curve is contained within a fiber. As a result, an HF curve has at least $m+1$ ends.

One may define a concept called  ``generalized HF curves" to include these possibilities. However, to define the cobordism maps on QHF, we  do require holomorphic curves with ends. Moreover, if a holomorphic curve  of the form $u=u_{\star} \bigcup v$, where $u_{\star} $ is an HF curve  and $v$  is a curve contained within a fiber, then   Theorem \ref{thm3} tells us that the ECH index of $[u]$ is at least two.  Since we only need to consider holomorphic curves with $I=0$ or $I=1$, for our purpose, there is no difference between  using ``HF curves'' and ``generalized HF curves'' .

On the other hand, if a sequence of  HF curves converges to a broken holomorphic curve $\mathbf{u}$ in the sense of \cite{FYHKE}, then each level of $\mathbf{u}$ is a generalized HF curve. See Lemma 2.9 of \cite{GHC2} for explanation of SFT compactness in the HF setting. 
\end{remark}

	\end{definition}
	Let $H_2(E_m, \textbf{y}_1,...,\textbf{y}_m,   \textbf{y}_0)$ be the set of continuous maps $$u: (\dot{F}, \partial \dot{F}) \to \left(\check{E}_m,  \mathcal{L}\bigcup_{i=1}^{m}\{\infty\} \times \textbf{y}_i \cup \{-\infty\} \times \textbf{y}_0 \right)$$satisfying the conditions  $ 2), 3), 4)$ in Definition \ref{definition1},  and modulo  an equivalence relation $\sim$, where $\check{E}_m$ is the compactification of $E_m$ by adding $\{\infty\}/\{-\infty\} \times [0,1]\times \Sigma$  to each positive/negative end.  Here  $u_1 \sim u_2  $ if and only if their compactifications   are equivalent in $H_2(\check{E}_m, \mathcal{L}\bigcup_{i=1}^{m}\{\infty\} \times \textbf{y}_i \cup \{-\infty\} \times \textbf{y}_0; \mathbb{Z})$.  	An element of  $H_2(E_m, \textbf{y}_1,...,\textbf{y}_m,   \textbf{y}_0)$   is called  \textbf{a relative homology class}. An easy generalization is that one could replace the Reeb chords  by the reference chords $\textbf{x}_H$ in the above definition. By definition, $H_2(E_m, \textbf{y}_1,...,\textbf{y}_m,   \textbf{y}_0)$ is an affine space of $H_2(E_m, \mathcal{L}_m, \mathbb{Z}). $ By the exact sequence
\begin{equation*}
...\to H_2(E_m, \mathbb{Z}) \xrightarrow{j_*}  H_2(E_m, \mathcal{L}_m; \mathbb{Z}) \xrightarrow{\partial_*} H_1(\mathcal{L}_m,  \mathbb{Z}) \xrightarrow{i_*}  H_1(E_m,  \mathbb{Z}) \to...,
\end{equation*}
and the diffeomorphism $(E_m, \mathcal{L}_m) \cong (\dot{D}_m \times \Sigma, \partial \dot{D}_m \times \underline{L})$,
$H_2(E_m, \mathcal{L}_m; \mathbb{Z})$ is generated by $[B_i^{\tau_j}]$ ($1\le i \le k+1$, $1\le j \le m+1$), where $\tau_j \in \partial \dot{D}_m$ and $B_i^{\tau_j}$ is the closure of $\Sigma \setminus (\pi_m^{-1}(\tau_j) \cap \mathcal{L}_m )= \bigcup_{i=1}^{k+1} \mathring{B}_i^{\tau_j}.$

Fix $A \in H_2(E_m, \textbf{y}_1,...,\textbf{y}_m,   \textbf{y}_0)$. We denote the moduli space of HF curves by $\mathcal{M}^J( \mathbf{y}_1, ...,\mathbf{y}_m; \mathbf{y}_0, A)$.

	\subsection{Fredholm index, ECH index and $J_0$ index} Now we  define   three  types of index defined for an HF curve, called Fredholm index, ECH index and $J_0$ index.  These definitions essentially follow  Section 4.4 and Section 4.5 of \cite{VPK}.
	
	To begin with, fix a trivialization of $u^* T \Sigma$ as follows. Fix a non-singular vector $v$ on $\underline{L}$.  By using the symplectic parallel transport,      $(v, j_{\Sigma}(v))$ gives  a trivialization of  $T\Sigma \vert_{\mathcal{L}}$, where $ j_{\Sigma}$ is a complex structure on $\Sigma$.    We extend the trivialization  arbitrarily along $\textbf{y}_{i}$.  Such a trivialization is denoted by $\tau$.
	
	Define a real line bundle $\mathfrak{L}$ over $\partial F$ as follows. Take  $\mathfrak{L} \vert_{\partial \dot{F}} : = u^*(T \mathcal{L} \cap T\Sigma)$. Extend $\mathfrak{L}$ to $\partial F -\partial \dot{F}$ by rotating  in  the counter-clockwise direction from  $u^*T L_{p_{j-1}}^i$ to  $u^*T L^i_{p_{j}}$ by the minimum amount. Then $(u^*T\Sigma, \mathfrak{L})$ forms  a bundle pair over $\partial  F$. With respect to the trivialization $\tau$, we have a  well-defined  Maslov index $\mu_{\tau}(u):= \mu(u^* T\Sigma, \mathfrak{L}, \tau)$ and 	relative Chern number  $ c_1(u^*T\Sigma, \tau). $
The number $2c_1(u^*T\Sigma, \tau) + \mu_{\tau}(u)$ is  independent  of the trivialization  $\tau$.
	
	The \textbf{Fredholm index} of an HF curve is defined by
	\begin{equation*}
		\operatorname{ind} u:= -\chi(F) +  2c_1(u^*T\Sigma, \tau) + \mu_{\tau}(u) + d(2-m).
	\end{equation*}
	The above index formula can be obtained by the doubling argument in Proposition 5.5.2  of \cite{VPK}.

To define the ECH index, we first need to define the relative self-intersection number as follows. 	Given $A \in   H_2(E_m, \textbf{y}_1,...,\textbf{y}_m,   \textbf{y}_0)$, an oriented immersed    surface   $C \subset  E_m$ is a $\tau$-representative of $A$ if
	\begin{enumerate}
		\item
		$C$ intersects the fibers positively along $\partial C$;
		\item
		$\pi_{[0,1] \times \Sigma } \vert_C$ is an embedding near infinity; 
		\item
		$C$ satisfies the $\tau$-trivial conditions in the sense of Definition 4.5.2 in \cite{VPK}.
	\end{enumerate}
	
	Let $C$ be a $\tau$-trivial representative of $A.$ Let $\psi$ be a section of the normal bundle $N_C$ such that $\psi \vert_{\partial C} =J\tau$.  Let $C'$ be a push-off of $C$ in the direction of $\psi$. Then the \textbf{relative self-intersection number } is defined by
	\begin{equation*}
		Q_{\tau}(A) := \#( C\cap C').
	\end{equation*}

Let  $A \in   H_2(E_m, \textbf{y}_1,...,\textbf{y}_m,   \textbf{y}_0)$  be a relative homology class. Define the \textbf{ECH index}  of $A$ by
\begin{equation*}
		I(A):=  c_1(T\Sigma \vert_A, \tau) +Q_{\tau} (A) +\mu_{\tau}(A) + d(1-m).
	\end{equation*}
Note that $I(A)$ is  independent of the choices of $\tau$ and $\tau$-representative of $A$.

The following theorem  summarizes  all the properties of the ECH index that we need.
\begin{theorem} 
 \label{thm3}
We have the following statements for the ECH index:
\begin{itemize}
\item
Let $u $ be an irreducible  $J$-holomorphic HF curve. Then the ECH index and the Fredholm index satisfy the following properties:
$$I(u) =\operatorname{ind} u + 2\delta(u), $$
where $\delta(u) \ge 0$   is a count of the singularities of $u$ with positive weight.
Moreover, $I(u)=\operatorname{ind}  u$ if and only if $u$ is embedded.

\item
Let $u=\bigcup_a u_a$ be an HF curve whose components $u_a$ are irreducible. Then
\begin{equation*}
I(u)=\sum_a I(u_a) +2\sum_{a<b}\#(u_a \cap u_b).
\end{equation*}
\item
If $J$ is generic in the sense of Lemma \ref{lem18}, then $I(u) \ge 0$.   
\item
Let $A, A' \in H_2(E_m, \mathbf{y}_1,..\mathbf{y}_m,   \mathbf{y}_0)$  be relative homology classes  such that
$$A'-A= \sum_{i=1}^{k+1} c_i [{B}^z_i] + n[\Sigma],$$
where $z\in \partial \dot{D}_m$ and  the $B_z^i$ are the closures of $\Sigma \setminus \underline{L}_z = \cup_{i=1}^{k+1} \mathring{B}_i^z$.  Then
\begin{equation*}
I(A')=I(A) +2\sum_{i=1}^{k+1} c_i +2n(k+1).
\end{equation*}
\end{itemize}
\end{theorem}

\begin{proof}
We prove the statements in the theorem one by one as follows.
\begin{itemize}
\item
Let $u$ be an irreducible HF curve. By the same argument as in  Lemma 4.5.9  of  \cite{VPK}, we have the following adjunction formula
\begin{equation} \label{eq36}
\begin{split}
c_1(u^*TE_m, (\tau, \partial_t)) &= c_1(du(TF), \partial_t) + c_1(N_u, J\tau)\\
&=\chi(F)-d + Q_{\tau}(u) -2\delta(u),
\end{split}
\end{equation}
where $N_u$ is the normal bundle of $u$ and $\partial_t$ is a trivialization of $T\dot{D}_m$ such that it agrees with $\partial_t$ over the ends. On the other hand,  we have $$c_1(u^*TE_m, (\tau, \partial_t))  = c_1(u^*T\Sigma, \tau) + c_1(u^*T\dot{D}_m,  \partial_t)  =c_1(u^*T\Sigma, \tau).$$
Combine the above two equations; then we obtain the ECH equality $I(u) =\operatorname{ind} u + 2\delta(u)$.

\item
To prove the second statement,  without loss of generality,  we  assume that $u=u_0 \bigcup u_1$, where $u_0$ and $u_1$ are distinct  irreducible HF curves.  Let $d$, $d_0$ and $d_1$  denote the degrees of  $u$, $u_0, u_1$ respectively. Then $d=d_0+d_1$.
By definition, the Chern number and  Maslov index are additive, and the relative self-intersection number is quadratic  in the sense that
\begin{equation} \label{eq37}
 Q_{\tau}(u) =Q_{\tau}(u_0) +Q_{\tau}(u_1) + 2\#(u_0 \cap u_1).
\end{equation}
Therefore, we have
\begin{equation*}
\begin{split}
I(u) &=  c_1(T\Sigma \vert_u, \tau) +Q_{\tau} (u) +\mu_{\tau}(u) + d(1-m)\\
&= c_1(T\Sigma \vert_{u_0}, \tau) + c_1(T\Sigma \vert_{u_1}, \tau)+ \mu_{\tau}(u_0)  +\mu_{\tau}(u_1) \\
 &+Q_{\tau} (u_0) +Q_{\tau} (u_1) + 2\#(u_0 \cap u_1) + d_0(1-m) +d_1(1-m) \\
&= I(u_0) +  I(u_1) + 2\#(u_0 \cap u_1).
\end{split}
\end{equation*}
\item
Let $u=\bigcup_a u_a$ be an HF curve, where each $u_a$ is irreducible.  Since $J$ is generic, $\operatorname{ind} u_a \ge 0$. By the first bullet, we have $I(u_a) \ge \operatorname{ind} u_a  \ge 0$. By intersection positivity of holomorphic curves, $\#(u_a \cap u_b) \ge 0$ for $a\ne b$.  Therefore, $I(u) \ge 0$ follows from the second bullet.

 \item
We now prove the final  statement of the theorem. Let $z, z'\in\partial \dot{D}_m$ be two points in  the same component of $\partial \dot{D}_m$. Note that   $$A+ \sum_{i=1}^{k+1} c_i [{B}^{z}_i] + n[\Sigma] = A + \sum_{i=1}^{k+1} c_i [{B}^{z'}_i] + n[\Sigma].$$
Hence, we may assume $z$ lies in the strip-like ends of $D_m$.  Let $u$ be a  $\tau$-representative of $A$. For $1\le i\le k$, by the argument in Lemma 2.4 of \cite{GHC2}, we have a $\tau$-representative $u'$ such that $[u']=A+[B_i^z]$.  Moreover, we have
\begin{equation} \label{eq38}
\begin{split}
&c_1(T\Sigma \vert_{u'}, \tau) = c_1(T\Sigma \vert_{u'}, \tau) +1, \mu_{\tau}(u') = \mu_{\tau}(u)\\
 &Q_{\tau}(u')= Q_{\tau}(u)+1,  \mbox{ and }\delta(u')=\delta(u).
\end{split}
\end{equation}
Therefore, $I(u') =I(u) +2$. Perform this construction $c_i$ times for each $1\le i \le k$.  Then we obtain  $I(A+ \sum_{i=1}^{k} c_i [{B}^z_i] )=I(A) + 2 \sum_{i=1}^{k} c_i. $ By definition,
\begin{equation} \label{eq48}
\begin{split}
I(A + n[\Sigma]) &= c_1(T\Sigma \vert_{A+n[\Sigma]}, \tau) +Q_{\tau} (A+n[\Sigma]) +\mu_{\tau}(A) + d(1-m)\\
&= c_1(T\Sigma \vert_{A}, \tau) +n c_1(T\Sigma \vert_{[\Sigma]}) + Q_{\tau} (A) + n^2[\Sigma]\cdot[\Sigma] \\
&+ 2n \#(A\cap \Sigma) +\mu_{\tau}(A) + d(1-m)\\
&= I(A) +n(\chi(\Sigma) + 2d)\\
&=I(A) + 2n(d-g+1) =I(A) +2n(k+1).
\end{split}
\end{equation}
Note that $[\Sigma] =\sum_{i=1}^{k+1}[B^z_i]$. Therefore, we have
\begin{equation} \label{eq50}
\begin{split}
I(A) +2c_{k+1}(k+1) &=  I(A + c_{k+1}[\Sigma])  =I(A + \sum_{i=1}^{k+1}c_{k+1}[B^z_i])\\
&=I(A + c_{k+1}[B^z_i]) +2 \sum_{i=1}^{k}c_{k+1} \\
&= I(A + c_{k+1}[B^z_i]) +2k c_{k+1}.
\end{split}
\end{equation}
This implies that  $ I(A + c_{k+1}[B^z_i]) =I(A)+ 2 c_{k+1}.$

\end{itemize}
\end{proof}

\begin{remark}
The first statement of Theorem \ref{thm3} is called  the ECH equality.    When $m=1$,  it agrees with  Theorem 4.5.13 in \cite{VPK}.  They are   analogues  of the ECH inequality   discovered by M. Hutchings  (Theorem 4.15 of \cite{H2}).

In contrast  to Theorem 4.15 of \cite{H2},  our result here is an equality rather  than an  inequality.   The reason is that the Reeb chords are simple in our setting.  Then the Maslov-index terms are the same in the ECH and Fredholm indices.  If one allows  the Reeb chords to be multiply covered,  then we get an inequality (see Theorem 7 of  \cite{CEWYZ}).
\end{remark}

We follow Hutchings's approach   to define the \textbf{$J_0$ index}.  It measures the Euler characteristic of the holomorphic curves.
 The construction of $J_0$ here more or less comes from the relative adjunction formula.  A similar concept, called the $J_+$ index in the usual Heegaard Floer homology,  is  defined by  C. Kutluhan, G. Matic, J. Van
Horn-Morris \cite{KMHW}.   Comparing  with  $J_0$ index, the $J_+$ in index \cite{KMHW}  also  includes extra  terms related
to the Reeb chords $\mathbf{y}_{\pm}$.

Fix a relative homology class $A \in  H_2(E_m, \textbf{y}_1,...,\textbf{y}_m,   \textbf{y}_0)$. The $J_0$ index of $A$ is defined by
\begin{equation*}
\begin{split}
J_0(A):=-c_1(TE_m\vert_A, (\tau, \partial_t)) +Q_{\tau}(A).
\end{split}
\end{equation*}

The following lemma summarize the properties of $J_0$. These properties are parallel to those  of  ECH  index in Theorem \ref{thm3}.
\begin{lemma} \label{lem12}
The index $J_0$ satisfies the following properties:
\begin{enumerate}
\item
Let $u: \dot{F} \to E_m$ be an irreducible  HF curve with degree $d$, then
\begin{equation*}
J_0(u)=-\chi(F) + d+ 2\delta(u).
\end{equation*}

\item
Let $u=\bigcup_a u_a$ be an HF curve and each $u_a$ is irreducible. Then
\begin{equation*}
J_0(u) = \sum_a J_0(u_a) +2 \sum_{a < b}  \#(u_a  \cap u_b).
\end{equation*}

\item
If a class $A$ supports an HF curve, then $J_0(A) \ge 0$.
\item
Let $A, A' \in H_2(E_m, \mathbf{y}_1,...,\mathbf{y}_m,   \mathbf{y}_0)$. Suppose that $A'-A=n[\Sigma]+\sum_{i=1}^{k+1} c_i[B^z_i]$. 
 Then
\begin{equation*}
J_0(A')=J_0(A) + 2c_{k+1}(d+g-1)+ 2n(d+g-1).
\end{equation*}
\end{enumerate}
\end{lemma}	
\begin{proof}
We  demonstrate the validity of these statements one by one.
\begin{itemize}
\item
By definition and the adjunction formula  (\ref{eq36}), we have
\begin{equation*}
\begin{split}
J_0(u)&=-c_1(TE_m\vert_u, (\tau, \partial_t)) +Q_{\tau}(u)\\
&=-\chi(F) +d -Q_{\tau}(u) +2\delta(u)  +Q_{\tau}(u)\\
&=-\chi(F) +d   +2\delta(u).
\end{split}
\end{equation*}
\item
To prove the second statement, without loss of generality, assume  that $u=u_0 \bigcup u_1$ has two irreducible components. Since   Chern number is additive and  the relative self-intersection is quadratic  (\ref{eq37}), we have
\begin{equation*}
\begin{split}
J_0(u) &=  -c_1(T\Sigma \vert_{u_0\cup u_1}, \tau) +Q_{\tau} (u_0 \cup u_1)\\
&= -c_1(T\Sigma \vert_{u_0}, \tau) +Q_{\tau} (u_0)  -  c_1(T\Sigma \vert_{u_1}, \tau)+ Q_{\tau} (u_1)  + 2\#(u_0 \cap u_1)\\
&= J_0(u_0) +  J_0(u_1) + 2\#(u_0 \cap u_1).
\end{split}
\end{equation*}
\item
If $u$ is irreducible, then by the first bullet, we have
\begin{equation*}
\begin{split}
J_0(u)=2g(F) -2 + \# \partial F +d +2\delta(u).
\end{split}
\end{equation*}
Since $u$ has at least one boundary, $d \ge 1$,  and $\delta(u) \ge 0$,   we have $$ \# \partial F +d +2\delta(u) \ge 2. $$ Then $J(u) \ge0$.  If $u =\bigcup_a u_a$  consists of several distinct   irreducible components, then $J_0(u) \ge 0$ follows from the second bullet and intersection positivity of holomorphic curves.

\item
Let $u$ be a  $\tau$-representative  of $A$.  From  the proof of  Theorem \ref{thm3},  we know that there is  $\tau$-representative  $u'$ with relative homology class  $[u'] =A+ \sum_{i=1}^{k} c_i[B^z_i]$ for $1\le i \le k$.  By the computations (\ref{eq38}) and definition, we have $J_0(u') = J_0(u)$. In other words, $J_0(A+ \sum_{i=1}^{k} c_i[B^z_i])  =J_0(A)$.  A geometric interpretation of this formula is that adding a disk does not change the topology  of a $d$-multisection.  We now compute contribution of $n[\Sigma]$ to the $J_0$ index.  By definition,
\begin{equation} \label{eq49}
\begin{split}
J_0(A + n[\Sigma]) &= -c_1(TE_m \vert_{A+n[\Sigma]}, \tau) +Q_{\tau} (A+n[\Sigma])  \\
&=- c_1(TE_m \vert_{A}, \tau) -n c_1(TE_m \vert_{[\Sigma]}) + Q_{\tau} (A) + n^2[\Sigma]\cdot[\Sigma] + 2n\#(A\cap \Sigma)\\
&=J_0(A) -n\chi(\Sigma)  + 2nd \\
&= J_0(A) +2n(d+g-1).
\end{split}
\end{equation}
Finally, the contribution of $c_{k+1}[B_{k+1}^z]$ follows from the same trick as (\ref{eq50}). We have $$J_0(A+ c_{k+1}[B^z_{k+1}])  = J_0(A) +2c_{k+1}(d+g-1). $$

\end{itemize}

\end{proof}

\subsection{QHF cobordism maps}
With the above preliminaries, we now define the cobordism maps   on QHF in Proposition \ref{lem1}. 
It  is similar to  the result in Section 4 of \cite{CHT}, where  V. Colin, K. Honda, and  Y. Tian establish  the $A_{\infty}$ structure on high dimensional Heegaard Floer homology. The definition of high dimensional Heegaard Floer homology is essentially the same as  QHF. The difference is that the symplectic manifolds and Lagrangian submanifolds are exact in the setting of \cite{CHT}, and hence no bubbles exist.  In our setting, we can rule out the bubbles by the index  computations in Lemma  \ref{lem15}, and the rest of argument is the same as those of \cite{CHT}.


First, note that every HF curve must be simple because its ends are asymptotic to   Reeb chords, and the Reeb chords are embedded. Then by the standard Sard-Smale argument, we have the following transversality result.
\begin{lemma} \label{lem18}
There exists a Baire subset $\mathcal{J}_{tame}^{reg}(E_m)$ of $\mathcal{J}_{tame}(E_m)$ such that  the moduli space $\mathcal{M}^J( \mathbf{y}_1, ...,\mathbf{y}_m; \mathbf{y}_0)$ is a manifold of expected dimension.
\end{lemma}
\begin{proof}
The proof follows from the same argument in  Lemma 9.12  of \cite{H1}.
\end{proof}
We call  an  almost complex structure in $\mathcal{J}_{tame}^{reg}(E_m)$ \textbf{a generic almost complex structure.}

Combining the above transversality result and the properties of ECH index in Theorem \ref{thm3}, we obtain the following compactness results for HF curves with lower ECH index.
\begin{lemma}  \label{lem15}
Let $J \in \mathcal{J}_{tame}^{reg}(E_m)$ be a generic almost complex structure.
\begin{itemize}
\item
If $I(A) =0$, then  $\mathcal{M}^J( \mathbf{y}_1, ...,\mathbf{y}_m; \mathbf{y}_0, A)$  is a set of finite points.
\item
If $I(A) =1$, then  $\mathcal{M}^J( \mathbf{y}_1, ...,\mathbf{y}_m; \mathbf{y}_0, A)$  is a $1$-dimensional manifold with boundaries 
\begin{equation*}
\begin{split}
&\partial \mathcal{M}^J( \mathbf{y}_1, ...,\mathbf{y}_m; \mathbf{y}_0, A) \\
&= \bigcup_{i=1}^m \bigcup_{A_{i1}\#A_{i2} =A, I(A_{i1})= 1, I(A_{i2})=0 }  \mathcal{M}^J(\mathbf{y}_i, \mathbf{y}'_i, A_{i1}) /\mathbb{R} \times  \mathcal{M}^J( \mathbf{y}_1, ...,\mathbf{y}'_i,...,\mathbf{y}_m; \mathbf{y}_0, A_{i2})   \\
&   \bigcup_{A_{1}\#A_{2} =A, I(A_{1})= 0, I(A_{2})=1 }  \mathcal{M}^J( \mathbf{y}_1,...,\mathbf{y}_m; \mathbf{y}'_0, A_{1}) \times   \mathcal{M}^J(\mathbf{y}'_0, \mathbf{y}_0, A_{2}) /\mathbb{R}.
\end{split}
\end{equation*}
\end{itemize}
\end{lemma}
\begin{proof}
We divide the proof into the following steps. 
\begin{itemize}
\item
Assume $I(A) =0$. By the ECH equality in Theorem \ref{thm3}, $I(u)=0$ implies that $\operatorname{ind} u =0$ and $u$ is embedded.
By Lemma \ref{lem18}, it suffices to show  that  $\mathcal{M}^J( \mathbf{y}_1, ...,\mathbf{y}_m; \mathbf{y}_0, A)$ is compact.

\item
Consider  a sequence of HF curves   $\{u_n: \dot{F}_n \to E_m\}_{n=1}^{\infty}$  in $\mathcal{M}^J( \textbf{y}_1, ...,\textbf{y}_m; \textbf{y}_0, A)$.  By Lemma \ref{lem12},   we may assume that the topological  types of $\{\dot{F}_n\}_{n=1}^{\infty}$ are  fixed.  Applying the SFT compactness  in
Section 11.3  of  \cite{FYHKE}  to $\{u_n\}_{n=1}^{\infty}$,  passing to a subsequence, $\{u_n\}_{n=1}^{\infty}$ converges to a broken holomorphic curve $\mathbf{u}$. We refer readers to Page 29--30 of \cite{GHC2} for explanation of the convergence in this setting. 

   Let $u^0$ denote the level in $E_m$.  Then $u^0= u^0_{\star}  \bigcup_i v_i$, where $ u^0_{\star} $ is an HF curve and $v_i$ are bubbles   arising   from pinching an arc or an interior simple curve in $\dot{F}_n$. Since $\pi_m$ is complex linear,  by the open mapping theorem, $v_i$ lies in a fiber $\pi_m^{-1}(\tau_i)$, where $\tau_i  \in  \dot{D}_m$.  If $\tau_i \in \partial  \dot{D}_m$, then $v_i$ is a holomorphic curves in $\Sigma$ with boundary in $\underline{L}_{\tau_i} = \pi_m^{-1}(\tau_i) \cap \mathcal{L}$ and its  homology class is 
$$[v_i]=\sum_{j=1}^{k+1}c_{ij} [B_j^{\tau_i}] \in H_2(\Sigma, \underline{L}_{\tau_i}, \mathbb{Z}),$$ where  $B_j^{\tau_i}$ are the closure of $\Sigma \setminus \underline{L}_{\tau_i} =\bigcup_{j=1}^{k+1}\mathring{B}_j^{\tau_i} .$
 If $\tau_i $ lies in the interior of $\dot{D}_m$, then $v_i$ is closed and its homology class is $n_i[\Sigma]$. Thus, we have $$[u^0]= [u^0_{\star}] +\sum_{i} \left(\sum_{j=1}^{k+1}c_{ij} [B_j^{\tau_i}] +n_i[\Sigma]\right).$$
By  Theorem \ref{thm3}, we have
\begin{equation} \label{eq40}
I(u^0) =I(u^0_{\star}) +2\sum_{i} \left( \sum_{j=1}^{k+1}c_{ij} + n_i(k+1) \right) \ge 2\sum_{i} \left( \sum_{j=1}^{k+1}c_{ij} + n_i(k+1) \right).
\end{equation}
For each  $1\le  j \le k+1$, let $s_j$ be a section of $E_m$ satisfying the following properties:
 \begin{enumerate}
 \item
 $s_j$ intersects the fibers positive transversely;
 \item
 For $\tau \in \partial \dot{D}_m$, $s_j(\tau) \in \mathring{B}^{\tau}_j$.
 \item
 $s_j=\mathbb{R}_{\pm} \times [0,1] \times z_j$ over the strip-like ends, where  $z_j \in \mathring{B}_j$.
 \end{enumerate}

 Define  $\mathfrak{n}_{j}(v) : =\#(s_j  \cap  v)$, where $v$ is a fiber bubble.   Note that $\mathfrak{n}_{j}(v)$ only depends on the homology class $[v] \in H_2(\Sigma, \underline{L}_{\tau}, \mathbb{Z})$. By definition, $\mathfrak{n}_{j}(B^{\tau}_{l}) =\delta_{jl} $.  Hence, $\mathfrak{n}_{j}(v_i) =c_{ij}$. On the other hand, $s_j$ intersects  fiber positive transversely.  Also, the orientation of $v_i$ is the same as the fiber because it is holomorphic.  Hence, $s_j $ intersects the $v_i$ positive transversely. Then  $\mathfrak{n}_{j}(v_i) =c_{ij} \ge 0$.  The same argument also implies that $n_i \ge 0$.  The above discussion also holds for the HF curves in the strip levels (see Lemma 2.10 of \cite{GHC2}).

  In sum,  the  ECH index of each level is nonnegative.  Moreover, if a level contains a bubble, then its ECH index is at least two. Because the total ECH index $I(A)$ is zero, then the  ECH index of each level  must be zero and   no bubbles exist.  The HF curves in strip levels must be the trivial strips; otherwise, their index are at least one which  violates   the  condition that $I(A)=0$.  In sum, $\mathcal{M}^J( \textbf{y}_1, ...,\textbf{y}_m; \textbf{y}_0, A)$ is compact.

\item
The proof of the second bullet basically is the same as the first item. By the same argument, the bubbles of $\mathbf{u} \in \partial \mathcal{M}^J( \mathbf{y}_1, ...,\mathbf{y}_m; \mathbf{y}_0, A)  $ can be ruled out.  Since the ECH index of each level is nonnegative, $\mathbf{u}$ only consists of two levels, one level has  ECH index one and  the other has zero ECH index. This leads to our second statement.
\end{itemize}
\end{proof}

Assume that $\underline{L}_{p_i} = \varphi_{H_i}(\underline{L})$. Define  reference chords  by $$\delta_i(t) : = \varphi_{H_{i}}(\textbf{x}_{\bar{H}_{i}\#H_{i-1}}(t))$$ for $1\le i \le m$ and  $\delta_0(t) =  \varphi_{H_{m}}(\textbf{x}_{\bar{H}_{m}\#H_{0}}(t)),$ where $\bar{H}_t(x) = -H_t(\varphi_H^t(x))$.
\begin{prop} \label{lem1}
Let $(\pi_m : E_m =\dot{D}_m \times \Sigma \to \dot{D}_m, \Omega_m) $ be the symplectic fiber bundle  with strip-like ends. Let $\mathcal{L}_m \subset \pi^{-1}(\partial D_m)$ be  Lagrangian submanifolds of $(E_m, \Omega_m)$ satisfying \ref{C0}, \ref{C1}, \ref{C2},  and \ref{C3}. Fix a reference relative  homology class $A_{ref} \in H_2(E_m,  \delta_1,... ,\delta_m,  \delta_0)$ and a generic almost complex structure $J \in \mathcal{J}_{tame}(E_m)$. Then  $(\pi_m: E_m \to D_m, \Omega_m, \mathcal{L}_m) $ induces an  $R$-module homomorphism
\begin{equation*}
HF_{A_{ref}}(E_m, \Omega_m, \mathcal{L}_m)_J: \bigotimes_{i=1}^mHF(\Sigma, \underline{L}_{p_{i-1}}, \underline{L}_{p_{i}}) \to HF(\Sigma, \underline{L}_{p_0},  \underline{L}_{p_{m}})
\end{equation*}
satisfying the following properties:
\begin{enumerate}

\item
(Invariance) Suppose that  there exists a family of symplectic forms  $\{\Omega_{\tau}\}_{\tau \in [0,1]}$ and a family of $\Omega_{\tau}$-Lagrangians $\{\mathcal{L}_{\tau}\}_{\tau \in [0,1]} \subset \partial E_m$ satisfying   \ref{C0}, \ref{C1}, \ref{C2}, \ref{C3}. Let  $\{J_{\tau} \}_{\tau \in [0,1]}$  be  a generic family of almost complex structures. Assume that  $\{(\Omega_{\tau}, \mathcal{L}_{\tau}, J_{\tau})\}_{\tau \in [0,1]}$ is $\tau$-independent over the strip-like ends.     Then
\begin{equation*}
HF_{A_{ref}}(E_m, \Omega_0, \mathcal{L}_0)_{J_0} = HF_{A_{ref}}(E_m, \Omega_1, \mathcal{L}_1)_{J_1}.
\end{equation*}
In particular, the cobordism maps are independent of the choice of almost complex structures.

\item
(Composition rule)  Suppose that the negative  end of $(E_m, \Omega_m, \mathcal{L}_m)$ agrees  with the $j$-th positive end of  $(E_n, \Omega_n, \mathcal{L}_n)$. Then we have
\begin{equation*}
HF_{A_2}(E_n, \Omega_n, \mathcal{L}_n) \circ HF_{A_1}(E_m, \Omega_m, \mathcal{L}_m)= HF_{A_1\#A_2}(E_{m+n-1}, \Omega_{m+n-1}, \mathcal{L}_{m+n-1}),
\end{equation*}
where $(E_{m+n-1}, \Omega_{m+n-1}, \mathcal{L}_{m+n-1})$ is the composition of $(E_m, \Omega_m, \mathcal{L}_m)$ and $ (E_n, \Omega_n, \mathcal{L}_n)$ defined in (\ref{eq33}).
\end{enumerate}
\end{prop}
\begin{proof}
At the  chain level, define
\begin{equation*}
\begin{split}
&CF_{A_{ref}}(E_m, \Omega_m, \mathcal{L}_m)_J(( \textbf{y}_1,[A_1])\otimes ...\otimes (\textbf{y}_m, [A_m]))\\
 =&\sum_{\mathbf{y}_0} \sum_{I(A)=0} \#\mathcal{M}^J( \textbf{y}_1, ...,\textbf{y}_m; \textbf{y}_0, A) (\textbf{y}_0, [A_0]).
\end{split}
\end{equation*}
Here $A_0$ is determined  by the relation $A_1\#..A_m\#A_{ref}\#(-A_0) =A$. By the first statement of Lemma \ref{lem15}, the above definition makes sense. By the second statement of  Lemma \ref{lem15} and  the gluing argument (see Appendix of \cite{RL1} for example),  $CF_{A_{ref}}(E_m, \Omega_m, \mathcal{L}_m)_J$ is a chain map.

The invariance and the composition rule follow from the standard homotopy and neck-stretching argument.  Again, the bubbles can be ruled  by the index reason as in Lemma \ref{lem15}.

 \end{proof}

\subsubsection{Reference relative  homology classes}  \label{section3.1}
Obviously, the cobordism maps depend on the choice of the reference relative homology class $A_{ref}$. For any two different reference homology classes, the cobordism maps defined by them are  differed from a shifting (\ref{eq28}).  To exclude this ambiguity,  we fix a reference relative homology class in the following way:

Let $\chi_+(s): \mathbb{R}_s  \to \mathbb{R} $ be a function such that $\chi_+=1$ when $s \le -R_0$ and $ \chi_+=0$ when $s \ge -1$. Define a diffeomorphism
\begin{equation*}
\begin{split}
F_+: &\mathbb{R}_- \times [0,1] \times \Sigma \to   \mathbb{R}_- \times [0,1] \times \Sigma\\
&(s, t, x) \to (s, t, \varphi_K \circ \varphi_{\chi_+(s)H} \circ (\varphi_{\chi_+(s)H}^{t})^{-1}(x)).
\end{split}
\end{equation*}
We view $F_+$ as a map on  $E_0$ by extending  $F_+$ to be $(z, x) \to (z, \varphi_K(x))$ over the rest of $E_0$.  Let $\mathcal{L}_+:= F_+(\partial \dot{D}_0 \times \underline{L}) \subset \partial E_0$ be a submanifold. Note that  $\mathcal{L}_+ \vert_{s \le -R_0} = \mathbb{R}_{s \le -R_0} \times (\{0\}\times \varphi_K \circ \varphi_H(\underline{L}) \cup \{1\} \times \varphi_K(\underline{L}) )$.
The surface $F_+(\dot{D}_0 \times \{\mathbf{x}\})$  represent a relative homology class  $A^+ \in H_2(E_0, \emptyset, \varphi_K (\mathbf{x}_{\bar{K} \# (K\#H)}))$.

For any Hamiltonian functions $H_1, H_2$,  we  find a suitable $H$ such that $H_1=H\#K$ and $H_2= K$.  So the above construction gives us a class $A_{H_1, H_2}^+ \in H_2(E_0, \emptyset,  \varphi_{H_2}(\textbf{x}_{\bar{H}_{2}\#H_{1}}) )$.

Let $\dot{D}^0$ be a disk with a strip-like positive end.  Define $E^0 := \dot{D}^0  \times \Sigma$.  By a  similar construction, we have a fiber-preserving diffeomorphism $F_-: E^0 \to E^0$.  Let $\mathcal{L}_- := F_-(\partial  \dot{D}^0 \times \underline{L})$. Then $A^-_{H_1, H_2} :=[F_- ( \dot{D}^0 \times \{\mathbf{x}\})]$ gives a relative homology class  in $H_2(E^0,  \varphi_{H_2}(\textbf{x}_{\bar{H}_{2}\#H_{1}}), \emptyset)$.

Using $A^{\pm}_{H_1, H_2} $, we determine a unique reference homology class  $A_{ref} \in H_2(E_m,  \delta_1,.. ,\delta_m,  \delta_0)$ as follows:  For $i$-th positive end of  $(E_m, \mathcal{L}_m)$, we  glue it with  $(E_0, \mathcal{L}_+)$ as in (\ref{eq33}), where  $ \mathcal{L}_+$ is determined by $H_{i-1}, H_i$.  Similarly, we glue the negative end of    $(E_m, \mathcal{L}_m)$ with $(E^0, \mathcal{L}_-)$.  Then this gives us a pair $(\overline{E} = \overline{D} \times \Sigma, \overline{\mathcal{L}})$, where $\overline{D}$ is a closed disk without puncture.   Note that $H_2(\overline{E}, \overline{\mathcal{L}}, \mathbb{Z}) \cong H_2(\overline{E}, \partial \overline{D} \times \underline{L}, \mathbb{Z}) $.  Under this identification, we have a canonical class $A_{can} =[\overline{D} \times \{\mathbf{x}\}] \in H_2(\overline{E}, \overline{\mathcal{L}}, \mathbb{Z})$.  We pick $A_{ref} \in H_2(E_m,  \delta_1,.. ,\delta_m,  \delta_0)$ to be  a  unique class such that
\begin{equation*}
A^{-}_{H_0, H_m} \# A_{ref}\#_{i=1}^m A^+_{H_{i-1}, H_i} =A_{can}.
\end{equation*}



\subsubsection{Continuation morphisms}
In this subsection, we recall the continuation morphisms defined in Proposition 2.14  of \cite{GHC2}.   It is a special case of Proposition \ref{lem1}.

In the case that $m=1$, we  identify $\pi_1: E_1 \to D_1$ with  $ \pi: \mathbb{R}_s \times [0,1]_t  \times \Sigma   \to \mathbb{R}_s \times [0,1]_t$.  Given two pairs of symplectomorphisms $(\varphi_{H_1}, \varphi_{K_1})$ and  $(\varphi_{H_2}, \varphi_{K_2})$,  we can use the same argument  in Lemma 6.1.1 of \cite{CHT}  to construct a pair $(\Omega_1, \mathcal{L}_1)$ such that
\begin{enumerate}
\item

$\Omega_1$ is a symplectic form such that $\Omega_1 \vert_{|s| \ge R_0}= \omega + ds \wedge dt$;
 \item
 $\mathcal{L}_1 \subset \mathbb{R} \times \{0, 1\} \times \Sigma $ are two  $d$ disjoint union of  $\Omega_1$-Lagrangian submanifolds;
\item
$\mathcal{L}_1 \vert_{s\ge R_0} =\left( \mathbb{R}_{s \ge R_0}  \times \{0\} \times \varphi_{H_1}(\underline{L})  \right) \cup (\mathbb{R}_{s \ge R_0}  \times \{1\} \times \varphi_{K_1}(\underline{L}))  $;

\item
$\mathcal{L}_1 \vert_{s \le -R_0} =\left( \mathbb{R}_{s \le -R_0}  \times \{0\} \times \varphi_{H_2}(\underline{L})  \right) \cup (\mathbb{R}_{s \le -R_0}  \times \{1\} \times \varphi_{K_2}(\underline{L}))  $.
 \end{enumerate}
We call the above triple $(E_1, \Omega_1, \mathcal{L}_1)$ a \textbf{Lagrangian cobordism} from $(\varphi_{H_1}(\underline{L}),  \varphi_{K_1}(\underline{L}))$ to $ (\varphi_{H_2}(\underline{L}),  \varphi_{K_2}(\underline{L}))$.

Recall that the reference class $A_{ref}$ is the unique class defined in Section \ref{section3.1}.  By the invariance property  in Proposition \ref{lem1}, the cobordism map $HF_{A_{ref}}(E_1, \Omega_1, \mathcal{L}_1)$ only depends on $\{(H_i, K_i)\}_{i=1,2}$. We call it a  \textbf{continuation morphism}, denoted by $\mathcal{I}^{H_1, H_2}_{K_1, K_2}$.
 Proposition \ref{lem1} implies that
the  continuation morphisms satisfy
\begin{equation}\label{eq64}
\mathcal{I}^{H_2, H_3}_{K_2, K_3}\circ \mathcal{I}^{H_1, H_2}_{K_1, K_2}=\mathcal{I}^{H_1, H_3}_{K_1, K_3}, \mbox{ and } \mathcal{I}^{H, H}_{K, K} =\operatorname{Id}.
\end{equation}
Thus,  $\mathcal{I}^{H_1, H_2}_{K_1, K_2}$ is an isomorphism.

 The direct limit of   $\{(HF(\Sigma, \varphi_{H_i}(\underline{L}), \varphi_{K_i}(\underline{L})),  \mathcal{I}^{H_i, H_j}_{K_i, K_j})\}$ is denoted by $HF(\Sigma, \underline{L})$. 
We  have a canonical isomorphism
 \begin{equation}
 j_{H, K} : HF(\Sigma, \varphi_H(\underline{L}), \varphi_K(\underline{L})) \to HF(\Sigma, \underline{L})
 \end{equation}
 that is  induced by the direct limit.

\paragraph{Notation.} To simplify the notation,  we write $\mathcal{I}^{H, K}_{0, 0}$ as  $\mathcal{I}_{H, K}$ instead. 

Let $H$ be a Hamiltonian function.  We consider another homomorphism
 \begin{equation}
I_H: CF(\Sigma,  \varphi_K(\underline{L}), \underline{L}) \to CF(\Sigma,  \varphi_{H\#K}(\underline{L}),  \varphi_H(\underline{L}))
 \end{equation}
 defined   by mapping $(\mathbf{y}, [A]) $ to  $(\varphi_H(\mathbf{y}), [\varphi_H(A)])$. Obviously, it  induces  an isomorphism $(I_H)_*$  at the homological level. We call it the \textbf{naturality isomorphism}.    In the following lemma, we show that it  is a special case of continuation  morphisms.

\begin{lemma} \label{lem9}
 The naturality isomorphisms satisfy the following diagram:
$$\begin{CD}
				HF(\Sigma,  \varphi_{K_1}(\underline{L} ),  \underline{L})  @> (I_{H_1})_*  >> HF(\Sigma,  \varphi_{H_1\#K_1}(\underline{L} ), \varphi_{H_1}(\underline{L}))  \\
				@VV\mathcal{I}^{K_1, K_2}_{0,0}  V @VV \mathcal{I}_{H_1, H_2}^{H_1\#K_1, H_2\#K_2}V\\
				HF(\Sigma,   \varphi_{K_2}(\underline{L} ), \underline{L})  @> (I_{H_2})_* >> HF(\Sigma,   \varphi_{H_2\#K_2}(\underline{L} ), \varphi_{H_2}(\underline{L})).
			\end{CD}$$
In particular, we have  $(I_{H_1})_*= \mathcal{I}_{  0, H_1}^{K_1, H_1\#K_1}$.
\end{lemma}
\begin{proof}
To prove the statement, we first split the diagram into two:
$$\begin{CD}
				HF(\Sigma,  \varphi_{K_1}(\underline{L}), \underline{L})  @> (I_{H_1})_* >> HF(\Sigma,   \varphi_{H_1\#K_1}(\underline{L} ), \varphi_{H_1}(\underline{L}))  \\
				@VV\mathcal{I}^{K_1, K_2}_{0,0} V @VV \mathcal{I}_{H_1, H_1}^{H_1\#K_1, H_1\#K_2}V\\
				HF(\Sigma,  \varphi_{K_2}(\underline{L} ),  \underline{L})  @> (I_{H_1})_* >> HF(\Sigma,  \varphi_{H_1\#K_2}(\underline{L} ), \varphi_{H_1}(\underline{L}))   \\
				@VV Id  V @VV \mathcal{I}_{H_1, H_2}^{H_1\#K_2, H_2\#K_2}V\\
				HF(\Sigma,   \varphi_{K_2}(\underline{L} ), \underline{L})  @> (I_{H_2})_* >> HF(\Sigma,   \varphi_{H_2\#K_2}(\underline{L} ), \varphi_{H_2}(\underline{L}))
			\end{CD}$$
To prove the first diagram, we define a diffeomorphism
\begin{equation*}
\begin{split}
F_{H_1}: &\mathbb{R} \times [0,1] \times \Sigma \to   \mathbb{R} \times [0,1] \times \Sigma\\
&(s, t, x) \to (s, t, \varphi_{H_1}(x))
\end{split}
\end{equation*}
Let $(\mathbb{R} \times [0,1] \times \Sigma, \Omega_1, \mathcal{L})$ be a Lagrangian cobordism from $(\varphi_{K_1}(\underline{L}), \underline{L})$ to $(\varphi_{K_2}(\underline{L}), \underline{L})$.  Let $u \in \mathcal{M}^J(\textbf{y}_+, \textbf{y}_-)$ be  an HF curve in $(\mathbb{R} \times [0,1] \times \Sigma, \Omega_1)$ with Lagrangian boundaries $\mathcal{L}$. Then $F_{H_1}(u)$ is a $F_{H_1*}J$-holomorphic HF curve in $(\mathbb{R} \times [0,1] \times \Sigma, (F_{H_1}^{-1})^*\Omega_1)$ with Lagrangian boundaries  $F_{H_1}(\mathcal{L})$.  This gives a one-to-one correspondence between the curves in $(E_1, \Omega_1, \mathcal{L})$ and curves in $(E_1, (F_{H_1}^{-1})^*\Omega_1, F_{H_1}(\mathcal{L}))$.   Note that  $F_{H_1}(u)$ is a holomorphic curve contributing  to the cobordism map $CF_{A_{ref}}(E_1,  (F_{H_1}^{-1})^*\Omega_1,  F_{H_1}(\mathcal{L}))$, and it  induces $\mathcal{I}_{H_1, H_1}^{H_1\#K_1, H_1\#K_2}$. Hence, the first diagram is true.

To prove the second diagram, the idea is the same.  Let $H_s: [0,1] \times \Sigma \to \mathbb{R} $ be a family of  Hamiltonian functions such that $H_s= H_1$ for $s \ge R_0$ and   $H_s= H_2$ for $s \le -R_0$.  Define a diffeomorphism
\begin{equation*}
\begin{split}
F_{\{H_s\}}: &\mathbb{R} \times [0,1] \times \Sigma \to   \mathbb{R} \times [0,1] \times \Sigma\\
&(s, t, x) \to (s, t, \varphi_{H_s}(x))
\end{split}
\end{equation*}
Let $\mathcal{L} = \mathbb{R} \times (  ( \{0\} \times \varphi_K(\underline{L}) ) \cup ( \{1\} \times \underline{L}) ) $ be  Lagrangian submanifolds  in  $(\mathbb{R} \times [0,1] \times \Sigma, \Omega=\omega+ ds \wedge dt)$.  Then  $F_{\{H_s\}}(\mathcal{L})$ is a disjoint  union of   $(F_{\{H_s\}}^{-1})^*\Omega$-Lagrangian submanifolds  such that
  $$F_{\{H_s\}}(\mathcal{L})=\begin{cases}
  \mathbb{R}_{\ge R_0} \times ( (\{0\} \times \varphi_{H_1\#K}(\underline{L}) \cup \{1\} \times \varphi_{H_1}(\underline{L}) )& \text{when $s \ge R_0$}\\    \mathbb{R}_{\le -R_0} \times ( (\{0\} \times \varphi_{H_2\#K}(\underline{L}) \cup \{1\} \times \varphi_{H_2}(\underline{L}))& \text{ when $s\le -R_0$}
   \end{cases}$$
  Therefore, we define the continuation morphism  $\mathcal{I}_{H_1, H_2}^{H_1\#K_2, H_2\#K_2}$ by counting the holomorphic curves in   $(\mathbb{R} \times [0,1] \times \Sigma, (F_{\{H_s\}}^{-1})^*\Omega, \mathcal{I}_{\{H_s\}}(\mathcal{L}))$. Similar as the previous case, the map $F_{\{H_s\}}$
establishes a one-to-one correspondence between the curves in $(\mathbb{R} \times [0,1] \times \Sigma, \Omega, \mathcal{L})$ and curves in $(\mathbb{R} \times [0,1] \times \Sigma, (F_{\{H_s\}}^{-1})^*\Omega, \mathcal{I}_{\{H_s\}}(\mathcal{L}))$. This gives us the second diagram.

To see  $(I_{H_1})_*= \mathcal{I}_{ 0, H_1}^{K_1, H_1\#K_1}$, we just need to take $K_2=K_1$ and $H_2=0$ in the diagram.
\end{proof}

\subsubsection{Quantum product on QHF}
In this subsection, we define the product structure on QHF by using the cobordism maps.

Consider  $E_2=\dot{D}_2 \times \Sigma$ with a symplectic form $\Omega_{E_2}= \omega + \omega_{D_2}$.  Take   $$\mathcal{L}_2= (\partial_1 \dot{D}_2 \times \varphi_{H_1}(\underline{L})) \cup (\partial_2 \dot{D}_2 \times \varphi_{H_2}(\underline{L}))  \cup (\partial_3 \dot{D}_2 \times \varphi_{H_3}(\underline{L})). $$
Define $\mu^{H_1, H_2, H_3}_2:=HF_{A_{ref}}(E_2, \Omega_2, \mathcal{L}_2)$, where $A_{ref}$ is the reference class in Section \ref{section3.1}.  Then  $\mu^{H_1,H_2, H_3}_2$ is a map
\begin{equation*}
\mu^{H_1,H_2, H_3}_2: HF(\Sigma, \varphi_{H_1}(\underline{L}), \varphi_{H_2}(\underline{L} )) \otimes HF(\Sigma, \varphi_{H_2}(\underline{L}), \varphi_{H_3}(\underline{L} )) \to HF(\Sigma, \varphi_{H_1}(\underline{L}), \varphi_{H_3}(\underline{L} )).
\end{equation*}
By Proposition \ref{lem1}, we have  the following diagram:
$$\begin{CD}
				HF(\Sigma, \varphi_{H_1}(\underline{L}), \varphi_{H_2}(\underline{L} )) \otimes HF(\Sigma, \varphi_{H_2}(\underline{L}), \varphi_{H_3}(\underline{L} )) @> \mu^{H_1, H_2, H_3}_2  >> HF(\Sigma, \varphi_{H_1}(\underline{L}), \varphi_{H_3}(\underline{L} ))  \\
				@VV\mathcal{I}^{H_1, K_1}_{H_2, K_2} \otimes \mathcal{I}^{H_2, K_2}_{H_3, K_3}V @VV \mathcal{I}^{H_1, K_1}_{H_3, K_3}V\\
				HF(\Sigma, \varphi_{K_1}(\underline{L}), \varphi_{K_2}(\underline{L} )) \otimes HF(\Sigma, \varphi_{K_2}(\underline{L}), \varphi_{K_3}(\underline{L} )) @> \mu^{K_1, K_2, K_3}_2 >> HF(\Sigma, \varphi_{K_1}(\underline{L}), \varphi_{K_3}(\underline{L} ))
			\end{CD}$$	
Therefore, $\mu^{H_1, H_2, H_3}_2$ descends to a bilinear map 
$$\mu_2:  HF(\Sigma, \underline{L})\otimes HF(\Sigma, \underline{L}) \to HF(\Sigma, \underline{L}). $$ We call $\mu_2$ the \textbf{quantum product} on QHF.

 \subsubsection{Unit}
 In this subsection, we define the unit  of  the  quantum product $\mu_2$.

Consider the case that $m=0$.  Let $\mathcal{L}_0 \subset \partial E_0 = \partial \dot{D}_0 \times \Sigma $ be  a  disjoint union of  $d$ submanifolds such that $$\mathcal{L}_0 \vert_{s \le -R_0} = \mathbb{R} \vert_{s\le -R_0} \times (\{0\} \times \varphi_H(\underline{L}) \cup  \{1\} \times \varphi_K(\underline{L})). $$
Take a symplectic form $\Omega_0$ such that $\Omega_0  \vert_{s \le -R_0}=\omega + ds \wedge dt$  and $\mathcal{L}_0$ is  a disjoint union of  $\Omega_0$-Lagrangian submanifolds.   More precisely, the tuple $(E_0, \Omega_0, \mathcal{L}_0)$ can be constructed as follows: First, we take a Lagrangian cobordism $(E_1, \Omega_1, \mathcal{L}_1)$ from $(\underline{L}, \underline{L})$ to  $(\varphi_H(\underline{L}), \varphi_K(\underline{L}))$.  Then  take   $(E_0, \Omega_0, \mathcal{L}_0)$  to be  the composition of  $(E_1, \Omega_1, \mathcal{L}_1)$ and $(E_0, \omega+ \omega_{D_0}, \partial \dot{D}_0 \times \underline{L})$.

These data induce  a cobordism map
\begin{equation*}
HF_{A_{ref}}(E_0, \Omega_0, \mathcal{L}_0): \mathbb{F}\to HF(\Sigma, \varphi_H(\underline{L}), \varphi_K(\underline{L})).
\end{equation*}
Again, $A_{ref}$ is the reference class  defined in Section \ref{section3.1}. Define $$e_{H, K} := HF_{A_{ref}}(E_0, \Omega_0,  \mathcal{L}_0)(1).$$  By Proposition \ref{lem1}, we have
\begin{equation*}
\begin{split}
\mathcal{I}^{H_1, H_2}_{K_1, K_2}(e_{H_1, K_1})
&=e_{H_2, K_2},\\
 \mu^{H_1, H_2, H_3}_2( e_{H_1, H_2} \otimes a ) =\mathcal{I}_{H_3, H_3}^{H_2, H_1}(a), & \mbox{ and }  \mu^{H_1, H_2, H_3}_2(b \otimes e_{H_2, H_3}) =\mathcal{I}_{H_2, H_3}^{H_1, H_1}(b),
\end{split}
\end{equation*}
where $a\in HF(\Sigma, \varphi_{H_2}(\underline{L}),   \varphi_{H_3}(\underline{L}))$ and $b\in HF(\Sigma, \varphi_{H_1}(\underline{L}),   \varphi_{H_2}(\underline{L}))$.   These identities  imply  that the following definition makes sense.

\begin{definition} \label{definition4}
The class $e_{H, K}$ descends to  a class $e_{\underline{L}} \in  HF(\Sigma, \underline{L} )$. We call it the  \textbf{HF unit}.  It is the  unit with respect to $\mu_2$  in the sense that   $\mu_2(e_{\underline{L}} \otimes a) =\mu_2(a \otimes e_{\underline{L}}) = a$.
\end{definition}

Similar to Lemma 5.6 of \cite{GHC2},  when $H$ is a suitable small Morse function,  the unit is represented by maximum points of $H \vert_{\underline{L}}$. We prove this as follows.

Fix  perfect Morse functions $f_{L_i}: L_i \to \mathbb{R}$ with a maximum point $y_i^+$ and a minimum point $y_i^-$. Extend $\bigcup_i f_{L_i}$ to be a Morse  function  $f: \Sigma \to \mathbb{R}$ satisfying the following conditions:
	\begin{enumerate} [label=\textbf{M.\arabic*}]
		\item \label{M1}
		$(f, g_{\Sigma})$ satisfies the Morse-Smale condition, where $g_{\Sigma}$ is a fixed metric on $\Sigma$.
		\item \label{M2}
		$f\vert_{L_i}  $ has a unique maximum $y_i^+$ and a unique minimum $y_i^{-}$.
		\item  \label{M3}
		$\{y_i^+\}$ are the only maximum points  of $f$. Also, $f\le 0$ and $f(y^+_i)=0$	for $1\le i \le d$.
		\item \label{M4}
		$f=f_{L_i} -\frac{1}{2}y^2$ in a  neighborhood of $L_i$, where $y$ is the coordinate of the normal
direction.
	\end{enumerate}
Take $H =1/\kappa  f$, where $\kappa \gg 1$.  By Lemma 5.1  of \cite{GHC2},   the set of Reeb chords of $\varphi_H$ is \begin{equation} \label{eq14}
\{ \textbf{y} = [0,1] \times (y_1, ...,y_d) \  \vert \    y_i \in \operatorname{Crit}(f\vert_{{L_i}}) \}
\end{equation}

For each  $\textbf{y}=[0,1] \times (y_1, ...,y_d) $, we construct a relative homology class $A_{\textbf{y}}$ as follows: Let $\eta = \bigcup_{i=1}^d\eta_i : \bigoplus_i[0,1]_{s} \to \underline{L}$ be a $d$-union of paths in $\underline{L}$, where  $\eta_i \subset L_i$ satisfies $\eta_i(1)=  y_i$ and $\eta_i(0) =x_i$.  Let  $u_i(s, t): = (s, t,   \varphi_H\circ (\varphi_{H}^t)^{-1} (\eta_i(s)))$. Then $u=\bigcup_{i=1}^d u_i$ is a $d$-multisection and it  gives arise a relative homology class  $A_{\textbf{y}} \in H_2(E,   \mathbf{y}, \mathbf{x}_H)$.

By  Lemma \ref{lem12},  it is straightforward  to show that
\begin{equation} \label{eq3}
\begin{split}
&\mathcal{A}_{H} (\textbf{y}, [A_{\textbf{y}}]+\sum_{i=1}^{k+1} c_i [B_i] )=  H(\textbf{y})  + \lambda \sum_{i=1}^k c_i + c_{k+1} \int_{B_{k+1}} \omega,\\
&J_0([A_{\textbf{y}}]+\sum_{i=1}^{k+1}c_i [B_i]) = 2c_{k+1}(g+d-1).
\end{split}
\end{equation}

The next  lemma is parallel to Lemma 5.8  of \cite{GHC2}. The proof is the same. But here we prove the same thing using the cylindrical formulation of QHF. 
\begin{lemma} \label{lem2}
Take $H =1/\kappa f$, where $\kappa$ is a sufficient large constant. Let $$\mathbf{y}_{\heartsuit} = [0,1 ] \times (y_1^+, ..., y_d^+),$$
where $y_i^+$ is the maximum point of $f_{L_i}$.      Let $A_{ref}$ be the reference homology class defined in Section \ref{section3.1}.    Then    we have a suitable pair $(\Omega_{E_0}, \mathcal{L}_0)$ such that for a generic  $J \in \mathcal{J}_{comp}(E_0)$, we have
$$CF_{A_{ref}}(E_0,   \Omega_{E_0}, \mathcal{L}_0)_J(1) =( \mathbf{y}_{\heartsuit}, [A_{\mathbf{y}_{\heartsuit}}]). $$
In particular, $( \mathbf{y}_{\heartsuit}, [A_{\mathbf{y}_{\heartsuit}}])$ is a cycle  representing   the unit.
\end{lemma}
\begin{proof}
To begin with, let us construct a symplectic form $\Omega_{E_0}$ and   Lagrangian  $\mathcal{L}_0$ explicitly  over $E_0$ as follows.

 Define a 2-form $\omega_{0}: = \omega + d(\chi(s) H\wedge dt)$  and a diffeomorphism
\begin{equation*}
\begin{split}
\Phi: &\mathbb{R}_{\le0} \times [0,1] \times \Sigma \to   \mathbb{R}_{\le 0}\times [0,1] \times \Sigma \\
&(s, t, x) \to (s, t, (\varphi_{{H}}^{\chi(s)t})^{-1}( x)),
\end{split}
\end{equation*}
where $\chi$ is a cutoff function such that $\chi(s)=0$ when $s \ge -1$ and $\chi(s)=1$ when $s \le -R_0$.
Because $\Phi=\operatorname{Id}$ when $s \ge -1$, we  extend it to be $\operatorname{Id}$ over the rest of $E_0$. Let $\varphi^t=\varphi_{H}^t$.  Note that $(\varphi^t)^* H =H$ because $H$ is  $t$-independent. By a direct computation, we have
\begin{equation*}
\begin{split}
&\Phi^{-1}_*(\partial_s)=\partial_s + t\dot{\chi}(s) X_{ H} \circ\varphi^{\chi(s)t}, \\
&\Phi^{-1}_*(\partial_t)=\partial_t + \chi(s)X_{ H} \circ\varphi^{\chi(s)t},\\
&  \Phi^{-1}_*(v) = \varphi^{\chi(s)t}_*(v).
\end{split}
\end{equation*}
Combining these ingredients, we get  a 2-form $$\omega_{E_0}: =(\Phi^{-1})^* \omega_0 = \omega + t\dot{\chi}(s) ds \wedge dH + \dot{\chi}(s) H ds\wedge dt  $$  satisfying  $\omega_{E_0}= \omega$ when $s \le -R_0$.   Define a $2$-form on $E_0$ by $\Omega_{E_0}:= \omega_{E_0}+ \omega_{D_0}$. It is symplectic provided that $R_0 \gg 1$. 
Also, $\mathcal{L} := \Phi(\partial \dot{D}_0 \times \varphi_{H}(\underline{L}))$ is a  $\Omega_{E_0}$-Lagrangian such that $$\mathcal{L} \vert_{s \le -R_0} = \mathbb{R}_{s\le -R_0} \times (\{0\} \times  \varphi_{H}(\underline{L} )\cup \{1\} \times  \underline{L} ).$$
The reference relative homology class $A_{ref}$ is represented by $\Phi(\dot{D}_0 \times \varphi_H(\mathbf{x}))$.

Take $J \in \mathcal{J}_{comp}(E_0)$, where  $\mathcal{J}_{comp}(E_0)$ is a subset of  $\mathcal{J}_{tame}(E_0)$ consisting of almost complex structures compatible with $\Omega_{E_0}$.   One can prove  Lemma 5.10 of \cite{GHC2} for the current setting. Thus,  $CF_{A_{ref}}(E_0,   \Omega_{E_0}, \mathcal{L})_J(1)$ is still well defined for $J \in \mathcal{J}_{comp}(E_0)$.

 Let $ \mathcal{M}^J(\emptyset, \mathbf{y}, A)$ be the moduli space of HF curves  in  $E_0$ with Lagrangian boundary condition $\mathcal{L}$.    
Note that
  \begin{equation*}
\begin{split}
&\int_{A_{ref}} \omega_{E_0} =  \int_{D_0 \times \varphi_{H}(\textbf{x})} \omega+ d(\chi(s)  H dt) =- H(\mathbf{x}) \mbox{ and }  J_0(A_{ref}) =0.
\end{split}
\end{equation*}
Let $A_0  \in H_2(E_1, \mathbf{y}, \mathbf{x}_H)$ be the class determined  by $A=A_{ref} \# (-A_0)$. Then
\begin{equation} \label{eq39}
\begin{split}
&\int u^*\omega_{E_0} = \int |d^{vert} u|^2 + \omega_{E_0}(d^{hor}u, J d^{hor}u) =-\mathcal{A}_{H}( \mathbf{y}, A_0),\\
& J_0(u) = J_0(A_{ref}) - J_0(A_0) =- J_0(A_0).
\end{split}
\end{equation}
where $d^{vert} u \in T^{vert}{E_0}$ and $d^{hor} u \in T^{hor}{E_0}$ are respectively the vertical and horizontal components of $du$.
By definition,  $T^{hor}E_0=\operatorname{Span}\{\partial_s -t\dot{\chi} X_{H}, \partial_t\} $.  Therefore, $\omega_{E_0} \vert_{T^{hor}E_0} = \dot{\chi} H \omega_{D_0}$.  By (\ref{M3}), $\dot{\chi} H \ge 0$.  Hence,  $$\int u^*\omega_{E_0} = \int |d^{vert} u|^2 + \dot{\chi} H |d^{hor}u|^2 \ge 0. $$  By the third bullet of Lemma \ref{lem12}, we have $J_0(u) \ge 0$. Combining these  with   (\ref{eq39}),
\begin{equation}  \label{eq47}
\begin{split}
&\int u^*\omega_{E_0} + \eta  J_0(u)  =  -\mathcal{A}^{\eta}_{H}( \mathbf{y}, A_0) \ge 0.
\end{split}
\end{equation}
Write $A_0=A_{\mathbf{y}} + \sum_{i=1}^{k+1} c_i [B_i]$. By Theorem \ref{thm3} and (\ref{eq3}), it is not difficult to show that
\begin{equation} \label{eq51}
\begin{split}
&0=I(u) = n(\mathbf{y}) - \sum_{i=1}^{k+1} 2c_i, \\
\mbox{and }&\mathcal{A}^{\eta}_{H}( \mathbf{y}, [A_0]) =\mathcal{A}_{H}( \mathbf{y}, [A_0])+\eta  J_0(u)\\
=&  H (\mathbf{y}) + \sum_{i=1}^{k+1} c_i \lambda + c_{k+1}\left(\int_{B_{k+1}} \omega +2\eta(d+g-1)\right)\\
=&H(\mathbf{y}) +  \lambda\sum_{i=1}^{k+1} c_i,
\end{split}
\end{equation}
where $n(\textbf{y})$ is the number of $y_i^-$-components.   By (\ref{eq47}) and (\ref{eq51}), we know that  $$ \textbf{y} =\textbf{y}_{\heartsuit}, \int u^*\omega_{{E_0}} =0 \mbox{ and } d^{vert} u=0.$$     Therefore, the  horizontal section $u=\dot{D}_0 \times \{\textbf{y}_{\heartsuit}\}$ is the only holomorphic curve contributing  to $CF_{A_{ref}}(E_0,   \Omega_{E_0}, \mathcal{L})_J(1)$. This leads to the conclusion of the lemma.  
\end{proof}

 From  Lemma \ref{lem2}, we  know that the definition of unit in Definition \ref{definition4} agrees with the Definition 3.7  of \cite{GHC2}.

\subsection{The $U$ maps on QHF} \label{U-map}
In this section, we define and compute the $U$ maps  on quantitative Heegaard Floer homology.  The construction of the $U$ map is  the same as the case for ECH (cf.  Section 2.5  of \cite{HT3}).

Fix  a point  $p =(0 , t_0,  z)  \in  E= \mathbb{R} \times [0,1] \times \Sigma$.   Given Reeb chords $\mathbf{y}_{\pm}$ and $A \in H_2(E, \mathbf{y}_+, \mathbf{y}_-)$, let $\mathcal{M}^{J}(\mathbf{y}_+, \mathbf{y}_-, A; p)$ denote the moduli space of HF curves passing through $p$.     We have the following compactness results about this moduli space.

\begin{lemma} \label{lem22}
For generic $J \in \mathcal{J}_E$ and $p$, we have the following statements:
\begin{itemize}
\item
Suppose that $I(A) =2$.   Then   $\mathcal{M}^{J}(\mathbf{y}_+, \mathbf{y}_-, A; p)$  is  a finite set of points.
\item
Suppose that $I(A) =3$.   Then   $\mathcal{M}^{J}(\mathbf{y}_+, \mathbf{y}_-, A; p)$   is a $1$-dimensional manifold with boundary 
\begin{equation*}
\begin{split}
\partial \mathcal{M}^{J}(\mathbf{y}_+, \mathbf{y}_-, A; p) = & \bigcup_{A=A_1\#A_2, I(A_1)=2, I(A_2)=1}\mathcal{M}^{J}(\mathbf{y}_+, \mathbf{y}_0, A_1; p)   \times \mathcal{M}^{J}(\mathbf{y}_0, \mathbf{y}_-, A_2)   /\mathbb{R}\\ &\bigcup_{A=A_1\#A_2, (A_1)=1, I(A_2)=2}   \mathcal{M}^{J}(\mathbf{y}_+, \mathbf{y}_0, A) /\mathbb{R}   \times \mathcal{M}^{J}(\mathbf{y}_0, \mathbf{y}_-, A; p)
\end{split}
\end{equation*}
\end{itemize}
\end{lemma}
\begin{proof}
Assume $I(A) =2$ or $I(A) =3$. 
Let $\{u_n\}_{n=1}^{\infty}$ be a sequence in $\mathcal{M}^{J}(\mathbf{y}_+, \mathbf{y}_-, A; p)$.   By Lemma \ref{lem12} and the SFT compactness (Section 11.3  of  \cite{FYHKE}),  passing to a subsequence,  $\{u_n\}_{n=1}^{\infty}$  converges to a broken  holomorphic curves $\mathbf{u} =\{u^1, ..., u^N\}$.   

By the proof of Lemma  2.10 of \cite{GHC2}, we have 
\begin{equation*}
[u^i] =[u^i_{\star}]  + \sum_{j=1}^{k+1} \left( c_{ij}[B_j] + c'_{ij} [\varphi_H(B_j)] \right) + m_i[\Sigma]
\end{equation*} 
for some $c_{ij}, c_{ij}', m_i \in \mathbb{Z}_{\ge0}$.  Here $u^i_{\star}$ is an HF curve and the rest of the classes come from bubbles.  By Theorem \ref{thm3}, we have 
\begin{equation} \label{eq89}
I(A) =I(\mathbf{u}) = \sum_{i=1}^N \left( I(u^i_{\star})  + 2\sum_{j=1}^{k+1}(c_{ij} + c_{ij}') + 2m_i(k+1)\right). 
\end{equation} 
If  some $m_i \ge 1$, then  $I(A) \ge 4$ which contradicts  the assumptions  $I(A)=2$ or $I(A)=3$.  
Thus, $m_i=0$ for $1\le i\le N$ and no closed bubbles exist. 

 Note that  the marked point $p$ is in the interior of $E$ and the bubbles with Lagrangian boundary conditions lie in $\mathbb{R}\times \{0, 1\} \times \Sigma$ (see Lemma 2.5 of \cite{GHC2}).  Thus, we have at least one HF curve $u^{i_0}_{\star}$ passing through the marked point $p$.  For generic $p$, the regularity of the HF curves and ECH equality imply that  
\begin{equation} \label{eq90}
I(u^{i_0}_{\star} ) \ge \operatorname{ind}  u^{i_0}_{\star}  \ge 2. 
\end{equation}
Combining (\ref{eq89}) and (\ref{eq90}), we have   $c_{ij} =c_{ij}' =0$ for all $i, j$. Thus, there are no bubbles. 

 If $I(A) =2$, then  $\mathbf{u}$ only consists of 1-level $u^{i_0}$, because the others levels have zero ECH index. In other words, $\mathcal{M}^{J}(\mathbf{y}_+, \mathbf{y}_-, A; p)$ is compact.  
 
 If $I(A) =3$, then  $\mathbf{u} = \{u^{i_0}, u^{i_0\pm 1}\}$ and $u^{i_0\pm 1}$ is an HF curve with $I=1$.

 The above compactness results,  together with    the gluing argument (see Appendix of \cite{RL1}),    leads to the second statement of the lemma. 
 \end{proof}

Define a map  $U_p : CF(\Sigma, \varphi_H(\underline{L}), \underline{L}) \to CF(\Sigma, \varphi_H(\underline{L}), \underline{L}) $ by 
\begin{equation*}
U_p (\mathbf{y}_+, [A_+]) :=\sum_{\mathbf{y}_-} \sum_{A \in H_2(E, \mathbf{y}_+, \mathbf{y}_-), I(A) =2} \# \mathcal{M}^{J}(\mathbf{y}_+, \mathbf{y}_-, A; p)  (\mathbf{y}_-, [(-A) \# A_+]). 
\end{equation*}
By the first statement of  Lemma \ref{lem22},  $U_p $ is well defined. By the   second  statement of  Lemma \ref{lem22},  $U_p$ is a chain map.     Thus, it descends to the homological level 
$$U_p : HF(\Sigma, \varphi_H(\underline{L}), \underline{L}) \to HF(\Sigma, \varphi_H(\underline{L}),\underline{L}).$$
We call this homomorphism  the \textbf{$U$ map}.

Next we show that the $U$ map is nothing but a geometric reformulation of the $T^{-1}$-action. In particular, it is independent of the choice of $p$. We  will suppress the subscript $p$.
 This actually coincides with the case in usual Heegaard Floer homology (cf.  Theorem 3.1.4 of \cite{VPK1}). 
\begin{thm} \label{thm9}
There is a  chain homotopy  
$$K: CF(\Sigma, \varphi_H(\underline{L}), \underline{L}) \to CF(\Sigma, \varphi_H(\underline{L}),\underline{L})$$ such that 
$$U_p - T^{-1} = d\circ K + K \circ d.$$
In other words, $U_p = T^{-1}$ at the homological level. 
\end{thm}
\begin{proof}
Fix  a marked point $z \in  \mathring{B}_m$, where $1\le m\le k+1$. Fix a generic path $\ell:[0, 1]_{\tau} \to  E$ such that $\ell(0)=p$ and  $\ell(1)=(0, 1, z)$.  Let  $A \in H_2(E, \mathbf{y}_+, \mathbf{y}_-)$ be a relative homology class such that $I(A)=2$. Let $\mathcal{M}^{J}(\mathbf{y}_+, \mathbf{y}_-, A; \ell)$ denote the moduli space of HF curves passing through the path $\ell$. 

For generic $J$ and $\ell$,  $\mathcal{M}^{J}(\mathbf{y}_+, \mathbf{y}_-, A; \ell)$ is a  $1$-dimensional manifold.  Let $\{u_n\}_{n=1}^{\infty}$ be a sequence in $\mathcal{M}^{J}(\mathbf{y}_+, \mathbf{y}_-, A; \ell)$.  As before, passing to a subsequence,  $\{u_n\}_{n=1}^{\infty}$  converges to a broken  holomorphic curve $\mathbf{u} =\{u^1, ..., u^N\}$.  Then   (\ref{eq89}) still holds.   Moreover, $m_i=0$ for  $1\le i\le N$ and no closed bubbles exist. 

Let $u^{i_0}$ be the level passing through $\ell$. Note that if   $u^{i_0}$ passes  through $\ell(\tau)$ and $\tau \ne 1$,  then $\ell(\tau) \in \operatorname{Im}u_{\star}^{i_0}$, because bubbles with Lagrangian boundary conditions lie in $\mathbb{R}\times \{0, 1\} \times \Sigma$.    On the other hand, if   $u^{i_0}$ passes  through  $\ell(1)$, then  $\ell(1)$ must lie inside the bubbles, because an HF curve will not pass through $\mathbb{R} \times\{0, 1\} \times \mathring{B}_i$. Consequently, we have the following three possibilities:
\begin{enumerate}
\item
If $u_{\star}^{i_0}$ passes  through  $\ell(0)=p$, then (\ref{eq90}) still holds.  Then
$\mathbf{u} =u_{\star}^{i_0} \in \mathcal{M}^{J}(\mathbf{y}_+, \mathbf{y}_-, A; p)$ is unbroken.  
\item
 If $u_{\star}^{i_0}$ passes  through  $\ell(\tau)$ and $0<\tau<1$,  then (\ref{eq90}) is replaced by 
 \begin{equation*} 
I(u^{i_0}_{\star} ) \ge \operatorname{ind}  u^{i_0}_{\star}  \ge 1. 
\end{equation*}
Combining the above inequality with (\ref{eq89}), we know that no bubbles can exist. Furthermore, either $\mathbf{u} =u^{i_0}_{\star}  $ is unbroken, or $\mathbf{u} =\{u^{i_0},  u^{i_0\pm1} \}$, where $ u^{i_0}, u^{i_0\pm1} $ are HF curves  with $I=1$. 
\item
If $u^{i_0}$  passes  through  $\ell(1) $, then  $\sum_{j=1}^{k+1}(c_{i_0j} + c_{i_0j}') \ge 1$ because only bubbles could pass through $\ell(1) $.  Then $I(A) =2$ and (\ref{eq89}) imply that  $\mathbf{y}_+=\mathbf{y}_-$,  $\mathbf{u} =u^{i_0}$ only consists of  1-level and $u^{i_0} = \mathbb{R} \times \mathbf{y}_+ \bigcup v$, where $v$ is a bubble passing through $\ell(1)=(0, 1, z)$.  Furthermore, $[v] =B_j$ or $[v] =\varphi_H(B_j)$. Since $z \in  \mathring{B}_m$,   only the disk $\{0\}\times \{1\} \times B_m$ can pass through $\ell(1)  =(0, 1, z)$.   Thus,  $j=m$ and the image of $v$ is $\{0\}\times \{1\} \times B_m$. 
\end{enumerate}

By the above compactness results   and the gluing argument (Appendix of \cite{RL1}),  we have 
\begin{equation*}
\begin{split}
\partial \mathcal{M}^{J}(\mathbf{y}_+, \mathbf{y}_-, A; \ell) =  &\mathcal{M}^{J}(\mathbf{y}_+, \mathbf{y}_-, A; p) \bigcup     \mathcal{M}^{J}(\mathbf{y}_+, \mathbf{y}_-, [\mathbb{R} \times \mathbf{y}_+])    \times \mathcal{M}^{J}(B_m; \ell(1))\\
& \bigcup_{A=A_1\#A_2, I(A_1)=1, I(A_2)=1}\mathcal{M}^{J}(\mathbf{y}_+, \mathbf{y}_0, A_1; \ell)  \times \mathcal{M}^{J}(\mathbf{y}_0, \mathbf{y}_-, A_2 )  /\mathbb{R}  \\ &\bigcup_{A=A_1\#A_2, I(A_1)=1, I(A_2)=1}   \mathcal{M}^{J}(\mathbf{y}_+, \mathbf{y}_0, A_1) /\mathbb{R}   \times \mathcal{M}^{J}(\mathbf{y}_0, \mathbf{y}_-, A_2; \ell). 
\end{split}
\end{equation*}
Here $\mathcal{M}^{J}(B_m; \ell(1))$ is the moduli space of holomorphic curves  with  homology class $B_m$ and passing through $\ell(1)$.   Actually, this moduli space only consists of one element $\{0\}\times \{1\} \times B_m$.  The term $  \mathcal{M}^{J}(\mathbf{y}_+, \mathbf{y}_-, [\mathbb{R} \times \mathbf{y}_+])    \times \mathcal{M}^{J}(B_m; \ell(1))$ is implicitly  empty  provided that  $\mathbf{y}_+ \ne  \mathbf{y}_-$.

Taking into account  the above factors, we obtain the conclusion of the theorem.
\end{proof}

\section{Proof of Theorem \ref{thm4}} \label{section4}
In this section, we study the properties  of the spectral invariants $c_{\underline{L}}$.  These  properties   and their proof  are parallel to those  in Theorem 1.13 of  \cite{CHMSS}.

\subsection{The HF action spectrum}
Fix a base point $\mathbf{x}$.  Define the \textbf{action spectrum} to be
\begin{equation} \label{eq59}
\operatorname{Spec}(H: \underline{L}, \mathbf{x}):=\{\mathcal{A}^{\eta}_H(\mathbf{y}, [A] ) \vert  A \in H_2(E, \mathbf{y}, \mathbf{x}_H) \}.
\end{equation}
For different base points $\mathbf{x}, \mathbf{x}'$, we have an isomorphism $$\Psi_{H, {\mathbf{x}, \mathbf{x}'}}:  H_2(E,  \mathbf{y}, \mathbf{x}_H) \to  H_2(E, \mathbf{y}, \mathbf{x}'_H)$$
preserving the action functional (see  (2.30)  of \cite{GHC2}). In particular, the  action spectrum is independent of the base point. So we suppress $\mathbf{x}$ from the notation.

  A Hamiltonian function $H$ is called \textbf{mean-normalized} if  $\int_{\Sigma} H_t \omega =0$ for any $t$.
\begin{definition} \label{definition3}
Two
mean-normalized Hamiltonian functions $ H^0, H^1$ are said to be homotopic if there exists a smooth path of Hamiltonian functions   $\{H^s\}_{s\in [0,1]}$ connecting $H^0$ to $H^1$ such that   $H^s$ is normalized and $\varphi_{H^s}=\varphi_{H^0}=\varphi_{H^1}$ for all  $s$.
\end{definition}

The following lemma shows that the spectrum is invariant  under homotopy.
\begin{lemma} \label{lem10}
If two mean-normalized Hamiltonian functions $H, K$ are homotopic, then we have
 $$\operatorname{Spec}(H: \underline{L})=\operatorname{Spec}(K: \underline{L}). $$
\end{lemma}
\begin{proof}
Fix a base point $\mathbf{x}=(x_1,...,x_d) \in \underline{L}$.    Let $\{\varphi_{s, t}:=\varphi_{H^s}^t\}_{s\in[0,1], t\in [0,1]}$   be a homotopy such that  $\varphi_{0,t}= \varphi^t_H$, $\varphi_{1,t}= \varphi^t_K$ and $\varphi_{H^s}^1 =\varphi_H=\varphi_K$ for all $s\in[0,1]$.   For a  fixed $t$, $\{\varphi_{s, t}\}_{s\in[0,1]} $ is also a family of Hamiltonian symplectomorphisms. Let $F^s_t$ be the Hamiltonian function  in   $s$-direction, i.e.,
\begin{equation*}
X_{F_t^s} =\partial_s \varphi_{s,t} \circ \varphi_{s,t}^{-1}.
\end{equation*}
$F_t^s$ is unique if we require that $F_t^s$ is mean-normalized.
Note that   $X_{F_t^s} =0 $ along $t=0,1$ because  $\varphi_{s, 0} =\operatorname{Id} $ and    $\varphi_{s, 1} =\varphi_H=\varphi_K =\varphi $.  By the mean-normalized condition, we have $F_0^s=F_1^s=0$.

Let $u_i(s,t) = (s, t, \varphi \circ  \varphi_{s,t}^{-1}(x_i))$. Note that $u_i(s, 0) \in \varphi(L_i)$ and $u_i(s, 1) \in L_i$ because   $\varphi_{s, 0} =\operatorname{Id} $ and    $\varphi_{s, 1} =\varphi $.  Then $u:=\cup_{i=1}^d u_i$ represents a class $A_0 \in H_2(E, \mathbf{x}_K, \mathbf{x}_H )$.  This induces an isomorphism    
$$\Psi_{A_0}: CF(\Sigma, \varphi_H(\underline{L}), \underline{L}) \to CF(\Sigma, \varphi_K(\underline{L}), \underline{L})$$ 
by mapping $(\mathbf{y}, [A])$ to  $(\mathbf{y}, [A\#(-A_0)])$.

Since $u$ is a disjoint union of strips, we have $J_0(A) =J_0(A\#(-A_0))$.
By a direct computation, we have
\begin{equation*}
\begin{split}
\int u_i^* \omega = & \int_0^1 \int_0^1 \omega(\partial_s \varphi_{s,t}^{-1}(x_i), \partial_t \varphi_{s,t}^{-1}(x_i)) ds \wedge dt\\
= & \int_0^1 \int_0^1 \omega(X_{F_t^s}(x_i), X_{H_t^s}(x_i)) ds \wedge dt =  \int_0^1 \int_0^1 \{F_t^s, H_t^s\}(x_i)ds \wedge dt
\end{split}
\end{equation*}
Because $H, K$ are mean-normalized,    $\partial_sH_t^s - \partial_t F_t^s - \{F_t^s, H_t^s\}= 0$ (see  (18.3.17) of \cite{Oh}).  Therefore,
\begin{equation*}
\begin{split}
\int u_i^* \omega &=  \int_0^1 \int_0^1 \left(\partial_sH_t^s(x_i) - \partial_t F_t^s(x_i) \right)ds \wedge dt\\
& = \int_0^1 H_t^1(x_i) dt -  \int_0^1 H_t^0(x_i) dt =  \int_0^1 K_t(x_i) dt -  \int_0^1 H_t(x_i) dt.
\end{split}
\end{equation*}
This implies that $\mathcal{A}^{\eta}_K( \mathbf{y}, [\Psi_{A_0}(A)]) = \mathcal{A}^{\eta}_H(\mathbf{y}, [A])$. In particular, we have $\operatorname{Spec}(H: \underline{L})=\operatorname{Spec}(K: \underline{L})$.

\end{proof}

\subsection{Proof of Theorem \ref{thm4}}
We now give the proof of  the properties in Theorem \ref{thm4} one by one.

\begin{proof}
\begin{itemize}
\item
(Spectrality)
Suppose that $\varphi_H$ is nondegenerate.  Then  $\operatorname{Spec}(H: \underline{L})$ is a discrete set over $\mathbb{R}$.  The  spectrality follows directly from the expression (\ref{eq20}).  For the case that $\varphi_H$ is degenerate, the statement can be deduced from the limit argument in \cite{LZ}.

\item(Hofer-Lipschitz continuity)
To prove the Hofer-Lipschitz, we first need to construct a Lagrangian cobordism so that we could estimate the energy of holomorphic curves.

Let $\chi(s): \mathbb{R}_s  \to \mathbb{R} $ be a  non-decreasing   cut-off function such that
\begin{equation} \label{eq29}
\chi(s)=
\begin{cases}
0&\mbox{if $s \le -R_0$}\\
1.  &\mbox{if $s \ge R_0$.}
\end{cases}
\end{equation}
Let  $H^s := \chi(s)H_+ + (1- \chi(s)) H_-$.
Define a diffeomorphism
\begin{equation} \label{eq13}
\begin{split}
F: &\mathbb{R} \times [0,1] \times \Sigma \to   \mathbb{R} \times [0,1] \times \Sigma\\
&(s, t, x) \to (s, t,  \varphi_{H^s} \circ (\varphi_{H^{s}}^{t})^{-1}(x)).
\end{split}
\end{equation}
 Let
\begin{equation*}
\begin{split}
&\mathcal{L} :=F(\mathbb{R} \times \{0, 1\} \times \underline{L} ),\\
&\omega_{E}:=(F^{-1})^*(\omega + d(H^s_t dt)) \mbox{ and } \Omega_{E} =\omega_{E} +ds \wedge dt.
\end{split}
\end{equation*}
Then $\omega_E$ is admissible and  $\mathcal{L} \subset \mathbb{R} \times \{0, 1\} \times \Sigma$ is a disjoint  union of  $\Omega_{E}$-Lagrangian submanifolds  such that
\begin{equation*}
\begin{split}
&\mathcal{L} \vert_{s \ge R_0} = \mathbb{R}_{s \ge R_0} \times\left( ( \{0\} \times \varphi_{H_+}(\underline{L})) \cup   ( \{1\} \times  \underline{L})\right), \\
&\mathcal{L} \vert_{s \le  -R_0} = \mathbb{R}_{s \le R_0} \times\left( ( \{0\} \times \varphi_{H_-}(\underline{L})) \cup   ( \{1\} \times  \underline{L})\right).
\end{split}
\end{equation*}
Let $A_{ref} =F(\mathbb{R} \times [0,1] \times \{\mathbf{x}\}) \in H_2(E_1, \mathbf{x}_{H_+}, \mathbf{x}_{H_-})$.  Take  a generic  $J \in \mathcal{J}_{comp}(E_1)$.  Then we have a  cobordism map $HF_{A_{ref}}(E_1, \Omega_E, \mathcal{L})_J= \mathcal{I}_{H_+,H_-}$.

  Let $u  \in \mathcal{M}^J(\textbf{y}_+, \textbf{y}_-)$ be an HF curve in $(E_1, \Omega, \mathcal{L})$. The energy of $u$ satisfies
\begin{equation} \label{eq35}
\begin{split}
 \int u^*\omega_{E}
 &=\int_{F^{-1}(u)} \omega + d_{\Sigma}H^s \wedge dt +  \dot{\chi}(s)(H_+-H_-) ds \wedge dt \\
 &\ge  \int_{F^{-1}(u)}  \dot{\chi}(s)(H_+-H_-) ds \wedge dt \\
 & \ge d \int_0^1 \min_{\Sigma}(H_+ - H_-) dt.
\end{split}
\end{equation}
The inequality in the second step ($\int_{F^{-1}(u)} \omega + d_{\Sigma}H^s \wedge dt  \ge 0$) follows the same argument as  in Lemma 3.8 of  \cite{CHS1}.

On the other hand,   we have
\begin{equation*}
\begin{split}
&\int u^* \omega_E    = \int_{A_+} \omega + \int_{A_{ref}} \omega_E -\int_{A_-} \omega\\
&J_0(u)   = J_0(A_+) +J_0(A_{ref}) - J_0(A_-)
\end{split}
\end{equation*}
due to  the relation $A_+\# A_{ref}\#(-A_-) =[u]$.  Note that $$\int_{A_{ref}} \omega_E = \int_0^1 H_+(t, \mathbf{x})dt - \int_0^1 H_-(t, \mathbf{x})dt \mbox{ and } J_0(A_{ref})=0.$$
By Lemma \ref{lem12} and (\ref{eq35}), we have
\begin{equation} \label{eq55}
d \int_{0}^1 \min_{\Sigma} (H_+ -H_-)dt \le \int u^* \omega_E + \eta J_0(u)= \mathcal{A}^{\eta}_{H_+} (\textbf{y}_+, A_+) - \mathcal{A}^{\eta}_{H_-} (\textbf{y}_-, A_-).
\end{equation}


Fix $a \ne 0 \in HF(\Sigma, \underline{L})$.  For any fixed $\delta$, take a cycle $\mathbf{c}_+ \in  CF(\Sigma,   \varphi_{H_+}(\underline{L}), (\underline{L}))$ representing $j_{H_+}^{-1}(a)$ and satisfying
\begin{equation*}
\mathcal{A}^{\eta}_{H}(\mathbf{c}_+) \le c_{\underline{L}}(H_+, a) + \delta.
\end{equation*}
Let $\mathbf{c}_- =  \mathcal{I}_{H_+,H_-}(\mathbf{c_+})$.  Then it is a cycle representing $j_{H_-}^{-1}(a)$. Take a summand $(\textbf{y}_-, [A_-])$  of  $\mathbf{c}_-$ such that $\mathcal{A}^{\eta}_H(\textbf{y}_-, [A_-]) = \mathcal{A}^{\eta}_H(\mathbf{c}_-) $.  Find a  summand   $(\textbf{y}_+, [A_+])$ of  $\mathbf{c}_+$ such that $ \langle \mathcal{I}_{H_+,H_-}(\textbf{y}_+, [A_+]), (\textbf{y}_-, [A_-]) \rangle =1$.  Then the  estimate (\ref{eq55}) implies that
\begin{equation*}
\begin{split}
 d \int_{0}^1 \min_{\Sigma} (H_+ - H_-)dt \le   c_{\underline{L}}(H_+, a) -c_{\underline{L}}(H_-, a) + \delta.
\end{split}
\end{equation*}
Taking  $\delta \to 0$, we obtain one half  of Hofer-Lipschitz property.  Interchange the positions of $H_+$ and $H_-$; then we obtain the other half.

\item
(Homotopy invariance) We now show that the spectral invariants are invariance under homotopy.
Let  $H$ and $K$ be  mean-normalized Hamiltonian functions  such that they are homotopic. Then,  we have a family of Hamiltonian functions $\{H_t^s\}_{s\in[0,1]}$ with $H^0_t=H_t$ and  $H^1_t=K_t$.  By Lemma \ref{lem10},  we have $$ \operatorname{Spec}(H: \underline{L})=\operatorname{Spec}(H^s: \underline{L})=\operatorname{Spec}(K: \underline{L}). $$
On the other hand,  $c_{\underline{L}}(H^s, a) $ is continuous with respect to $s$.  Moreover, $\operatorname{Spec}(H: \underline{L})$ is a nowhere dense set of $\mathbb{R}$. Therefore,  $c_{\underline{L}}(H^s, a) $ must be  a constant.

\item
(Shift property) Consider  a family of  functions $H^s= H + s c$, where $0\le s\le 1$. Since $\varphi_{H^s} =\varphi_H$ for all $s$,  the  chain complex $CF(\Sigma, \varphi_{H^s}(\underline{L}), \underline{L})$ is independent of $s$.   Note that  $$\mathcal{A}^{\eta}_{H+sc}(\textbf{y}, [A]) =\mathcal{A}^{\eta}_{H}(\textbf{y}, [A]) + sd \int_0^1 c(t)dt.$$ Therefore, 
 $c_{\underline{L}}(H^s, a) - s  d\int_0^1 c(t)dt \in \operatorname{Spec}(H: \underline{L})$.
By  the Hofer-Lipschitz property and the fact that $\operatorname{Spec}(H: \underline{L})$ is a nowhere dense,  $c_{\underline{L}}(H^s, a) - s d \int_0^1 c(t)dt$ is a constant. Taking $s =0$, we know that the constant is  $c_{\underline{L}}(H, a)$.

\item (Lagrangian control property)
We now prove the Lagrangian control property.
Let $H_t$ be a Hamiltonian such that $H_{t} \vert_{L_i} =c_i(t)$.   Then $X_{H_t}$ is tangent to $L_i$ along $L_i$. Hence,  $\varphi_H(L_i) =L_i$. The Reeb chords  correspond  to $\mathbf{y} \in \underline{L}$. By  assumption \ref{assumption6}, we have
\begin{equation*}
\begin{split}
\operatorname{Spec}(H: \underline{L})& = \{   m_0\lambda + m_1(1-k\lambda) + m_1 2\eta(d+g-1)+\sum_{i=1}^d \int_0^1 c_i(t) dt   \vert m_0, m_1 \in \mathbb{Z} \}\\
&= \{  m \lambda +\sum_{i=1}^d \int_0^1 c_i(t) dt   \vert m \in \mathbb{Z} \}.
\end{split}
\end{equation*}
Define a family of Hamiltonian functions $\{H^s :=sH\}_{s \in [0,1]}$.  By the spectrality, we have
$c_{\underline{L}}(H^s, a)  = m_0 \lambda  + \sum_{i=1}^d \int_0^1 sc_i(t) dt $. Here $m_0 \in \mathbb{Z}$ must be  a constant due  to the Hofer-Lipschitz continuity.  We know that $ m_0 \lambda  =  c_{\underline{L}}(0, a) $ by taking $s=0$.  Then the Lagrangian control property follows from taking $s=1$.

\item
(Triangle inequality) We now prove the triangle inequality.
First, we introduce   an operation  on Hamiltonian functions  called the \textbf{join}.  The join of $H$ and $K $ is defined by
\begin{equation*}
H_t \diamond K_t(x)=
\begin{cases}
2\rho'(2t)K_{\rho(2t)}(x)&\mbox{if $0\le t \le \frac{1}{2}$}\\
2\rho'(2t-1)H_{\rho(2t-1)}(x)  &\mbox{if $ \frac{1}{2} \le t \le 1$},
\end{cases}
\end{equation*}
where $\rho: [0,1] \to [0,1]$ is a fixed non-decreasing smooth function that  is equal to 0 near 0 and equal to 1 near 1.
Similar to  the composition, the time 1-flow of $H_t \diamond K_t$ is $\varphi_H \circ  \varphi_K$. We first prove the triangle inequality for  $H \diamond K$ instead of  $ H \#K$.

Let $a, b \in HF(\Sigma, \underline{L})$.  Take
\begin{equation*}
\begin{split}
&\Omega_2= \omega + \omega_{D_2} \\
&\mathcal{L}_2 = \left(\partial_1 \dot{D}_2 \times \varphi_{H} \circ \varphi_K (\underline{L}) \right)\cup \left(\partial_2 \dot{D}_2  \times \varphi_{H} (\underline{L}) \right)  \cup  \left(\partial_3 \dot{D}_2  \times \underline{L}\right).
\end{split}
\end{equation*}
These  induce the quantum  product  $$\mu_2: HF(\Sigma, \varphi_{H} \circ \varphi_K(\underline{L}), \varphi_{H}(\underline{L} )) \otimes HF(\Sigma, \varphi_{H}(\underline{L}), \underline{L} ) \to HF(\Sigma, \varphi_{H} \circ \varphi_{K}(\underline{L}), \underline{L} ). $$

Let us first consider the following special case: Suppose that there is a base point $\mathbf{x}=(x_1,...,x_d) \in \underline{L}$ such that
\begin{equation} \label{eq10}
\begin{split}
&d_{\Sigma}H_t(x_i)=d_{\Sigma}K_t(x_i) =0, \\
\mbox{ and }& \nabla^2H_t(x_i), \nabla^2K_t(x_i)  \mbox{ are non-degenerate. }
\end{split}
\end{equation}
for $1\le i \le d$.
 This assumption implies that  $\varphi_H^{t}(x_i) =x_i$, $\varphi_K^{t}(x_i) =x_i$ and $d_{\Sigma}(H_t\diamond K_t)(x_i)=0$.  In particular, $\mathbf{x}$ is a non-degenerate Reeb chord of $\varphi_H$, $\varphi_K$ and $\varphi_H \circ \varphi_K$.  Also, the reference chords become $\mathbf{x}_H= \varphi_H(\mathbf{x}_K)=\mathbf{x}_{H\diamond K}=\mathbf{x}$.

Take  $A_{ref} =[D_2 \times \{\mathbf{x}\}] \in H_2(\varphi_H(\mathbf{x}_K),  \mathbf{x}_H,\mathbf{x}_{H\diamond K} )$  be  the reference relative homology class. By definition,  we have
\begin{equation} \label{eq8}
\begin{split}
\int_{A_{ref}} \omega = \int_{D_2 \times\{ \mathbf{x}\}} \omega =0 \mbox{ and } J_0(A_{ref})=0.
\end{split}
\end{equation}

 Let  $u \in \mathcal{M}^J(\mathbf{y}_1, \mathbf{y}_2; \mathbf{y}_0, A)$ be an HF curve with $I=0$. Here the relative homology class $A$ satisfies  $A_1\#A_2\#A_{ref}\#(-A_0) = A$.   Therefore,  the energy and $J_0$ index of $u$  are 
\begin{equation} \label{eq7}
\begin{split}
 &\int u^*\omega = \int_{A_1} \omega  + \int_{A_2} \omega + \int_{A_{ref}} \omega -  \int_{A_0} \omega \\
 &J_0(u) =J_0(A_1) + J_0(A_2) + J_0(A_{ref})    - J_0(A_0).   
\end{split}
\end{equation}

Take $J \in \mathcal{J}_{comp}(E_2)$.   Then $\int u^*\omega= \int |d^{vert} u|^2 \ge 0$. By Lemma \ref{lem12}, $J_0(u) \ge 0$. Combine these facts  with   (\ref{eq8}), (\ref{eq7}); then we have
\begin{equation}\label{eq15}
\begin{split}
\mathcal{A}^{\eta}_{H\diamond K}( \textbf{y}_0, [A_0]) \le  \mathcal{A}^{\eta}_{K}( \varphi_H^{-1}(\textbf{y}_1), [\varphi_H^{-1}(A_1)]) + \mathcal{A}^{\eta}_{H}( \textbf{y}_2, [A_2]).
\end{split}
\end{equation}

   Assume that $\mu_2(a\otimes b) \ne 0$. Let $\mathbf{c}_0\in CF(\Sigma,   \varphi_{H} \circ \varphi_{K}(\underline{L}), \underline{L})$,    $\textbf{c}_1\in CF(\Sigma,  \varphi_{H} \circ \varphi_{K}(\underline{L}), \varphi_H(\underline{L}))$, and $\textbf{c}_2\in CF(\Sigma,  \varphi_{H}(\underline{L}), \underline{L} )$  be cycles  representing   $j_{H\diamond K}^{-1}( \mu_2(a \otimes b)),$   $j_{H\diamond K, H}^{-1}(a)$, and $j_H^{-1}(b)$ respectively.   By Lemma \ref{lem9},  $ \varphi_H^{-1}(\textbf{c}_1 )$  is a cycle representing   $j_K^{-1}(a)$.    We choose $\textbf{c}_1, \textbf{c}_2$ such  that
\begin{equation*}
\begin{split}
&\mathcal{A}^{\eta}_{K}(\varphi_H^{-1}(\textbf{c}_1) ) \le c_{\underline{L}}(K, a) + \delta \mbox{ and } \mathcal{A}^{\eta}_{H}(\textbf{c}_2) \le c_{\underline{L}}(H, b) + \delta.
\end{split}
\end{equation*}

Therefore, (\ref{eq15}) implies that  $\mathcal{A}^{\eta}_{H\diamond K}( \mathbf{c}_0) \le   \mathcal{A}^{\eta}_{K}( \varphi_H^{-1}( \mathbf{c}_1)) +\mathcal{A}^{\eta}_{H}(  \mathbf{c}_2) .  $ Take $\delta \to 0$. We have
\begin{equation*}
\begin{split}
 c_{\underline{L}}(H\diamond K, \mu_2(a \otimes b))  \le  c_{\underline{L}}(K, a) + c_{\underline{L}}(H, b).
\end{split}
\end{equation*}

For general Hamiltonians $H_t, K_t$, we construct approximations  $H_t^{\delta}$, $K_t^{\delta}$ satisfying the assumptions (\ref{eq10}) as follows.

Fix local coordinates $(x, y)$ around $x_i$.  Then we can write
$$H_t(x, y) = H_t(0) + \partial_xH_t(0)x +\partial_yH_t(0)y +R_t(x, y),$$
where $R_t(x, y)$ is the higher  order terms.
We may assume that $\nabla^2 H_t(0)$ is non-degenerate; otherwise, we can achieve this by perturbing $H_t$ using  a small Morse function with a  critical point at $x_i$.

Pick a cut-off function $\chi_{\delta}(r): \mathbb{R}_+ \to \mathbb{R}$ such that $\chi_{\delta} (0)=1$, $\chi_{\delta}'(0)=0$ and  $\chi_{\delta} =0$ for $r \ge  \delta$, where $r=\sqrt{x^2+y^2}$.  Define $H^{\delta}_t$ by
$$H^{\delta}_t(x, y) = H_t(0) + (1-\chi_{\delta}(r))(\partial_xH_t(0)x +\partial_yH_t(0)y) +R_t(x, y).$$
We perform the same construction for $K_t$.
Apparently, we have
\begin{equation*}
\begin{split}
&d_{\Sigma}H^{\delta}_t(x_i)=d_{\Sigma}K^{\delta}_t(x_i)=0,\\
&\nabla^2H^{\delta}_t(x_i)= \nabla^2H_t(x_i), \  \  \nabla^2K^{\delta}_t(x_i)= \nabla^2K_t(x_i),\\
&|H^{\delta}_t-H_t| \le c_0 \delta, \  \ |K^{\delta}_t-K_t| \le c_0 \delta, \\
&|H_t\diamond K_t - H^{\delta}_t\diamond K^{\delta}_t| \le c_0\delta.\\
\end{split}
\end{equation*}

Apply the triangle inequality to $H^{\delta}_t, K^{\delta}_t, H^{\delta}_t\diamond K^{\delta}_t$, and then  by the Hofer-Lipschitz continuity, we have
\begin{equation*}
\begin{split}
 c_{\underline{L}}(H\diamond K, \mu_2(a \otimes b))  \le  c_{\underline{L}}(K, a) + c_{\underline{L}}(H, b) + O(\delta).
\end{split}
\end{equation*}
Note that  the above construction works for any $\delta$; hence,  we can  take $\delta \to 0$.

Since the normalization of  $H \diamond K$ and $H \#K$ are homotopic, we replace   $H \diamond K$  in the    triangle inequality by $H\#K. $

 \item (Normalization)
 To see $ c_{\underline{L}}(0, e_{\underline{L}}) =0$,   first  note that  we have
 \begin{equation*}
\begin{split}
 c_{\underline{L}}(0, e_{\underline{L}})=c_{\underline{L}}( 0, \mu_2(e_{\underline{L}} \otimes e_{\underline{L}}))  \le  c_{\underline{L}}(0, e_{\underline{L}}) + c_{\underline{L}}(0, e_{\underline{L}}).
\end{split}
\end{equation*}
Hence, we get $c_{\underline{L} }(0, e_{\underline{L}}) \ge 0.$ On the other hand, Lemma \ref{lem2} and  (\ref{eq20}) imply that $ c_{\underline{L}}(0, e_{\underline{L}}) =\lim_{\kappa \to \infty} c_{\underline{L}}(1 /\kappa f, e_{\underline{L}})\le  \lim_{\kappa \to \infty} \mathcal{A}_{1 /\kappa f}^{\eta} ((\mathbf{y}_{\heartsuit}, [A_{\mathbf{y}_{\heartsuit} }])) =0$.  

\item (Weyl Law)
The proof of the Weyl Law  relies on the  Hofer-Lipschitz and the Lagrangian control properties.  We have obtained these properties.  One can follow the same argument in (Page 12--13)  \cite{CHMSS} to prove the Weyl Law. We skip the details here.

\end{itemize}
\end{proof}

\section{Open-closed morphisms} \label{open-closed}
In this section, we prove Theorem \ref{thm2}. Instead of  proving it directly, we first establish the following theorem, from which Theorem \ref{thm2} follows easily.

\begin{thm} \label{thm6}
Let $\underline{L}$ be an admissible link and $\varphi_H$  be  a $d$-nondegenerate Hamiltonian symplectomorphism. 		 	Then for a generic admissible almost complex structure $J \in \mathcal{J}_{tame}(W, \Omega_H)$ (see Definition \ref{definition7}), we have  a $\Lambda$-module homomorphism
\begin{equation*}
\widetilde{\mathcal{OC}}(\underline{L}, H)_J :   HF(\Sigma,  \varphi_H(\underline{L}), \underline{L})_J \to \widetilde{PFH}(\Sigma, \varphi_H, d)_J
\end{equation*}
 satisfying  the following properties:
		\begin{itemize}
			\item (\textbf{Partial invariance})
			Suppose that $\varphi_H, \varphi_G$ satisfy the following conditions (see Definition \ref{definition2}):
			\begin{enumerate} [label=\textbf{$\spadesuit$.\arabic*}]
				\item  \label{assumption1}
				Each  periodic orbit of $\varphi_{H}$ with degree less than or equal to $d$ is either $d$-negative elliptic or hyperbolic.
				
				\item \label{assumption2}
				Each periodic orbit of  $\varphi_{G}$  with degree  less than or equal to $d$  is either $d$-positive elliptic or hyperbolic.	
			\end{enumerate}
	
Then for any generic admissible almost complex structures   $J_H \in \mathcal{J}_{tame}(W, \Omega_{H})$ and $J_G \in \mathcal{J}_{tame}(W, \Omega_{G})$, we have  the following commutative diagram:
\begin{equation}			\label{eq67}
\begin{CD}
				HF(\Sigma,  \varphi_H(\underline{L}),   \underline{L})_{J_H} @>\widetilde{\mathcal{OC}}(\underline{L}, H)_{J_H}>>   \widetilde{PFH}(\Sigma,  \varphi_H, d)_{J_H} \\
				@VV \mathcal{I}_{H, G}V @VV \mathfrak{I}_{H, G}V\\
				HF(\Sigma,  \varphi_G(\underline{L}),  \underline{L})_{J_G}  @> \widetilde{\mathcal{OC}}(\underline{L}, G)_{J_G}>>    \widetilde{PFH}(\Sigma,  \varphi_G, d)_{J_G}
			\end{CD}
\end{equation}

\item (\textbf{Commutation with the $U$ maps})
\begin{equation*}
\widetilde{\mathcal{OC}} (\underline{L}, H )_J  \circ U =U \circ \widetilde{\mathcal{OC}} (\underline{L}, H )_J.  
\end{equation*}	
			
			\item
			(\textbf{Non-vanishing})	
			There  exist     nonzero classes    $\sigma_{\underline{L}}\in HF(\Sigma, \underline{L})$  and    $\mathfrak{d}  \in  \widetilde{PFH}(\Sigma,d)$  such that 	  if $\varphi_H$  satisfies   condition   \ref{assumption1} or   condition \ref{assumption2},  then we have 
\begin{equation*}
\widetilde{\mathcal{OC}} (\underline{L}, H )_J ( j_{H}^{-1}(\sigma_{\underline{L}}))= \mathfrak{j}_{ H}^{-1}(\mathfrak{d}),
\end{equation*}
		where  $j_H$ and $\mathfrak{j}_H$ are   the canonical isomorphisms in   (\ref{eq22}). Furthermore, we have  			\begin{equation*}
			c_{\underline{L}}(0, \sigma_{\underline{L}}) =0 \mbox{ and } c_d^{pfh}(0, \mathfrak{d})=0. 
			\end{equation*}
		\end{itemize}
	\end{thm}

The construction of $\widetilde{\mathcal{OC}}(\underline{L}, H)_J$  parallels the construction in Section 6  of  \cite{VPK}  and that of  the  closed-open morphisms in \cite{GHC2}.  Therefore, we will just outline the construction of the open-closed morphisms and the proof of partial invariance in Theorem \ref{thm6}. We will focus on proving   nonvanishing of the open-closed morphisms.

\begin{remark} \label{remark6}
A consequence of the  partial invariance  (\ref{eq67}) is that 
\begin{equation*}
\mathfrak{j}_H \circ \widetilde{\mathcal{OC}}(\underline{L}, H)_{J_H} \circ j_H^{-1}= \mathfrak{j}_G \circ \widetilde{\mathcal{OC}}(\underline{L}, G)_{J_G} \circ j_G^{-1}
\end{equation*}
 provided that $\varphi_H$ satisfies condition  \ref{assumption1}  and  $\varphi_G$ satisfies  condition   \ref{assumption2}.   
Let 
$$\mathcal{H}_{\spadesuit} : = \{H \in C^{\infty}(S^1 \times \Sigma) \vert \varphi_H  \mbox{ satisfies  \ref{assumption1} or \ref{assumption2} }\}.$$
Then  the map $\mathfrak{j}_H \circ \widetilde{\mathcal{OC}}(\underline{L}, H)_{J_H} \circ j_H^{-1}: HF(\Sigma, \underline{L}) \to \widetilde{PFH}(\Sigma, d)$ is independent of the  choice of the Hamiltonian function  in $\mathcal{H}_{\spadesuit} $. Therefore, if $H \in  \mathcal{H}_{\spadesuit}$,   then  $\widetilde{\mathcal{OC}}(\underline{L}, H)_{J}$ is independent of the choice of almost complex structure $J$.  We will   suppress the almost complex structure  from the notation when  it  plays no role.

The author is grateful to an anonymous referee for pointing out this observation. This remark also applies to the closed-open morphisms defined in \cite{GHC2}, and Remark 6.1 therein is redundant. 
\end{remark}

 \begin{remark} \label{remark2}
 The assumptions \ref{assumption1} and \ref{assumption2} come from the holomorphic curve definition of the PFH cobordism maps.
For technical reasons (see Section 5.5 of \cite{H4}),   the cobordism maps on PFH are defined   using the Seiberg-Witten theory \cite{KM} and the isomorphism ``PFH=SWF'' \cite{LT}.  Nevertheless,  the proof of the partial invariance in  Theorem \ref{thm6}   consists of   neck-stretching and homotopy   arguments  for holomorphic curves in an open-closed symplectic manifold. Thus,  we   need  a holomorphic curves  definition of the PFH cobordism maps. The assumptions \ref{assumption1},  \ref{assumption2} are used to guarantee that the PFH cobordism maps can be defined by counting holomorphic curves in the special cases (\ref{eq23}).  According to the results in \cite{GHC}, the Seiberg-Witten definition agrees with the holomorphic curves definition in these special cases. We believe  that the assumptions  \ref{assumption1},  \ref{assumption2} can  be removed if one could define the PFH cobordism maps by purely  holomorphic curve methods.

   \end{remark}

\subsection{Open-closed  symplectic cobordisms} \label{open-closed-cobordism}
To begin with, let us  introduce the open-closed  symplectic manifold and the Lagrangian submanifolds. The construction follows \cite{VPK}.  Define a base surface  $B \subset \mathbb{R}_s \times (\mathbb{R}_t/(2 \mathbb{Z}))$  by $B:=  \mathbb{R}_s \times (\mathbb{R}_t/(2 \mathbb{Z})) \setminus B^c$, where $B^c$ is $(2, \infty)_s \times [1,2]_t $ with the corners rounded.   See Figure \ref{figure3} for the picture of $B$.
	\begin{figure}[h]
		\begin{center}
			\includegraphics[width=9cm, height=3.5cm]{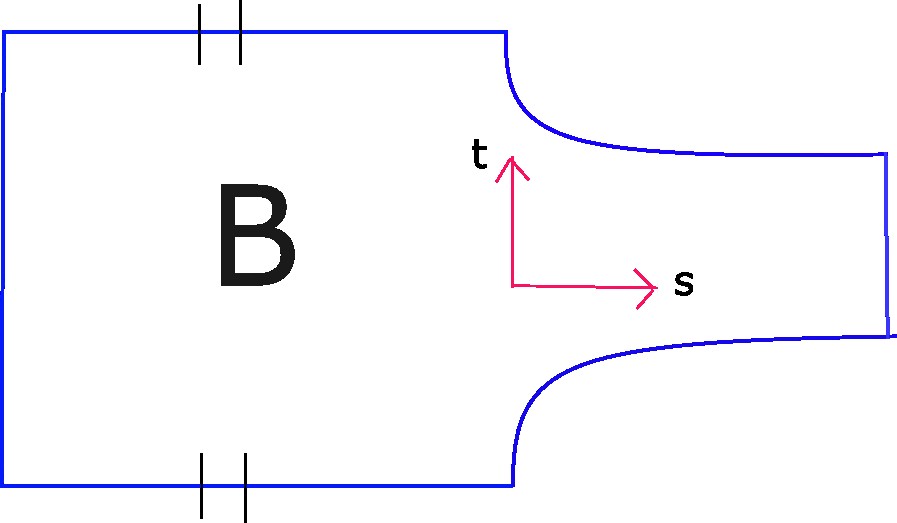}
		\end{center}
		\caption{The open-closed surface}
		\label{figure3}
	\end{figure}

 Let $Y_{\varphi_H}: =[0, 2] \times \Sigma / (0, \varphi_H(x)) \sim (2, x)$ 	be the mapping torus of $\varphi_H$.  Then  $\pi: \mathbb{R}_s \times Y_{\varphi_{H}} \to \mathbb{R}_s \times  (\mathbb{R}_t/(2 \mathbb{Z}))$ is a surface  bundle over the cylinder.
	Define  a surface bundle  $W_H$ by
	$$\pi_W= \pi \vert_{W}: W_H:=\pi^{-1}(B) \to B. $$
 The symplectic form $\Omega_H$ on $W_H$ is defined to be the restriction of $\omega_{\varphi_H} + ds\wedge dt$.  Note that $W_H$  is fiberwise diffeomorphic to $B\times\Sigma$.    So we denote $W_H$ by $W$ instead when the context is clear.

	We place a copy of $\underline{L}$ on the fiber $\pi_W^{-1}(3,1)$ and take its parallel transport along $\partial B$ using the symplectic connection.  The parallel transport sweeps out an $\Omega_H$-Lagrangian submanifold $\mathcal{L}_{H}$ in $W_H$.   Then  $\mathcal{L}_{H}$ consists of $d$ disjoint connected components.  Moreover, we have
	\begin{equation*}
		\begin{split}
			&\mathcal{L}_{H} \vert_{\mathbb{R}_{s \ge 3 } \times\{0\} } = \mathbb{R}_{s \ge 3} \times \{0\} \times \varphi_H(\underline{L})\\
			&\mathcal{L}_{H} \vert\vert_{\mathbb{R}_{s \ge 3 } \times\{1\} }= \mathbb{R}_{s \ge 3} \times \{1\} \times \underline{L}.
		\end{split}
	\end{equation*}
	We call the  triple $(W_H, \Omega_H, \mathcal{L}_{H})$   an \textbf{open-closed cobordism}. The concepts of almost complex structures and holomorphic curves of   $(W_H, \Omega_H, \mathcal{L}_{H})$  are defined similarly  to the case of HF curves.

\begin{definition}\label{definition7}
Let $\mathcal{J}_{tame}(W, \Omega_H)$ be the set of almost complex structures satisfying the following conditions:
\begin{enumerate}
  \item
  $J \vert_{\mathbb{R}_{\ge 3} \times [0,1] \times \Sigma}$ and $J \vert_{\mathbb{R}_{\le -1} \times Y_{\varphi_H}}$ are  admissible.
  \item
  $d\pi_W$ is $(J, j_B)$-complex linear, i.e., $d\pi_W \circ J =j_{B} \circ d\pi_W$, where $j_B$ is  the complex structure of $B$ that maps $\partial_s$ to $\partial_t$.
  \item
  $J$ is  $\Omega_H$-tame.
\end{enumerate}
\end{definition}
	
\begin{definition} [Definition 5.4.3 of \cite{VPK}] \label{definition5}
		Fix a Reeb chord $\mathbf{y}$ and an orbit set $\alpha$ with degree $d$. Let $(\dot{F},j)$ be a Riemann surface (possibly disconnected) with   punctures.  A \textbf{d-multisection} in $W$ is a smooth map $u: (\dot{F}, \partial \dot{F}) \to W$ such that
		\begin{enumerate}
\item
Each irreducible component of $\dot{F}$ has at least one puncture. 
			\item
			$u(\partial {\dot{F}}) \subset \mathcal{L}_{H }$. Write $\mathcal{L}_H = \bigcup_{i=1}^d L^i_{H} $, where $L_H^i$  is a   connected component of $\mathcal{L}_{H}$. For each $1 \le i\le d$,   $u^{-1}(L_{H}^i)$ consists of exactly one component of $\partial \dot{F}$.
			\item
			$u$ is asymptotic to $\mathbf{y}$ as $s \to \infty $.
			\item
			$u$ is asymptotic to $\alpha$ as $s \to -\infty $.
			
		\end{enumerate}
Fix $J \in \mathcal{J}_{tame}(W, \Omega_H)$. A $J$-holomorphic $d$-multisection is called an \textbf{HF-PFH curve}. We call  the integer  $d$
 the degree of an HF-PFH curve.

	\end{definition}
As in the case of HF curves (Remark \ref{remark5}),   an HF-PFH curve is an authentic ``multisection'',  because each irreducible component of $\dot{F}$ contains at least one puncture.
   Consequently,  an HF-PFH curve  has at least one positive end and one negative end.

Let  $$Z_{\textbf{y}, \alpha} : = \mathcal{L}_{H} \bigcup (\{ \infty \} \times \mathbf{y}) \bigcup ( \{- \infty \}  \times \alpha) \subset \check{W},$$
where $\check{W} =W \bigcup (\{\infty\} \times [0,1] \times \Sigma) \bigcup (\{-\infty\} \times Y_{\varphi_H})$ is the compactification.
 We use  $H_2(W,   \mathbf{y}, \alpha)$  to denote the equivalence classes   of continuous  maps  $u: (\dot{F}, \partial \dot{F}) \to (\check{W}, Z_{\mathbf{y}, \alpha}) $ satisfying  2), 3), 4) in Definition \ref{definition5}. Two  maps are equivalent if they represent  the same element in $H_2(\check{W},Z_{\textbf{y}, \alpha}; \mathbb{Z})$.   Note that  $H_2(W, \textbf{y}, \alpha)$ is an affine space of $H_2(W, \mathcal{L}_{H}; \mathbb{Z})$.  By the exact sequence
\begin{equation*}
...\to H_2(W, \mathbb{Z}) \xrightarrow{j_*}  H_2(W, \mathcal{L}_{H}; \mathbb{Z}) \xrightarrow{\partial_*} H_1(\mathcal{L}_{H},  \mathbb{Z}) \xrightarrow{i_*}  H_1(W,  \mathbb{Z}) \to...,
\end{equation*}
we know that  $H_2(W, \mathcal{L}_{H}; \mathbb{Z})$ is generated by $H_1(S^1, \mathbb{Z}) \otimes H_1(\Sigma, \mathbb{Z})$,  $[\Sigma]$ and $\ker i_*$.  Since $(W, \mathcal{L}_H)$ is diffeomorphic to $(B \times \Sigma, \partial B \times \underline{L})$, it is easy to check that  $\ker i_*$ is generated by    $[B_i]$ ($1 \le i\le k$), where  $[B_i]$ is the class represented by the parallel translation of $B_i \subset \pi_W^{-1}(3, 1)$. Because $\sum_{i=1}^{k+1}[B_i] =[\Sigma]$,   $H_2(W, \mathcal{L}_{H}; \mathbb{Z})$ is generated by   $[B_i]$ ($1\le i \le k+1$) and  $H_1(S^1, \mathbb{Z}) \otimes H_1(\Sigma, \mathbb{Z})$.
Therefore, the difference  of any two relative homology  classes  can be written as
	\begin{equation*}
		\mathcal{Z}'- \mathcal{Z}=\sum_{i=1}^{k+1} c_i [B_i] + [S],
	\end{equation*}
where  $[S]$ is a  class  in the $H_1(S^1, \mathbb{Z}) \otimes H_1(\Sigma, \mathbb{Z})$-component of $H_2(Y_{\varphi_H}, \mathbb{Z})$.

    We remark that the HF-PFH curves are simple because they are asymptotic to    Reeb chords. Thus, we have the following transversality result.
\begin{lemma}
There exists a Baire subset $\mathcal{J}^{reg}_{tame}(W, \Omega_H)$ of $\mathcal{J}_{tame}(W, \Omega_H)$ such that for  $J \in \mathcal{J}^{reg}_{tame}(W, \Omega_H)$, the $J$-holomorphic HF-PFH curves are Fredholm regular.
\end{lemma}
\begin{proof}
 The proof is the same as  Lemma 9.12  of  \cite{H1}.
\end{proof}
As before, we call   almost complex structures  in $\mathcal{J}^{reg}_{tame}(W, \Omega_H)$ \textbf{generic}.

\subsection{ECH index and $J_0$ index} \label{open-closed-ECH-index}
The ECH index and $J_0$ index   can also be adapted to the open-closed setting.

Fix a non-vanishing vector field on $\underline{L}$. This gives a trivialization $\tau$ of $T\Sigma \vert_{\underline{L}}$. We extend it to $T\Sigma \vert_{\mathcal{L}_H}$ by using the symplectic parallel transport.  We then extend the trivialization of $T\Sigma \vert_{\mathcal{L}_H}$  in an arbitrary manner   along $\{\infty\} \times \textbf{y} $ and along $\{-\infty\} \times \alpha$. Then we  define the relative Chern number $c_1(u^*T\Sigma, \tau)$.  This is the obstruction to extending $\tau$ to $u$.

Define a real line bundle $\mathfrak{L}$ of $T\Sigma$ along $\mathcal{L}_H \cup \{\infty\}   \times \textbf{y}$ as follows.  We set $\mathfrak{L} \vert_{\mathcal{L}_H} : = T\mathcal{L}_H \cap T\Sigma$. Then extend $\mathfrak{L}$  across $ \{\infty\}   \times \textbf{y}$ by rotating in the counterclockwise direction from $T\varphi_H(\underline{L})$ to $T\underline{L}$ in $T\Sigma$ by the minimum amount.  With respect to the trivialization $\tau$,  we obtain a Maslov index for the bundle pair $(u^* \mathfrak{L}, u^*T\Sigma)$, denoted by $\mu_{\tau}(u)$.

The \textbf{Fredholm index}  of   an HF-PFH curve $u$ is
\begin{equation*}
		\begin{split}
		 \operatorname{ind}u :=-\chi(\dot{F}) -d  +2 c_1(u^*T\Sigma, \tau) + \mu_{\tau}(u) - CZ^{ind}_{\tau}(\alpha).
		\end{split}
	\end{equation*}	
The notation  $CZ^{ind}_{\tau}(\alpha)$ is explained as follows. Let $\alpha=\{(\alpha_i, m_i)\}$. Suppose that for each $i$, $u$ has  $k_i$ negative ends  and each end is asymptotic to $\alpha_i^{q_j}$. Then the  total multiplicity is $m_i=\sum_{j=1}^{k_i} q_j$. Define  $$CZ^{ind}_{\tau}(\alpha):=\sum_i \sum_{j=1}^{k_i} CZ_{\tau}(\alpha_i^{q_j}),$$ where $CZ_{\tau}$ is the Conley-Zehnder  index.

Given $\mathcal{Z} \in  H_2(W, \textbf{y}, \alpha)$, we have the relative self-intersection number $Q_{\tau}(\mathcal{Z})$ defined as before.  The  \textbf{ECH index}   is defined by   (Definition 5.6.5  of \cite{VPK})
\begin{equation*}
		\begin{split}
		I(\mathcal{Z}):= c_{1}(TW \vert_{\mathcal{Z}}, \tau) +Q_{\tau}(\mathcal{Z}) + \mu_{\tau}(\mathcal{Z}) - CZ^{ech}_{\tau}(\alpha) -d,	
		\end{split}
	\end{equation*}	
where $CZ_{\tau}^{ech}(\alpha) : = \sum_i \sum_{p=1}^{m_i} CZ_{\tau}(\alpha_{i}^{p})$.

The index inequalities in Theorem \ref{thm3} still hold in the open-closed setting.
\begin{lemma} (Cf. Theorem 5.6.9 of \cite{VPK}, Lemma 5.2 of \cite{GHC2}) \label{lem3}
The ECH index satisfies the following properties:
\begin{itemize}
\item
Let $u \in \mathcal{M}^J(\mathbf{y}, \alpha )$ be an irreducible HF-PFH curve in $(W , \Omega_H, \mathcal{L}_H)$. Then we have
\begin{equation*}
		\begin{split}
		 &I(u) \ge \operatorname{ind} u  + 2 \delta(u).\\
		\end{split}
	\end{equation*}	
 Moreover, equality  holds only if $u$ satisfies the ECH partition condition.

 \item
 If $u=\bigcup_a u_a$  is an  HF-PFH curve consisting  of several  (distinct) irreducible components, then
\begin{equation*}
		\begin{split}
		  I(u) \ge \sum_a I(u_a) + 2\sum_{a < b} \# (u_a \cap u_b).
		\end{split}
	\end{equation*}
\item
Let $u$ be an HF-PFH curve. Then     $I(u) \ge 0$  provided that $J$ is generic.

\item
Let $\mathcal{Z}, \mathcal{Z} \in H_2(W, \mathbf{y}, \alpha)$  be relative  homology classes  such that $$\mathcal{Z}'-\mathcal{Z}=n[\Sigma]+\sum_{i=1}^{k+1} c_i[B_i] + [S], $$ where $[S] \in H_1(S^1, \mathbb{Z}) \otimes H_1(\Sigma, \mathbb{Z})$.   Then  we have
		\begin{equation} \label{eq5}
\begin{split}
			&I(\mathcal{Z}')=I(\mathcal{Z}) + \sum_{i=1}^{k+1} 2c_i  + 2n(k+1).
\end{split}
		\end{equation}
 \end{itemize}
\end{lemma}
In this paper, we don't need the details of the ECH partition condition.   For readers who are interested  in it, please refer to Definition 4.1 of \cite{H1} and Definition 4.13 of \cite{H2}.  

\begin{proof}
\begin{itemize}
\item
We  begin by proving  the ECH inequality for HF-PFH curves.  Let  $u$ be an irreducible HF-PFH curve.
Note that   $$c_{1}(u^*TW , (\tau, \partial_t))  =c_{1}(u^*T\Sigma , \tau ) + c_{1}(u^*TB , \partial_t)  =c_{1}(u^*T\Sigma , \tau ) + \chi(B)  =c_{1}(u^*T\Sigma , \tau ) . $$
By definition and adjunction formula (Lemma 5.6.3  of \cite{VPK}), we have
\begin{equation*}
		\begin{split}
		 &I(u) - \operatorname{ind} u  =2\delta(u)+ w_{\tau}(u) +CZ^{ind}_{\tau}(\alpha)  - CZ^{ech}_{\tau}(\alpha), \\
		\end{split}
	\end{equation*}	
where $w_{\tau}(u)$ is  the total writhe of the
braids $u(\dot{F})\cap \{s\} \times Y_{\varphi_H}$ for $s\ll -1$  with respect to $\tau$. See Definition 2.8 of \cite{H2} for its definition.

 By Lemma 6.13 of \cite{H1}, we have $ w_{\tau}(u) +CZ^{ind}_{\tau}(\alpha) \ge CZ^{ech}_{\tau}(\alpha)$ and equality holds only if $u$ satisfies the ECH partition condition.  This implies the first bullet.

\item
To prove the second statement, without loss of generality, assume that $u=u_0\bigcup u_1$ has two distinct irreducible components, where   $u_i \in \mathcal{M}^J(\mathbf{y}_i, \alpha_i)$. By Lemma 8.5 of \cite{H1},
\begin{equation} \label{eq41}
Q_{\tau}(u_0, u_1) = \# (u_0 \cap u_1) -l_{\tau}(u_0, u_1),
\end{equation}
 where  $l_{\tau}(u_0, u_1)$ is  the total  linking number  of the
braids $u_0(\dot{F}_0)\cap \{s\} \times Y_{\varphi_H}$ and $u_1(\dot{F}_1)\cap \{s\} \times Y_{\varphi_H}$ for $s\ll -1$  with respect to $\tau$  (see Definition 2.9 of \cite{H2}).
 Because  Chern number and Maslov  index are  additive and  the  relative intersection number is quadratic, we have
\begin{equation*}
		\begin{split}
		 &I(u_0 \bigcup u_1)  - I(u_0) -I(u_1)    \\
		 =&2\#(u_0 \cap u_1) -2l_{\tau}(u_0, u_1) +CZ_{\tau}^{ech}(\alpha_0)  + CZ_{\tau}^{ech}(\alpha_1) -CZ_{\tau}^{ech}(\alpha)     \\
		\end{split}
	\end{equation*}	
By Lemma 4.17 and  Lemma 5.10 of \cite{H2}, we have $$ CZ_{\tau}^{ech}(\alpha_0)  + CZ_{\tau}^{ech}(\alpha_1) \ge 2l_{\tau}(u_0, u_1)+ CZ_{\tau}^{ech}(\alpha). $$ Then we  get the second bullet.

\item
We now show that the ECH index is nonnegative when $J$ is generic. By the first bullet and the genericity of $J$, we have $I(u_a) \ge \operatorname{ind} u_a +2\delta(u_a) \ge 0$.  By the intersection positivity of holomorphic curves, we have $\#(u_a \cap u_b) \ge0$.  Therefore, the third bullet follows from the second bullet.

\item
The proof  of the last statement is essentially the same as   the proof of the fourth bullet  in Theorem \ref{thm3}. Let $u$ be a $\tau$-trivial representative of $\mathcal{Z}$. For $1\le i \le k$, we modify an end of $u$, denote the result by $u'$ such that $[u']= \mathcal{Z} +[B_i]$ (see Lemma 2.4 of \cite{GHC2} for the construction). Then   (\ref{eq38}) still holds.  It tells us that adding $[B_i]$ to $\mathcal{Z}$  increases  the ECH index by $2$. 

To see the contribution from $[\Sigma] + [S]$, we just need to repeat the  computations in (\ref{eq48}) and (\ref{eq49}).  By using $\sum_{i=1}^{k+1} [B_i] =[\Sigma]$ and the trick in (\ref{eq50}), we know that adding $[B_{k+1}]$ to $\mathcal{Z}$ increases the ECH index by $2$.
\end{itemize}
\end{proof}

Define the  \textbf{$J_0$ index} of $\mathcal{Z}$ by
\begin{equation*}
		\begin{split}
		J_0(\mathcal{Z}):= -c_{1}(TW \vert_{\mathcal{Z}}, \tau) +Q_{\tau}(\mathcal{Z})  - CZ^{J_0}_{\tau}(\gamma),	
		\end{split}
	\end{equation*}	
where $CZ^{J_0}_{\tau}(\alpha) = \sum_i \sum_{p=1}^{m_i-1} \mu_{\tau}(\alpha_i^{p})$. The following lemma  is an analogue of Lemma \ref{lem12}.

\begin{lemma}
\label{lem17}
The $J_0$ index satisfies the following properties:
\begin{itemize}
\item
Let $u \in \mathcal{M}^J(\mathbf{y}, \alpha )$ be an irreducible HF-PFH curve. Then we have
\begin{equation*}
		\begin{split}
       J_0(u)\ge 2(g(F) -1    + \delta(u)) + \#\partial F + |\alpha|,
		\end{split}
	\end{equation*}	
where $|\alpha| $ is a quantity satisfying $|\alpha| \ge 1$ provided that $\alpha$ is nonempty (see Definition 6.4 of \cite{H2}).

 \item
 If $u=\bigcup_a u_a$  is an  HF-PFH curve consisting  of several  (distinct) irreducible components, then
\begin{equation*}
		\begin{split}
		  J_0(u) \ge \sum_a J_0(u_a) +2  \sum_{a <b} \# (u_a \cap u_b).
		\end{split}
	\end{equation*}
\item
Let $u$ be an HF-PFH curve.  Then $J_0(u) \ge 0$.

\item
		Let $\mathcal{Z}, \mathcal{Z} \in H_2(W, \mathbf{y}, \alpha)$  be relative  homology classes  such that $$\mathcal{Z}'-\mathcal{Z}=n[\Sigma]+\sum_{i=1}^{k+1} c_i[B_i] + [S], $$ where $[S] \in H_1(S^1, \mathbb{Z}) \otimes H_1(\Sigma, \mathbb{Z})$.   Then  we have
		\begin{equation}
\begin{split}
J_0(\mathcal{Z}')=J_0(\mathcal{Z}) + 2c_{k+1} (d+g-1)+ 2n(d+g-1).
\end{split}
		\end{equation}
 \end{itemize}
\end{lemma}
\begin{proof}
 \begin{itemize}
 \item
By definition and adjunction formula (Lemma 5.6.3  of \cite{VPK}), we obtain
\begin{equation*}
		\begin{split}
		 J_0(u)  &=-\chi(\dot{F}) + w_{\tau}(u)  + 2\delta(u) - CZ^{J_0}_{\tau}(\alpha) \\
&=2(g(F)-1 + \delta(u)) +\#\partial F + \#\Gamma + w_{\tau}(u)   - CZ^{J_0}_{\tau}(\alpha),
		\end{split}
	\end{equation*}	
where $\Gamma$ is the set of interior punctures.
By (6.2) of \cite{H2}, we have  $$  \#\Gamma + w_{\tau}(u)   - CZ^{J_0}_{\tau}(\alpha) \ge |\alpha|.$$ Hence, the inequality in the first statement holds.

\item
Again, assume that $u=u_0\bigcup u_1$ has two distinct irreducible components, where   $u_i \in \mathcal{M}^J(\mathbf{y}_i, \alpha_i)$.
Because the Chern number and Maslov index are additive and the relative intersection number is quadratic, and by (\ref{eq41}), we have
\begin{equation*}
		\begin{split}
		 &J_0(u_0 \bigcup u_1)  - J_0(u_0) -J_0(u_1)    \\
		 =&2\#(u_0 \cap u_1) -2l_{\tau}(u_0, u_1) +CZ_{\tau}^{J_0}(\alpha_0)  + CZ_{\tau}^{J_0}(\alpha_1) -CZ_{\tau}^{J_0}(\alpha).    
		\end{split}
	\end{equation*}	
By Lemma 4.17 and Lemma 6.15 of \cite{H2}, we have $ CZ_{\tau}^{J_0}(\alpha_0)  + CZ_{\tau}^{J_0}(\alpha_1) \ge 2l_{\tau}(u_0, u_1)+ CZ_{\tau}^{J_0}(\alpha) $.  Then we  get the second bullet.

\item
Because each irreducible HF-PFH curve  has at least one  boundary component  and each $\alpha_a$ is nonempty,  by the first bullet, we have $ J_0(u_a)\ge 0.$ Then $J_0(u) \ge  0$   follows from the second bullet and intersection positivity of holomorphic curves.

\item
The proof of the fourth statement is the same as that in Lemma \ref{lem3}.
 \end{itemize}
\end{proof}

\subsection{Construction and invariance of $\widetilde{\mathcal{OC}}$} \label{construction-invariance}
In this subsection, we outline the construction of the open-closed morphisms. Also, we will explain why  why they satisfy the partial invariance and commute with the $U$ maps.

To define the open-closed morphisms, the key ingredient is the following compactness result.
\begin{lemma}\label{lem16}
Let $J \in \mathcal{J}_{tame}(W, \Omega_H)$ be a generic almost complex structure. Then we have the  following statements. 
\begin{enumerate}
  \item
 If $I(\mathcal{Z})=0$, then  $\mathcal{M}^J(\mathbf{y}, \alpha, \mathcal{Z})$ is a finite set of points.
 \item
Suppose that  $I(\mathcal{Z})=1$ and $\alpha$ is a PFH generator. Let $\{u_n\}_{n=1}^{\infty} \subset \mathcal{M}^J(\mathbf{y}, \alpha, \mathcal{Z})$ be a sequence of HF-PFH curves. Then $\{u_n\}_{n=1}^{\infty}$ converges to a broken holomorphic curve $\mathbf{u}$ in the sense of SFT \cite{FYHKE}. Moreover, $\mathbf{u}$ belongs to one of the following types:
\begin{enumerate}
  \item
  $\mathbf{u} \in  \mathcal{M}^J(\mathbf{y}, \alpha, \mathcal{Z})$;
  \item
    $\mathbf{u} =\{u^0, u^1\}$, where   $u^0 \in \mathcal{M}^J( \mathbf{y}', \alpha)$ is an embedded  HF-PFH curve  with $I=\operatorname{ind} =0$,  and $u^1 \in  \mathcal{M}^J(\mathbf{y}, \mathbf{y}')$ is an embedded  HF curve  with 
    $I=\operatorname{ind} =1.$
  \item
  $\mathbf{u} =\{u^{1}, v_1,..., v_k, u^0\}$, where   $u^0 \in  \mathcal{M}^J( \mathbf{y}, \beta)$ is an embedded  HF-PFH curve  with $I=\operatorname{ind} =0$, $u^1 \in \mathcal{M}^J( \beta, \alpha)$ is  a    PFH curve  with $I=\operatorname{ind} =1$, and $v^i \in \mathcal{M}^J(\beta, \beta)$ are connectors  (branched covers of trivial cylinders) with $\operatorname{ind} =0$.
\end{enumerate}
\end{enumerate}
\end{lemma}
\begin{proof}
 Let $\{u_n\}_{n=1}^{\infty} \subset \mathcal{M}^J(\mathbf{y}, \alpha, \mathcal{Z})$ be a sequence of HF-PFH curves.  By the first two bullets of Lemma \ref{lem17}, we may assume that the   domains of $\{u_n\}_{n=1}^{\infty} $ have a fixed topological type.
By  SFT compactness \cite{FYHKE}, after passing to a subsequence, $\{u_n\}_{n=1}^{\infty} $  converges to a broken holomorphic curve $\mathbf{u} = \{u^{-N_-}, ..., u^0,..., u^{N_+}\}$, where $u^0$ is  a curve in $W$, $u^{i}$ are curves in $\mathbb{R} \times Y_{\varphi_H}$ for $i<0$, and $u^{i}$ are curves in $ \mathbb{R} \times [0,1] \times \Sigma$ for $i>0$.  Moreover, we have
\begin{equation} \label{eq72}
\sum_{i=-N_-}^{N_+} I([u^i]) = I([\mathbf{u}]) =I(\mathcal{Z}).
\end{equation}

 Decompose $u^0=u^0_{\star} \bigcup v$, where $u^0_{\star}$ is an HF-PFH curve and $v$ is a bubble.  Without loss of generality, assume that $v$ is irreducible.  By  the open mapping theorem, the image of $v$ is  inside a fiber $\pi_W^{-1}(b)$.  If $b \in \partial B$, then the homology class of $v$ is $[v] =\sum_{i=1}^{k+1} c_i [B_i^b]$, where $\bigcup_{i=1}^{k+1} \mathring{B}_i^b = \Sigma\setminus (\mathcal{L}\cap  \pi_W^{-1}(b))$ and $B_i^b$ is the closure of $\mathring{B}_i^b$. Fix $z_i \in \mathring{B}_i$. Define
$$n_{z_i}(v): =\#(\mathbb{R}\times  \Psi_H(S^1 \times z_i) \cap v),$$
where $\Psi_H$ is the trivialization (\ref{eq34}). 
Here  we regard $W$ as a submanifold of $\mathbb{R} \times Y_{\varphi_H}$. The intersection number $n_{z_i}$ only depends on the homology class of $v$. By definition,  we have $n_{z_i}(B_j^b) =\delta_{ij}$. Hence, $n_{z_i}(v) =c_i$.  Because $v$ is holomorphic, the orientation of $v$ is the same as the fibers. Since $\mathbb{R}\times  \Psi_H(S^1 \times z_i)$ intersects the fibers positively and transversely,  $c_i= n_{z_i}(v)  \ge 0$.   By the first and fourth bullets of Lemma \ref{lem3}, we have
 \begin{equation} \label{eq73}
   I(u^0)=I(u^0_{\star} )  +2 \sum_{i=1}^{k+1}c_i \ge 0.
 \end{equation}
The above argument also works for $b \in \mathring{B}$. Combining (\ref{eq73}) with the proof of Lemma \ref{lem15}, we know that each level of $\mathbf{u}$ has nonnegative ECH index and the bubbles contribute at least two to the ECH index.

Suppose that $I(\mathcal{Z}) =0$.  Then  (\ref{eq72}) implies that no bubbles exist and $I(u^i) =0$ for  $-N_- \le i\le N_+$.  For $i>0$, the HF curves with zero ECH index   are just unions of trivial strips which are ruled out. For $i<0$, $u^i$ are  branched covers of the trivial cylinders. By Lemma 1.7 of \cite{HT1},  $\operatorname{ind} u^i  \ge 0$.  Because the Fredholm indices are additive,   we have 
$$\sum_{i=-N_-}^0\operatorname{ind} u^i =\operatorname{ind} u_n $$ for $n \gg1$.    By the ECH inequality in Lemma \ref{lem3}, $\operatorname{ind} u_n =0$.  Therefore,  $\operatorname{ind} u^i =0$ for each $i$.
By the first bullet of Lemma \ref{lem3},  the negative ends of $u_n$ satisfy the ECH partition condition.  So does  $u^{-N_-}$.  By  exercise 3.14 of  \cite{H4}, if  the negative ends of a $\operatorname{ind}=0$ branched cover of a trivial cylinder  satisfy  the ECH partition condition, then  the covering must be trivial. In other words, $u^{-N_-}$   must consist of trivial covers of  the trivial cylinders  which are also ruled out. In sum, $ \mathcal{M}^J( \mathbf{y}, \alpha, \mathcal{Z})$ is compact.

If $I(\mathcal{Z}) =1$, then the same argument can also be used  to rule out the bubbles because each bubble  increases the ECH index by $2$.  Then,  it remains only to repeat the argument  in Theorem 6.1.4. of \cite{VPK}.
\end{proof}

Recall that $W$ is a subset of $\mathbb{R}\times  Y_{\varphi_H}$. Let  $\mathcal{Z}_{ref}\in H_2(W, \mathbf{x}_H,  \gamma^{\mathbf{x}}_H)$   be a reference relative homology class that is represented by $(\mathbb{R} \times \Psi_H(S^1 \times \mathbf{x}) ) \cap W$. The open-closed morphism  at the chain level is defined  by
\begin{equation*}
\widetilde{\mathcal{OC}}(\underline{L}, H)_J (\mathbf{y}, [A]) := \sum_{(\alpha, [Z])}  \sum_{\mathcal{Z}, I(\mathcal{Z}) =0}\#\mathcal{M}^J(\textbf{y}, \alpha, \mathcal{Z}) (\alpha, [Z]),
\end{equation*}
 The class $Z$ is characterized   by $A\#\mathcal{Z}_{ref} \#(-Z) =\mathcal{Z}. $
 By Lemma \ref{lem16},  the above equation is well defined.  Using Hutchings-Taubes's gluing analysis \cite{HT1, HT2} and the compactness result in the second bullet of Lemma \ref{lem16},  $\widetilde{\mathcal{OC}}(\underline{L}, H)_J$ is a chain map.  We refer  the reader to Section 6.5 of \cite{VPK} for a nice overview of  Hutchings-Taubes's gluing argument. The authors of   \cite{VPK} also explain why the gluing argument can be adapted to the open-closed setting therein.  Therefore, $ \widetilde{\mathcal{OC}}(\underline{L}, H)_J$ descends to a homomorphism at the homological level
 \begin{equation*}
\widetilde{\mathcal{OC}}(\underline{L}, H)_J: HF(\Sigma, \varphi_H(\underline{L}), \underline{L}) \to \widetilde{PFH}(\Sigma,  \varphi_H, d).
\end{equation*}

In what follows,  we outline the proof of partial invariance and of the fact that     $\widetilde{\mathcal{OC}}(\underline{L}, H)_J$ intertwines the $U$ maps   in Theorem \ref{thm6}. The arguments  consist of the following key steps:
 \begin{itemize}
 \item
 Consider a family of  tuples  $(\Omega_{\tau}, \mathcal{L}_{\tau}, J_{\tau})_{\tau \in [0,1]}$, where $\Omega_{\tau}$ is a symplectic form of  $W$, $J_{\tau}\in \mathcal{J}_{tame}(W, \Omega_{\tau})$, and  $ \mathcal{L}_{\tau} \subset \partial W$ is a disjoint union of $d$  $\Omega_{\tau}$-Lagrangian submanifolds.  Moreover, $\mathcal{L}_{\tau} \cap   \Sigma$ is Hamiltonian isotopic  to $\underline{L}$  and  
 $$(\mathcal{L}_{\tau}, J_{\tau})  \vert_{s \ge R_0} =\left( \mathbb{R}_{s \ge R_0} \times \left( \{0\} \times\varphi_H(\underline{L}) \cup \{1\} \times \underline{L} \right), J \right) \mbox{ and }  J \vert_{s \le -R_0} =J.$$ 
 If $J_{\tau_{\star}}$($\tau_{\star} \in [0,1]$) is generic, then we can define a homomorphism $\widetilde{\mathcal{OC}}(\Omega_{\tau_{\star}}, \mathcal{L}_{\tau_{\star}})_{J_{\tau_{\star}}}$ by counting $I=0$ HF-PFH curves as before.

 \item
 Suppose that $J_0, J_1$ are generic, and the family $ \{J_{\tau}\}_{\tau \in [0,1]}$ is generic in the sense that $\mathcal{M} : =\{(u, \tau) \vert \tau \in [0,1], u \in \mathcal{M}^{J_{\tau} }(\textbf{y}, \alpha, \mathcal{Z})\}$  is a manifold of expected dimension.  Consequently, any $J_{\tau}$ HF-PFH curve has Fredholm index at least $-1$. Then we define a map  
 $$K:CF(\Sigma,   \varphi_H(\underline{L}), \underline{L} ) \to \widetilde{PFC}(\Sigma,  \varphi_H, d)$$
  by counting $I=-1$ HF-PFH curves.    We have a similar compactness result  as in Lemma \ref{lem16} (see Lemma 4.8 of \cite{GHC2} for its counterpart in closed-open setting).  Using the compactness result and Hutchings-Taubes's gluing analysis \cite{HT1, HT2}, $K$ is a chain homotopy, i.e.,
 $$\widetilde{\mathcal{OC}}(\Omega_{0}, \mathcal{L}_0)_{J_{0}} - \widetilde{\mathcal{OC}}(\Omega_{1}, \mathcal{L}_{1})_{J_{1}} =K \circ d_J + \partial_J \circ K. $$
\item
Assume that $\varphi_H$ satisfies \ref{assumption1} and  $\varphi_G$ satisfies \ref{assumption2}. Let $(E_1, \Omega_1, \mathcal{L}_1)$ be a Lagrangian cobordism from $(\varphi_G(\underline{L}), \underline{L})$ to $(\varphi_H(\underline{L}), \underline{L})$.  Let $(X, \Omega_X)$ be  a symplectic cobordism from $(Y_{\varphi_H}, \omega_{\varphi_H})$ to $(Y_{\varphi_G}, \omega_{\varphi_G})$ defined by (\ref{eq23}).   Consider the $R$-stretched composition of  $(E_1, \Omega_1, \mathcal{L}_1)$, $(W_H, \Omega_H, \mathcal{L}_H)$ and  $(X, \Omega_X)$, denoted by  $(W_R, \Omega_R, \mathcal{L}_R)$.   Let $J_R$ be a generic family of almost complex structures converging to generic almost complex structures $J_1, J_H, J_X$ on $E_1, W, X$ respectively as $R \to \infty$.

By the second bullet, we have
\begin{equation} \label{eq42}
\widetilde{\mathcal{OC}}(\Omega_{R=0}, \mathcal{L}_{R=0})_{J_{R=0}} =\widetilde{\mathcal{OC}}(\underline{L}, G)_{J_G}.
\end{equation}
As $R \to \infty$, the $I=0$ HF-PFH curves in $(W_R, \Omega_R, \mathcal{L}_R)$  converge  to  a broken holomorphic curve $\mathbf{u}$.  Under assumptions \ref{assumption1}, \ref{assumption2}, the PFH curves in $(X, \Omega_X)$ have nonnegative ECH index  (see Section 7.1 of \cite{GHC}).   By Lemma \ref{lem3}, the bubbles contribute at least two to the ECH index. Combining the above two facts  with Theorems \ref{thm3} and Lemma \ref{lem3}, the holomorphic curves in each level have nonnegative ECH index.  As a result, these holomorphic curves have  zero ECH index and no bubbles exist. Each curve at a level   of $\mathbf{u}$    is either embedded or a branched cover.  See  Lemma 4.10 and Lemma 4.11 of \cite{GHC2} for its counterparts. By Hutchings-Taubes's gluing argument \cite{HT1, HT2},   we have
\begin{equation} \label{eq43}
\widetilde{\mathcal{OC}}(\Omega_R, \mathcal{L}_R) = PFH_{Z_{ref}}(X,\Omega_X)_{J_X}
\circ \widetilde{\mathcal{OC}}(\underline{L}, H)_{J_H} \circ \mathcal{I}_{G, H}   
\end{equation}
for $R \gg 1$.  Here $PFH_{Z_{ref}}(X,\Omega_X)_{J_X}$ is the PFH cobordism map  defined by counting embedded holomorphic curves in $X$. It follows from Theorem 2 of \cite{GHC} that  $PFH_{Z_{ref}}(X,\Omega_X)_{J_X}$  is well defined. Moreover,  by Theorem 3  of  \cite{GHC},    we have $$PFH_{Z_{ref}}(X,\Omega_X)_{J_X} = PFH^{sw}_{Z_{ref}}(X,\Omega_X) =\mathfrak{I}_{H, G} . $$ 
Again, by the second bullet, we have  $\widetilde{\mathcal{OC}}(\Omega_{R=0}, \mathcal{L}_{R=0})_{J_{R=0}}  =\widetilde{\mathcal{OC}}(\Omega_{R}, \mathcal{L}_{R})_{J_{R}} $.  Then the partial invariance follows from (\ref{eq42}) and (\ref{eq43}).

\item
To show that $\widetilde{\mathcal{OC}} (\underline{L}, H )_J $ commutes with the $U$ map,  the strategy is the same as Theorem 5.9.1 of \cite{VPK1}. We first define a ``pointed open-closed morphism" as follows. Fix a generic point $p $ in the interior of $W$.  Let $\mathcal{M}^J(\mathbf{y}, \alpha, \mathcal{Z}; p)$ be the moduli space of HF-PFH curves passing through the marked point $p$.   Combining the index computation in  Lemma \ref{lem16} and the argument in Lemma \ref{lem22},  one can show that 
\begin{enumerate}
\item
If $I(\mathcal{Z}) =2$, then $\mathcal{M}^J(\mathbf{y}, \alpha, \mathcal{Z}; p)$ is   a finite set of points. 
\item
Suppose $I(\mathcal{Z}) =3$. Then each broken holomorphic curve $\mathbf{u}$ in $\partial \mathcal{M}^J(\mathbf{y}, \alpha, \mathcal{Z}; p)$ consists of an $I=2$ HF-PFH curve $u^0$ in $W$ passing through $p$ and an  $I=1$   HF or PFH curve   $u^1$, possibly  together with $\operatorname{ind}=0$ connectors between $u^0$ and $u^1$. 
\end{enumerate} 
Then  the   ``pointed open-closed morphism"  $\widetilde{\mathcal{OC}} (\underline{L}, H; p )_J$ is defined as counting the moduli spaces   $  \mathcal{M}^J(\mathbf{y}, \alpha, \mathcal{Z}; p)$  with $I(\mathcal{Z}) =2$.   The first  assertion implies that   $\widetilde{\mathcal{OC}} (\underline{L}, H; p )_J$  is well defined. The  second  assertion,  together with the gluing analysis \cite{RL1, HT1, HT2},  shows  that   $\widetilde{\mathcal{OC}} (\underline{L}, H; p )_J$ is a chain map.  As a result, $\widetilde{\mathcal{OC}} (\underline{L}, H; p )_J$  descends to homology. 
\item
Next, we show that  
\begin{equation*}
\begin{split}
\widetilde{\mathcal{OC}} (\underline{L}, H; p )_J  \mbox{ is chain homotopic to } \widetilde{\mathcal{OC}} (\underline{L}, H )_J\circ U \mbox{ and } & U \circ \widetilde{\mathcal{OC}} (\underline{L}, H )_J.  
  \end{split}
\end{equation*}

Let $\ell: [0, \infty)_{\tau}\to W$ be a generic path such that $\ell(0)=p$ and  $\ell(\tau)=(\tau, t_+, z_+)$ when $
\tau \gg 1$, where   $(t_+, z_+)$ is   fixed.   Consider    the moduli space of HF-PFH curves passing through the marked point $\ell$, denoted by $\mathcal{M}^J(\mathbf{y}, \alpha, \mathcal{Z}; \ell)$.   Then this moduli space is an   $(I(\mathcal{Z}) -1)$-dimensional manifold.  Consider a sequence of HF-PFH curves $u_{\tau}$ such that $\ell(\tau) \in \operatorname{Im}u_{\tau}$.   

By arguments similar to those in Lemma      \ref{lem16}  and  Lemma  \ref{lem22}, we have the following statements:
\begin{enumerate}
\item
If $I(\mathcal{Z}) =1$, then $\mathcal{M}^J(\mathbf{y}, \alpha, \mathcal{Z}; \ell)$ is  a set of finite points. 

\item
If  $I(\mathcal{Z}) =2$ and  $\tau \to 0$,   then $u_{\tau}$ converges to an  HF-PFH curve in  $\mathcal{M}^J(\mathbf{y}, \alpha, \mathcal{Z}; p)$.  

\item
If  $I(\mathcal{Z}) =2$  and  $\tau \to \tau_{\star} \in (0, \infty)$ and $u_{\tau}$ converges to a broken holomorphic curve $\mathbf{u}$,    then $\mathbf{u}$ consists of an $I=1$ HF-PFH curve   $u^0$ passing through $\ell$, an $I=1$ HF curve or PFH curve  in the  symplectization  level $u^1$, possibly with additional levels  consisting of $\operatorname{ind} =0$ connectors. 
\item
If  $I(\mathcal{Z}) =2$  and   $\tau \to \infty$,  $u_{\tau}$ converges to a broken holomorphic curve $\mathbf{u}=\{u^0, u^1\}$. Here $u^0$ is an HF-PFH curve with $I=0$ and $u^1$ is an  $I=2$ HF curve passing through  the marked point $p_+=(0, t_+, z_+)$.  
\end{enumerate}
Define a map   $$K:CF(\Sigma,   \varphi_H(\underline{L}), \underline{L} ) \to \widetilde{PFC}(\Sigma,  \varphi_H, d)$$
by counting  $\mathcal{M}^J(\mathbf{y}, \alpha, \mathcal{Z}; \ell)$ with $I(\mathcal{Z})=1$. 
Again,   the above assertions   together with the gluing analysis \cite{RL1, HT1, HT2} show that   
$$\widetilde{\mathcal{OC}} (\underline{L}, H; p )_J =  \widetilde{\mathcal{OC}} (\underline{L}, H) \circ U  + K \circ d + \partial \circ K.$$ A similar argument can be applied to the case where $U \circ \widetilde{\mathcal{OC}} (\underline{L}, H )_J $ is homotopic to  $ \widetilde{\mathcal{OC}} (\underline{L}, H; p)_J$.  

In sum, $\widetilde{\mathcal{OC}} (\underline{L}, H; p )_J = \widetilde{\mathcal{OC}} (\underline{L}, H )_J\circ U = U \circ \widetilde{\mathcal{OC}} (\underline{L}, H )_J $ on homology. 

 \end{itemize}

Finally, we explain why $\widetilde{\mathcal{OC}}(\underline{L}, H)_J $ is a $\Lambda$-module homomorphism. 
By definition, we have 
\begin{equation*}
\begin{split}
&\langle \widetilde{\mathcal{OC}}(\underline{L}, H)_J (S\cdot(\mathbf{y}, [A])), S\cdot (\alpha, [Z] ) \rangle\\
 =&\langle \widetilde{\mathcal{OC}}(\underline{L}, H)_J ((\mathbf{y}, [A+\Sigma])), (\alpha, [Z+\Sigma] ) \rangle\\
=&   \sum_{\mathcal{Z} \in [A]\# \mathcal{Z}_{ref} \#(-[Z])}\#\mathcal{M}^J(\textbf{y}, \alpha, \mathcal{Z}) = \langle \widetilde{\mathcal{OC}}(\underline{L}, H)_J ((\mathbf{y}, [A])), (\alpha, [Z] ) \rangle.
\end{split} 
\end{equation*}
Hence,   we have 
\begin{equation*}
\begin{split}
 &\widetilde{\mathcal{OC}}(\underline{L}, H)_J (S\cdot(\mathbf{y}, [A])) = \sum_{(\alpha, [Z])}\langle \widetilde{\mathcal{OC}}(\underline{L}, H)_J (S\cdot(\mathbf{y}, [A])), S\cdot (\alpha, [Z] )\rangle  S\cdot (\alpha, [Z])\\
=&\sum_{(\alpha, [Z])}\langle \widetilde{\mathcal{OC}}(\underline{L}, H)_J ((\mathbf{y}, [A])), (\alpha, [Z] ) \rangle S\cdot (\alpha, [Z])=S \cdot   \widetilde{\mathcal{OC}}(\underline{L}, H)_J ((\mathbf{y}, [A])).
\end{split} 
\end{equation*}
Consequently, $\widetilde{\mathcal{OC}}(\underline{L}, H)_J $ is a $\Lambda$-module homomorphism.

\section{Computations of $\widetilde{\mathcal{OC}}$} \label{computing-open-closed}
In this section, we  compute the open-closed morphism for  a special Hamiltonian function $H$ satisfying \ref{assumption1}.  Using  partial invariance, we deduce the non-vanishing result under the assumptions  \ref{assumption1} and  \ref{assumption2}.   The main idea here is   the same as  that in \cite{GHC2}.

 \subsection{A suitable choice of Hamiltonian  functions}
 To begin with, we need to clarify our choice of   a special Hamiltonian function $H$ satisfying \ref{assumption1}. 
 
Suppose that $f$ is a  Morse function satisfying \ref{M1}, \ref{M2}, \ref{M3},  and \ref{M4}. For a sufficiently large constant $ \kappa \gg 1$,  define $H_{1/\kappa} := - 1/\kappa f$.  $H_{1/\kappa}$ is a slight perturbation of the height function in Figure \ref{figure1}. This is a nice candidate for computations because we can describe the periodic orbits and Reeb chords in terms of the critical points of $f$,  and the indices  of holomorphic curves are computable.   However,  $H_{1/\kappa}$ does not satisfy  \ref{assumption1} or \ref{assumption2}.  We need to follow   the discussion in Section 5.1 of \cite{GHC2} to modify $H_{1/\kappa}$.

  Fix  numbers $ \kappa_0 \gg 1$ and $ \delta, \delta_0>0$. Following    the construction in Section 5.1 of \cite{GHC2}, we  take   a smooth function $\varepsilon : \Sigma  \to \mathbb{R}$ such that $0< 1/\kappa \le \varepsilon \le  1/\kappa_0$,     and    define a  new autonomous Hamiltonian function 
  $$H_{\varepsilon} :=- \varepsilon f$$ 
  satisfying   the following properties:
\begin{enumerate} [label=\textbf{F.\arabic*}]
\item \label{F1}
There is  a collection of open disks $\mathcal{U}^{\delta+ \delta_0} = \bigcup_p U^{\delta + \delta_0}_p$ with radius $\delta+ \delta_0$  such that $H_{\varepsilon} \vert_{\Sigma\setminus \mathcal{U}^{\delta+ \delta_0}}=H_{1/\kappa} \vert_{\Sigma\setminus \mathcal{U}^{\delta+ \delta_0}} $, where $p$ runs over all the local maxima  of $-f$ and  $U^{\delta + \delta_0}_p $ is a   $(\delta+ \delta_0)$-neighbourhood  of $p$.

 \item  \label{F2}
$( H_{\varepsilon}, g_{\Sigma})$ still satisfies  the Morse-Smale conditions. Moreover, $\operatorname{Crit}(H_{\varepsilon}) = \operatorname{Crit}(- f)$.

\item \label{F3}
 $\varphi_{H_{\varepsilon}}$ is  $d$-nondegenerate.  Furthermore,  the periodic orbits of  $\varphi_{H_{\varepsilon}}$  with period at most  $d$ are covers of the constant orbits at critical points of $H_{\varepsilon}$.

\item \label{F4}
For each local maximum $p$, $\varphi_{H_{\varepsilon}}$ has a family  of periodic orbits $\{\gamma_{r_0, \theta}(t)\}_{\theta \in [0, 1]}$ that foliates $S_t^1 \times \partial \mathcal{U}^{r_0}_p$, where $\delta + \delta_0  \le r_0 \le \delta +2 \delta_0$. Moreover, the period of  $\gamma_{r_0, \theta}(t)$  is  strictly greater than $d$.

\item \label{F5}
The Reeb chords  of $\varphi_{H_{\varepsilon}}$ are  still
 corresponding to the critical points of $\bigcup_{i=1}^df_{L_i}$.   See (\ref{eq14}).
\end{enumerate}

By Proposition  3.7 of \cite{GHC}, we perturb $H_{\varepsilon}$ to a new Hamiltonian function $H_{\varepsilon}'$ (which may depend on $t$) such that it satisfies the following properties:
\begin{enumerate} [label=\textbf{G.\arabic*}]
\item \label{G1}
  $H'_{\varepsilon} \vert_{\Sigma \setminus \mathcal{U}^{\delta}}=H_{\varepsilon} \vert_{\Sigma\setminus \mathcal{U}^{\delta}} $.

\item
$H_{\varepsilon} '$ still satisfies  \ref{F4} and \ref{F5}.
\item
$|H_{\varepsilon} ' - H_{\varepsilon} | \le c_0 \delta$ and   $|d H_{\varepsilon} ' -dH_{\varepsilon} | \le c_0 \delta$.

 \item
The periodic orbits of $\varphi_{H'_{\varepsilon}}$   with period less than or equal to $d$ are either hyperbolic or $d$-negative elliptic.  In other words,   $\varphi_{H'_{\varepsilon}}$ is  $d$-nondegenerate  and    satisfies \ref{assumption1}.
\end{enumerate}

\begin{remark}
Because we take $H_{\varepsilon} =-\varepsilon f$, the maximum points $\{y_i^+\}$ of $f$ are the minimum points of  $H_{\varepsilon}$.  We use $\{y_-^i\}$ to denote the minimum points  of $H_{\varepsilon}$ from now on.
\end{remark}


Let $y$ be a critical point of $H_{\varepsilon}$. Let $\gamma_{y}$ denote the constant simple periodic orbit at the critical point $y$. Note that the period of $\gamma_y$ is $1$.  We define special PFH generators and a Reeb chord as follows:
\begin{enumerate}
\item
Let $I=(i_1, ...,i_d)$. Here we allow $i_j=i_k$ for $j\ne k$. Define $$\alpha_I:=\gamma_{y_{-}^{i_1}} ... \gamma_{y_{-}^{i_d}}.$$  When $I=(1,2, ..., d)$, we denote $\alpha_I$ by $\alpha_{\diamondsuit}$.   Here and below, we use multiplicative notation for orbit sets. 

\item
$\mathbf{y}_{\diamondsuit}:=[0,1] \times (y_-^1, ...,y_-^d)$.
\end{enumerate}
 Let $\alpha=\gamma_{p_1} \cdots \gamma_{p_d}$ be   an  orbit set, where $p_i  \in \operatorname{Crit}(H_{\varepsilon})$.  Define a relative homology class $Z_{\alpha}$ as follows:  Let $\eta=\bigsqcup_{i=1}^d \eta_i: \bigsqcup_{i=1}^d [0,1] \to \Sigma$ be a union of paths with $d$ components such that $ \eta_i(1) =p_i$ and $ \eta_i(0) =x_i$. Define  a relative homology class by
	\begin{equation}
		Z_{\alpha}:=[\Psi_{H_{\varepsilon}}(S^1 \times \eta)] \in H_2(Y_{\varphi_{H_{\varepsilon}}}, \alpha, \gamma_{H_{\varepsilon}}^{\mathbf{x}}).
	\end{equation}
Let  $\alpha=\gamma_{p_1} \cdots \gamma_{p_d}$ and $\beta=\gamma_{q_1} \cdots \gamma_{q_d}$ be two orbit sets, where  $p_i, q_j  \in \operatorname{Crit}(H_{\varepsilon})$.  
Define $$Z_{\alpha, \beta}:= Z_{\alpha}\#(-Z_{\beta}) \in H_2(Y_{\varphi_{H_{\varepsilon}}}, \alpha, \beta).$$

The following lemma tells us that $(\textbf{y}_{\diamondsuit}, [A_{\textbf{y}_{\diamondsuit}}])$ is a cycle.
\begin{lemma} \label{lem7}
Let $d_J$ be the differential of $CF(\Sigma,  \varphi_{H'_{\varepsilon}}(\underline{L}), \underline{L} ) $. Then $d_J=0$.  In particular,  $(\mathbf{y}_{\diamondsuit}, [A_{\mathbf{y}_{\diamondsuit}}])$ is a cycle, where $A_{\mathbf{y}_{\diamondsuit}}$ is the relative homology class  defined  in  (\ref{eq3}).
\end{lemma}
\begin{proof}
By \ref{F5}, we know that $CF(\Sigma, \varphi_{H'_{\varepsilon}}(\underline{L}),  \underline{L}) \cong \bigoplus^{2^d} R$. According to  Lemma 6.8 in \cite{CHMSS} and  the isomorphism (\ref{eq30}), we know that 
$$CF(\Sigma,  \varphi_{H'_{\varepsilon}}(\underline{L}),  \underline{L} ) \cong H^*(\mathbb{T}^d, R) \cong HF(\Sigma,   \varphi_{H'_{\varepsilon}}(\underline{L}), \underline{L})$$
as $R$-modules. Since $$\operatorname{dim}_{R} HF(\Sigma,   \varphi_{H'_{\varepsilon}}(\underline{L}), \underline{L} )  \le \operatorname{dim}_{R}  \ker d_J \le \operatorname{dim}_{R} CF(\Sigma,   \varphi_{H'_{\varepsilon}}(\underline{L}), \underline{L}),$$ we  must have $d_J= 0$.
\end{proof}

  \subsection{Proof  of  the nonvanishing result}
By Lemma \ref{lem7}, $\widetilde{\mathcal{OC}}(\underline{L}, H'_{\varepsilon})_J (\mathbf{y}_{\diamondsuit}, [A_{\mathbf{y}_{\diamondsuit}}])$ is a cycle. To prove  the non-vanishing result, it suffices to show that this cycle is not exact. Before giving the technical lemmas, we explain the role of each step.

Our argument rests on the following three observations. 
\begin{enumerate}[label=\textbf{O.\arabic*}]
\item \label{O1}
For a suitable choice of almost complex structure, the ``constant holomorphic curve'' at $ \mathbf{y}_{\diamondsuit}= (y_-^1,\ldots,y_-^d)$ is a Fredholm-regular $I=0$ holomorphic curve from $ (\mathbf{y}_{\diamondsuit},[A_{\mathbf{y}_{\diamondsuit}}])$ to $(\alpha_{\diamondsuit},[Z_{\diamondsuit}])$ (Lemma \ref{lem21}).
\item  \label{O2}
The energy of  each  holomorphic curve is nonnegative, and the ``constant holomorphic curve" is the unique minimizer  of the energy (Lemma \ref{lem8}). These facts follow from the method of P. Seidel \cite{Sei}. Following his terminology, we call the ``constant holomorphic curve'' a horizontal section.
\item  \label{O3}
Using Lemma \ref{lem3} and Lemma \ref{lem17},  we know that ECH index and energy satisfy the monotonicity condition:
$$\frac{\lambda}{2}\left(I(\mathcal{Z}_1) - I(\mathcal{Z}_2) \right) =\int_{\mathcal{Z}_1 \#(-\mathcal{Z}_2)} \omega_{\varphi_H} + \eta J_0(\mathcal{Z}_1\#(-\mathcal{Z}_2)),$$
where $\mathcal{Z}_1, \mathcal{Z}_2 \in H_2(W, \mathbf{y}, \alpha)$. 
It is important to note that  if $\mathcal{Z}_2$ is represented by the horizontal   sections, then  the ECH index of   $\mathcal{Z}_1$ is positively proportional to its  energy plus $\eta J_0$-index. 

\end{enumerate}

Roughly speaking, Observations \ref{O1} and \ref{O2} are used to find a nonzero leading term in the open-closed morphisms, and Observation \ref{O3} is used to exclude certain ``higher order terms".   When we use the Morse function $H_{1/\kappa}$ for computation, one can verify these observations   and  use  them  to  prove (\ref{eq87}). But now we are  working with  $H_{\varepsilon}'$,  which is not a Morse function, certain adjustments are required.  
For example, we cannot verify Observation \ref{O3} if we consider a holomorphic curve in   a symplectic cobordism from $(Y_{\varphi_{H'_{\varepsilon}}} , \omega_{\varphi_{H'_{\varepsilon}}})$ to $(Y_{\varphi_{H_{\varepsilon}}} , \omega_{\varphi_{H_{\varepsilon}}})$,   because the dynamics  of $\varphi_{H'_{\varepsilon}}$ is not clear. 
However, if  we only  consider the holomorphic currents with a  positive end   at  $\alpha_I $ and negative ends at  periodic orbits of $\varphi_{H_{\varepsilon}}$,   then the ECH index and energy   of  the holomorphic currents are computable,  and thus Observation \ref{O3} can be verified.

Lemmas \ref{lem14}, \ref{lem4}, \ref{lem5}, \ref{lem6}, and \ref{lem11}  adapt and repeatedly  apply    these three observations in varied situations. Their logical roles are summarized  as follows.
\begin{itemize}
\item
In Lemma \ref{lem14},
we show that $\widetilde{\mathcal{OC}}(\underline{L}, H'_{\varepsilon})_J (\mathbf{y}_{\diamondsuit}, [A_{\mathbf{y}_{\diamondsuit}}])$ splits as
$$ 
\widetilde{\mathcal{OC}}(\underline{L}, H'_{\varepsilon})_J (\mathbf{y}_{\diamondsuit}, [A_{\mathbf{y}_{\diamondsuit}}])=  (\alpha_{\diamondsuit}, [Z_{\diamondsuit}])+  \mbox{$(\beta, Z)$-terms, $ \beta  \ne  \alpha_I$}. $$
 The leading term $(\alpha_{\diamondsuit},[Z_{\diamondsuit}])$ follows from Observations \ref{O1} and \ref{O2}. The less immediate assertion is $\beta\ne\alpha_I$. If we have a holomorphic curve from  $\mathbf{y}_{\diamondsuit}$ to  $\beta=\alpha_I$, then the curve  would have positive energy by Observation \ref{O2}, hence positive ECH index by Observation \ref{O3}, contradicting  the definition of $\widetilde{\mathcal{OC}}(\underline{L}, H'_{\varepsilon})_J$.

 \item 
To show that $\widetilde{\mathcal{OC}}(\underline{L}, H'_{\varepsilon})_J (\mathbf{y}_{\diamondsuit}, [A_{\mathbf{y}_{\diamondsuit}}])$ is a nonexact cycle, in Lemmas \ref{lem4}, \ref{lem5}, and \ref{lem6}, we pass to the PFH cobordism map induced by a symplectic cobordism $(X,\Omega_X)$ from $(Y_{\varphi_{H'_{\varepsilon}}},\omega_{\varphi_{H'_{\varepsilon}}})$ to
$(Y_{\varphi_{H_{\varepsilon}}}, \omega_{\varphi_{H_{\varepsilon}}})$  defining the continuation morphism $\mathfrak{I}_{H'_{\varepsilon}, H_{\varepsilon}}$.  We  are going to compute  $\mathfrak{I}_{H'_{\varepsilon},H_{\varepsilon}}\circ\widetilde{\mathcal{OC}}(\underline{L},H'_{\varepsilon})_J (\mathbf{y}_{\diamondsuit},[A_{\mathbf{y}_{\diamondsuit}}])$.   
  \item
  Lemma \ref{lem4} adapts Observations \ref{O1} and \ref{O2} to this PFH cobordism setting. The  energy  of holomorphic curves in $(X,\Omega_X)$ need not be nonnegative.  We  only consider holomorphic curves with a positive end at $\alpha_{\diamondsuit}$.   Then the energy  of the curves could be negative only if they   cross the barriers  $\mathbb{R} \times \gamma_{r_0, \theta}$ in \ref{F4}.  By intersection positivity of holomorphic curves, we show that crossing  the barriers  actually increases the energy.   The underlying logic is that the intersection number  with  $\mathbb{R} \times \gamma_{r_0, \theta}$ is positively proportional to the  energy. A similar idea is also used  in Lemma \ref{lem6}.

 \item
To compute $\mathfrak{I}_{H'_{\varepsilon},H_{\varepsilon}}\circ\widetilde{\mathcal{OC}}(\underline{L},H'_{\varepsilon})_J (\mathbf{y}_{\diamondsuit},[A_{\mathbf{y}_{\diamondsuit}}])$, we separate the leading term   $\mathfrak{I}_{H'_{\varepsilon},H_{\varepsilon}}(\alpha_{\diamondsuit},[Z_{\diamondsuit}])$ from the remaining terms  $\mathfrak{I}_{H'_{\varepsilon},H_{\varepsilon}}(\beta,[Z])$. Lemma \ref{lem5} is an analogue of  Lemma \ref{lem14}. It  treats $\mathfrak{I}_{H'_{\varepsilon},H_{\varepsilon}}(\alpha_{\diamondsuit},[Z_{\diamondsuit}])$ by adapting Observations \ref{O2} and \ref{O3}.

 \item
Lemma \ref{lem6} shows that $\mathfrak{I}_{H'_{\varepsilon},H_{\varepsilon}}(\beta,[Z])=\sum(\beta',[Z'])$ with $\beta'\ne\alpha_{I}$. The key issue is to exclude  the  $\alpha_{I}$ term. The dynamics of the orbit set $\beta$ is   unknown,  so we  cannot  analyze a  holomorphic current  with a positive end at  $\beta$ directly. Instead, we concatenate it with an  HF-PFH curve from $\mathbf{y}_{\diamondsuit}$ to $\beta$ contributing to  $\langle \widetilde{\mathcal{OC}}(\underline{L},H'_{\varepsilon})_J (\mathbf{y}_{\diamondsuit},[A_{\mathbf{y}_{\diamondsuit}}]),  (\beta,[Z]) \rangle$. The resulting current from $\mathbf{y}_{\diamondsuit}$ to $\alpha_{I}$ and its index and energy are computable. We find that its existence would contradict  Observations \ref{O2} and \ref{O3}.
  
  \item  
The terms $\sum(\beta',[Z'])$ in Lemma \ref{lem5} and Lemma \ref{lem6} with $\beta'\ne\alpha_I$ form the higher-order part. Note that $\beta'$ here are orbit sets of the Morse function $H_{\varepsilon}$.   Lemma \ref{lem11} shows that the PFH cobordism map induced by (\ref{eq6}) kills them, since every holomorphic current ending at $\beta'$ has positive ECH index. The same lemma adapts Observations \ref{O1} and \ref{O2} to show that this cobordism map sends
$(\alpha_{\diamondsuit},[Z_{\diamondsuit}])$ to $1$. Thus, we obtain (\ref{eq104}), and   the original cycle cannot be exact and this finishes  the proof of nonvanishing. 
 \end{itemize}

Now we begin to prove the mentioned  lemmas. To establish Observations \ref{O1} and \ref{O2}, we consider   a smaller set of almost complex structures   $ \mathcal{J}(W, \Omega_{H'_{\varepsilon}}) $ instead. Here
 $ \mathcal{J}(W, \Omega_{H'_{\varepsilon}})  \subset  \mathcal{J}_{tame}(W, \Omega_{H'_{\varepsilon}}) $  is   a set of almost complex structures   that are restrictions of admissible  almost complex structures   in  $\mathcal{J}(Y_{\varphi_{ H'_{\varepsilon}}}, \omega_{\varphi_{H'_{\varepsilon}}} )$.
The reason  for  using such an  almost complex structure  $J  \in \mathcal{J}(W, \Omega_{H'_{\varepsilon}}) $  is that  $u_{y_i} $ is a $J$-holomorphic HF-PFH curve  in $\mathcal{M}^J(y_i, \gamma_{y_i})$, where  $u_{y_i} $  is  the restriction of $\mathbb{R} \times \gamma_{y_i}$ to $W$. It is called a \textbf{horizontal section of    $(W, \Omega_{H_{\varepsilon}'}, \mathcal{L}_{{H'_{\varepsilon}}}, J)$}. Moreover, it is easy to check  that 
$$I(u_{y_-^i} ) =\operatorname{ind} u_{y_-^i} =0$$ 
by definition.

The following lemma tells us that  the open-closed  morphism at the chain level is  still  well defined for   $J \in \mathcal{J}(W, \Omega_H)$.
\begin{lemma} \label{lem21}
We have the following transversality statements. 
\begin{itemize}
   \item
   There is a Baire subset of $\mathcal{J}(W, \Omega_H)$, denoted by $\mathcal{J}^{reg}(W, \Omega_H)$. For  $J \in \mathcal{J}^{reg}(W, \Omega_H)$,   if $u$ is  $J$-holomorphic  HF-PFH curve   which is not a horizontal section, then $u$ is Fredholm regular.
   \item
   For  $J \in \mathcal{J}(W, \Omega_H)$, if $u$ is a horizontal section with $\operatorname{ind} u=0$, then $u$ is Fredholm regular.
 \end{itemize}
\end{lemma}
\begin{proof}
 The proof is the same as the proof of Lemma 5.8 in \cite{GHC2}. Also see Lemma 9.12 of \cite{H1} and    Lemma 2.11 of \cite{Sei}. 
\end{proof}	

The Fredholm regularity   implies that for  $J' \in \mathcal{J}_{tame}(W, \Omega_{H'_{\varepsilon}})$ that is close to   $J \in \mathcal{J}(W, \Omega_{H'_{\varepsilon}})$ in the $C^\infty$ topology,  we have $\widetilde{\mathcal{OC}}(\underline{L}, H_{\varepsilon}')_{J'}= \widetilde{\mathcal{OC}}(\underline{L}, H_{\varepsilon}')_J$. Thus, we work with  $J \in \mathcal{J}(W, \Omega_{H'_{\varepsilon}})$ from now on.

Another advantage of using   $J \in \mathcal{J}(W, \Omega_{H})$ is that  the energy of HF-PFH curves is nonnegative. Moreover, the horizontal sections are characterized  by   energy.
	\begin{lemma} \label{lem8}
		Let    $J \in \mathcal{J}(W, \Omega_H)$. Let $u: \dot{F} \to W$ be a $J$-holomorphic HF-PFH curve in   $(W, \Omega_{H}, \mathcal{L}_{{H}})$.  Then the $\omega_{\varphi_H}$-energy satisfies $$E_{\omega_{\varphi_{H}} }(u) :=\int_{\dot{F}} u^* \omega_{\varphi_H}  \ge 0.$$ Moreover, when $H=H'_{\varepsilon}$, $E_{\omega_{\varphi_H}} (u) =0$ if and only if $u$ is a union of  the horizontal sections.
	\end{lemma}
\begin{proof}
The proof is the same as Lemma 6.6 in \cite{GHC2}.
\end{proof}

The horizontal section  $u_{{\diamondsuit}}:=\bigcup_{i=1}^d u_{y_-^i}$ represents  a relative homology class $\mathcal{Z}_{hor} \in H_2(W, \mathbf{y}_{\diamondsuit}, \alpha_{\diamondsuit})$.   We take the reference relative homology class to be $$\mathcal{Z}_{ref} = (-A_{\textbf{y}_{\diamondsuit}}) \# \mathcal{Z}_{hor} \#Z_{\alpha_{\diamondsuit}} \in H_2(W, \textbf{x}_H, \gamma_H^{\textbf{x}}). $$ Using the horizontal sections, we obtain the leading term of $\widetilde{\mathcal{OC}}(\underline{L}, {H'_{\varepsilon}})_J (\mathbf{y}_{\diamondsuit}, [A_{\mathbf{y}_{\diamondsuit}}])$ in the following lemma.
\begin{lemma} \label{lem14}
For a generic $J \in \mathcal{J}(W, \Omega_{H'_{\varepsilon}})$, we have
\begin{equation*}
\widetilde{\mathcal{OC}}(\underline{L}, {H'_{\varepsilon}})_J (\mathbf{y}_{\diamondsuit}, [A_{\mathbf{y}_{\diamondsuit}}])=  (\alpha_{\diamondsuit}, [Z_{\alpha_{\diamondsuit}}]) +\sum (\beta, [Z]),
\end{equation*}
where $(\beta, [Z])$ satisfies   $\beta \ne \alpha_{I}$  and $\mathbb{A}_{H'_{\varepsilon}} (\alpha_{\diamondsuit}, [Z_{\alpha_{\diamondsuit}}]) + \eta J_0([Z_{\alpha_{\diamondsuit}} -Z])> \mathbb{A}_{H'_{\varepsilon}} (\beta, [Z])$.
\end{lemma}
\begin{proof}
 Consider the moduli space of HF-PFH curves $\mathcal{M}^J(\textbf{y}_{\diamondsuit}, \alpha_{\diamondsuit}, \mathcal{Z})$ with $I(\mathcal{Z})=0$.  Then
 $$\mathcal{Z}=\mathcal{Z}_{hor} \# Z_{\alpha_{\diamondsuit}, \alpha_I} +  \sum_{i=1}^{k+1} c_{i} [B_i] + n[\Sigma] +[S]$$
 for some $c_i, n \in\mathbb{Z}$ and $[S]\in H_1(S^1, \mathbb{Z}) \otimes H_1(\Sigma, \mathbb{Z})$.
 Let $u \in \mathcal{M}^J(\textbf{y}_{\diamondsuit}, \alpha_I, \mathcal{Z})$. By definition, $I(\mathcal{Z}_{hor}) = I(u_{\diamondsuit}) =I(Z_{\alpha_{\diamondsuit}, \alpha_I})=0$.
 By the fourth bullet of Lemma \ref{lem3} and Lemma \ref{lem17}, we have
  \begin{equation} \label{eq44}
  \begin{split}
  &I(u)=2\sum_{i=1}^{k+1} c_i  + 2n(k+1),  \\
  &J_0(u)=2c_{k+1}(d+g-1)+2n(d+g-1).
  \end{split}
  \end{equation}
Combining (\ref{eq44}) and  \ref{assumption6}, we have
 \begin{equation}\label{eq45}
 \begin{split}
 E_{\omega_{H'_{\varepsilon}}}(u) + \eta J_0(u)
 &= \sum_{i=1}^{k} \lambda c_i  +c_{k+1} \int_{B_{k+1}} \omega+ 2 \eta c_{k+1}(d+g-1)+ n + \eta2n(d+g-1) \\
 &= \lambda \left(\sum_{i=1}^{k+1}  c_i + (k+1)n \right) = \frac{\lambda}{2} I(u).
 \end{split}
  \end{equation}
 By (\ref{eq45}) and $I(u)=0$, we have $ E_{\omega_{H'_{\varepsilon}}}(u) + \eta J_0(u)=0$.  Since $J_0(u) \ge 0$ (Lemma \ref{lem17}) and $E_{\omega_{H'_{\varepsilon}}}(u)  \ge 0$ (Lemma \ref{lem8}),  $ E_{\omega_{H'_{\varepsilon}}}(u) =  J_0(u)=0$.   Lemma \ref{lem8} implies   that  $u  =\cup_{i=1}^d u_{y_-^i}$ is a union of horizontal sections. In other words, we must have $\alpha_I= \alpha_{\diamondsuit}$ and  the union of horizontal sections $u_{{\diamondsuit}}$ is the unique element in   $\mathcal{M}^J(\textbf{y}_{\diamondsuit}, \alpha_{\diamondsuit})$ with $I=0$. Therefore, we obtain 
 $$\langle \mathcal{OC}(\underline{L}, {H'_{\varepsilon}})_J (\mathbf{y}_{\diamondsuit}, [A_{\mathbf{y}_{\diamondsuit}}]),  (\alpha_{\diamondsuit}, [Z_{\alpha_{\diamondsuit}}]) \rangle =1.$$

If $u$ is an  HF-PFH curve in   $\mathcal{M}^J(\textbf{y}_{\diamondsuit}, \beta)$ and $\beta \ne \alpha_{{\diamondsuit}}$, then $E_{ \omega_{\varphi_{H'_{\varepsilon}}} }(u) >0$; otherwise, by Lemma \ref{lem8}, $u$ is horizontal and $u$ must be asymptotic to $\alpha_{{\diamondsuit}}$.  By  Lemma \ref{lem17}, we have
\begin{equation*}
\begin{split}
0<E_{\omega_{H'_{\varepsilon}}}(u) +\eta J_0(u)&=\int_{\mathcal{Z}_{ref}} \omega_{\varphi_{H'_{\varepsilon}}} +\int_{A_{\textbf{y}_{\diamondsuit}}} \omega -\int_Z  \omega_{\varphi_{H'_{\varepsilon}}} + \eta ( J_0(\mathcal{Z}_{ref}) + J_0(A_{\mathbf{y}_{\diamondsuit}}) -J_0(Z))\\
&=\int_{Z_{\alpha_{\diamondsuit}} }  \omega_{\varphi_{H'_{\varepsilon}}} - \int_Z  \omega_{\varphi_{H'_{\varepsilon}}} + \eta (J_0([Z_{\alpha_{\diamondsuit}}]) -J_0([Z])) \\
&=\mathbb{A}_{H'_{\varepsilon}} (\alpha_{\diamondsuit}, Z_{\alpha_{\diamondsuit}}) - \mathbb{A}_{H'_{\varepsilon}} (\beta, [Z]) + \eta J_0([Z_{\alpha_{\diamondsuit}} -Z]).
\end{split}
\end{equation*}
In the above computations, we have used the facts that 
$$J_0(\mathcal{Z}_{hor}) =0 \mbox{ and } \int_{\mathcal{Z}_{hor}}  \omega_{\varphi_{H'_{\varepsilon}}}=0. $$
\end{proof}

Let 
\begin{equation} \label{eq91}
\mathfrak{c}_{\diamondsuit}:=\widetilde{\mathcal{OC}}( \underline{L} , {H'_{\varepsilon}})_J (\mathbf{y}_{{{\diamondsuit}}}, [A_{\mathbf{y}_{\diamondsuit}}]) =(\alpha_{\diamondsuit}, [Z_{\alpha_{\diamondsuit}}]) +\sum (\beta, [Z])
\end{equation}
  By Lemma \ref{lem7}, $\mathfrak{c}_{\diamondsuit} $ is a cycle.  However, it is difficult to determine  whether $ \mathfrak{c}_{\diamondsuit}$ is exact    at this stage, because we do not yet   know  the differential on $\widetilde{PFH}(\Sigma, \varphi_{H'_{\varepsilon}}, d)$.
 To show that $\mathfrak{c}_{\diamondsuit}$ is non-exact,  the strategy is to find the corresponding cycle $\mathfrak{c}_{\diamondsuit}' \in \widetilde{PFC}(\Sigma, \varphi_{H_{\varepsilon}}, d)$,  because  the  elements  in  $\widetilde{PFC}(\Sigma, \varphi_{H_{\varepsilon}}, d)$   can  be determined more easily. Thus, we need to  compute the cobordism map  $\mathfrak{I}_{H'_{\varepsilon}, H_{\varepsilon}}$. To this end, we need to introduce some definitions.

Let $(X, \Omega_X)$ be the symplectic cobordism defined in (\ref{eq23}).  Take $H_+= H'_{\varepsilon} $ and $H_-=H_{\varepsilon}.$
By \ref{G1},  $\Omega_X=\omega + dH_{\varepsilon} \wedge dt + ds \wedge dt$ is $\mathbb{R}$-invariant over  $\mathbb{R} \times S^1_t \times (\Sigma \setminus \mathcal{U}^{\delta+ \delta_0})$. This region is called a \textbf{product region}.  Let 
$$Z_{ref} :=[\mathbb{R} \times S^1 \times \mathbf{x}] \in H_2(X, \gamma_{H'_{\varepsilon}}^{\mathbf{x}},  \gamma_{H_{\varepsilon}}^{\mathbf{x}})$$
 be the reference homology class.

Let $\mathcal{J}_{comp}(X, \Omega_X)$ be the set of $\Omega_X$-compatible  almost complex structures  such that
\begin{enumerate}
\item
$J_X \vert_{s \ge R_0} \in \mathcal{J}(Y_{\varphi_{H'_{\varepsilon}}}, \omega_{\varphi_{H'_{\varepsilon}}}) $  and $J_X \vert_{s \le 0} \in \mathcal{J} ( Y_{\varphi_{H_{\varepsilon}}}, , \omega_{\varphi_{H_{\varepsilon}}}) $.
\item
$j\circ \pi_*= \pi_* \circ J_X $, where $j$ is the  complex structure on $\mathbb{R}_s \times S^1_t $ such  that $j(\partial_s) =\partial_t$.

\end{enumerate}
Given orbit sets $\alpha_{\pm}$, let $\overline{\mathcal{M}_i^{J_X}}(\alpha_+, \alpha_-)$ denote the moduli space of broken holomorphic currents in $X$ with ECH index $i$ (see Definition 1.6 of \cite{HT}).

In the following lemmas, we   compute  $PFC_{Z_{ref}}^{sw}(X, \Omega_X)_{J_X}(\mathfrak{c}_{\diamondsuit}). $
\begin{lemma} \label{lem4}
Let $J_{X} \in \mathcal{J}_{comp}(X, \Omega_X)$ be an almost complex structure such that it is $\mathbb{R}$-invariant in the product region $\mathbb{R} \times S^1_t \times (\Sigma \setminus  \mathcal{U}^{\delta+ \delta_0})$.  Then $\overline{\mathcal{M}_0^{J_{X}}}(\alpha_{{\diamondsuit}}, \alpha_I) =\emptyset$ unless $\alpha_I =\alpha_{\diamondsuit}$.  When $\alpha_I =\alpha_{\diamondsuit}$,  the trivial cylinder $\mathbb{R} \times \alpha_{\diamondsuit}$ is the unique element in $\overline{\mathcal{M}_0^{J_{X}}}(\alpha_{\diamondsuit}, \alpha_{\diamondsuit})$.
\end{lemma}
\begin{proof}
 Let $\mathcal{C} \in   \overline{\mathcal{M}_0^{J_{X}}}(\alpha_{{\diamondsuit}}, \alpha_I) $ be a  broken holomorphic current. Let $Z  \in H_2(X, \alpha_{\diamondsuit}, \alpha_I)$ denote the  relative homology class of $\mathcal{C}$.  Then $Z$ can be written as $ Z_{\alpha_{\diamondsuit}, \alpha_I} + n(Z) [\Sigma ] + [S]$, where $[S] \in H_1(S^1,\mathbb{Z}) \otimes H_1(\Sigma, \mathbb{Z})$. It is easy to show that
\begin{equation*}
\begin{split}
&I(\alpha_{\diamondsuit}, \alpha_I, Z) =2n(Z)(k+1)   \mbox{ and }  \int_Z \omega_X = n(Z).
\end{split}
\end{equation*}
Then $n(Z) =0$ because  $I=0$.   Also,  by definition, we have
\begin{equation*}
\begin{split}
\#(\mathcal{C}  \cap \mathbb{R} \times \gamma_{r_0, \theta}) =\# ((Z_{\alpha_{\diamondsuit}, \alpha_I} +  [S]) \cap  \mathbb{R} \times \gamma_{r_0, \theta} ) = 0.
\end{split}
\end{equation*}
Note that the above intersection numbers are well defined because $\gamma_{r_0, \theta}$ and $\alpha_I$ are disjoint.
Because $\mathbb{R} \times \gamma_{r_0, \theta} $ is holomorphic by the choice of $J_X$,  the above equality implies that $\mathcal{C}$ does not intersect  $\mathbb{R} \times \gamma_{r_0, \theta} $. Recall that the family of orbits   $\gamma_{r_0, \theta} $ foliate  $S^1 \times \partial \mathcal{U}^{r_0}$. Consequently,    $\mathcal{C}$  is contained in the product region $\mathbb{R} \times S^1_t \times (\Sigma \setminus \mathcal{U}^{\delta+ \delta_0})$.  Then  $  \int_{\mathcal{C}} \omega_X =0$ implies that $\mathcal{C}$ is a union of trivial cylinders (Proposition 9.1 of \cite{H1}). Thus,  we must have $\alpha_I=\alpha_{\diamondsuit}$.
\end{proof}

\begin{lemma} \label{lem5}
Let $J_X $ be a  generic almost complex structure  in $\mathcal{J}_{comp}(X, \Omega_X)$ such that  $J_X$ is $\mathbb{R}$-invariant in the product region $\mathbb{R} \times S^1_t \times (\Sigma \setminus \mathcal{U}^{\delta+ \delta_0})$.    Then we have
\begin{equation*}
PFC_{Z_{ref}}^{sw}(X, \Omega_X)_{J_X} (\alpha_{\diamondsuit}, [Z_{\alpha_{\diamondsuit}}]) =  (\alpha_{\diamondsuit}, [Z_{\alpha_{\diamondsuit}}]) +\sum (\beta', [Z']),
\end{equation*}
where $(\beta', [Z'])$ satisfies $\beta' \ne \alpha_I$ and
$\mathbb{A}_{H'_{\varepsilon}}(\alpha_{\diamondsuit}, [Z_{\alpha_{\diamondsuit}}])- \mathbb{A}_{H_{\varepsilon}}( \beta', [Z']) \ge \frac{1}{4(k+1)}$.
\end{lemma}
\begin{proof}
By the   holomorphic axioms (Theorem 1 of \cite{GHC} and Appendix of  \cite{GHC2}) and Lemma \ref{lem4}, we know that
\begin{equation*}
\langle PFC_{Z_{ref}}^{sw}(X, \Omega_X)_{J_X} (\alpha_{\diamondsuit}, [Z_{\alpha_{\diamondsuit}}]), (\alpha_I, [Z])\rangle =0
\end{equation*}
when  $(\alpha_I, [Z]) \ne (\alpha_{\diamondsuit}, [Z_{\alpha_{\diamondsuit}}])$, and
\begin{equation*}
\langle PFC_{Z_{ref}}^{sw}(X, \Omega_X)_{J_X} (\alpha_{\diamondsuit}, [Z_{\alpha_{\diamondsuit}}]), (\alpha_{\diamondsuit}, [Z_{\alpha_{\diamondsuit}}]) \rangle=1.
\end{equation*}

 Assume that $\langle PFC_{Z_{ref}}^{sw}(X, \Omega_X)_{J_X}  (\alpha_{\diamondsuit}, [Z_{\alpha_{\diamondsuit}}]), (\beta', [Z']) \rangle =1 $ for some $(\beta', [Z'])$ and $\beta' \ne \alpha_I$. Again  by the holomorphic axioms, we have a holomorphic curve $\mathcal{C}  \in \overline{\mathcal{M}_0^{J_X}}(\alpha_{\diamondsuit}, \beta')$. The relative homology class  of  $\mathcal{C} $ is $Z_{\alpha_{\diamondsuit}, \beta'} + n[\Sigma] +[S]$.

It is easy to check that
	\begin{equation} \label{eq9}
		\begin{split}
			 &I(\mathcal{C}) = -h(\beta') -2e_+(\beta') +  2n(k+1)=0,\\
			 &\mathbb{A}_{H'_{\varepsilon}}(\alpha_{\diamondsuit}, [Z_{\alpha_{\diamondsuit}}]) - \mathbb{A}_{H_{\varepsilon}}(\beta', [Z]) = -H_{\varepsilon}(\beta' ) +n,
		\end{split}
	\end{equation}
where  $h(\beta')$ is the   total multiplicity of all the hyperbolic orbits in $\beta'$ and  $e_+(\beta')$ is the  total multiplicity   of all the  elliptic  orbits at local maxima  of $H_{\varepsilon}$.

Because $\beta' \ne \alpha_I$, we have $h(\beta') + 2e_+(\beta') \ge 1$. Therefore, (\ref{eq9}) implies that
\begin{equation*}
		\begin{split}
			 \mathbb{A}_{H_{\varepsilon}}(\alpha_{\diamondsuit}, Z_{\alpha_{\diamondsuit}}) - \mathbb{A}_{H_{\varepsilon}}(\beta', Z') =& -H_{\varepsilon}(\beta') + \frac{h(\beta') + 2e_+(\beta')}{2(k+1)} \\
			 \ge& \frac{1}{2(k+1)} +O({d}/{\kappa_0}) \ge \frac{1}{4(k+1)}.
		\end{split}
	\end{equation*}

\end{proof}

\begin{lemma} \label{lem6}
Let $(\beta, [Z])$ be a term of $\mathfrak{c}_{\diamondsuit}$ in (\ref{eq91}).  Let $J_X$ be the almost complex structure in Lemma \ref{lem5}. Then we have
\begin{equation*}
PFC_{Z_{ref}}^{sw}(X, \Omega_X)_{J_X} (\beta, [Z]) =  \sum (\beta', [Z']),
\end{equation*}
where $(\beta', [Z'])$ satisfies  $\beta' \ne \alpha_I$, $\mathbb{A}_{H_{\varepsilon}}(\alpha_{\diamondsuit}, [Z_{\alpha_{\diamondsuit}}]) -  \mathbb{A}_{H_{\varepsilon}}( \beta', [Z'])  \ge  \frac{1}{4(k+1)}$  and
$$\mathbb{A}_{H_{\varepsilon}}(\alpha_{\diamondsuit}, [Z_{\alpha_{\diamondsuit}}])- \mathbb{A}_{H_{\varepsilon}}( \beta', [Z'])  +  \eta J_0([Z_{\alpha_{\diamondsuit}} -Z'])  \ge  \frac{1}{4} \int_{B_{k+1}} \omega. $$
\end{lemma}
\begin{proof}
First, we show that $\beta'$ cannot be $\alpha_I$.
Assume that 
\begin{equation*}
\langle PFC_{Z_{ref}}^{sw}(X, \Omega_X)_{J_X} (\beta, [Z]), (\alpha_I,  [Z']) \rangle =1.
\end{equation*} 
Then we  have a  broken holomorphic  current   $\mathcal{C} =(C, \mathcal{C}_0)$, where $C \in \mathcal{M}^J(\textbf{y}_{\diamondsuit}, \beta)$ is an  HF-PFH curve  with $I=0$ and $\mathcal{C}_0 \in \overline{\mathcal{M}_0^{J_{X}}}(\beta, \alpha_{I})$. Let  $\mathcal{Z} \in H_2(W, \textbf{y}_{\diamondsuit},  \alpha_I)$ denote the  relative homology class of   $\mathcal{C} $.

		Reintroduce the  periodic orbits $\gamma^i_{r_0, \theta_0}$  near the local maximums  of $H_{\varepsilon}$.  The superscript ``$i$'' indicates that the local maximum lies in the  domain $\mathring{B}_i$, where $1 \le i \le k+1$. In particular,     $\gamma^i_{r_0, \theta_0}$ lies in $S^1 \times \mathring{B}_i$.  Regarding $W$  as a  subset of $\mathbb{R} \times Y_{\varphi_{H_{\varepsilon}'}}$,  define  curves  $v_i:= (\mathbb{R}\times \gamma_{r_0, \theta_0}^i) \cap W$. Note that each $v_i$ is  $J$-holomorphic and $\partial v_i$ is disjoint from the Lagrangian $\mathcal{L}_{{H'_{\varepsilon}}}$.
  Then  for any relative homology class $\mathcal{Z}' \in H_2(W, \textbf{y}_{\diamondsuit}, \alpha_I)$, we have a well-defined intersection number
		\begin{equation*}
			n_i(\mathcal{Z}') : =\# (\mathcal{Z}' \cap v_i).
		\end{equation*}
	
	The relative homology class $\mathcal{Z} \in H_2(W, \textbf{y}_{\diamondsuit}, \alpha_I)$  can be written as 
	$$\mathcal{Z}= \mathcal{Z}_{hor} \# Z_{\alpha_{\diamondsuit}, \alpha_I} + \sum_{i=1}^{k+1} c_i[B_i] +n[\Sigma] +[S], $$  where $[S] \in H_1(S^1, \mathbb{Z}) \otimes H_1(\Sigma, \mathbb{Z})$ and  $\mathcal{Z}_{hor}$ is the class represented by the union of horizontal sections.  By Lemma \ref{lem3},  the ECH index of $\mathcal{Z}$ is
\begin{equation} \label{eq11}
I(\mathcal{Z})  = \sum_{i=1}^{k+1} 2c_i + 2n(k+1)= I(C)+ I(\mathcal{C}_0) = 0.
\end{equation}	
	 Let $q_i$ denote the period of $\gamma^i_{r_0, \theta_0}$.   From the construction in \cite{GHC2}, the period of $\gamma^i_{r_0, \theta_0}$ is determined by the function $\varepsilon$.  For a suitable choice of $\varepsilon$, we can choose $q_i=q$ for $1 \le i\le k+1$. By definition, we have
	 \begin{equation} \label{eq12}
	 	n_i( \mathcal{Z}_{hor} \#Z_{\alpha_{\diamondsuit}, \alpha_I} )=0,  \  n_i([B_j])=\delta_{ij}q,  \  n_i([S])=0,   \mbox{ and } n_i([\Sigma])=q.
	 \end{equation}
	 for $1\le i, j \le k+1$.
	 From  (\ref{eq11}) and (\ref{eq12}), we know that
 \begin{equation*}
 \#(\mathcal{C} \cap (\bigsqcup_{i=1}^{k+1} v_i) )=\sum_{i=1}^{k+1} n_i(\mathcal{Z})=   \sum_{i=1}^{k+1} c_i q +(k+1)nq=0.
\end{equation*}
  By the intersection positivity of holomorphic curves, $\mathcal{C}$ does not  intersect  $\mathbb{R} \times \gamma^i_{r_0, \theta_0}$. In particular, $\mathcal{C}_0$ lies inside the product region of $X $. Therefore,  $ \int_{\mathcal{C}_0} \omega_X=  \int_{\mathcal{C}_0}{\omega_{\varphi_{H'_{\varepsilon}}}}  \ge 0$.   By Lemma \ref{lem13} and Lemma  \ref{lem17}, $J_0 (\mathcal{Z}) =J_0(C) + J_0 (\mathcal{C}_0) \ge 0. $   According to  (\ref{eq11}),  Lemmas \ref{lem3} and \ref{assumption6},  we have
\begin{equation*}
\begin{split}
\int_{\mathcal{Z}} \omega_{\varphi_{H'_{\varepsilon}}} + \eta J_0 (\mathcal{Z})   &= \int_{\mathcal{Z}_{hor} \#Z_{\alpha_{\diamondsuit}, \alpha_I}} \omega_{\varphi_{H'_{\varepsilon}}} + \lambda \sum_{i=1}^{k+1} c_i  + n + 2n\eta(d+g-1)\\
&= \lambda \left( \sum_{i=1}^{k+1} c_i +n(k+1)\right) =\frac{\lambda}{2}I(\mathcal{Z}) =0.
\end{split}
\end{equation*}
On the other hand,  since Lemma \ref{lem8}, $ \int_{\mathcal{C}_0}{\omega_{\varphi_{H'_{\varepsilon}}}}  \ge 0$ and $J_0(\mathcal{Z}) \ge 0$, we have
\begin{equation*}
\int_{\mathcal{Z}} \omega_{\varphi_{H'_{\varepsilon}}} + \eta J_0 (\mathcal{Z}) \ge  \int_C{\omega_{\varphi_{H'_{\varepsilon}}}}  +   \int_{\mathcal{C}_0}{\omega_{\varphi_{H'_{\varepsilon}}}}   >0.
\end{equation*}
We obtain a contradiction. Hence, $\beta'\ne \alpha_I. $

Now we consider the case that  
$$\langle PFC_{Z_{ref}}^{sw}(X, \Omega_X)_{J_X} (\beta, [Z]),    (\beta', [Z']) \rangle=1   \mbox{ and } \beta' \ne \alpha_I. $$ As before,  we  have a broken holomorphic current   $\mathcal{C} =(C, \mathcal{C}_0)$, where $C \in \mathcal{M}^J(\textbf{y}_{\diamondsuit}, \beta, \mathcal{Z}_{hor} \#Z_{\alpha_{\diamondsuit}} \#(-Z))$ is an  HF-PFH curve  and $\mathcal{C}_0 \in \overline{\mathcal{M}^{J_{X}}}(\beta, \beta')$ with relative homology class $Z\#(-Z')$.   Therefore, the relative homology class of $\mathcal{C}$ is $$\mathcal{Z}_{hor} \#Z_{\alpha_{\diamondsuit}}\#(-Z') =\mathcal{Z}_{hor} \#Z_{\alpha_{\diamondsuit}, \beta'} + \sum_{i=1}^{k+1} c_i[B_i] +  n[\Sigma] +[S]$$ for some $c_i, n \in \mathbb{Z}$ and $[S] \in H_1(S^1, \mathbb{Z}) \otimes H_1(\Sigma, \mathbb{Z}). $

We now  prove the estimates on  
$$\mathbb{A}_{H_{\varepsilon}}(\alpha_{\diamondsuit}, [Z_{\alpha_{\diamondsuit}}])- \mathbb{A}_{H_{\varepsilon}}( \beta', [Z'])  \mbox{ and }\mathbb{A}_{H_{\varepsilon}}(\alpha_{\diamondsuit}, [Z_{\alpha_{\diamondsuit}}])- \mathbb{A}_{H_{\varepsilon}}( \beta', [Z'])  +  \eta J_0([Z_{\alpha_{\diamondsuit}} -Z']). $$   
    Suppose that $\beta'$ has $E_+$ distinct simple orbits (ignoring the multiplicity) at the local maximums and $E_-$ distinct simple orbits at the local minimums.
As in  (\ref{eq9}), we have
 \begin{equation} \label{eq56}
 \begin{split}
&0=I(\mathcal{C}) = I(C) + I(\mathcal{C}_0) =   -h(\beta') -2e_+(\beta') +2n(k+1)\\
&J_0(\mathcal{C})=d-h(\beta') -2e_+(\beta') + E_+-E_-+ 2n(d+g-1)\\\
&\mathcal{A}_{H'_{\varepsilon}}( \textbf{y}_{\diamondsuit}, [A_{ \textbf{y}_{\diamondsuit}}]) - \mathbb{A}_{H_{\varepsilon}}(\beta', [Z']) =\int_{C} \omega_{\varphi_{H'_{\varepsilon}}}   + \int_{\mathcal{C}_0} \omega_X = -H_{\varepsilon}(\beta') +n,
 \end{split}
	\end{equation}
where $h(\beta')$ is   the total multiplicity of  the hyperbolic orbits and $e_+(\beta')$ is  the total multiplicity  of the elliptic  orbits at the local maximums.

Note that $ \mathcal{A}_{H'_{\varepsilon}}(\textbf{y}_{\diamondsuit}, [A_{\textbf{y}_{\diamondsuit}}]) =\mathbb{A}_{H_{\varepsilon}}(\alpha_{\diamondsuit}, [Z_{\alpha_{\diamondsuit}}])$. Since $\beta' \ne \alpha_I$, $h(\beta') +2e_+(\beta') \ge 1$.   By the first and third equations of (\ref{eq56}), we have 
\begin{equation*}
\mathbb{A}_{H_{\varepsilon}}(\alpha_{\diamondsuit}, [Z_{\alpha_{\diamondsuit}}]) - \mathbb{A}_{H_{\varepsilon}}(\beta', [Z'])  = -H_{\varepsilon}(\beta') + \frac{h(\beta') + 2e_+(\beta') }{2(k+1)} \ge \frac{1}{2(k+1)} +O(d \varepsilon) \ge \frac{1}{4(k+1)}.
\end{equation*}

 By (\ref{eq56}), we have
 \begin{equation*}
 \begin{split}
&\mathbb{A}_{H_{\varepsilon}}(\alpha_{\diamondsuit}, [Z_{\alpha_{\diamondsuit}}])- \mathbb{A}_{H_{\varepsilon}}( \beta', [Z']) +\eta J_0(\mathcal{C})\\
=&\mathbb{A}_{H_{\varepsilon}}(\alpha_{\diamondsuit}, [Z_{\alpha_{\diamondsuit}}])- \mathbb{A}_{H_{\varepsilon}}( \beta', [Z']) +\eta J_0([Z_{\alpha_{\diamondsuit}} -Z'])\\
=&   -H_{\varepsilon}(\beta') +n + 2n\eta(d+g-1)+ \eta(d-h(\beta') -2e_+(\beta') + E_+-E_-)\\
\ge & -H_{\varepsilon}(\beta')    + \lambda n(k+1) - \eta (h(\beta') + 2e_+(\beta'))\\
=&  -H_{\varepsilon}(\beta')  + (\frac{\lambda}{2}  - \eta) (h(\beta') +2e_+(\beta')).
 \end{split}
	\end{equation*}
 By assumption \ref{assumption6},  we have
 \begin{equation}
\frac{\lambda}{2} - \eta =\eta(2g+k-2) + \frac{1}{2}\int_{B_{k+1}} \omega.
\end{equation}
If $g\ge1$ or if $g=0$ and $k \ge 2$, then $\frac{\lambda}{2} - \eta \ge  \frac{1}{2}\int_{B_{k+1}} \omega$. 
 If $g=0$ and $k=1$, assumption \ref{assumption6} implies that $\eta=0$ and $\lambda=\int_{B_1} \omega =\int_{B_2} \omega$. Hence,  we have $\frac{\lambda}{2} -\eta \ge \frac{1}{2}\int_{B_{k+1}} \omega$ in all cases.
 Therefore, $$ \mathbb{A}_{H_{\varepsilon}}(\alpha_{\diamondsuit}, [Z_{\alpha_{\diamondsuit}}])- \mathbb{A}_{H_{\varepsilon}}( \beta', [Z']) +\eta J_0([Z_{\alpha_{\diamondsuit}} -Z']) \ge -H_{\varepsilon}(\beta')  + \frac{1}{2} \int_{B_{k+1}} \omega   \ge  \frac{1}{4} \int_{B_{k+1}} \omega.$$

\end{proof}

\begin{lemma} \label{lem11}
Let $\mathfrak{c}_{\diamondsuit}' := PFC_{Z_{ref}}^{sw}(X, \Omega_X)(\mathfrak{c}_{\diamondsuit})$. Then the cycle $\mathfrak{c}_{\diamondsuit}'$ is non-exact, i.e., it represents a non-zero class in $\widetilde{PFH}(\Sigma, \varphi_{H_{\varepsilon}}, d)$.
\end{lemma}
\begin{proof}
Let $(X_-, \Omega_{X_-})$ be the symplectic cobordism from $(Y_{\varphi_{H_{\varepsilon}}}, \omega_{\varphi_{H_{\varepsilon}}})$  to $\emptyset$ defined in (\ref{eq6}).   Fix a generic $J_{X_-} \in \mathcal{J}_{comp}(X_-, \Omega_{X_-})$.  Define  a homomorphism
\begin{equation*}
\begin{split}
PFC_{Z_{ref}}(X_-, \Omega_{X_-})_{J_{X_-}}:& \widetilde{PFC}(\Sigma, \varphi_{H_{\varepsilon}}, d) \to \mathbb{F}\\
&(\alpha, [Z]) \to\#\mathcal{M}^{J_{X_-}}(\alpha, \emptyset, [Z\#Z_{ref}]).
\end{split}
\end{equation*}
by counting $I=0$ (unbroken)  holomorphic curves  in  $(X_-, \Omega_{X_-})$, where $Z_{ref} \in H_2(X_-, \gamma_H^{\mathbf{x}}, \emptyset)$.  We claim that  $PFC_{Z_{ref}}(X_-, \Omega_{X_-})_{J_{X_-}}$ is well defined and 
\begin{equation} \label{eq16}
		\begin{split}
			 &PFC_{Z_{ref}}(X_-, \Omega_{X_-})_{J_{X_-}}(\alpha_I, [Z_I]) =1,\\
			 &PFC_{Z_{ref}}(X_-, \Omega_{X_-})_{J_{X_-}}(\beta', [Z'])=0 \mbox{ \ \ for $(\beta', [Z'])\ne (\alpha_I, [Z_I])$}.
		\end{split}
	\end{equation}
Furthermorse, $PFC_{Z_{ref}}(X_-, \Omega_{X_-})_{J_{X_-}} $ is a chain map and descends to a homomorphism on homology:
\begin{equation*}
PFH_{Z_{ref}}(X_-, \Omega_{X_-})_{J_{X_-}} : \widetilde{PFH}(\Sigma, \varphi_{H_{\varepsilon}}, d) \to \mathbb{F}. 
\end{equation*}
Using Taubes's techniques \cite{T1, T2} and C. Gerig's generalization \cite{CG, CG2},  the map $PFH_{Z_{ref}}(X_-, \Omega_{X_-})_{J_{X_-}} $ should agree with the PFH cobordism map $PFH^{sw}_{Z_{ref}}(X_-, \Omega_{X_-}) $ (see Remark 1.3 of \cite{GHC1}). But we don't need this to prove  the  lemma.

By the above claims together with  Lemmas \ref{lem5} and \ref{lem6},  we have 
$$PFC_{Z_{ref}}(X_-, \Omega_{X_-})_{J_{X_-}}(\mathfrak{c}'_{\diamondsuit}) =1.$$  As a result,  $\mathfrak{c}_{\diamondsuit}'$ is non-exact. 
 


We will not go through all the details, but let us outline how to obtain  (\ref{eq16}) and well definiteness of $PFC_{Z_{ref}}(X_-, \Omega_{X_-})_{J_{X_-}}$.   Basically, the  idea is the same as  Lemma \ref{lem2}.  First, we take the reference class to  be $Z_{ref} := [B_- \times \mathbf{x}] \in H_2(X_-, \gamma_H^{\mathbf{x}},  \emptyset)$.  By the same computation as in Lemma 3.3  of  \cite{GHC1} (or Lemma 3.1 in \cite{GHC3} for a similar computation
in the ECH setting),   we have
\begin{equation} \label{eq52}
\begin{split}
&I(Z\#Z_{ref}) =2e_+(\alpha) + h(\alpha)  + 2n(Z)(k+1)\\
&\int_{Z\#Z_{ref}} \omega_{X_-} = H_{\varepsilon}(\alpha) + n(Z),
 \end{split}
\end{equation}
where $n(Z)$ is the unique integer   such that $[Z] =[Z_{\alpha}] + n(Z) [\Sigma]$. 
As in Lemma \ref{lem8},  if  $J_{X_-}$  preserves  the horizontal and vertical bundles of $X_-$, then the   energy of    any $J_{X_-}$-holomorphic current  is nonnegative. Therefore,  by the second equation in (\ref{eq52}), $n(Z) \ge 0$ provided that  the moduli space $ \mathcal{M}^{J_{X_-}}(\alpha, \emptyset, Z\#Z_{ref}) $ is nonempty.
It follows from the  ECH index  formula in (\ref{eq52})  that $I(\mathcal{C}) \ge 0$ for any holomorphic current. Moreover,  $I(\mathcal{C}) = 0$ only  if $\mathcal{C}$ is asymptotic to   $\alpha_I$.  Thus, we obtain the second identity of (\ref{eq16}) immediately.

Recall that the key difficulty   of defining PFH cobordism maps by holomorphic curves is that the ECH index could be negative in  symplectic cobordisms (see Section 5.5 of \cite{H4}).  In our situation, this possibility  has been  eliminated by the last paragraph's discussion. 
Applying Gerig’s analysis (Proposition 3.13 of \cite{CG}),  one can show that     $\mathcal{M}_0^{J_{X_-}}(\alpha, \emptyset)$ is a finite set. Thus, 
 $PFC_{Z_{ref}}(X_-, \Omega_{X_-})_{J_{X_-}}$ is well defined.    By SFT compactness,  a sequence of holomorphic curves in $\mathcal{M}_1^{J_{X_-}}(\alpha, \emptyset)$ converges to a broken holomorphic curve  $\mathbf{u}.$ The same index analysis  (Proposition 3.23 of \cite{CG}) shows  that  $\mathbf{u}$ consists of an  $I=0$ curve  at cobordism level,   an  $I=1$ curve  at  symplectization level,  and
possibly additional levels between them consisting of connectors. Applying the gluing analysis \cite{HT1, HT2},     $PFC_{Z_{ref}}(X_-, \Omega_{X_-})_{J_{X_-}}$  is a chain map.

Choose a suitable  $J_{X_-}$ such that $u_I=B_- \times    \{y_-^{i_1}, ...,y_-^{i_d}\}$ is holomorphic.  Let $u \in \mathcal{M}^{J_{X_-}}(\alpha, \emptyset, Z_{ref}\# Z)$ with $I=0$. According to (\ref{eq52}), $\alpha=\alpha_I$ and $\int u^*\omega_{X_-} =0$.    As in Lemma \ref{lem8},  $\int u^*\omega_{X_-} =0$ implies that $u$ is horizontal, i.e., $d^{vert}u =0$. As a result, $u=u_I$. This leads to the first identity  of (\ref{eq16}).


\end{proof}
Define 
\begin{equation} \label{eq84}
\sigma_{\underline{L}}: =j_{H'_{\varepsilon}}([(\mathbf{y}_{{{\diamondsuit}}}, [A_{\mathbf{y}_{\diamondsuit}}])]) \in HF(\Sigma, \underline{L}) \mbox{ and }  \mathfrak{d} :=\mathfrak{j}_{H'_{\varepsilon}}([\mathfrak{c}_{\diamondsuit}]) \in  \widetilde{PFH}(\Sigma, d),
\end{equation}
where $\mathfrak{c}_{\diamondsuit}$ is the cycle   defined in (\ref{eq91}).  The class $\sigma_{\underline{L}} $ is well defined by Lemma \ref{lem7}.   (\ref{eq91}) and (\ref{eq84}) imply that 
\begin{equation} \label{eq79}
 \widetilde{\mathcal{OC}}( \underline{L} , {H'_{\varepsilon}})_J(j_{H'_{\varepsilon}}^{-1} (\sigma_{\underline{L}}))= \mathfrak{j}_{H'_{\varepsilon}}^{-1}(\mathfrak{d}). 
\end{equation}
 By Lemma \ref{lem11}, $\mathfrak{d}$ is nonzero .   

Recall that $\varphi_{H'_{\varepsilon}}$ satisfies \ref{assumption1}. 
Let $H$ be a Hamiltonian   function such that $\varphi_H$ satisfies  \ref{assumption1} or \ref{assumption2}. By the partial invariance in Theorem \ref{thm6}, Remark \ref{remark6} and (\ref{eq79}),  we have 
\begin{equation*}
\widetilde{\mathcal{OC}} (\underline{L}, H )_J (j_{H}^{-1}(\sigma_{\underline{L}})) = \mathfrak{j}_{ H}^{-1}(\mathfrak{d}). 
\end{equation*}


The remaining task  is to prove the  normalization property of the spectral invariants in   (\ref{eq92}). Combining (\ref{eq20}) with the fact that the differential defined by $H'_{\varepsilon}$ vanishes (Lemma \ref{lem7}), we have
\begin{equation*}
 c_{\underline{L}}(H'_{\varepsilon}, \sigma_{\underline{L}}) =  \mathcal{A}^{\eta}_{H'_{\varepsilon}}(\mathbf{y}_{\diamondsuit}, [A_{\mathbf{y}_{\diamondsuit}}]) =0.
 \end{equation*}
 Note that the filtered QHF groups defined by $H_{\varepsilon}$ and $H'_{\varepsilon}$ are canonically isomorphic. Hence $ c_{\underline{L}}(H_{\varepsilon}, \sigma_{\underline{L}})=0$ as well.  For $n \gg 1$ such that $\frac{1}{n} < \varepsilon$, the Reeb chords of $H_{\frac{1}{n}} =-\frac{1}{n} f$  still correspond to  the critical points of $f$. Moreover,
 \begin{equation*}
 \mathbf{c}= \mathcal{I}_{H_{\varepsilon}, H_{\frac{1}{n}}} (\mathbf{y}_{\diamondsuit}, [A_{\mathbf{y}_{\diamondsuit}}]) =(\mathbf{y}_{\diamondsuit}, [A_{\mathbf{y}_{\diamondsuit}}]) + \sum (\mathbf{y}, [A]).  
 \end{equation*}
The leading term $(\mathbf{y}_{\diamondsuit}, [A_{\mathbf{y}_{\diamondsuit}}]) $ is given by the  trivial  holomorphic strips.  By the energy estimates (\ref{eq55}),  the other terms $\sum (\mathbf{y}, [A])$ have action  strictly less than zero.  Thus, $\mathcal{A}^{\eta}_{H_{\frac{1}{n}}}( \mathbf{c})=0$. 
By the argument of Lemma  \ref{lem7}, the differential of $H_{\frac{1}{n}}$  vanishes. Then 
$$c_{\underline{L}}(H_{\frac{1}{n}}, \sigma_{\underline{L}}) =  \mathcal{A}^{\eta}_{H_{\frac{1}{n}}}( \mathbf{c})= 0. $$ By taking  $n \to \infty$, we obtain $ c_{\underline{L}}(0, \sigma_{\underline{L}})=0$.

By Lemma \ref{lem5} and Lemma \ref{lem6}, we have  $\mathbb{A}_{H_{\varepsilon}}(\mathfrak{c}') =\mathbb{A}_{H_{\varepsilon}}(\alpha_{\diamondsuit}, [Z_{\alpha_{\diamondsuit}}])=0 $.  According to Example 2.18 of \cite{EH}, the PFH differential defined by $H_{\varepsilon}$ is zero.  This can be proved by an argument similar to the one used for computing the ECH of prequantization bundles \cite{Fa, NW}. Therefore,   $c_{d}^{pfh}(H_{\varepsilon}, \mathfrak{d}) =\mathbb{A}_{H_{\varepsilon}}(\mathfrak{c}') =0$.   
For $n \gg 1$ such that $\frac{1}{n}< \varepsilon$,  the Reeb orbits still correspond  to the critical points  of $f$. As before, we have 
 $$\mathfrak{I}_{H_{\varepsilon}, H_{\frac{1}{n}}} ((\alpha_{\diamondsuit}, [Z_{\alpha_{\diamondsuit}}]))  =(\alpha_{\diamondsuit}, [Z_{\alpha_{\diamondsuit}}]) + \sum (\beta'', [Z'']),  $$
where the action of $(\beta'', [Z''])$ is strictly less than zero. Hence, $c_{d}^{pfh}(H_{\frac{1}{n}}, \mathfrak{d} )=0$ and $$c_{d}^{pfh}(0, \mathfrak{d} ) = \lim_{n\to \infty}c_{d}^{pfh}(H_{\frac{1}{n}}, \mathfrak{d} )=0.$$

So far, we have finished  the proof of Theorem \ref{thm6}.

\section{Proof of main results} \label{section6}
In this section, we apply the previous results to prove the main theorems.
 \subsection{Proof of Theorem \ref{thm2}} \label{proof of thm2}
 Now we prove  Theorem \ref{thm2}  using   Theorem \ref{thm6}.
Fix $K \in \mathcal{H}_{\spadesuit}$, where  $\mathcal{H}_{\spadesuit}$ is defined in Remark \ref{remark6}. 
For any Hamiltonian function $H$, define
\begin{equation} \label{eq66}
\begin{split}
\mathcal{OC}(\underline{L}, H) : = &\mathfrak{I}_{K, H} \circ  \widetilde{\mathcal{OC}}(\underline{L},  K)\circ \mathcal{I}_{H, K} \\
=&  \mathfrak{j}_H^{-1} \circ \mathfrak{j}_{K} \circ  \widetilde{\mathcal{OC}}(\underline{L}, K) \circ  {j}^{-1}_{K } \circ j_H .
\end{split}
\end{equation}
By (\ref{eq63}) and (\ref{eq64}), $\mathcal{OC}(\underline{L}, H)$ satisfies the invariance property.     By the invariance property in Theorem \ref{thm6} (see also Remark \ref{remark6}), we have
\begin{equation*}
\mathcal{OC}(\underline{L}) (\sigma_{\underline{L}})  =   \mathfrak{j}_K \circ \widetilde{\mathcal{OC}}(\underline{L},  K)  \circ j^{-1}_K=     \mathfrak{j}_{H'_{\varepsilon}} \circ  \widetilde{\mathcal{OC}}(\underline{L}, H'_{\varepsilon}) \circ  {j}^{-1}_{H'_{\varepsilon} }. 
\end{equation*}
Combining the above equation with  (\ref{eq79}),   we have $\mathcal{OC}(\underline{L}) (\sigma_{\underline{L}}) =\mathfrak{d}$.  Since   $\mathfrak{I}_{KH},   \mathcal{I}_{H, K }$   and $\widetilde{\mathcal{OC}}(\underline{L}, K) $  are $\Lambda$-module homomorphisms and commute with the $U$ maps, so is $\mathcal{OC}(\underline{L}, H)$.   
 By Theorem \ref{thm9}, we have 
\begin{equation} \label{eq96}
\begin{split}
&U^{d-g+1} (\mathfrak{d}) =U^{d-g+1}  \circ \widetilde{\mathcal{OC}}(\underline{L})(\sigma_{\underline{L}}) =\widetilde{\mathcal{OC}}(\underline{L})( U^{d-g+1}  \sigma_{\underline{L}}) \\
=&\widetilde{\mathcal{OC}}(\underline{L}, H)( T^{-(d-g+1)}  \sigma_{\underline{L}}) = \widetilde{\mathcal{OC}}(\underline{L})_J( S^{-1} \sigma_{\underline{L}}) =S^{-1}\widetilde{\mathcal{OC}}(\underline{L})_J(   \sigma_{\underline{L}}) =S^{-1}\mathfrak{d}. 
\end{split} 
\end{equation}

By the partial invariance in Theorem \ref{thm6}  and (\ref{eq66}), we have  
$$ {\mathcal{OC}} (\underline{L}, H ) = \widetilde{\mathcal{OC}} (\underline{L}, H )$$
provided that $ H \in \mathcal{H}_{\spadesuit}$.  
Given a Hamiltonian function $H$ and a sufficiently small  number  $\delta>0$,   Proposition 3.7 of \cite{GHC} provides  a function $H^{\delta} \in  \mathcal{H}_{\spadesuit}$  satisfying  
\begin{equation}\label{eq68}
|H^{\delta}-H| + |dH-dH^{\delta}|_{g_{S^1 \times \Sigma}} \le c_0\delta,
\end{equation}
where $c_0$ is a universal constant. 

To prove (\ref{eq65}), we assume that the link $\underline{L}$ is $0$-admissible.   Assume that  
$$ {\mathcal{OC}} (\underline{L}, H^{\delta} ) (j_{H^{\delta}}^{-1}(\sigma))=\widetilde{\mathcal{OC}} (\underline{L}, H^{\delta} ) (j_{H^{\delta}}^{-1}(\sigma))= \mathfrak{j}_{ H^{\delta}}^{-1}(\mathfrak{v}) \ne 0.$$  
For $\kappa \gg 1$, we have a cycle $\mathbf{c}=\sum (\mathbf{y}, [A])$ such that it  represents $ j_{H^{\delta}}^{-1}(\sigma) $ and satisfies 
\begin{equation*}
\mathcal{A}_{H^{\delta}} (\mathbf{y}, [A]) \le  c_{\underline{L}}(H^{\delta}, \sigma ) + 1/ \kappa.
\end{equation*}
Then $ \widetilde{\mathcal{OC}} (\underline{L}, H^{\delta} )(\mathbf{c}) =\sum (\alpha, [Z])$ is a cycle representing $\mathfrak{j}_{ H^{\delta}}^{-1}(\mathfrak{v})$.  By  the  definition of $\widetilde{\mathcal{OC}} (\underline{L}, H^{\delta} )$, there is an HF-PFH curve $u \in \mathcal{M}^J(\mathbf{y}, \alpha)$  with relative homology class  $[u]=A\#\mathcal{Z}_{ref} \#(-Z)$.

Recall that
 $\mathcal{Z}_{ref} =[ (\mathbb{R} \times \Psi_{H^{\delta}}(S^1 \times \mathbf{x}) ) \cap W]\in H_2(W, \mathbf{x}_{H^{\delta}},  \gamma^{\mathbf{x}}_{H^{\delta}})$. Therefore, $ \int_{\mathcal{Z}_{ref}} \omega_{\varphi_{H^{\delta}}}  =0.$  
 By Lemma \ref{lem21}, we  can take $J \in \mathcal{J}(W, \Omega_{H^{\delta}})$ to define the open-closed morphism.  Because the holomorphic curves have nonnegative energy (Lemma \ref{lem8}),  we have
 \begin{equation*}
 \begin{split}
 0 \le  \int u^* \omega_{\varphi_{H^{\delta}}} &= \int_{A}   \omega_{\varphi_{H^{\delta}}} +  \int_{\mathcal{Z}_{ref}} \omega_{\varphi_{H^{\delta}}} - \int_{Z}   \omega_{\varphi_{H^{\delta}}}\\
 &=\mathcal{A}_{H^{\delta}} (\mathbf{y}, [A])  -\mathbb{A}_{H^{\delta}} (\alpha, [Z])\\
 &\le c_{\underline{L}}(H^{\delta},  \sigma ) + 1/ \kappa-\mathbb{A}_{H^{\delta}} (\alpha, [Z]).
    \end{split}
 \end{equation*}
 Therefore,  $c^{pfh}_d(H^{\delta}, \mathfrak{v}) \le c_{\underline{L}}(H^{\delta}, \sigma ) + 1/ \kappa$. Let  $\kappa \to \infty$. Then  $c^{pfh}_d(H^{\delta},  \mathfrak{v}) \le c_{\underline{L}}(H^{\delta},  \sigma).$   By the Hofer-Lipschitz continuity (Theorem \ref{thm4} and Theorem 3.1  of  \cite{CHS2}) and the estimate (\ref{eq68}),   we have $c^{pfh}_d(H,  \mathfrak{v}) \le c_{\underline{L}}(H,  \sigma)$  by  taking  $\delta \to 0$.

\subsection{Proof of Theorem \ref{thm7}}  \label{proof of thm7}
 To begin with, we show that 
$\mathcal{OC}(\underline{L}, H; R)$ is an $R$-module homomorphism. 
 Since 
$\mathcal{OC}(\underline{L}, H)$ is a $\Lambda$-module homomorphism and $S=T^{d-g+1}$, it suffices to show that $$T^b \cdot  \mathcal{OC}(\underline{L}, H; R)=    \mathcal{OC}(\underline{L}, H; R)  T^b$$ for $0\le b \le d-g$. 

Let $0\le b \le d-g$ and $ \sigma \in HF(\Sigma,\underline{L})$. Then  we have 
\begin{equation*}
\begin{split}
T^b \cdot  \mathcal{OC}(\underline{L}, H; R)  (\sigma) & =\sum_{i=0}^{d-g} \mathcal{OC}(\underline{L}, H)(T^i \cdot \sigma )T^{-i+b}= \sum_{j=-b}^{d-g-b} \mathcal{OC}(\underline{L}, H)(T^{j +b}\cdot \sigma )T^{-j} \\
&= \sum_{j=-b}^{-1} \mathcal{OC}(\underline{L}, H)(T^{j +b} \cdot \sigma)T^{-j} +\mathcal{OC}(\underline{L}, H)  (T^b \cdot \sigma )\\
&  - \sum_{j=d-g+1-b}^{d-g}\mathcal{OC}(\underline{L}, H)(T^{j +b}\cdot \sigma )T^{-j}  
\end{split}
\end{equation*}
Since $\mathcal{OC}(\underline{L}, H)$ is a $\Lambda$-module  homomorphism and $T^{d-g+1}=S$, we have 
\begin{equation*}
\begin{split}
\sum_{j=d-g+1-b}^{d-g}\mathcal{OC}(\underline{L}, H)(T^{j +b}\cdot \sigma)T^{-j}  
=&\sum_{i=-b}^{-1}\mathcal{OC}(\underline{L}, H)(T^{i +b+ (d-g+1)}\cdot \sigma)T^{-(d-g+1)-i} \\
=&\sum_{i=-b}^{-1}\mathcal{OC}(\underline{L}, H)(ST^{i +b} \cdot \sigma)S^{-1}T^{-i} \\
=&\sum_{i=-b}^{-1}\mathcal{OC}(\underline{L}, H)(T^{i +b} 
\cdot \sigma) T^{-i}. 
\end{split}
\end{equation*}
The above computations  imply  that $T^b \cdot   \mathcal{OC}(\underline{L}, H; R) =   \mathcal{OC}(\underline{L}, H; R) (T^b \cdot ). $    
 
Let $\mathfrak{d}_i=\mathcal{OC}(\underline{L}) (T^i \sigma_{\underline{L}} )$ be  the  $T^{-i}$ coefficient of $\mathfrak{d}_R= \mathcal{OC}(\underline{L}, H; R)(\sigma_{\underline{L}})$.    By Theorem \ref{thm2} and Theorem \ref{thm9}, we have 
\begin{equation*}
\begin{split}
U^i \mathfrak{d}_i = U^i \circ \mathcal{OC}(\underline{L}) (T^i\sigma_{\underline{L}}) = \mathcal{OC}(\underline{L}) (U^i T^i\sigma_{\underline{L}})  = \mathcal{OC}(\underline{L}) ( \sigma_{\underline{L}})  =\mathfrak{d} \ne 0. 
\end{split}
\end{equation*}
 Thus, $\mathfrak{d}_i \ne 0$.  The same  calculation  in (\ref{eq96}) shows that   $U^{d-g+1} \mathfrak{d}_R= S^{-1}\mathfrak{d}_R$.

Before moving on, we state the following observation.  It will be used to prove the normalization   property $c_d^{pfh}(0,\mathfrak{d}_R; R)=0.$
\begin{lemma} \label{lem23}
For any $i \in \mathbb{Z}$, we have  
$$\operatorname{gr}( \mathcal{OC}(\underline{L}) (T^i \cdot \sigma_{\underline{L}})  ) =\operatorname{gr}( \mathfrak{d} )   +2i.  $$
\end{lemma}
\begin{proof}
Because the continuation morphisms preserve the grading, it suffices to compute the grading of  $\mathcal{OC}(\underline{L}, H'_{\varepsilon})(T^{i}\cdot  (\mathbf{y}_{\diamondsuit}, [A_{\mathbf{y}_{\diamondsuit}}]))$. 

Assume $i \ge 0$. Let $(\alpha, [Z])$ be an anchored PFH generator  such that 
\begin{equation*}
\langle   \mathcal{OC}(\underline{L}, H'_{\varepsilon})(T^{i}\cdot  (\mathbf{y}_{\diamondsuit}, [A_{\mathbf{y}_{\diamondsuit}}])) ,  (\alpha, [Z])  \rangle = \langle   \widetilde{\mathcal{OC}}(\underline{L}, H'_{\varepsilon})(T^{i}\cdot  (\mathbf{y}_{\diamondsuit}, [A_{\mathbf{y}_{\diamondsuit}}])) ,  (\alpha, [Z])  \rangle \ne 0. 
\end{equation*}
Then we have an HF-PFH curve $u \in \mathcal{M}^J( \mathbf{y}_{\diamondsuit}, \alpha)$   with relative homology class 
$$ \left(A_{\mathbf{y}_{\diamondsuit}} + \sum_{m=1}^i B_{j_m} \right)  \#\mathcal{Z}_{ref} \# (-Z).$$  Moreover,  $I(u)=0$.  By Lemma \ref{lem3}, we have 
$$I(A_{\mathbf{y}_{\diamondsuit}}   \#\mathcal{Z}_{ref} \#  (-Z)) + 2i =I(u)=0.  $$
On the other hand, 
 $$I(A_{\mathbf{y}_{\diamondsuit}}   \#\mathcal{Z}_{ref} \#  (-Z))  = I(A_{\mathbf{y}_{\diamondsuit}}   \#\mathcal{Z}_{ref} \#  (-Z_{\alpha_{\diamondsuit}}))  + I(Z_{\alpha_{\diamondsuit}} \#(-Z))=  I(Z_{\alpha_{\diamondsuit}} \#(-Z)). $$ 
Thus,  $$ I(Z \#(-Z_{\alpha_{\diamondsuit}} )) = \operatorname{gr}( (\alpha, Z)) -   \operatorname{gr}((\alpha_{\diamondsuit}, Z_{\alpha_{\diamondsuit}})) =   2i. $$
The above relation holds for   $i \ge 0$. Taking $i=0$, we know that  $\operatorname{gr}(\mathfrak{d}) = \operatorname{gr}((\alpha_{\diamondsuit}, Z_{\alpha_{\diamondsuit}})) $. For $i>0$, the above relation implies the conclusion  of the lemma. 

If $i<0$,  the result follows from the case  $i\ge 0$ by using  the fact   that $\mathcal{OC}(\underline{L})$ is $\Lambda$-homomorphism. 
\end{proof}

For $1\le i \le d-g $, let $\mathfrak{d}_i =  \mathcal{OC}( \underline{L}) (T^i \cdot  \sigma_{\underline{L}})$  denote the  $T^{-i}$ coefficient  of $\mathfrak{d}_R.$
   Write $$  \mathcal{OC}(\underline{L}, H_{\varepsilon})(T^{i}\cdot  (\mathbf{y}_{\diamondsuit}, [A_{\mathbf{y}_{\diamondsuit}}])) =\sum (\alpha, [Z_{\alpha} + n_{\alpha}\Sigma]),$$
where $(\alpha, [Z_{\alpha} + n_{\alpha} \Sigma])$  is an anchored  PFH generator of $H_{\varepsilon}$.  By Lemma \ref{lem23},  we have  $$I((Z_{\alpha} + n_{\alpha}\Sigma) \#(-Z_{\alpha_{\diamondsuit}})) =2i.$$  On the other hand,  it follows  from (\ref{eq52}) that  $I( Z_{\alpha}  \#(-Z_{\alpha_{\diamondsuit}})) \ge 0$. Therefore,
$$n_{\alpha} =\frac{1}{2(d-g+1)} (2i -I(Z_{\alpha} \#(- Z_{\alpha_{\diamondsuit}} ))) \le \frac{d-g}{d-g+1} <1.$$  In other words, $n_{\alpha} \le 0$.  By Hofer-Lipschitz property  of the spectral invariants, we have 
\begin{equation*}
\begin{split}
c_d^{pfh}(0, \mathfrak{d}_i) -\frac{i}{d-g+1} & \le c_d^{pfh}(H_{\varepsilon}, \mathfrak{d}_i) -\frac{i}{d-g+1} + d O(\varepsilon) \\
&\le  \mathbb{A}_{H_{\varepsilon}}( (\alpha, [Z_{\alpha}+ n_{\alpha}\Sigma] ) )  -\frac{i}{d-g+1} + d O(\varepsilon)  \\
& \le \mathbb{A}_{H_{\varepsilon}}( (\alpha, [Z_{\alpha}] ) )  -\frac{1}{d-g+1}  + d O(\varepsilon) \\
&=d O(\varepsilon) -\frac{1}{d-g+1}<0. 
\end{split}
\end{equation*}
Consequently,  $c_d^{pfh}(0, \mathfrak{d}_R; R)  =   \max_{ 0\le i\le d-g}\{c_d^{pfh}(0,\mathfrak{d}_i) -\frac{i}{d-g+1}\}=  c_d^{pfh}(0, \mathfrak{d})=0.$

Suppose that $\underline{L}$ is $0$-admissible. Let $\sigma \in HF(\Sigma, \underline{L})$ be a  nonzero class  such that $\mathfrak{v}=\mathcal{OC}(\underline{L}, H; R)(\sigma) \ne 0$.  Let $\mathfrak{v}_i$  be  the $T^{-i} $ coefficient  of $\mathfrak{v}$, namely, $\mathfrak{v}_i = \mathcal{OC}(\underline{L}, H)( T^i \cdot \sigma).  $ Theorem \ref{thm2} implies that 
$$c_d^{pfh}(H,\mathfrak{v}_i) \le c_{\underline{L}}(H, T^i \cdot \sigma)= c_{\underline{L}}(H, \sigma) +i \lambda=c_{\underline{L}}(H, \sigma) + \frac{i}{d-g+1}.$$
Therefore, 
\begin{equation*}
c_d^{pfh}(H,\mathfrak{v}; R) = \max_{0\le i\le d-g}\{c_d^{pfh}(H,\mathfrak{v}_i) -\frac{i}{d-g+1}\} \le   c_{\underline{L}}(H, \sigma). 
\end{equation*}

\subsection{Proof of Theorem \ref{thm8}}
In this subsection,  we prove Theorem \ref{thm8}. To this end, we need the following lemma about the grading. 
 \begin{lemma} \label{lem24}
We have  $\operatorname{gr}(\mathfrak{e})=d$ and  $\operatorname{gr}( \mathfrak{d}) =-d$. In particular, 
 \begin{equation*} 
\operatorname{gr}(\mathfrak{e}) -\operatorname{gr}( \mathfrak{d}) =2d.
\end{equation*}
 \end{lemma}
 \begin{proof}
 Let $\mathfrak{e}_{H}:=  \mathfrak{j}_H^{-1}(\mathfrak{e})$ and $\mathfrak{d}_{H}  :=\mathfrak{j}_H^{-1}(\mathfrak{d})$.   
 Because the cobordism maps  $\mathfrak{I}_{H, G}$
 preserve  the grading,  it suffices to  compute  the grading   of $ \mathfrak{e}_{H}, \mathfrak{d}_{H} $ for  a special  $H$.   Take $H=H_{\varepsilon}$. 
 By definition,  we have 
 $$\operatorname{gr}(\alpha, [Z]) =I(Z\#Z_{ref}) - d,$$ where  $I(Z\#Z_{ref})$ is the ECH index  given by (\ref{eq52}). 
By  Lemma \ref{lem5} and Lemma \ref{lem6}, $\mathfrak{d}_{H_{\varepsilon}}$ is represented by 
 $$\mathfrak{c}_{\diamondsuit}'=(\alpha_{\diamondsuit}, [Z_{\alpha_{\diamondsuit}}]) +\sum (\beta', [Z']), $$ where $\beta' \ne \alpha_I.$  By (\ref{eq52}),
 we have 
 \begin{equation}   \label{eq54}
 \operatorname{gr}( \mathfrak{d} _{H_{\varepsilon}})=\operatorname{gr}(\alpha_{\diamondsuit}, [Z_{\alpha_{\diamondsuit}}])=-d.
 \end{equation}

Recall that  $ \mathfrak{e}_{H_{\varepsilon}}  =PFH^{sw}_{Z_{ref}}(X_+, \Omega_{X_+})(1)$ (see Remark 6.1 of \cite{GHC2}), where $(X_+, \Omega_{X_+})$ is the symplectic cobordism (\ref{eq98}).
  Assume that $PFC^{sw}_{Z_{ref}}(X_+, \Omega_{X_+})(1) =\sum (\alpha_+, [Z_+])$.
 For $Z_+ \in H_2(X_+, \emptyset, \alpha_+)$,  by the same argument  as  in  Lemma 3.3 of \cite{GHC1} or Lemma 3.1 of \cite{GHC3}, we have
 \begin{equation} \label{eq53}
 \begin{split}
 &I(Z_+) =2d-2e_+(\alpha_+) - h(\alpha_+) +2n(Z_+) (d -g +1),\\
 &\int_{Z_+} \omega_{X_+}  = -H_{\varepsilon}(\alpha_+)   +n(Z_+).
 \end{split}
 \end{equation}
 By the  holomorphic curve axioms, we have a broken holomorphic current  $\mathcal{C} \in \overline{\mathcal{M}^{J_{X_+}}}(\emptyset, \alpha_+, Z_+)$ with $I(Z_+)=0$.  The energy of $\mathcal{C}$ is nonnegative. Then $n(Z_+) \ge 0$.  The ECH index formula  (\ref{eq53}), together with  $I(Z_+) =0 $,  implies   that $e_+(\alpha_+)=d$.  It follows from  (\ref{eq52}) that 
  \begin{equation} \label{eq93}
 \operatorname{gr}(\mathfrak{e}_{H_{\varepsilon}})=\operatorname{gr}((\alpha_+, [Z_+])) =d.
  \end{equation}
Then the conclusion of the  lemma follows from (\ref{eq54}) and (\ref{eq93}). 
 \end{proof}

\begin{proof}[Proof of Theorem \ref{thm8}]
Now assume $\Sigma=\mathbb{S}^2$.    By Theorem \ref{thm7},  we know   $\mathcal{OC}(\underline{L}) (T^i \cdot \sigma_{\underline{L}})  \ne 0$.  By Lemma \ref{lem23},  Lemma \ref{lem24} and Proposition \ref{lem25}, we have 
$$\mathcal{OC}(\underline{L}) (T^i \cdot \sigma_{\underline{L}})  =U^{d-i}\mathfrak{e}. $$  

\end{proof}


\subsection{Comparing  PFH and HF spectral invariants}
In this section, we prove Theorem \ref{thm0}, Theorem \ref{thm5}.  Hence, we assume that the link $\underline{L}$ is $0$-admissible.

\begin{proof} [Proof of  Theorem \ref{thm5}]

The inequality $ c_{\underline{L}} (H,e_{\underline{L}}) \le c^{pfh}_{d}(H,   \mathfrak{e})$ is Theorem 3 of \cite{GHC2}. By  Theorem \ref{thm2} and Theorem \ref{thm7}, we have $  c^{pfh}_{d}(H, \mathfrak{d}) \le c^{pfh}_{d}(H, \mathfrak{d}_R; R) \le c_{\underline{L}}(H, \sigma_{\underline{L}})$.

It remains to prove  $c_{\underline{L}} (H, \sigma_{\underline{L}}) \le c_{\underline{L}} (H, e_{\underline{L}}).$     By the triangle inequality, we have
 \begin{equation*}
\begin{split}
c_{\underline{L}} (H, \sigma_{\underline{L}}) = c_{\underline{L}} (H, \mu_2(e_{\underline{L}} \otimes \sigma_{\underline{L}}))\le   c_{\underline{L}}(H, e_{\underline{L}}) +  c_{\underline{L}}(0, \sigma_{\underline{L}})  =  c_{\underline{L}}(H, e_{\underline{L}}).
\end{split}
\end{equation*}
\end{proof}

Next, we prove Theorem \ref{thm0}   using Theorem \ref{thm5}.
\begin{proof} [Proof of  Theorem  \ref{thm0}]
The  standard energy estimate implies  that the $U$ map decreases the PFH spectral invariants. 
By Proposition \ref{lem25}, Theorem \ref{thm5} and Theorem \ref{thm8}, we have 
\begin{equation} \label{eq77}
\begin{split}
c_{\underline{L}}  (H, \sigma_{\underline{L}})  \ge  c^{pfh}_{d}(H,   \mathfrak{d}_R; R)   \ge  c^{pfh}_{d}(H,   \mathfrak{d}) =  c^{pfh}_{d}(H,  U^{d} \mathfrak{e}) \ge   c^{pfh}_{d}(H,    U^{d+1}   \mathfrak{e}) = c^{pfh}_{d}(H,    S^{-1} \mathfrak{e}).  \\
\end{split}
\end{equation}
 According to  Proposition 4.2 of \cite{EH}, we have 
 \begin{equation}\label{eq78}
 c^{pfh}_{d}(H,    S^{-1} \mathfrak{e})=c^{pfh}_{d}(H,    \mathfrak{e})-1.  
 \end{equation}
In sum, we have 
\begin{equation*}
\begin{split}
c^{pfh}_{d}(H, \mathfrak{e})-1 \le c^{pfh}_{d}(H,   \mathfrak{d}_R; R)    \le c_{\underline{L}}(H, \sigma_{\underline{L}}) \le c_{\underline{L}}(H, e_{\underline{L}}) \le c^{pfh}_{d}(H,  \mathfrak{e}).
\end{split}
\end{equation*}
\end{proof}

%
%
%
%

\subsection{Quasimorphisms }
In this section, we show that $\mu^{pfh}_d(\cdot, \mathfrak{e})$ is a quasimorphism on $Ham(\mathbb{S}^2, \omega)$.    This result could be deduced from the equivalence  between PFH spectral invariants and link spectral invariants in   Theorem \ref{thm0}  and the corresponding result in Theorem 7.6 of \cite{CHMSS}.  Here we provide an alternative proof by using the duality in Floer homology. The argument is due to  M. Entov and L. Polterovich \cite{EP}.

To begin with,  let us recall some facts about the duality in Floer homology.
Let $\mathfrak{c}$ be a   \textbf{graded filtered Floer-Novikov complex} over a field $\mathbb{F}$ in the sense of \cite{U}.  We can associate a graded chain complex $(C_*(\mathfrak{c}), \partial)$ with $\mathfrak c$. One can define the homology and spectral numbers for   $(C_*(\mathfrak{c}), \partial)$.   Roughly speaking,  $\mathfrak{c}$ is an abstract  complex that  is characterized by the common properties of Floer homology.
We remark   that   the PFH chain complex is an example  of a  graded filtered Floer-Novikov complex.

For $\mathfrak{c}$, M. Usher defines another   graded filtered Floer-Novikov complex   $\mathfrak{c}^{op}$ called the opposite complex.  
Roughly speaking, the homology of  $(C_*(\mathfrak{c}^{op}), \delta)$ is  Poincaré dual to $H_*(C_*(\mathfrak{c}))$ in the following sense:
There is a non-degenerate   pairing 
$$\Delta: H_{-k}(C_*(\mathfrak{c}^{op})) \times H_k(C_*(\mathfrak{c})) \to \mathbb{F}.$$ We refer the reader  to \cite{U}  details of the notions of a graded filtered Floer-Novikov complex and   opposite complex.

Let $ \mathfrak{c}_1,  \mathfrak{c}_2$ be   graded filtered Floer-Novikov complexes.  Let $I: C_*(\mathfrak{c}_1) \to  C_*(\mathfrak{c}_2)$  be a 0-degree  chain map given by
\begin{equation*}
I p_1 = \sum_{p_2} n(p_1, p_2)p_2,
\end{equation*}
where $p_i$ are generators of  $C_*(\mathfrak{c}_i)$ and $n(p_1,p_2) \in \mathbb{F}$.   Define  $I^{op}: C_*(\mathfrak{c}^{op}_2) \to  C_*(\mathfrak{c}^{op}_1)$  by  $$I^{op}p_2=\sum_{p_1} n(p_1, p_2)p_1. $$

\begin{lemma} \label{lem20}
The map  $I^{op}: C_*(\mathfrak{c}^{op}_2) \to  C_*(\mathfrak{c}^{op}_1)$ satisfies the following properties:
\begin{itemize}
\item
 $I^{op}$ is a chain map.  It descends to a map $I^{op}_* :  H_*(C_*(\mathfrak{c}^{op}_2))\to  H_*(C_*(\mathfrak{c}^{op}_1)).$
\item
Let $I_1: C_*(\mathfrak{c}_1) \to  C_*(\mathfrak{c}_2)$ and $I_2: C_*(\mathfrak{c}_2) \to  C_*(\mathfrak{c}_3)$ be two  0-degree  chain maps. Then 
$(I_2 \circ I_1)^{op}= I_1^{op} \circ I_2^{op}$.  In particular, if $I_*$ is an isomorphism, so is $I^{op}_*$.

\item
Let $a \in H_{-k}(C_{*}(\mathfrak{c}^{op}_2))$ and  $b \in H_k(C_*(\mathfrak{c}_1))$. Then  we have
$$\Delta(a, I_*(b)) = \Delta(I^{op}_*(a), b). $$
\end{itemize}
\end{lemma}
The proof of this lemma  is straightforward; we leave the details to the reader.

Now we construct the opposite complex of  $(PFC_*(\mathbb{S}^2, \varphi_H, d), \partial_J)$.  
Let  $\bar{H}_t = -H_{1-t}$. This is  a Hamiltonian  function  generating  $\varphi_H^{-1}$.   Define  a diffeomorphism
\begin{equation*}
\begin{split}
\iota: &  S_t^1 \times \Sigma \to    S_{\tau}^1 \times \Sigma\\
&(t, x) \to (1-t, x).
\end{split}
\end{equation*}
Note that $ (\iota^{-1})^*(\omega + dH_t \wedge dt) = \omega + d \bar{H}_{\tau} \wedge d\tau. $ If $\gamma$ is a  periodic orbit of $\varphi_H$, then $\bar{\gamma} :=\iota \circ \gamma$ is a   periodic orbit of $\varphi_H^{-1}$. Here we orient  $\bar{\gamma}$ so that it is positively transverse to $\Sigma$.

Recall  the symplectic cobordism    $(X=\mathbb{R} \times S^1 \times \Sigma, \Omega_X =\omega + d(H_t^s dt) + ds\wedge dt)$ (\ref{eq23}).  We  extend the map $\iota$ by
\begin{equation*}
\begin{split}
\iota: &\mathbb{R}_s \times S_t^1 \times \Sigma \to  \mathbb{R}_r \times S_{\tau}^1 \times \Sigma\\
&(s, t, x) \to (-s, 1-t, x).
\end{split}
\end{equation*}
Note that $(\iota^{-1})^*\Omega_X = \omega - d(H^{-r}_{1-\tau} \wedge d\tau) + dr \wedge d\tau$.    Therefore, $(X, (\iota^{-1})^*\Omega_X )$ is a symplectic cobordism  from $(Y_{\varphi_{\bar{H}_-}}, \omega_{\varphi_{\bar{H}_-}})$ to  $(Y_{\varphi_{\bar{H}_+}}, \omega_{\varphi_{\bar{H}_+}})$.

Consider the case in which  $H_t^s=H_t$. Let $PFC_*(\mathbb{S}^2, \varphi_H^{-1}, d)$ be the  complex generated by $(\bar{\alpha}, -\iota_*Z)$.    Note that $\iota_*Z \in H_2(Y_{\varphi^{-1}_H},   \gamma_{\bar{H}}^{\mathbf{x}}, \bar{\alpha})$.  Here   $-\iota_* Z $ denotes    the unique class in    $H_2(Y_{\varphi^{-1}_H},     \bar{\alpha}, \gamma_{\bar{H}}^{\mathbf{x}})$  such that $(-\iota_*Z) \# \iota_*Z =[\mathbb{R} \times \bar{\alpha}]$.
 Note that we have
\begin{enumerate} [label=\textbf{E.\arabic*}]
\item
 $\mathbb{A}_{\bar{H}}  (\bar{\alpha}, -\iota_*Z) = - \mathbb{A}_H(\alpha, [Z])$.

 \item
 $\operatorname{gr} (\bar{\alpha}, -\iota_*Z) = -\operatorname{gr}(\alpha, Z). $ 

\item \label{E3}
Let $u \in \mathcal{M}^J(\alpha_+, \alpha_-, Z)$ be a holomorphic curve in $ \left(X,  \Omega_X \right) $. Then $\bar{u}:=\iota\circ u  \in \mathcal{M}^{\bar{J}}(\bar{\alpha}_-, \bar{\alpha}_+, \iota_*Z)$  is  a holomorphic curve in $ \left(X,  (\iota^{-1})^*\Omega_X \right) $, where $\bar{J} = \iota_* \circ J \circ \iota_*^{-1}$.  This establishes a one-to-one correspondence between $ \mathcal{M}^J(\alpha_+, \alpha_-, Z)$ and $\mathcal{M}^{\bar{J}}(\bar{\alpha}_-, \bar{\alpha}_+, \iota_*Z)$.

\end{enumerate}
These three points imply  that     $(PFC_*(\mathbb{S}^2, \varphi_H^{-1}, d), \partial_{\bar{J}} )$ is the opposite complex  of  $(PFC_*(\mathbb{S}^2, \varphi_H, d), \partial_J)$.

The  pairing
$\Delta:   PFC_{-k}(\mathbb{S}^2, \varphi_H^{-1}, d) \otimes PFC_k(\mathbb{S}^2, \varphi_H, d)    \to \mathbb{F}$
is defined by
\begin{equation*}
\begin{split}
\Delta \left(\sum a_{(\bar{\alpha}, -\iota_*[Z])} (\bar{\alpha}, -\iota_*[Z]), \sum b_{(\beta, [Z])} (\beta, [Z]) \right) := \sum_{ (\alpha, [Z])} a_{(\bar{\alpha}, -\iota_*[Z])}  b_{(\alpha, [Z])}.
\end{split}
\end{equation*}
This pairing  descends to homology.   By Corollary 1.4 of  \cite{U} (also see Lemma 2.2 of \cite{EP}), we have
\begin{equation*}
\begin{split}
c^{pfh}_{d }(H, \mathfrak{e}) = -\inf \{c^{pfh}_{d}(\bar{H}, \mathfrak{v} ) \vert \mathfrak{v} \in PFH_{-k}(\mathbb{S}^2, \varphi_H^{-1}, d), \Delta(\mathfrak{v} ,   \mathfrak{e}_H) \ne 0 \},
\end{split}
\end{equation*}
where  $\mathfrak{e}_{H} =  \mathfrak{j}_H^{-1}(\mathfrak{e})$.

The key to proving  Theorem \ref{thm1} is the following lemma.
\begin{lemma} \label{lem19}
For any Hamiltonian function $H \in C^{\infty}(S^1 \times \mathbb{S}^2)$, we have
\begin{equation*}
\begin{split}
&c^{pfh}_{d }(H,  \mathfrak{e} ) + c^{pfh}_{d }(\bar{H},  \mathfrak{e})\le 1, \\
\mbox{ and }& c_{\underline{L} }(H, e_{\underline{L}}) +  c_{\underline{L} }(\bar{H}, e_{\underline{L}}) \le  1.
\end{split}
\end{equation*}
\end{lemma}
\begin{proof}
By Theorem \ref{thm5}, it suffices to prove the first inequality.

Let $g: \mathbb{S}^2\to \mathbb{R}$ be a Morse function with two critical points $x_+, x_-$, where $x_+$ is the maximum point and $x_-$ is the minimum point. Let $\bar{G}_{\epsilon}: =\epsilon g$. Take $\mathbf{x}=(x_-, ..., x_-)$ to  be the base point.
By  (\ref{eq69}) and (\ref{eq53}), we have
\begin{equation*}
\operatorname{gr}((\gamma_{x_+}^{d_+} \gamma^{d_-}_{x_-}, Z_{\gamma_{x_+}^{d_+} \gamma^{d_-}_{x_-}}+n[\mathbb{S}^2])) = 2d_+  + 2n(d+1) -d,
\end{equation*}
where $d_{\pm} \ge 0$ such that $d_++d_-=d$. The grading formula implies that $\partial =0$.   It is important to note that  $ (\gamma^d_{x_+}, [Z_{\gamma_{x_+}^{d} }])$ is the only element with $\operatorname{gr}=d$,  and  $(\gamma^d_{x_-}, [Z_{\gamma_{x_-}^{d} }])$ is the only element with $\operatorname{gr}=-d$. By Lemma \ref{lem24}, we have  $\mathfrak{e}_{ \bar{G}_{\epsilon}} = (\gamma^d_{x_+}, [Z_{\gamma_{x_+}^{d} }])$ and  $\mathfrak{d}_{ \bar{G}_{\epsilon}} = (\gamma^d_{x_-}, [Z_{\gamma_{x_-}^{d} }]).$
 Then for any $H$,  we have $\mathfrak{e}_{\bar{H}} = \mathfrak{I}_{\bar{G}_{\epsilon}, \bar{H}}( (\gamma^d_{x_+}, [Z_{\gamma_{x_+}^{d} }]))$.

By the observation \ref{E3}, we have  $( \mathfrak{I} _{\bar{G}_{\epsilon}, \bar{H}})^{op} =\mathfrak{I}_{H, G_{\epsilon}}$. By Lemma \ref{lem20}, we have 
\begin{equation*}
\Delta(\mathfrak{v} , \mathfrak{e}_{ \bar{H}}  )= \Delta(\mathfrak{v} , \mathfrak{I}_{\bar{G}_{\epsilon}, \bar{H}}( \gamma^d_{x_+}, [Z_{\gamma_{x_+}^{d} }])) = \Delta( \mathfrak{I}_{H, G_{\epsilon}}(\mathfrak{v} ), (\gamma^d_{x_+}, [Z_{\gamma_{x_+}^{d} }])).
\end{equation*}
Because  $(\bar{\gamma}_{x_+}^{d} , [-\iota_*Z_{\gamma^d_{x_+}}])$ is the only class with $\operatorname{gr}=-d$, $ \Delta(\mathfrak{v} , \mathfrak{e}_{\bar{H}}  )\ne 0$ if and only if $  \mathfrak{I} _{H, G_{\epsilon}}(\mathfrak{v} )=(\bar{\gamma}_{x_+}^{d} , [-\iota_*Z_{\gamma^d_{x_+}}])$. Therefore, $\mathfrak{v} =\mathfrak{d}_{H} $.
Combining this with (\ref{eq77}) and  (\ref{eq78}), we have
\begin{equation*}
-c^{pfh}_{d}(\bar{H}, \mathfrak{e}) =  c^{pfh}_{d}(H, \mathfrak{d})  \ge c^{pfh}_{d}(H, \mathfrak{e} ) -1.
\end{equation*}
\end{proof}

\begin{proof} [Proof of Theorem \ref{thm1}]
By the triangle inequality  in Theorem \ref{thm4} and Lemma \ref{lem19}, we have
\begin{equation*}
\begin{split}
& c_{\underline{L}}(H, e_{\underline{L}})  +  c_{\underline{L}}(K, e_{\underline{L}})\\
=& c_{\underline{L}}(H, e_{\underline{L}})  +  c_{\underline{L}}(\bar{H} \diamond H \diamond K, e_{\underline{L}})\\
\le &  c_{\underline{L}}(H, e_{\underline{L}})  + c_{\underline{L}}(\bar{H}, e_{\underline{L}})  +  c_{\underline{L} }(H \diamond K, e_{\underline{L}}) \le   c_{\underline{L}}(H \diamond K, e_{\underline{L}})+1.
 \end{split}
\end{equation*}
The above inequality and the   triangle inequality imply  that $\mu_{\underline{L}}$ is a quasimorphism with defect $2$. So is $\mu_{d}^{pfh}(\cdot, \mathfrak{e})$.
\end{proof}

\section{Quantum homology correspondences} \label{section7}
This section gives a heuristic discussion of the correspondence between PFH and quantum homology, as well as the correspondence between the open-closed morphisms and the quantum inclusion. Furthermore, we give a possible way to compute the  class $\mathfrak{d}_R$ for the torus case. 

The discussion draws on the methods of \cite{GHC4, MW,   Zhou}, for which $\mathbb{Q} $-coefficients and the condition $d>2g-2$ are essential. It is not known whether analogous results hold with $\mathbb{Z}/2\mathbb{Z}$-coefficients or when $d\le 2g-2$. See Remark 1.5 and the discussion on page 3 of \cite{GHC4}. Therefore, unless otherwise stated, we assume throughout this section that $\mathbb{F}=\mathbb{Q} $ and $d >2g-2$.

 \subsection{Isomorphisms between PFH and quantum homology}
In view of the tautological correspondence  between the holomorphic curves in $\mathbb{R} \times S^1 \times \Sigma$ and  holomorphic cylinders in   $\operatorname{Sym}^d\Sigma$,  one expects the following conjecture. 

\textbf{Conjecture 1. } There are  ring  isomorphisms  
\begin{equation}\label{eq95}
\begin{split}
&\mathfrak{J}: QH(\operatorname{Sym}^d\Sigma;R) \cong \widetilde{PFH}(\Sigma, d) \otimes_{\Lambda} R,\\
& \mathcal{J}:  QH(\operatorname{Sym}^d \underline{L}; R) \cong HF(\Sigma, \underline{L}), 
\end{split}
\end{equation}
where   $QH(L; R)$ and $QH(\operatorname{Sym}^d\Sigma;R)$ are respectively the Lagrangian quantum homology and quantum  homology (see \cite{BC} for their definition).

The second isomorphism $\mathcal{J}$ has been constructed using Theorem 1 of \cite{GHC2}; at present, however, it is known only to be an $R$-module isomorphism. Motivated by V. Munoz and B.-L. Wang (Theorem 2.1 of \cite{MW}) and by S. Zhou's unpublished work \cite{Zhou}, we  propose a construction for the first isomorphism $\mathfrak{J}$.

Let $\mathbb{A}(\Sigma) : =\mathbb{Q}[U] \otimes \Lambda^*H^1(\Sigma, \mathbb{Q})$. Each element $\Theta \in \mathbb{A}(\Sigma)$ acts as an endomorphism of PFH, and the $U$ map forms part of this action. See Section 1.a of \cite{T6} for the definition of the $H^1(\Sigma, \mathbb{F})$-action. The PFH cobordism maps can also be extended to incorporate this action:
\begin{equation*} 
\begin{split}
PFH_{Z_{ref}}^{sw}(X, \Omega_{X} ) : \mathbb{A}(\Sigma) \otimes \Lambda \otimes    \widetilde{PFH}(\Sigma, \varphi_+, d) \to  \widetilde{PFH}(\Sigma, \varphi_-, d).
\end{split}
\end{equation*}
 
One could identify $H^*(\operatorname{Sym}^d \Sigma, \mathbb{Q})$ with an appropriate subset of $ \mathbb{A}(\Sigma)$. We then make the following claims, without proof:
\begin{itemize}
\item
The restriction of $PFH_{Z_{ref}}^{sw}(X_+, \Omega_{X_+} ) $ to $H^*(\operatorname{Sym}^d \Sigma, \mathbb{Q})$ gives an isomorphism
\begin{equation}  \label{eq102}
\begin{split}
PFH_{Z_{ref}}^{sw}(X_+, \Omega_{X_+} ) : H^*(\operatorname{Sym}^d \Sigma, \mathbb{Q})\otimes \Lambda \cong H^*(\operatorname{Sym}^d \Sigma, \Lambda) \to  \widetilde{PFH}(\Sigma, \varphi_H, d), 
\end{split}
\end{equation}
where $(X_+, \Omega_{X_+})$ is the symplectic cobordism in (\ref{eq98}).

\item
There exists a nondegenerate pairing
$$ \langle \cdot, \cdot \rangle_{pfh} : \widetilde{PFH}(\Sigma,\varphi_H,  d)  \otimes \widetilde{PFH}(\Sigma, \varphi_H,  d) \to \mathbb{Q}  $$
defined in terms of the PFH cobordism maps induced by the symplectic cobordisms (\ref{eq6}, \ref{eq98}).

\item
By the composition law for PFH cobordism maps, the isomorphism (\ref{eq102}) descends to the isomorphism $\mathfrak{J}$ after extending to $R$ and applying Poincaré duality. This is an analogue of the PSS isomorphism in Hamiltonian Floer homology. Similarly, $\langle \cdot, \cdot \rangle_{pfh} $ descends to a pairing on $\widetilde{PFH}(\Sigma, d)\otimes_{\Lambda} R$.

\item
The pullback $ \mathfrak{J}^*\langle \cdot, \cdot \rangle_{pfh}$ agrees with the intersection pairing on $H^*(\operatorname{Sym}^d \Sigma,  R)$.
\end{itemize}
Evidence for these claims is provided by the parallel result proved in the ECH setting in \cite{GHC4}. We refer the reader to the outline on pages 7--8 of \cite{GHC4}. More specifically, the arguments for the claims above should parallel those in Section 5.1 of \cite{GHC4},  and  the definition of $\langle \cdot, \cdot \rangle_{pfh} $ is analogous to Definition 5.10 of \cite{GHC4}.

\begin{remark}
Presuming that the statements above hold, we obtain only that $\mathfrak{J} $ and $ \mathcal{J}$ are $R$-module isomorphisms. Conjecture (\ref{eq95}) asserts the stronger statement that they are ring isomorphisms with respect to the quantum products. We \textbf{cannot} prove these ring isomorphisms at present.
\end{remark}

Similar to the story here,  Biran and Cornea       construct a quantum inclusion   \cite{BC}: 
$$i_L: QH(L;R)\longrightarrow QH( M;R),$$
where $M$ is a  closed symplectic manifold  and $L \subset M$ is a monotone Lagrangian.  When the link  $\underline{L}$ is $0$-admissible with $d$ components, $\operatorname{Sym}^d \underline{L}$ is a monotone Lagrangian in $\operatorname{Sym}^d \Sigma$ (Cf. Lemma 4.21 of \cite{CHMSS}).

\textbf{Conjecture 2.}  Building on the preceding discussion, we conjecture that the following diagram commutes:
 \begin{equation}	\label{eq103}	 
\begin{CD}
				QH(\operatorname{Sym}^d\underline{L}; R) @> i_{\operatorname{Sym}^d \underline{L}}>>  QH(\operatorname{Sym}^d\Sigma; R) \\
				@VV \mathcal{J}V @VV \mathfrak{J}V\\
				HF(\Sigma,   \underline{L})  @> \mathcal{OC}(\underline{L}; R)>>    \widetilde{PFH}(\Sigma,   d)\otimes_{\Lambda} R. 
			\end{CD}
\end{equation}

Consider the case in which $\Sigma=\mathbb{S}^2$ and $\underline{L}$ is $0$-admissible with $d$ components.  Then $\operatorname{Sym}^d\mathbb{S}^2\cong\mathbb{CP}^d$ and $  \operatorname{Sym}^d\underline{L}$ is Hamiltonian isotopic to the monotone Clifford torus $L: = \mathbb{T}^d_{clif}$.     In the concrete case  $d=2$, Theorem 2.3.2 of \cite{BC} gives, after rewriting their Novikov variable as $T$,
\begin{equation} \label{eq94} 
i_L([\operatorname{pt}_{L}])=[\operatorname{pt}]+  T^{-1}h +[\mathbb{CP}^2]T^{-2} = \sum_{i=0}^2 h^{*{2-i}}T^{-i},
\end{equation}
where  $h =[\mathbb{CP}^1]$ and  $[\operatorname{pt}_{L}]$ is the degree-zero generator (point class) of $QH(\mathbb{T}^2_{clif};R)$, and ``$*$" is the quantum product.  

On the other hand,  (\ref{eq94})    has exactly the same form as the computation of  $ \mathcal{OC}(\underline{L}; R)$      in Theorem \ref{thm8}.    This provides  evidence   that Conjecture (\ref{eq103}) holds. 
In particular, motivated by Theorem \ref{thm8}, we conjecture that $$i_{\mathbb{T}^d_{\mathrm{clif}}}([\operatorname{pt}_{\mathbb{T}^d_{\mathrm{clif}}}]) =
\sum_{i=0}^{d} h^{*(d-i)}T^{-i}$$  for  $d \in \mathbb{Z}_{\ge 1}$.

 \subsection{A   characterization  of  open-closed morphisms}
This subsection describes a heuristic characterization of the open-closed morphisms, which is analogous to the corresponding characterization  for the quantum inclusion. 
 
 By Theorem A of
\cite{BC}, the quantum inclusion is characterized by
\begin{equation} \label{eq105}
\langle \operatorname{PD}(a),i_L(x)\rangle= \epsilon_L(a\circledast x), 
\end{equation}
where $\epsilon_L$ is the \textbf{quantum augmentation.}   Hence, for any
$x\in QH(L; R)$ satisfying $\epsilon_L(x)\neq0$, taking
$a=[M]$  and using
$ [M]\circledast x=x$ gives
$$
\langle 1,i_L(x)\rangle=\epsilon_L(x)\neq0,
$$
and therefore $i_L(x)\neq0$. In particular, whenever the quantum point
class $[\operatorname{pt}_L]$ is defined and has nonvanishing  augmentation,
$i_L([\operatorname{pt}_L])\neq0.$

 The proof of the nonvanishing assertion in Theorem \ref{thm2} uses a formally similar detection mechanism.  Recall that our underlying strategy is to show that 
 \begin{equation*} 
 PFH_{Z_{ref}}^{sw}(X_-, \Omega_{X_-}) \circ \mathcal{OC}(\underline{L}, H) ( {j}_H^{-1} (\sigma_{\underline{L}}) ) \ne 0, 
 \end{equation*} 
 where  $(X_-, \Omega_{X_-})$ is the symplectic cobordism  (\ref{eq6}).  As a result, $\mathcal{OC}(\underline{L})(\sigma_{\underline{L}}) \ne 0.$ More generally,  we have  $\mathcal{OC}(\underline{L}; R)(\sigma) \ne 0$ whenever    \begin{equation} \label{eq99}
  PFH_{Z_{ref}}^{sw}(X_-, \Omega_{X_-})\circ \Theta \circ \mathcal{OC}(\underline{L}, H; R) ( {j}_H^{-1} (\sigma) ) \ne 0 
   \end{equation} 
for some  $\Theta \in \mathbb{A}(\Sigma)$.

Heuristically, we have a characterization of  $\mathcal{OC}(\underline{L}; R)$ which is similar to (\ref{eq105}). We explain this as follows. Using the definition of   $ \langle \cdot, \cdot \rangle_{pfh}$, (\ref{eq99})  can equivalently be written as
\begin{equation*}
\begin{split}
&PFH_{Z_{ref}}^{sw}(X_-, \Omega_{X_-})\circ \Theta \circ \mathcal{OC}(\underline{L}, H; R) ( {j}_H^{-1} (\sigma) )\\
=& \langle \mathfrak{J}(\Theta), \mathcal{OC}(\underline{L}; R) (\sigma )  \rangle_{pfh} = \langle \Theta, \mathfrak{J}^{-1} \circ  \mathcal{OC}(\underline{L}; R) (\sigma )   \rangle, 
\end{split}
\end{equation*}
 where  $\langle \cdot, \cdot   \rangle$ is the intersection pairing on $H^*(\operatorname{Sym}^d \Sigma, R)$. 
 
On the other hand,    the  gluing argument suggests that   
$$PFH^{sw}_{Z_{ref}}(X_-, \Omega_{X_-}) \circ \Theta \circ \mathcal{OC}(\underline{L}, H; R)  $$   is equivalent to counting holomorphic curves  in $( \dot{D}^0 \times \Sigma, \mathcal{L})$ passing through  marked points and $1$-cycles, where  $\dot{D}^0$  is  a disk with a strip-like positive end, and $\mathcal{L} \subset \partial \dot{D}^0$  is a suitable Lagrangian  such that $\mathcal{L} =\mathbb{R}_+ \times \left( \{0\}\times \varphi_H(\underline{L}) \cup \{1\} \times \underline{L} \right)$ over the end.  In other words, we should have 
 $$\langle \Theta, \mathfrak{J}^{-1} \circ \mathcal{OC}(\underline{L}; R) (\sigma)   \rangle = \epsilon_{\underline{L}}(\Theta \otimes  \sigma ), $$  
 where  $\epsilon_{ \underline{L}}$  is a  certain   augmentation-like  functional. Note that this form    is exactly the same as  Biran-Cornea's characterization  (\ref{eq105}), and  $ \mathcal{OC}(\underline{L}; R) $ is nontrivial if and only if $\epsilon_{ \underline{L}}$ is nontrivial.   In view of the tautological correspondence (cf. Section 3 of \cite{GHC2}),  $\epsilon_{ \underline{L}}$ should be  equivalent to  $\epsilon_{\operatorname{Sym}^d\underline{L}}$ under the isomorphism (\ref{eq95}).

Presuming  Conjecture  (\ref{eq103}),    we  have   
\begin{equation*}  
\langle \Theta, \mathfrak{J}^{-1} \circ  \mathcal{OC}(\underline{L}; R) (\sigma )   \rangle  = \langle \Theta, i_{\operatorname{Sym}^d\underline{L}}(\mathcal{J}^{-1}(\sigma))\rangle =\epsilon_{\operatorname{Sym}^d\underline{L}} (\Theta \circledast  \mathcal{J}^{-1}(\sigma)). 
\end{equation*}
This is consistent with the discussion in the last paragraph.

 \subsection{The class $\mathfrak{d}_R$ in torus case}\label{torus}
 For the case $\Sigma=\mathbb{T}^2$,  we conjecture that 
\begin{equation} \label{eq100}
\mathfrak{d}_R =\sum_{i=0}^{d-1}  c_i \left(\theta U^{d-1-i}  \mathfrak{e} \right) T^{-i}, 
\end{equation}
for some $c_i \ne 0$, 
where $\theta \in \Lambda^*H^1(\mathbb{T}^2, \mathbb{Q})$ is the generator with degree $2$,  and its action lowers the PFH grading by $2$.  
 
We explain the evidence for this conjecture.
 Suppose   the conjectural ring isomorphisms (\ref{eq95}) hold.  According to the  computations of  A. Bertram and M. Thaddeus (Corollary 7.4 of \cite{BT}),  one expects the  following   presentation: 
\begin{equation} \label{eq101}
\bigoplus_{k = \operatorname{gr}(\mathfrak{e})   \!  \!  \!  \!  \! \mod 2  } \widetilde{PFH}_{k}(\mathbb{T}^2,  d) \cong  \left(
\frac{\Lambda[U,\theta]}{\left(\theta^2,\,U^d-\theta U^{d-1}-S^{-1}\right)}\right)\mathfrak{e}.  
\end{equation}
In  Zhou's  unpublished work \cite{Zhou},  S. Zhou establishes this presentation,  without using Conjecture  (\ref{eq95}).  Presuming (\ref{eq101}) is true, then one can verify that 
$$ \ker\left(U^d-S^{-1}\right) = \bigoplus_{j=0}^{d-1}
\Lambda\cdot\theta U^j\mathfrak{e}. $$
By Theorem \ref{thm7},  each  $T^{-i}$ coefficient  $\mathfrak{d}_i$ of  $\mathfrak{d}_R$  belongs to  $ \ker\left(U^d-S^{-1}\right).$
Combining  this with  the grading information in Lemma \ref{lem23} and  Lemma \ref{lem24},  we obtain  
$$\mathfrak{d}_i =   c_i \theta U^{d-1-i}  \mathfrak{e}$$ 
for some $c_i \ne 0 \in \mathbb{Q}$,  which yields (\ref{eq100}).

 \begin{remark}
  Zhou  \cite{Zhou} and Bertram-Thaddeus \cite{BT} prove the presentation  (\ref{eq101}) with  $\mathbb{Q}$-coefficients.  We don't know whether   (\ref{eq101})  still holds   with $\mathbb{Z}/2\mathbb{Z}$-coefficients. 
  
  If    (\ref{eq100}) also is true for $\mathbb{Z}/2\mathbb{Z}$-coefficients, then $c_i=1$.  In particular,  $\mathfrak{d}_R$  is  independent of the admissible link.
\end{remark}
 
In higher genus, the analogous presentation and the available grading information are insufficient to determine $\mathfrak{d}_R$. Additional information about $\mathfrak{d}_R$ is therefore required. Nevertheless, we conjecture that $\mathfrak{d}$ and $\mathfrak{d}_R$ are independent of the admissible link in every genus.

	\appendix
	\renewcommand{\appendixname}{Appendix~\Appendix{section}}

\section{Spectral invariants under Novikov completion} \label{appendixA}
In this appendix, we clarify the relation between the spectral invariants defined here and  those  in \cite{GHC2}. 
\begin{lemma}\label{A1}
For any $H \in C^{\infty}(S^1\times \Sigma)$, we have 
\begin{equation*}
c_d^{pfh}(H, \mathfrak{v})^{o} =c_d^{pfh}(H, \iota^{pfh}_*(\mathfrak{v})) \mbox{ and } c_{\underline{L}}(H, \sigma)^{o} =c_{\underline{L}}(H,   \iota_{*}^{hf}(\sigma)),
\end{equation*}
where $\mathfrak{v} \in \widetilde{PFH}^{o}(\Sigma,   d)$ and $\sigma \in HF^{o}(\Sigma, \underline{L}).$  
\end{lemma}
\begin{proof}
We only  prove the QHF case,  because the proof for  the  PFH case is similar. 
By Hofer-Lipschitz continuity, it suffices to prove  the case in which  $\varphi_H$ is nondegenerate.

By (\ref{eq20}),  we have a cycle  $\mathbf{c}_{0} \in CF^o(\Sigma, \varphi_H(\underline{L}), \underline{L}) $   representing the class ($j_H^o)^{-1}(\sigma)$.  Then $\iota^{hf}( \mathbf{c}_0)$ is a cycle representing  $(\iota^{hf}_H)_*((j_H^o)^{-1}(\sigma)) = j_H^{-1}(\iota^{hf}_*(\sigma))$.  By definition, we have 
$$ c_{\underline{L}}(H, \iota^{hf}_*(\sigma))  \le \mathcal{A}_H(\mathbf{c}_0).$$
Taking the infimum over all $\mathbf c_0$  immediately gives
\begin{equation*}
 c_{\underline{L}}(H, \iota^{hf}_*(\sigma)) \le   c_{\underline{L}}(H,  \sigma)^{o}.
\end{equation*}

 For each intersection point $\mathbf{y}$, fix a relative homology class $A_{\mathbf{y}} \in H_2(E, \mathbf{y}, \mathbf{x}_H)$. Denote $\hat{\mathbf{y}}  :=(\mathbf{y}, [A_{\mathbf{y}}])$.  Then  we have the following isomorphisms 
\begin{equation*}
CF^o(\Sigma, \varphi_H(\underline{L}), \underline{L}) \cong \oplus_{\mathbf{y}} \mathbb{F}[T^{-1}, T] \hat{\mathbf{y}}  \mbox { and } CF(\Sigma, \varphi_H(\underline{L}), \underline{L}) \cong \oplus_{\mathbf{y}} R \hat{\mathbf{y}} 
\end{equation*}
defined by sending $(\mathbf{y}, [A_{\mathbf{y}} ]+ b\mathcal{B})$ to $T^b \hat{\mathbf{y}}$.

For an  integer $N$,   define the $N$-truncation of $R$ by 
$$R_{\ge N} :=\{r \in R: r=\sum_{k \ge N } a_kT^k\} \mbox{ and } R_{< N} :=\{r \in R: r=\sum_{k < N } a_kT^k\}.$$
So every $r\in R$ can be uniquely decomposed as $r_{\ge N} + r_{<N}$.  Similarly,  a chain $\mathbf{c} = \sum_{\mathbf{y}}  r_{\mathbf{y}} \hat{\mathbf{y}} \in CF(\Sigma, \varphi_H(\underline{L}), \underline{L}) $ can be uniquely decomposed as $\mathbf{c}_{\ge N} + \mathbf{c}_{<N}$,
where 
$$\mathbf{c}_{\ge N} := \sum_{\mathbf{y}}  (r_{\mathbf{y}})_{\ge N} \hat{\mathbf{y}} \mbox{ and } \mathbf{c}_{< N} := \sum_{\mathbf{y}}  (r_{\mathbf{y}})_{< N} \hat{\mathbf{y}}.  $$

 Let $\mathbf{c}  \in CF(\Sigma, \varphi_H(\underline{L}), \underline{L}) $ be a cycle representing the class $j_H^{-1}(\iota^{hf}_*(\sigma))$ with $\mathcal{A}_H(\mathbf{c} ) = c_{\underline{L}}(H, \iota_*^{hf}(\sigma))$.  Here we set  $\mathbf{c} =0$  if  $\iota^{hf}_*(\sigma) =0$. Therefore, there is a chain $\mathbf{b} \in  CF(\Sigma, \varphi_H(\underline{L}), \underline{L}) $ such that 
$$\mathbf{c} =\mathbf{c}_0 + d\mathbf{b}. $$
 Choose a sufficiently negative integer $N$  such that 
$ (\mathbf{c}_0)_{\ge N }  =\mathbf{c}_0  \mbox{ and } (\mathbf{c}_0)_{<N}=0. $
Then taking the $N$-truncations of the chains gives
\begin{equation}\label{eq81}
\mathbf{c}_{\ge N} =\mathbf{c}_0 +  (d\mathbf{b})_{\ge N} \mbox{ and }\mathbf{c}_{< N} = (d\mathbf{b})_{< N}.
\end{equation}
By the monotonicity (Lemma 2.7 of \cite{GHC2}),   $\langle d \hat{\mathbf{y}}, \hat{\mathbf{y}}' \rangle \in \mathbb{F}[T^{-1}, T]$. Therefore, there is a constant $N_0 \in \mathbb{Z}_{\ge 1}$  depending only on the complex $(CF(\Sigma, \varphi_H(\underline{L}), \underline{L}), d)$ such that $\langle d \hat{\mathbf{y}}, \hat{\mathbf{y}}' \rangle  \in R_{\ge -N_0} \cap R_{< N_0}$ for all $\hat{\mathbf{y}}, \hat{\mathbf{y}}'$.  
Note that for $r_i \in R_{\ge N_i} (i=1, 2)$ and $r_j \in R_{<N_j} (j=3, 4)$, we have $r_1r_2 \in R_{\ge  N_1+N_2}$ and $r_3r_4 \in R_{<N_3+N_4}$. 
Hence, 
\begin{equation}\label{eq82}
 (d\mathbf{b})_{\ge N}  =d (\mathbf{b}_{\ge N +N_0}) + (d\mathbf{b}_{< N +N_0})_{\ge N}  .
\end{equation}
By (\ref{eq81}) and (\ref{eq82}),  
\begin{equation*} 
\mathbf{c}_{\ge N}  - (d \mathbf{b}_{< N +N_0})_{\ge N} =\mathbf{c}_0 + d ( \mathbf{b}_{\ge N +N_0}) \in CF^o(\Sigma, \varphi_H(\underline{L}), \underline{L})  
\end{equation*}
is a cycle representing $\sigma$. Note that the $R$-coefficients in $(d \mathbf{b}_{< N +N_0})_{\ge N} $ are in $R_{\ge N} \cap R_{<N+2N_0}$. 

 Because 
$$\mathcal{A}_H( \sum_{\mathbf{y}} ( r_{\mathbf{y}})_{< N+2N_0} \hat{\mathbf{y}})=\max_{\mathbf{y}}\{\mathcal{A}_H( \hat{\mathbf{y}}) + | (r_{\mathbf{y}})_{<N+2N_0}|  \lambda \}< \max_{\mathbf{y}} {    \mathcal{A}_H (\hat{\mathbf{y}})} + (N +2N_0)\lambda,$$
we have  $\lim_{N \to -\infty }\mathcal{A}_H((d \mathbf{b}_{< N +N_0})_{\ge N} ) =\lim_{N\to -\infty} \mathcal{A}_H(\mathbf{c}_{<N}) =-\infty$.

 Now assume  $\mathbf{c} \ne 0$.  Then  for sufficiently negative $N$,  we have   $$\mathcal{A}_H(\mathbf{c}_{\ge N}  - (d \mathbf{b}_{< N +N_0})_{\ge N} ) =\mathcal{A}_H(\mathbf{c}_{\ge N}) =\mathcal{A}_H(\mathbf{c})  = c_{\underline{L}}(H, \iota^{hf}_*(\sigma)).$$
This implies that 
\begin{equation*}
 c_{\underline{L}}(H, \iota^{hf}_*(\sigma)) =\mathcal{A}_H(\mathbf{c}_{\ge N}  - (d \mathbf{b}_{< N +N_0})_{\ge N} )  \ge  c_{\underline{L}}(H,  \sigma)^{o}.
\end{equation*}

If $\mathbf{c} =0$, i.e.,  $\iota^{hf}_*(\sigma) =0,$ then
\begin{equation*}
c_{\underline{L}}(H,  \sigma)^{o} \le \mathcal{A}_H(  (d \mathbf{b}_{< N +N_0})_{\ge N} )      \to -\infty 
\end{equation*}
as $N \to -\infty$. Thus, $c_{\underline{L}}(H,  \sigma)^{o}  =-\infty = c_{\underline{L}}(H, \iota^{hf}_*(\sigma))$.
\end{proof}

\section{Morphisms and coefficient extension} \label{appendixB}
This appendix clarifies the relationship between the formulations   of the open-closed morphisms in Theorem \ref{thm2}, Theorem \ref{thm7} and those in \cite{BC}.  
 
We use the terminology of filtered Floer--Novikov complexes from Section 2 of 
\cite{U}.  Given a graded filtered Floer-Novikov complex $\mathfrak{c}$,  then we have a   principal $\Gamma$-bundle,    
$$ \pi:P \longrightarrow S,  $$
 where $\Gamma$ is an  abelian deck-transformation
group.  The finite set  $S$ consists  of unanchored generators,
whereas a point of $P$ is an anchored generator.  The associated Novikov
ring is the appropriate completion $R_{\Gamma}$ of the group ring
$\mathbb{F}[\Gamma]$.   $C_*(\mathfrak{c})$ is generated by elements of  $P$ satisfying a certain finiteness condition.  The differential is 
$$\partial p =\sum_{q \in P } n(p, q) q. $$
Schematically, $P$ consists of anchored PFH generators and $S$ consists of   PFH generators. $\Gamma =\mathbb{Z}[\Sigma]$ and the  $\Gamma$-action is adding $[\Sigma]$   to the cappings. Also, $R_{\Gamma}=\Lambda$ and   $C_*(\mathfrak{c}) = \widetilde{PFC}_*(\Sigma, \varphi_H, d).$ 

Let  $\mathfrak{c}_{\pm} $ be two graded filtered Floer-Novikov complexes.   Assume $\Gamma_{+}=\Gamma_- = \Gamma$. Suppose that we have a function  $n_F:P_+\times P_-\longrightarrow\mathbb{F}$ satisfying 
\begin{itemize} 
\item
For each $p_+$, there are only finitely many $p_-\in P_-$ such that $ n_F(p_+, p_-)\ne 0  $.
\item
For $p_{\pm} \in P_{\pm}$, we have 
$$\sum_{p \in P_+}   n_+(p_+,  p) n_F(p, p_-)  -  \sum_{q \in P_-}n_F(p_+, q) n_-(q, p_-)  =0. $$ 
\item
$n_F(g\cdot p,g\cdot q)=n_F(p,q)$ for $g \in \Gamma$. 
\end{itemize}
Then we can define an  $R_{\Gamma}$-equivariant chain map
\begin{equation} \label{eq108}
  F(p_+):=\sum_{p_- \in P_-}n_F(p_+,p_-)p_-. 
  \end{equation}

There is an alternative way to construct  a complex using the data in  filtered Floer--Novikov complexes $\mathfrak{c}_{\pm}$ as follows.   Fix a lifting $p_{x_{\pm}}$ for each  $x_{\pm}$. Define complexes 
$$\overline{C}_*(\mathfrak{c}_{\pm}) =\bigoplus_{x_{\pm} \in S_{\pm}}R_{\Gamma} x_{\pm} \mbox{ with differential } \partial_{\pm} x_{\pm} =\sum_{y_{\pm} \in S_{\pm}}  \left( \sum_{g \in \Gamma_{\pm}}n_{\pm}(p_{x_{\pm}}, g\cdot p_{y_{\pm}}) g  \right) y_{\pm}.$$
Similarly, define 
\begin{equation} \label{eq107}
\overline{F}(x_+):=\sum_{x_- \in S_-} \left( \sum_{g \in \Gamma} n_F(p_{x_+},g\cdot p_{x_-}) g \right)x_-. 
\end{equation}
 The quantum homologies in \cite{BC} are defined by the above approach. 

Note that we have chain-complex isomorphisms
$$\Phi_{\pm}:  \overline{C}_*(\mathfrak{c}_{\pm}) \to {C}_*(\mathfrak{c}_{\pm}),   \ \ \ g x_{\pm} \to  g\cdot p_{x_{\pm}}, \  \ \ g\in \Gamma.  $$ It is easy to  verify that   $\Phi_-^{-1} \circ F \circ \Phi_{+} =\overline{F}. $ Thus, these two constructions are equivalent.

Now consider the situation that $\Gamma_-$ is a proper subgroup of $\Gamma_+$.  Then  $\Lambda: =R_{\Gamma_-} $ is a subfiled of $R:=R_{\Gamma_+}$.  If the function $n_F$ is $\Gamma_-$-invariant,  by regarding $C_*(\mathfrak{c}_{\pm})$ as  $\Lambda$-modules, then the map $F$ (\ref{eq108})  still is  well defined and it is $\Lambda$-equivariant.   As an example,   when   $C_*(\mathfrak{c}_+)$ is the QHF complex and  $C_*(\mathfrak{c}_-)$ is the PFH complex, 
then the counting  for  defining the open-closed morphisms 
$$n_{\mathcal{OC}}((\mathbf{y}, [A]), (\alpha, [Z])): = \sum_{\mathcal{Z} =A\#\mathcal{Z}_{ref} \#(-Z), I(\mathcal{Z} ) =0} \#\mathcal{M}^J(\mathbf{y}, \alpha,  \mathcal{Z})$$
is an example satisfying the above condition. The finiteness condition follows from the monotonicity \ref{O3}. Thus, $F= \widetilde{\mathcal{OC}}(\underline{L}, H)_J$ is this case. 

 However, $\overline{F}$ is not a  well-defined chain map, because the summation in (\ref{eq107}) only runs  over $\Gamma_-$, and misses many possible nontrivial values of  $n_F$. 
A  natural   solution is to replace  the summation in (\ref{eq107}) such that it  runs over $\Gamma_+$.   To make this possible, we need to tensor $\overline{C}_*(\mathfrak{c}_{-}) $ with $R$.  Then, we define 
\begin{equation*}
\begin{split}
\overline{F}& : \overline{C}_*(\mathfrak{c}_{+}) \to \overline{C}_*(\mathfrak{c}_{-}) \otimes_{\Lambda}R  \\
&x_+ \to  \sum_{x_- \in S_-} \left( \sum_{g \in \Gamma_+}  n_F(g^{-1} \cdot p_{x_+}, p_{x_-}) g \right)x_-. 
\end{split}
\end{equation*}  
One can check that it is  an $R$-equivariant chain map.  
This has the same algebraic form as the Novikov-weighted construction of the quantum inclusion in \cite{BC}.

Then we obtain an $R$-extension of $F$ by defining 
\begin{equation*}
\begin{split}
F_R & :  {C}_*(\mathfrak{c}_{+}) \to {C}_*(\mathfrak{c}_{-}) \otimes_{\Lambda}R,  \  \  F_R : = (\Phi_- \otimes_{\Lambda} \operatorname{Id}_{R})  \circ \overline{F} \circ \Phi_+^{-1}. 
\end{split}
\end{equation*}
Because $(\Phi_- \otimes_{\Lambda} \operatorname{Id}_{R})$, $\overline{F} $ and $\Phi_+^{-1}$ are $R$-equivariant, so is $F_R$. 
Let $$\Pi: R \to \Lambda, \  \ \sum_{g\in \Gamma_+} a_g g \to \sum_{g\in \Gamma_-} a_g g $$ denote the  natural  projection. 
Then $F= (\operatorname{Id} \otimes \Pi ) \circ F_R$. 

Suppose the index $[\Gamma_+: \Gamma_-]=k$ is finite. Choose representatives  $\{g_i\}_{i=1}^k$  of the cosets in $\Gamma_+/\Gamma_-$ with $g_1=1$.  Then $R=\oplus_{i=1}^k \Lambda g_i$ and   $\Pi(g_ig_j^{-1}) = \delta_{ij}$. It is straightforward to verify that 
$$F_R(p)= \sum_{i=1}^k F(g_i\cdot p) \otimes g_i^{-1}. $$
This leads to our definition of  $\mathcal{OC}(\underline{L}; R)$  in (\ref{eq106}).

Shenzhen University

\verb| Email address: ghchen@szu.edu.cn |	
\end{document}